\documentclass[11pt,a4paper,reqno]{amsart}

\pdfoutput=1

\usepackage[utf8]{inputenc}
\usepackage{epsfig,amsmath,latexsym,amssymb}
\usepackage{amsthm}
\usepackage{bm}
\usepackage[numbers,sort&compress,longnamesfirst]{natbib}
\usepackage{graphicx}
\usepackage{url}
\usepackage[lmargin=1in, rmargin=1in, tmargin=1in, bmargin=0.9in]{geometry}
\usepackage{color}
\usepackage[colorlinks,citecolor=blue,urlcolor=blue,filecolor=blue]{hyperref}
\usepackage{multirow}
\usepackage{enumitem}
\usepackage{lscape}
\usepackage{dsfont}
\usepackage[ruled,vlined]{algorithm2e}
\usepackage{algorithmic}
\usepackage{placeins}
\usepackage{xcolor}

\usepackage{times}

\DeclareFontFamily{U}{mathx}{\hyphenchar\font45}
\DeclareFontShape{U}{mathx}{m}{n}{
      <5> <6> <7> <8> <9> <10>
      <10.95> <12> <14.4> <17.28> <20.74> <24.88>
      mathx10
      }{}
\DeclareSymbolFont{mathx}{U}{mathx}{m}{n}
\DeclareMathSymbol{\bigtimes}{1}{mathx}{"91}

\newcommand{\overbar}[1]{\mkern 6.5mu\overline{\mkern-5.5mu#1\mkern-2.5mu}\mkern 2.5mu}

\newcommand{\overtilde}[1]{\mkern 6.5mu\widetilde{\mkern-5.5mu#1\mkern-2.5mu}\mkern 2.5mu}

\usepackage{amsfonts,bbm}
\newcommand{\Td}{\mathrm{d}}

\newcommand{\TL}{\mathrm{L}}

\newcommand{\TR}{\mathrm{R}}

\newcommand{\BIb}{{\boldsymbol{b}}}

\newcommand{\BIe}{{\boldsymbol{e}}}

\newcommand{\BIg}{{\boldsymbol{g}}}

\newcommand{\BIj}{{\boldsymbol{j}}}
\newcommand{\BIk}{{\boldsymbol{k}}}
\newcommand{\BIl}{{\boldsymbol{l}}}

\newcommand{\BIp}{{\boldsymbol{p}}}
\newcommand{\BIq}{{\boldsymbol{q}}}

\newcommand{\BIs}{{\boldsymbol{s}}}

\newcommand{\BIu}{{\boldsymbol{u}}}
\newcommand{\BIv}{{\boldsymbol{v}}}
\newcommand{\BIw}{{\boldsymbol{w}}}
\newcommand{\BIx}{{\boldsymbol{x}}}
\newcommand{\BIy}{{\boldsymbol{y}}}
\newcommand{\BIz}{{\boldsymbol{z}}}

\newcommand{\BIJ}{{\boldsymbol{J}}}
\newcommand{\BIK}{{\boldsymbol{K}}}
\newcommand{\BIL}{{\boldsymbol{L}}}

\newcommand{\BIT}{{\boldsymbol{T}}}

\newcommand{\CA}{\mathcal{A}}
\newcommand{\CB}{\mathcal{B}}
\newcommand{\CC}{\mathcal{C}}
\newcommand{\CCCD}{\mathcal{D}}
\newcommand{\CE}{\mathcal{E}}
\newcommand{\CF}{\mathcal{F}}
\newcommand{\CG}{\mathcal{G}}
\newcommand{\CH}{\mathcal{H}}
\newcommand{\CI}{\mathcal{I}}

\newcommand{\CL}{\mathcal{L}}

\newcommand{\CP}{\mathcal{P}}
\newcommand{\CQ}{\mathcal{Q}}

\newcommand{\CV}{\mathcal{V}}

\newcommand{\CX}{\mathcal{X}}
\newcommand{\CY}{\mathcal{Y}}
\newcommand{\CZ}{\mathcal{Z}}

\newcommand{\BCX}{{\boldsymbol{\mathcal{X}}}}

\newcommand{\FC}{\mathfrak{C}}

\newcommand{\FF}{\mathfrak{F}}

\newcommand{\FL}{\mathfrak{L}}

\newcommand{\SFB}{\mathsf{B}}

\newcommand{\SFL}{\mathsf{L}}

\newcommand{\SFO}{\mathsf{O}}

\newcommand{\SFT}{\mathsf{T}}
\newcommand{\SFU}{\mathsf{U}}

\newcommand{\Bzeta}{{\boldsymbol{\zeta}}}

\newcommand{\Biota}{{\boldsymbol{\iota}}}

\newcommand{\Blambda}{{\boldsymbol{\lambda}}}

\newcommand{\Bsigma}{{\boldsymbol{\sigma}}}

\newcommand{\Bphi}{{\boldsymbol{\phi}}}

\newcommand{\R}{\mathbb{R}} % real number
\newcommand{\Q}{\mathbb{Q}} % rational number
\newcommand{\N}{\mathbb{N}} % natural number
\newcommand{\vecone}{\mathbf{1}} % vector of ones
\newcommand{\veczero}{\mathbf{0}} % vector of zeros
\newcommand{\vecinfty}{\boldsymbol{\infty}} % vector of infinities
\newcommand{\INDI}{\mathbbm{1}} % indicator function
\newcommand{\TRANSP}{\mathsf{T}} % transpose
\newcommand{\DIFF}{\Td} % differentiation sign
\newcommand{\DIFFX}[1]{\,\Td{#1}} % differentiation with a measure or a variable name
\newcommand{\DIFFM}[2]{\,{#1}({#2})}% % differentiation with a measure and a variable
\newcommand{\lspan}{\mathrm{span}} % span
\newcommand{\PROB}{\mathbb{P}} % probability measure
\newcommand{\EXP}{\mathbb{E}} % expectation
\newcommand{\clos}{\mathrm{cl}} % closure
\newcommand{\inter}{\mathrm{int}} % interior
\newcommand{\relint}{\mathrm{relint}} % relative interior
\newcommand{\relbd}{\mathrm{relbd}} % relative boundary
\newcommand{\support}{\mathrm{supp}} % support (of measure)
\newcommand{\conv}{\mathrm{conv}} % convex hull
\newcommand{\aff}{\mathrm{aff}} % affine hull
\newcommand{\cone}{\mathrm{cone}} % conic hull
\newcommand{\specialoverline}[1]{\makebox[0pt]{$\phantom{#1}\overline{\phantom{#1}}$}#1}

\DeclareMathOperator*{\argmin}{arg\,min} % argument that minimizes
\DeclareMathOperator*{\argmax}{arg\,max} % argument that maximizes
\DeclareMathOperator*{\minimize}{\mathrm{minimize}} % minimize, for optimization problems
\DeclareMathOperator*{\maximize}{\mathrm{maximize}} % maximize, for optimization problems

\newtheorem{theorem}{Theorem}[section]
\newtheorem{corollary}[theorem]{Corollary}
\newtheorem{lemma}[theorem]{Lemma}
\newtheorem{definition}[theorem]{Definition}
\newtheorem{remark}[theorem]{Remark}
\newtheorem{proposition}[theorem]{Proposition}

\newtheorem{setting}[theorem]{Setting}
\newtheorem{assumption}[theorem]{Assumption}

\numberwithin{equation}{section}

\numberwithin{figure}{section}

\numberwithin{table}{section}

\DontPrintSemicolon

\begin{document}

\author[A. Neufeld]{Ariel Neufeld}
\author[Q. Xiang]{Qikun Xiang}

\address{Division of Mathematical Sciences, Nanyang Technological University, 21 Nanyang Link, 637371 Singapore}
\email{ariel.neufeld@ntu.edu.sg}

\address{Division of Mathematical Sciences, Nanyang Technological University, 21 Nanyang Link, 637371 Singapore}
\email{qikun.xiang@ntu.edu.sg}

\date{}

\title[Numerical method for multi-marginal optimal transport beyond discrete measures]{Numerical method for feasible and approximately \\ optimal solutions of multi-marginal optimal transport \\ beyond discrete measures}

\thanks{To appear in Foundations of Computational Mathematics.}

\thanks{AN and QX gratefully acknowledge the financial support by the MOE AcRF Tier 2 Grant \textit{MOE-T2EP20222-0013}.}



\begin{abstract}
We propose a numerical algorithm for the computation of multi-marginal optimal transport (MMOT) problems involving general probability measures that are not necessarily discrete. 
By developing a relaxation scheme in which marginal constraints are replaced by finitely many linear constraints and by proving a specifically tailored duality result for this setting, we approximate the MMOT problem by a linear semi-infinite optimization problem. 
Moreover, we are able to recover a feasible and approximately optimal solution of the MMOT problem, and its sub-optimality can be controlled to be arbitrarily close to 0 under mild conditions. 
The developed approximation scheme leads to a numerical algorithm which can compute a feasible and approximately optimal solution of the MMOT problem with arbitrarily small sub-optimality. 
Besides the approximately optimal solution, the algorithm also computes upper and lower bounds for the optimal value of the MMOT problem. 
The difference between the computed bounds provides an explicit sub-optimality bound for the computed approximately optimal solution. 
We demonstrate the proposed algorithm in three numerical experiments involving an MMOT problem that stems from fluid dynamics, the Wasserstein barycenter problem, and a large-scale MMOT problem with 100 marginals.
We observe that our algorithm is capable of computing high-quality solutions of these MMOT problems and the computed sub-optimality bounds are much less conservative than their theoretical upper bounds in all the experiments.  
Moreover, we compare our algorithm with existing regularization-based algorithms to showcase its advantages.
\end{abstract}

\maketitle

\section{Introduction}\label{sec:introduction}

In this paper, we develop a numerical method for the computation of multi-marginal optimal transport (MMOT) problems involving general probability measures that are not necessarily discrete. 
Given \mbox{$N\in\N$} Borel probability measures $\mu_1,\ldots,\mu_N$ on Polish spaces $\CX_1,\ldots,\CX_N$ and a cost function \sloppy{\mbox{$f:\CX_1\times\cdots\times\CX_N\to\R$}}, 
we solve the following minimization problem:
\begin{align}
\inf_{\mu\in\Gamma(\mu_1,\ldots,\mu_N)}\bigg\{\int_{\CX_1\times\cdots\times\CX_N}f(x_1,\ldots,x_N)\DIFFM{\mu}{\DIFF x_1,\ldots,\DIFF x_N}\bigg\},
\label{eqn:mmot}
\tag{$\mathsf{OT}$}
\end{align}
where $\Gamma(\mu_1,\ldots,\mu_N)$ denotes the set of couplings of $\mu_1,\ldots,\mu_N$ that is defined below.

\begin{definition}[Coupling]\label{def:coupling}
For $k\in\N$ Polish spaces $\CY_1,\ldots,\CY_k$ and $k$ probability measures \sloppy{${\nu_1\in\CP(\CY_1)},\ldots,\nu_k\in\CP(\CY_k)$}, let $\Gamma(\nu_1,\ldots,\nu_k)$ denote the set of couplings of $\nu_1,\ldots,\nu_k$, defined as 
\begin{align*}
\Gamma(\nu_1,\ldots,\nu_k):=\Big\{\gamma\in\CP(\CY_1\times\cdots\times \CY_k):\text{the marginal of }\gamma\text{ on }\CY_j\text{ is }\nu_j\text{ for }j=1,\ldots,k\Big\}.%
\end{align*}%
\end{definition}%

This is an extension of the classical two-marginal (i.e., $N=2$) optimal transport problem of Monge and Kantorovich, which has been thoroughly studied in the literature; see, e.g., \citep{villani2003topics, villani2008optimal, mccann2014glimpse, peyre2019computational} as well as the recent survey by \citet{benamou2021optimal} and the references therein for applications of optimal transport and discussions about its computation. 
For various theoretical results in the general multi-marginal case (i.e., $N>2$), we refer the reader to the duality results by \citet{kellerer1984duality}, the results on the conditions for the existence of {Monge-type} optimal solutions by \citet{pass2024general}, as well as the survey by \mbox{\citet{pass2015multi}} and the references therein.

The MMOT problem serves as the basis of several related problems such as the Wasserstein barycenter problem \citep{agueh2011barycenters} and the martingale optimal transport problem \citep{beiglbock2013model}. 
The original MMOT problem and its various extensions have many theoretical and practical applications, including but not limited to:
theoretical economics \citep{chiappori2010hedonic, carlier2010matching, galichon2016optimal}, 
density functional theory (DFT) in quantum mechanics \citep{buttazzo2012optimal, cotar2013density, chen2014numerical, cotar2018smoothing, friesecke2022genetic, hu2023global},
computational fluid mechanics \citep{benamou2000computational, benamou2019generalized},
mathematical finance \citep{beiglbock2013model, MOTduality, DonlinskySoner, eckstein2019robust, aquino2019bounds, guyon2024dispersion, henry2019martingale}, 
robust risk management \citep{eckstein2018robust, ennaji2024robust},
statistics \citep{srivastava2015wasp, srivastava2018scalable},
machine learning \citep{rabin2011wasserstein},
tomographic image reconstruction~\citep{abraham2017tomographic},
signal processing \citep{elvander2020multi, haasler2021multi}, 
and operations research \citep{chen2018distributionally, gao2017distributionally, gao2017data}.

There exists a vast literature on the computational aspect of MMOT and related problems. 
Many such studies focus on the case where the marginals are discrete measures with finite support \citep{benamou2015iterative, altschuler2023polynomial, tupitsa2020multimarginal, puccetti2020on}, or where non-discrete marginals are replaced by their discrete approximations \citep{carlier2015numerical, guo2019computational, eckstein2019robust, friesecke2022genetic, hu2023global}. 
When all marginals have finite support, the MMOT problem corresponds to a linear programming problem typically involving a large number of decision variables. 
Moreover, when replacing non-discrete marginals with their discrete approximations, the optimal solutions of the linear programming problem and its dual are not feasible solutions of the original MMOT problem and its dual without discretization. 
This is a crucial shortcoming of discretization-based methods since one is only able to obtain an \textit{infeasible} solution of the MMOT problem that approximates the actual optimal solution, and the approximation error can only be controlled using theoretical estimates that are typically over-conservative in practice.

Instead of discretizing the marginals, \citet{alfonsi2021approximation} and \citet{alfonsi2022constrained} have explored an alternative approximation scheme of MMOT and related problems via relaxing the marginal constraints into a finite collection of linear constraints with respect to test functions. 
It is shown that there exists a discrete probability measure with sparse support that optimizes the relaxed MMOT problem.
Despite that \citet{alfonsi2021approximation} have established the convergence of the optimal solutions of the relaxed MMOT problems to a true optimal solution of the MMOT problem in the limit, a discrete optimal solution of any particular relaxation remains \textit{infeasible} for the MMOT problem.
Moreover, the numerical algorithms in \citep{alfonsi2021approximation} and \citep{alfonsi2022constrained} approximate a discrete optimal solution of the relaxed MMOT problem by optimizing over the positions and probabilities of a finite collection of atoms (i.e., Dirac measures), which corresponds to solving a non-convex optimization problem. 
Hence, there is no guarantee that these algorithms will converge to a globally optimal solution of the relaxed MMOT problem, and it is difficult to quantify and control the approximation error.
\citet{zhou2024efficient} developed an algorithm based on dual gradient ascent for computing a lower bound for MMOT problems with non-discrete marginals when the cost function has a pair-wise structure.
However, the difference between the computed lower bound and the optimal value of the MMOT problem is non-trivial and hard to control unless the cost function has a tree structure.

There are numerical methods for solving MMOT and related problems which use regularization to speed up the computation. 
Most notably, \citet{cuturi2013sinkhorn} proposed to use entropic regularization and the Sinkhorn algorithm for solving the classical optimal transport problem (i.e., when $N=2$). 
See also \citep{nutz2021entropic, eckstein2021quantitative} for the theoretical properties of entropic regularization and the Sinkhorn algorithm. 
While most regularization-based methods deal with discrete marginals, see, e.g., \citep{benamou2019generalized, peyre2019computational, tupitsa2020multimarginal, fan2024parallel}, there are also regularization-based methods for solving MMOT problems with non-discrete marginals. 
These methods involve approximating an infinite-dimensional optimization problem by finite-dimensional parametrizations such as
with reproducing kernel Hilbert spaces, see, e.g., \citep{genevay2016stochastic},
and with deep neural networks, see, e.g., \citep{eckstein2018robust, eckstein2019computation, eckstein2019robust, aquino2019bounds, aquino2020minmax, henry2019martingale, cohen2020estimating}. 
A limitation of regularization-based methods is that the regularization term introduces a bias (see, e.g., \citep[Section~3.3]{benamou2021optimal}) which only goes to 0 asymptotically when the regularization parameter goes to~0. 
When the regularization parameter is close to~0, numerical instability may arise and complicate the computational procedure; see \citep[Section~4.4]{peyre2019computational} for a detailed discussion of this issue in the Sinkhorn algorithm and some remedies. 
Moreover, using deep neural networks to parametrize infinite-dimensional decision variables incurs practical challenges when training the neural networks due to the non-convexity of the objective function, and thus the trained neural networks may represent optimal solutions of the MMOT problem poorly.

In this paper, we tackle the MMOT problem \eqref{eqn:mmot} in its original form without discretization or regularization, and we develop a numerical algorithm that is able to compute \textit{feasible} and \textit{approximately optimal} solutions of \eqref{eqn:mmot}. 
Let us present an overview of our numerical method.
Our method relaxes the marginal constraints in \eqref{eqn:mmot} by linear constraints with respect to a \textit{finite} collection of properly chosen \textit{test functions}, that is, for $i=1,\ldots,N$, rather than requiring the marginal of $\mu$ on $\CX_i$ to be equal to $\mu_i$, we allow the marginal of $\mu$ on $\CX_i$ to be any $\nu_i\in\CP(\CX_i)$ that satisfies the linear constraints $\int_{\CX_i}g\DIFFX{\nu_i}=\int_{\CX_i}g\DIFFX{\mu}_i$ $\forall g\in\CG_i$
with respect to a finite collection $\CG_i$ of test functions on~$\CX_i$.
We call $[\mu_i]_{\CG_i}:=\big\{\nu_i\in\CP(\CX_i):\int_{\CX_i}g\DIFFX{\nu_i}=\int_{\CX_i}g\DIFFX{\mu_i}\;\forall g\in\CG_i\big\}$ the \textit{moment set} centered at $\mu_i$ characterized by the test functions~$\CG_i$. 
The dual of this relaxed MMOT problem corresponds to a linear semi-infinite programming problem which is easier to tackle numerically compared to the original MMOT problem.
Approximately solving this relaxation of \eqref{eqn:mmot} then results in a lower bound for the optimal value of \eqref{eqn:mmot} as well as an infeasible solution $\hat{\mu}\notin\Gamma(\mu_1,\ldots,\mu_N)$.
Subsequently, we recover a feasible solution $\tilde{\mu}$ of \eqref{eqn:mmot} via an operation that we call \textit{reassembly} derived by applying the ``gluing'' operation to the infeasible solution $\hat{\mu}\in\CP(\CX_1\times\cdots\times\CX_N)$ and the pair-wise optimal couplings between the marginals $\hat{\mu}_1,\ldots,\hat{\mu}_N$ of $\hat{\mu}$ and the target marginals $\mu_1,\ldots,\mu_N$ that realize the \mbox{1-Wasserstein} metrics $W_1(\hat{\mu}_1,\mu_1),\ldots,W_1(\hat{\mu}_N,\mu_N)$. 
The constructed feasible solution $\tilde{\mu}$ provides an upper bound for the optimal value of \eqref{eqn:mmot} and the sub-optimality of $\tilde{\mu}$ can be controlled by 
the supremum \mbox{1-Wasserstein} metric between $\mu_i$ and any $\nu_i\in[\mu_i]_{\CG_i}$.
Through thoughtful designs of the test functions $\CG_1,\ldots,\CG_N$, we are able to control the sub-optimality of $\tilde{\mu}$ to be arbitrarily close to~0.
Below is a summary of our main contributions.%
\begin{enumerate}[label=(\arabic*), topsep=0pt]%
\item We develop a relaxation scheme \eqref{eqn:mmotlb} for \eqref{eqn:mmot}, and we introduce the notion of \textit{reassembly} (Definition~\ref{def:reassembly}) for constructing a feasible solution of \eqref{eqn:mmot} when the marginals $\mu_1,\ldots,\mu_N$ are non-discrete (see Theorem~\ref{thm:lowerbound}). 
Moreover, we show that an optimal solution of \eqref{eqn:mmot} can be obtained in the limit if the sub-optimality of the constructed feasible solution can be controlled to shrink to~0 (see Theorem~\ref{thm:lowerboundconverge}).

\item We provide duality results tailored to \eqref{eqn:mmotlb} and its dual \eqref{eqn:mmotlb-dual-lsip} (see Theorem~\ref{thm:duality}). 
Moreover, we analyze the support sparsity of (approximately) optimal solutions of \eqref{eqn:mmotlb} (see Proposition~\ref{prop:mmotlb-sparsity}) as well as the computational complexity of \eqref{eqn:mmotlb-dual-lsip} (see Theorem~\ref{thm:mmot-complexity}).

\item In the compact Euclidean case, i.e., when $\CX_1,\ldots,\CX_N$ are compact subsets of Euclidean spaces, 
we control the sub-optimality of the constructed feasible solution of \eqref{eqn:mmot} to be arbitrarily close to~0
by explicitly constructing moment sets on $\CX_1,\ldots,\CX_N$ whose supremum
\mbox{1-Wasserstein} metrics to $\mu_1,\ldots,\mu_N$ can be controlled to be arbitrarily close to~0, 
and we quantify the number of test functions needed to achieve $\epsilon$-optimality for any $\epsilon>0$ (see Setting~\ref{sett:compact} and Proposition~\ref{prop:momentset-errorcontrol}). 
More generally, 
we extend these results to the unbounded Euclidean case, where
we explicitly construct finitely many test functions such that the supremum \mbox{$p$-Wasserstein} metric (for $p\in[1,\infty)$) within the resulting moment set can be controlled to be arbitrarily close to~0 (see Proposition~\ref{prop:cover-cube} and Theorem~\ref{thm:momentset-euclidean}). 
This yields a non-asymptotic analog of the results of \citet{breeden1978prices} and \citet{talponen2014note} for the model risk involved in pricing financial derivatives given market-implied information (see Corollary~\ref{cor:breedenlitzenberger}). 
These extended results are of independent interest and are presented in Section~\ref{sec:momentset-extended}.

\item We develop a numerical algorithm (Algorithm~\ref{alg:cp-mmot}) for solving the linear semi-infinite optimization problem \eqref{eqn:mmotlb-dual-lsip}. 
Subsequently, we develop a numerical algorithm (Algorithm~\ref{alg:mmot}) which, for any given $\epsilon>0$, is not only capable of computing a \textit{feasible} $\epsilon$-optimal solution of \eqref{eqn:mmot} but also capable of computing upper and lower bounds for its optimal value such that the difference between the bounds is at most $\epsilon$ (see Theorem~\ref{thm:mmotalgo}). 

\item We perform three numerical experiments involving problem instances that include an MMOT problem that stems from fluid dynamics, the Wasserstein barycenter problem, and a large-scale MMOT problem with $N=100$ marginals,
and we compare the proposed algorithm with two regularization-based algorithms by \citet*{genevay2016stochastic} and \citet*{eckstein2019computation}.
The numerical results show that the proposed algorithm outperforms the two existing algorithms as it 
produces feasible high-quality solutions
and computes sub-optimality estimates that are much less conservative than purely theoretically derived sub-optimality estimates.
\end{enumerate}%

The rest of this paper is organized as follows. 
Section~\ref{ssec:settings} introduces the notations in the paper and the settings of the MMOT problem. 
In Section~\ref{ssec:relaxation}, we develop the notions of \textit{moment set} and \textit{reassembly}, and introduce the approximation scheme for \eqref{eqn:mmot}. 
Section~\ref{ssec:duality} presents the duality results specifically tailored to the relaxed MMOT problem, while Section~\ref{ssec:mmot-complexity} presents an analysis of the computational complexity of the resulting linear semi-infinite optimization problem.
In Section~\ref{ssec:reassemblysemidiscrete}, we develop results on the explicit construction of \textit{reassembly}.
In Section~\ref{ssec:momentsetdd}, we explicitly construct test functions in the compact Euclidean case to control the supremum \mbox{1-Wasserstein} metric within the resulting \textit{moment set}.
Section~\ref{sec:numerics} presents our numerical algorithm for approximately solving \eqref{eqn:mmot} and its theoretical properties.
Section~\ref{sec:experiments} showcases the performance of our numerical algorithm in three numerical experiments.
Our extended results on controlling the supremum $p$-Wasserstein metric within a moment set are presented in Section~\ref{sec:momentset-extended}.
Lastly, Section~\ref{sec:proof-theory} contains the proofs of the theoretical results in this paper.

\section{Approximation of multi-marginal optimal transport}\label{sec:theory}

\subsection{Settings and notations}\label{ssec:settings}
Throughout this paper, all vectors are assumed to be column vectors. 
We denote vectors and vector-valued functions by boldface symbols.
In particular, for $n\in\N$, we denote by $\veczero_n$ the vector in $\R^n$ with all entries equal to zero, i.e., $\veczero_n:=(\underbrace{0,\ldots,0}_{n\text{ times}})^\TRANSP$, and we denote by $\vecone_n$ the vector in $\R^n$ with all entries equal to one, i.e., $\vecone_n:=(\underbrace{1,\ldots,1}_{n\text{ times}})^\TRANSP$. We also use $\veczero$ and $\vecone$ when the dimension is unambiguous.
We denote by $\langle\,\cdot\,,\cdot\,\rangle$ the Euclidean dot product, i.e., $\langle\BIx,\BIy\rangle:=\BIx^\TRANSP\BIy$, and we denote by $\|\cdot\|_p$ the $p$-norm of a vector for $p\in[1,\infty]$.
Moreover, we denote $a\vee b:=\max\{a,b\}$, $a\wedge b:=\min\{a,b\}$, and $(a)^+:=\max\{a,0\}$ for all $a,b\in\R$.
Furthermore, we let $\lceil\,\cdot\,\rceil$ denote the ceiling function, i.e., for $a\in\R$, $\lceil a\rceil$ is the smallest integer greater than or equal to~$a$.
A subset of a Euclidean space is called a polyhedron or a polyhedral convex set if it is the intersection of finitely many closed half-spaces. 
In particular, a subset of a Euclidean space is called a polytope if it is a bounded polyhedron. 
For a subset $A$ of a Euclidean space, let $\aff(A)$, $\conv(A)$, $\cone(A)$ denote the affine hull, convex hull, and conic hull of $A$, respectively. Moreover, let $\clos(A)$, $\inter(A)$, $\relint(A)$, $\relbd(A)$ denote the closure, interior, relative interior, and relative boundary of $A$, respectively. 

For a Polish space $\CY$ with its corresponding metric $d_{\CY}(\,\cdot\,,\cdot\,)$, let $\CB(\CY)$ denote the Borel subsets of $\CY$, let $\CP(\CY)$ denote the set of Borel probability measures on $\CY$, and let $\CP_p(\CY)$ denote the Wasserstein space of order~$p$ on $\CY$ for $p\in[1,\infty)$, which is defined by 
\begin{align*}
\CP_p(\CY):=\bigg\{\mu\in\CP(\CY):\exists \hat{z}\in\CY\text{ such that }\int_{\CY}d_{\CY}(\hat{z},z)^p\DIFFM{\mu}{\DIFF z}<\infty\bigg\}.
\end{align*}%
Moreover, for any $z\in\CY$, let $\delta_{z}$ denote the Dirac measure at $z$.
For $\nu\in\CP(\CY)$, let $\support(\nu)$ denote the support of $\nu$ and let $\CL^1(\CY,\nu)$ denote the set of $\nu$-integrable functions on $\CY$. 
For any $\mu,\nu\in\CP(\CY)$ and $p\in[1,\infty)$, let $W_p(\mu,\nu)$ denote the Wasserstein metric of order~$p$ between $\mu$ and $\nu$, which is defined by
\begin{align*}
W_p(\mu,\nu):=\left(\inf_{\gamma\in\Gamma(\mu,\nu)}\bigg\{\int_{\CY\times\CY}d_{\CY}(x,z)^p\DIFFM{\gamma}{\DIFF x,\DIFF z}\bigg\}\right)^{\frac{1}{p}}.
\end{align*}
In particular, $W_p(\mu,\nu)<\infty$ if $\mu,\nu\in\CP_p(\CY)$.

In this paper, we consider $N\in\N$ Polish spaces $(\CX_1,d_{\CX_1}),\ldots,(\CX_N,d_{\CX_N})$ with $N$ associated probability measures $\mu_1,\ldots,\mu_N$. This is presented in the following assumption. 
\begin{assumption}\label{asp:productspace}
$(\CX_1,d_{\CX_1}),\ldots,(\CX_N,d_{\CX_N})$ are $N\in\N$ Polish spaces and $\BCX:=\CX_1\times\cdots\times\CX_N$ is a Polish space equipped with the 1-product metric
\begin{align}
d_{\BCX}\big((x_1,\ldots,x_N),(x'_1,\ldots,x'_N)\big):=\sum_{i=1}^Nd_{\CX_i}(x_i,x'_i)\qquad \forall (x_1,\ldots,x_N),(x'_1,\ldots,x'_N)\in\BCX.
\label{eqn:1prod-metric}
\end{align}
Moreover, $\mu_1\in\CP_1(\CX_1),\ldots,\mu_N\in\CP_1(\CX_N)$.
\end{assumption}

For $i=1,\ldots,N$, let $\pi_i:\BCX\to\CX_i$ denote the projection function onto the $i$-th component. 
For $\mu\in\CP(\BCX)$ and $i=1,\ldots,N$, let $\mu\circ\pi_i^{-1}$ denote the push-forward of $\mu$ under $\pi_i$, which will also be referred to as the $i$-th marginal of $\mu$.
We work under the following technical assumption about the cost function \sloppy{${f:\BCX\to\R}$} in \eqref{eqn:mmot}. 

\begin{assumption}\label{asp:mmotexistence}
In addition to Assumption~\ref{asp:productspace}, $f:\BCX\to\R$ is a function that satisfies:
\begin{enumerate}[label=(\alph*)]
\item $f$ is lower semi-continuous;
\item there exist upper semi-continuous functions $(h_i:\CX_i\to\R)_{i=1:N}$ which satisfy $h_i\in\CL^1(\CX_i,\mu_i)$ $\forall 1\le i\le N$ and $\sum_{i=1}^N {h_i\circ\pi_i(\BIx)}\le f(\BIx)$ $\forall \BIx\in\BCX$.
\end{enumerate}
\end{assumption}

By a multi-marginal extension of \citep[Theorem~4.1]{villani2008optimal}, \eqref{eqn:mmot} attains an optimal solution. 

\subsection{The approximation scheme}\label{ssec:relaxation}

In this subsection, we develop a scheme to approximate \eqref{eqn:mmot} such that the approximation error can be controlled. 
The idea of the approximation scheme is to replace the marginal constraints in \eqref{eqn:mmot} with finitely many linear equality constraints. 
The feasible set of the resulting relaxation of \eqref{eqn:mmot}
contains probability measures on $\BCX$ whose marginals belong to convex subsets of $\CP_1(\CX_1),\ldots,\CP_1(\CX_N)$ known as \textit{moment sets}, see, e.g., \citep{winkler1988extreme}. 
The notion of moment set is formally defined as follows.

\begin{definition}[Moment set]\label{def:momentset}
Let $(\CY,d_{\CY})$ be a Polish space and let $p\in[1,\infty)$. For a collection $\CG$ of $\R$-valued Borel measurable functions on $\CY$, let $\CP_p(\CY;\CG):=\big\{\mu\in\CP_p(\CY):\CG\subseteq\CL^p(\CY,\mu)\big\}$. Let $\overset{\CG}{\sim}$ be defined as the following equivalence relation on $\CP_p(\CY;\CG)$: for all $\mu,\nu\in\CP_p(\CY;\CG)$,
\begin{align}
\mu\overset{\CG}{\sim}\nu \quad \Leftrightarrow \quad \forall g\in\CG,\; \int_{\CY}g\DIFFX{\mu}=\int_{\CY}g\DIFFX{\nu}.
\label{eqn:momentequivalence}
\end{align}
For every $\mu\in\CP_p(\CY;\CG)$, let $[\mu]_{\CG}:=\Big\{\nu\in\CP_p(\CY;\CG):\nu\overset{\CG}{\sim}\mu\Big\}$ be the equivalence class of $\mu$ under~$\overset{\CG}{\sim}$. 
We call $[\mu]_{\CG}$ the moment set centered at $\mu$ characterized by the test functions $\CG$. 
In addition, let $\specialoverline{W}_{p}(\mu,[\mu]_{\CG})$ denote the supremum $W_p$-metric between $\mu$ and members of $[\mu]_{\CG}$, i.e., 
\begin{align*}
\specialoverline{W}_{p}(\mu,[\mu]_{\CG}):=\sup_{\nu\in[\mu]_{\CG}}\big\{W_p(\mu,\nu)\big\}.
%\label{eqn:supdistance}
\end{align*}
Let $\lspan_1(\CG)$ denote the set of finite linear combinations of functions in $\CG$ plus a constant intercept, i.e., $\lspan_1(\CG):=\big\{y_0+\sum_{j=1}^{k}y_jg_{j}:k\in\N_0,\;(y_j)_{j=0:k}\subset\R,\;(g_{j})_{j=1:k}\subseteq\CG\big\}$.
It thus holds by (\ref{eqn:momentequivalence}) that if $\nu\in[\mu]_{\CG}$, then $\int_{\CY}g\DIFFX{\mu}=\int_{\CY}g\DIFFX{\nu}$ for all $g\in\lspan_1(\CG)$. In particular, we have $\mu\overset{\CG}{\sim}\nu$ if and only if $\mu\;\overset{\lspan_1(\CG)}{\scalebox{3.5}[1]{$\sim$}}\;\nu$, and $[\mu]_{\CG}=[\mu]_{\lspan_1(\CG)}$.
Furthermore, under Assumption~\ref{asp:productspace},
for any collections 
$\CG_1\subseteq\CL^1(\CX_1,\mu_1),\ldots,\CG_N\subseteq\CL^1(\CX_N,\mu_N)$ of functions on $\CX_1,\ldots,\CX_N$,
we slightly abuse the notation and define $\Gamma\big([\mu_1]_{\CG_1},\ldots,[\mu_N]_{\CG_N}\big)$ as follows:
\begin{align*}
\Gamma\big([\mu_1]_{\CG_1},\ldots,[\mu_N]_{\CG_N}\big):=\Big\{\gamma\in\Gamma(\nu_1,\ldots,\nu_N):\nu_i\in[\mu_i]_{\CG_i}\;\forall 1\le i\le N\Big\}.
%\label{eqn:momentset-productspace}
\end{align*}%
\end{definition}%

Now, for given marginals $\mu_1\in\CP_1(\CX_1),\ldots,\mu_N\in\CP_1(\CX_N)$ and cost function $f:\BCX\to\R$ satisfying Assumption~\ref{asp:mmotexistence}, as well as $N$ collections of test functions $\CG_1\subseteq\CL^1(\CX_1,\mu_1),\ldots,\CG_N\subseteq\CL^1(\CX_N,\mu_N)$, we aim to solve the following relaxation of \eqref{eqn:mmot}:
\begin{align}
\inf_{\mu\in\Gamma([\mu_1]_{\CG_1},\ldots,[\mu_N]_{\CG_N})}\bigg\{\int_{\BCX}f\DIFFX{\mu}\bigg\}.
\label{eqn:mmotlb}
\tag{$\mathsf{OT}_{\mathsf{relax}}$}
\end{align}%

\begin{remark}
In general, the infimum in \eqref{eqn:mmotlb} is not necessarily attained. This is demonstrated by the following examples, where we let $N=1$, $\BCX=\CX_1=\R_+$, and thus $\Gamma\big([\mu_1]_{\CG_1}\big)=[\mu_1]_{\CG_1}$. 
In the first example, the set $[\mu_1]_{\CG_1}$ lacks tightness. 
In the second example, the set $[\mu_1]_{\CG_1}$ is tight but lacks weak closedness. 
\begin{itemize}
\item Example 1: let $f(x):=\frac{2}{x+2}$ $\forall x\in\R_+$, let $\CG_1:=\{g\}$ where $g(x):=\frac{1}{x+1}$ $\forall x\in\R_+$, and let $\mu_1\in\CP_1(\R_+)$ satisfy $\int_{\R_+}g\DIFFX{\mu_1}=\frac{1}{2}$. 
In this case, $\int_{\R_+}f\DIFFX{\nu}>\frac{1}{2}$ for any $\nu\in[\mu_1]_{\CG_1}$. 
However, if we let $\nu_n:=\frac{n-1}{2n}\delta_{0}+\frac{n+1}{2n}\delta_{n}$ for all $n\in\N$, then $\nu_n\in[\mu_1]_{\CG_1}$ for all $n\in\N$, but $\lim_{n\to\infty}\int_{\R_+}f\DIFFX{\nu_n}=\frac{1}{2}$. Hence, the infimum in \eqref{eqn:mmotlb} is not attained. We remark that the sequence $(\nu_n)_{n\in\N}$ does not converge weakly to any probability measure. 

\item Example 2: let $f(x):=x\wedge 1$ $\forall x\in\R_+$, let $\CG_1:=\{g\}$ where $g(x):=x$ $\forall x\in\R_+$, and let $\mu_1\in\CP_1(\R_+)$ satisfy $\int_{\R_+}g\DIFFX{\mu_1}=1$. 
In this case, the only $\nu\in\CP_1(\R_+)$ that satisfies $\int_{\R_+}f\DIFFX{\nu}=\nobreak0$ is $\nu=\delta_0\notin[\mu_1]_{\CG_1}$. 
Hence, $\int_{\R_+}f\DIFFX{\nu}>0$ for any $\nu\in[\mu_1]_{\CG_1}$. 
However, if we let $\nu_n:=\frac{n-1}{n}\delta_{0}+\frac{1}{n}\delta_{n}$ for all $n\in\N$, then $\nu_n\in[\mu_1]_{\CG_1}$ for all $n\in\N$, but $\lim_{n\to\infty}\int_{\R_+}f\DIFFX{\nu_n}=0$.
Hence, the infimum in \eqref{eqn:mmotlb} is not attained. Observe that here the sequence $(\nu_n)_{n\in\N}$ converges weakly to $\delta_{0}$, which is not in $[\mu_1]_{\CG_1}$. 
\end{itemize}%
\end{remark}%

When $|\CG_i|=m_i\in\N$ for $i=1,\ldots,N$,
that is, when the collections of test functions $\CG_1,\ldots,\CG_N$ are finite,
\eqref{eqn:mmotlb} is a linear optimization problem over the space of probability measures on $\BCX$ subject to $\sum_{i=1}^{N}m_i$ linear equality constraints. 
In this case, it is well-known that there exists an approximately optimal solution of \eqref{eqn:mmotlb} which is supported on at most $\big(2+\sum_{i=1}^Nm_i\big)$ points; see, e.g., \citep[Theorem~3.1]{alfonsi2021approximation}. 
This is detailed in the next proposition.

\begin{proposition}[Sparsely supported approximately optimal solutions of \eqref{eqn:mmotlb}]\label{prop:mmotlb-sparsity}
Let Assumption~\ref{asp:mmotexistence} hold. 
For $i=1,\ldots,N$, let $\CG_i=\big\{g_{i,1},\ldots,g_{i,m_i}\big\}\subset\CL^1(\CX_i,\mu_i)$ where $m_i\in\N$, and let $m:=\sum_{i=1}^Nm_i$.
\widowpenalty-1000
Then, the following statements hold.
\begin{enumerate}[label=(\roman*),beginpenalty=10000]
\item\label{props:mmotlb-sparsity-disc}
For any $\epsilon_0>0$, there exist $1\le q\le m+2$, $(a_k)_{k=1:q}\subset(0,1)$ with $\sum_{k=1}^qa_k=1$, and $(\BIx_k)_{k=1:q}\subseteq\BCX$, such that $\hat{\mu}:=\sum_{k=1}^qa_k\delta_{\BIx_k}\in\CP(\BCX)$ is an $\epsilon_0$-optimal solution of~\eqref{eqn:mmotlb},
i.e., 
$\int_{\BCX}f\DIFFX{\hat{\mu}}\le \inf_{\mu\in\Gamma([\mu_1]_{\CG_1},\ldots,[\mu_N]_{\CG_N})}\big\{\int_{\BCX}f\DIFFX{\mu}\big\}+\epsilon_0$.

\item\label{props:mmotlb-sparsity-cont}
Suppose in addition that $\CX_i$ is compact and $g_{i,j}$ is continuous for $j=1,\ldots,m_i$, $i=1,\ldots,N$. 
Then, there exist $1\le q\le m+2$, $(a_k)_{k=1:q}\subset(0,1)$ with $\sum_{k=1}^qa_k=1$, and $(\BIx_k)_{k=1:q}\subseteq\BCX$, such that $\hat{\mu}:=\sum_{k=1}^qa_k\delta_{\BIx_k}\in\CP(\BCX)$ is an optimal solution of \eqref{eqn:mmotlb}.
\end{enumerate}
\end{proposition}

\begin{proof}[Proof of Proposition~\ref{prop:mmotlb-sparsity}]
See Section~\ref{ssec:proof-approximation}.
\end{proof}

To control the approximation error of \eqref{eqn:mmotlb}, we need to introduce additional assumptions on the cost function $f$ and the marginals $\mu_1,\ldots,\mu_N$ besides Assumption~\ref{asp:mmotexistence}. Since these assumptions depend on additional terms that affect the approximation error, let us define the set $\CA\big(L_f,D,\underline{f}_1,\overline{f}_1,\ldots,\underline{f}_N,\overline{f}_N\big)$ as follows.

\begin{definition}\label{def:controlsets}
Let Assumption~\ref{asp:mmotexistence} hold. 
For $L_f>0$, $D\in\CB(\BCX)$, and Borel measurable functions $\underline{f}_i:\CX_i\to\R$, $\overline{f}_i:\CX_i\to\R$ with $\underline{f}_i\le\overline{f}_i$ for $i=1,\ldots,N$, we say that
$(\mu_1,\ldots,\mu_N,f)\in\CA\big(L_f,D,\underline{f}_1,\overline{f}_1,\ldots,\underline{f}_N,\overline{f}_N\big)$ if the following conditions hold:
\begin{enumerate}[label=(\alph*)]
\item $f:\BCX\to\R$ restricted to $D$ is $L_f$-Lipschitz continuous, i.e., $\big|f(\BIx)-f(\BIx')\big|\le L_fd_{\BCX}(\BIx,\BIx')$ $\forall \BIx,\BIx'\in D$;
\item for $i=1,\ldots,N$, $\underline{f}_i$ and $\overline{f}_i$ are $\mu_i$-integrable, and it holds that
\begin{align*}
\sum_{i=1}^N\underline{f}_i\circ\pi_i(\BIx)\le f(\BIx)\INDI_{\BCX\setminus D}(\BIx) \le \sum_{i=1}^N\overline{f}_i\circ\pi_i(\BIx) \qquad \forall \BIx\in\BCX.
\end{align*}
\end{enumerate}%
\end{definition}%

Under the assumption that $(\mu_1,\ldots,\mu_N,f)\in\CA\big(L_f,D,\underline{f}_1,\overline{f}_1,\ldots,\underline{f}_N,\overline{f}_N\big)$,
we are able to construct an approximately optimal solution of \eqref{eqn:mmot} via an approximately optimal solution of \eqref{eqn:mmotlb},
where the approximation error can be controlled by the supremum $W_1$-metrics $\big(\specialoverline{W}_1(\mu_i,[\mu_i]_{\CG_i})\big)_{i=1:N}$.
This construction utilizes an application of the following gluing lemma, which we call \textit{reassembly}.

\begin{lemma}[Gluing lemma {\citep[Lemma~7.6]{villani2003topics}}]\label{lem:gluing}
Suppose that $\nu_1,\nu_2,\nu_3$ are three probability measures on three Polish spaces $(\CY_1,d_{\CY_1})$, $(\CY_2,d_{\CY_2})$, $(\CY_3,d_{\CY_3})$, respectively, and suppose that $\gamma_{12}\in\Gamma(\nu_1,\nu_2)$, $\gamma_{23}\in\Gamma(\nu_2,\nu_3)$. 
Then, there exists a probability measure $\gamma_{123}\in\CP(\CY_1\times \CY_2\times \CY_3)$ with marginal $\gamma_{12}$ on $\CY_1\times \CY_2$ and marginal $\gamma_{23}$ on $\CY_2\times \CY_3$. 
\end{lemma}

Using the gluing lemma, one can \textit{reassemble} a coupling of $\mu_1,\ldots,\mu_N$ from any coupling of $\hat{\mu}_1\in\CP_1(\CX_1),\ldots,\hat{\mu}_N\in\CP_1(\CX_N)$. This is detailed in Definition~\ref{def:reassembly} and Lemma~\ref{lem:reassembly} below.%

\begin{definition}[Reassembly]\label{def:reassembly}
Let Assumption~\ref{asp:productspace} hold and let $\mu_1,\ldots,\mu_N$ be defined as in Assumption~\ref{asp:productspace}. 
For any $\hat{\mu}\in\CP_1(\BCX)$, let its marginals on $\CX_1,\ldots,\CX_N$ be denoted by $\hat{\mu}_1,\ldots,\hat{\mu}_N$, respectively. 
Moreover, for $i=1,\ldots,N$, let $\bar{\CX}_i:=\CX_i$ in order to differentiate $\CX_i$ and its copy. 
A probability measure $\tilde{\mu}\in\Gamma(\mu_1,\ldots,\mu_N)\subseteq\CP_1(\BCX)$ is called a reassembly of $\hat{\mu}$ with marginals $\mu_1,\ldots,\mu_N$ if there exists $\gamma\in\CP(\CX_1\times\cdots\times\CX_N\times\bar{\CX}_1\times\cdots\times\bar{\CX}_N)$ which satisfies the following conditions:\widowpenalty-1000
\begin{enumerate}[label=(\alph*),beginpenalty=100000]
\item the marginal of $\gamma$ on $\CX_1\times\cdots\times\CX_N$ is $\hat{\mu}$;

\item\label{defs:reassembly-W1}
for $i=1,\ldots,N$, the marginal $\gamma$ on $\CX_i\times\bar{\CX}_i$, denoted by $\gamma_i$, is an optimal coupling of $\hat{\mu}_i$ and $\mu_i$ under the cost $d_{\CX_i}$, i.e., $\gamma_i\in\Gamma(\hat{\mu}_i,\mu_i)$ satisfies $\int_{\CX_i\times\bar{\CX}_i}d_{\CX_i}(x,z)\DIFFM{\gamma_i}{\DIFF x,\DIFF z}=W_1(\hat{\mu}_i,\mu_i)$;

\item the marginal of $\gamma$ on $\bar{\CX}_1\times\cdots\times\bar{\CX}_N$ is $\tilde{\mu}$.
\end{enumerate}
Let $R(\hat{\mu};\mu_1,\ldots,\mu_N)\subseteq\Gamma(\mu_1,\ldots,\mu_N)$ denote the set of reassemblies of $\hat{\mu}$ with marginals $\mu_1,\ldots,\mu_N$.%
\end{definition}%

For $i=1,\ldots,N$, the condition~\ref{defs:reassembly-W1} in Definition~\ref{def:reassembly} requires the marginal of $\gamma$ on $\CX_i\times\bar{\CX}_i$ to be an optimal coupling of $\hat{\mu}_i$ and $\mu_i$ which realizes their $W_1$-metric.
This enables us to control the approximation error of \eqref{eqn:mmotlb} via the Lipschitz constant~$L_f$ (see Definition~\ref{def:controlsets}) and the supremum $W_1$-metrics
$\big(\specialoverline{W}_1(\mu_i,[\mu_i]_{\CG_i})\big)_{i=1:N}$,
as we will show later in Theorem~\ref{thm:lowerbound}.

Lemma~\ref{lem:reassembly} shows that $R(\hat{\mu};\mu_1,\ldots,\mu_N)\subseteq\Gamma(\mu_1,\ldots,\mu_N)\subseteq\CP_1(\BCX)$ is non-empty. 

\begin{lemma}[Existence of reassembly]\label{lem:reassembly}
Let Assumption~\ref{asp:productspace} hold and let $\mu_1,\ldots,\mu_N$ be defined as in Assumption~\ref{asp:productspace}. 
Then, for any $\hat{\mu}\in\CP_1(\BCX)$, there exists a reassembly $\tilde{\mu}\in R(\hat{\mu};\mu_1,\ldots,\mu_N)$ of $\hat{\mu}$ with marginals $\mu_1,\ldots,\mu_N$. 
\end{lemma}

\begin{proof}[Proof of Lemma~\ref{lem:reassembly}]
See Section~\ref{ssec:proof-approximation}.
\end{proof}

\begin{remark}
In general, explicit construction of a reassembly $\tilde{\mu}\in R(\hat{\mu};\mu_1,\ldots,\mu_N)$ is highly non-trivial due to the difficulty in explicitly constructing an optimal coupling between two arbitrarily probability measures. 
Nonetheless, such construction is tractable under specific assumptions, as we will show in Section~\ref{ssec:reassemblysemidiscrete}. 
\end{remark}

Theorem~\ref{thm:lowerbound} below is the main result of this subsection. 
It states that an approximately optimal solution of \eqref{eqn:mmot} can be constructed from an approximately optimal solution of \eqref{eqn:mmotlb} through reassembly, and that \eqref{eqn:mmotlb} gives a lower bound for \eqref{eqn:mmot} whose quality depends on $L_f,D,\underline{f}_1,\overline{f}_1,\ldots,\underline{f}_N,\overline{f}_N$ and $\CG_1,\ldots,\CG_N$.

\begin{theorem}[Approximation of multi-marginal optimal transport]\label{thm:lowerbound}
Let Assumption~\ref{asp:mmotexistence} hold and let $\CG_i\subseteq\CL^1(\CX_i,\mu_i)$ for $i=1,\ldots,N$.
Moreover, let $L_f>0$, $D\in\CB(\BCX)$, let $(\underline{f}_i:\CX_i\to\R)_{i=1:N}$, $(\overline{f}_i:\CX_i\to\R)_{i=1:N}$ be Borel measurable functions such that $\underline{f}_i\in\lspan_1(\CG_i)$, $\underline{f}_i\le\overline{f}_i$ for $i=\nobreak1,\ldots,N$, and
$(\mu_1,\ldots,\mu_N,f)\in\CA\big(L_f,D,\underline{f}_1,\overline{f}_1,\ldots,\underline{f}_N,\overline{f}_N\big)$.
Furthermore, suppose that there exist $(\hat{x}_1,\ldots,\hat{x}_N)\in D$, $h_1\in\lspan_1(\CG_1),\ldots,h_N\in\lspan_1(\CG_N)$ such that $d_{\CX_i}(\hat{x}_i,\cdot\,)\le h_i(\,\cdot\,)$ for $i=\nobreak1,\ldots,N$. 
Then, the following statements hold. 
\begin{enumerate}[label=(\roman*)]
\item\label{thms:lowerbound-control}
Let $\hat{\mu}_i\in[\mu_i]_{\CG_i}$ for $i=1,\ldots,N$. 
If $\hat{\mu}\in\Gamma(\hat{\mu}_1,\ldots,\hat{\mu}_N)$ satisfies $f\in\nobreak\CL^1(\BCX,\hat{\mu})$,
and $\tilde{\mu}\in\nobreak R(\hat{\mu};\mu_1,\ldots,\mu_N)$, 
then $f\in\CL^1(\BCX,\tilde{\mu})$,
and
the following inequality holds:
\begin{align*}
\int_{\BCX}f\DIFFX{\tilde{\mu}}-\int_{\BCX}f\DIFFX{\hat{\mu}}\le \sum_{i=1}^N \left(L_f W_1(\mu_i,\hat{\mu}_i)+\int_{\CX_i}\overline{f}_i-\underline{f}_i \DIFFX{\mu_i}\right).
\end{align*}

\item\label{thms:lowerbound-boundedness}
The infimum in \eqref{eqn:mmotlb} is not equal to $-\infty$. 

\item\label{thms:lowerbound-w1radfinite}
For $i=1,\ldots,N$, $\specialoverline{W}_{1}(\mu_i,[\mu_i]_{\CG_i})<\infty$. 

\item\label{thms:lowerbound-epsilonoptimal}
Let $\epsilon_{0}>0$ and let $\epsilon_{\mathsf{apx}}:=\epsilon_{0}+\sum_{i=1}^N \Big(L_f \specialoverline{W}_{1}(\mu_i,[\mu_i]_{\CG_i})+\int_{\CX_i}\overline{f}_i-\underline{f}_i \DIFFX{\mu_i}\Big)$.
Suppose that $\hat{\mu}\in\Gamma\big([\mu_1]_{\CG_1},\ldots,[\mu_N]_{\CG_N}\big)$ is an $\epsilon_{0}$-optimal solution of \eqref{eqn:mmotlb}, i.e., 
\begin{align}
\int_{\BCX}f\DIFFX{\hat{\mu}}\le\inf_{\mu\in\Gamma([\mu_1]_{\CG_1},\ldots,[\mu_N]_{\CG_N})}\bigg\{\int_{\BCX}f\DIFFX{\mu}\bigg\}+\epsilon_{0}.
\label{eqn:lowerbound-epsilon-optimizer}
\end{align}
Then, every $\tilde{\mu}\in R(\hat{\mu};\mu_1,\ldots,\mu_N)$ is an $\epsilon_{\mathsf{apx}}$-optimal solution of \eqref{eqn:mmot}, i.e., 
\begin{align*}
\int_{\BCX}f\DIFFX{\tilde{\mu}}\le\inf_{\mu\in\Gamma(\mu_1,\ldots,\mu_N)}\bigg\{\int_{\BCX}f\DIFFX{\mu}\bigg\}+\epsilon_{\mathsf{apx}}.
\end{align*}

\item\label{thms:lowerbound-mmotcontrol}
The following inequalities hold:
\begin{align}
\begin{split}
0&\le\inf_{\mu\in\Gamma(\mu_1,\ldots,\mu_N)}\bigg\{\int_{\BCX}f\DIFFX{\mu}\bigg\}-\inf_{\mu\in\Gamma([\mu_1]_{\CG_1},\ldots,[\mu_N]_{\CG_N})}\bigg\{\int_{\BCX}f\DIFFX{\mu}\bigg\}\\
&\le \sum_{i=1}^N \left(L_f \specialoverline{W}_{1}(\mu_i,[\mu_i]_{\CG_i})+\int_{\CX_i}\overline{f}_i-\underline{f}_i \DIFFX{\mu_i}\right)<\infty.
\end{split}
\label{eqn:lowerbound}
\end{align}
\end{enumerate}
\end{theorem} 

\begin{proof}[Proof of Theorem~\ref{thm:lowerbound}]
See Section~\ref{ssec:proof-approximation}.
\end{proof}

A special case of Theorem~\ref{thm:lowerbound} is when $f$ is $L_f$-Lipschitz continuous on $\BCX$ for some $L_f>0$. In this case, $(\mu_1,\ldots,\mu_N,f)\in\CA(L_f,\BCX,0,0,\ldots,0,0)$ and the error terms in Theorem~\ref{thm:lowerbound} can be simplified accordingly. 

Theorem~\ref{thm:lowerbound}\ref{thms:lowerbound-epsilonoptimal} and Theorem~\ref{thm:lowerbound}\ref{thms:lowerbound-mmotcontrol} show that whenever one finds an $\epsilon_0$-optimal solution $\hat{\mu}$ of \eqref{eqn:mmotlb} and constructs a reassembly $\tilde{\mu}\in R(\hat{\mu};\mu_1,\ldots,\mu_N)$, one can obtain lower and upper bounds
\begin{align*}
\int_{\BCX}f\DIFFX{\hat{\mu}}-\epsilon_0\le\inf_{\mu\in\Gamma(\mu_1,\ldots,\mu_N)}\bigg\{\int_{\BCX}f\DIFFX{\mu}\bigg\}\le \int_{\BCX}f\DIFFX{\tilde{\mu}}
\end{align*}
for \eqref{eqn:mmot} where 
the difference between them can be controlled by
$\int_{\BCX}f\DIFFX{\tilde{\mu}}-\left(\int_{\BCX}f\DIFFX{\hat{\mu}}-\epsilon_0\right)\le \epsilon_0+\sum_{i=1}^N \Big(L_f \specialoverline{W}_{1}(\mu_i,[\mu_i]_{\CG_i})+\int_{\CX_i}\overline{f}_i-\underline{f}_i \DIFFX{\mu_i}\Big)$.

An observation from Theorem~\ref{thm:lowerbound} is that if one can find increasingly better choices of $L_f$, $D$, $\underline{f}_1$, $\overline{f}_1,\ldots,\underline{f}_N$, $\overline{f}_N$, and $\CG_1,\ldots,\CG_N$ such that $\sum_{i=1}^N \left(L_f \specialoverline{W}_{1}(\mu_i,[\mu_i]_{\CG_i})+\int_{\CX_i}\overline{f}_i-\underline{f}_i \DIFFX{\mu_i}\right)$ shrinks to 0, then one can obtain an optimal solution of \eqref{eqn:mmot} in the limit. This is detailed in the next theorem. 
We will discuss how one can control $\sum_{i=1}^N \left(L_f \specialoverline{W}_{1}(\mu_i,[\mu_i]_{\CG_i})+\int_{\CX_i}\overline{f}_i-\underline{f}_i \DIFFX{\mu_i}\right)$ to be arbitrarily close to~0 in Section~\ref{ssec:momentsetdd}. 

\begin{theorem}[Optimal solution of multi-marginal optimal transport]\label{thm:lowerboundconverge}
Let Assumption~\ref{asp:mmotexistence} hold and let $\big(L_f^{(l)}$, $D^{(l)}$, $\underline{f}^{(l)}_1$, $\overline{f}^{(l)}_1,\ldots,\underline{f}^{(l)}_N$, $\overline{f}^{(l)}_N,\CG^{(l)}_1,\ldots,\CG^{(l)}_N\big)_{l\in\N}$ be such that:
\begin{enumerate}[label=(\alph*)]
\item for each $l\in\N$, $(\mu_1,\ldots,\mu_N,f)\in\CA\big(L_f^{(l)},D^{(l)},\underline{f}^{(l)}_1,\overline{f}^{(l)}_1,\ldots,\underline{f}^{(l)}_N,\overline{f}^{(l)}_N\big)$;
\item for each $l\in\N$ and for $i=1,\ldots,N$, $\underline{f}^{(l)}_i\in\lspan_1\big(\CG^{(l)}_i\big)\subseteq\CL^1(\CX_i,\mu_i)$;
\item for each $l\in\N$, there exist $(\hat{x}_1,\ldots,\hat{x}_N)\in D^{(l)}$, $h_1\in\lspan_1\big(\CG_1^{(l)}\big),\ldots,h_N\in\lspan_1\big(\CG_N^{(l)}\big)$ such that $d_{\CX_i}(\hat{x}_i,\cdot\,)\le h_i(\,\cdot\,)$ for $i=1,\ldots,N$; 
\item $\lim_{l\to\infty}\sum_{i=1}^N \left(L^{(l)}_f \specialoverline{W}_{1}\big(\mu_i,[\mu_i]_{\CG^{(l)}_i}\big)+\int_{\CX_i}\overline{f}^{(l)}_i-\underline{f}^{(l)}_i \DIFFX{\mu_i}\right)= 0$. 
\end{enumerate}
Let $\big(\epsilon_0^{(l)}\big)_{l\in\N}\subset(0,\infty)$ be such that $\lim_{l\to\infty}\epsilon_0^{(l)}=0$.
For each $l\in\N$, let $\hat{\mu}^{(l)}\in\Gamma\big([\mu_1]_{\CG^{(l)}_1},\ldots,[\mu_N]_{\CG^{(l)}_N}\big)$ be an $\epsilon_0^{(l)}$-optimal solution of $\inf_{\mu\in\Gamma\big([\mu_1]_{\CG^{(l)}_1},\ldots,[\mu_N]_{\CG^{(l)}_N}\big)}\big\{\int_{\BCX}f\DIFFX{\mu}\big\}$, i.e., 
\begin{align}
\int_{\BCX}f\DIFFX{\hat{\mu}^{(l)}}\le\inf_{\mu\in\Gamma\big([\mu_1]_{\CG^{(l)}_1},\ldots,[\mu_N]_{\CG^{(l)}_N}\big)}\bigg\{\int_{\BCX}f\DIFFX{\mu}\bigg\}+\epsilon_0^{(l)},
\label{eqn:relaxed-epsilon-optimizer}
\end{align}
and let $\tilde{\mu}^{(l)}\in R(\hat{\mu}^{(l)};\mu_1,\ldots,\mu_N)$ be a reassembly of $\hat{\mu}^{(l)}$ with marginals $\mu_1,\ldots,\mu_N$. Then, $\big(\tilde{\mu}^{(l)}\big)_{l\in\N}$ has at least one accumulation point in $\big(\CP_1(\BCX),W_1\big)$, and every accumulation point of $\big(\tilde{\mu}^{(l)}\big)_{l\in\N}$ is an optimal solution of~\eqref{eqn:mmot}.
\end{theorem}

\begin{proof}[Proof of Theorem~\ref{thm:lowerboundconverge}]
See Section~\ref{ssec:proof-approximation}.
\end{proof}

\begin{remark}\label{rmk:comparison-with-Alphonsi}
    The relaxation of the marginal constraints in an MMOT problem into finitely many moment-based constraints has previously been considered by \citet{alfonsi2021approximation} and \citet{alfonsi2022constrained} under a slightly different setting. 
    In their setting, apart from the moment-based linear equality constraints with respect to finitely many test functions, they impose inequality constraints with respect to functions that have faster growth than the test functions in order to guarantee that the relaxed MMOT problem attains an optimal solution. 
    \citet{alfonsi2021approximation} have shown that optimal solutions of the relaxed MMOT problems converge to a true optimal solution of the MMOT problem in the limit when the number of test functions goes to infinity, under the assumption that the class of test functions satisfies certain conditions.
    However, every probability measure in such an optimizing sequence is in general infeasible for the MMOT problem, and feasibility is only attained at their limit. 
    In contrast, we are able to construct a feasible solution of \eqref{eqn:mmot} with respect to any finite collections of test functions $\CG_1,\ldots,\CG_N$, where the sub-optimality of this feasible solution can be quantified and controlled based on the choices of the test functions; see also Remark~\ref{rmk:comparison-with-Alphonsi-2} for further discussions with regard to the construction of test functions for controlling the sub-optimality.
\end{remark}%

\subsection{Duality results}\label{ssec:duality}

We have derived a relaxation of \eqref{eqn:mmot} given by \eqref{eqn:mmotlb} in Section~\ref{ssec:relaxation}. 
This subsection is dedicated to the analysis of the dual optimization problem of \eqref{eqn:mmotlb} when the total number of test functions in $\CG_1,\ldots,\CG_N$ is finite, which is a linear semi-infinite programming (LSIP) problem. 
Let Assumption~\ref{asp:mmotexistence} hold, let the moment set $[\mu_i]_{\CG_i}$ be characterized by $m_i\in\N$ test functions $\CG_i:=\big\{g_{i,1},\ldots,g_{i,m_i}\big\}\subset\CL^1(\CX_i,\mu_i)$ for $i=1,\ldots,N$, and let $m:=\sum_{i=1}^Nm_i$. 
For notational simplicity, let the vector-valued functions $\BIg_1:\CX_1\to\R^{m_1},\ldots,\BIg_N:\CX_N\to\R^{m_N}$, and $\BIg:\BCX\to\R^m$ be defined as
\begin{align}
\begin{split}
\BIg_i(x_i)&:=\big(g_{i,1}(x_i),\ldots,g_{i,m_i}(x_i)\big)^\TRANSP \hspace{6.8pt} \qquad \forall x_i\in\CX_i,\;\forall 1\le i\le N,\\
\BIg(x_1,\ldots,x_N)&:=\big(\BIg_1(x_1)^\TRANSP,\ldots,\BIg_N(x_N)^\TRANSP\big)^\TRANSP \qquad\forall (x_1,\ldots,x_N)\in\BCX.
\end{split}
\label{eqn:g-vecdef}
\end{align}
Moreover, let the vectors $\bar{\BIg}_1\in\R^{m_1},\ldots,\bar{\BIg}_N\in\R^{m_N}$, and $\bar{\BIg}\in\R^m$ be defined as
\begin{align}
\begin{split}
\bar{\BIg}_i&:=\textstyle\big(\int_{\CX_i}g_{i,1}\DIFFX{\mu_i},\ldots,\int_{\CX_i}g_{i,m_i}\DIFFX{\mu_i}\big)^\TRANSP \qquad\forall 1\le i\le N,\\
\bar{\BIg}&:=(\bar{\BIg}_1^\TRANSP,\ldots,\bar{\BIg}_N^\TRANSP)^\TRANSP.
\end{split}
\label{eqn:v-vecdef}
\end{align}
Then, the dual optimization problem of \eqref{eqn:mmotlb} is an LSIP problem given by
\begin{align}
\begin{split}
\maximize_{y_0,\,\BIy}\quad & y_0+\langle\bar{\BIg},\BIy\rangle\\
\text{subject to}\quad & y_0+\langle\BIg(\BIx),\BIy\rangle\le f(\BIx) \qquad \forall \BIx\in\BCX,\\
& y_0\in\R,\qquad \BIy\in\R^m.
\end{split}
\tag{$\SFO\SFT^*_{\mathsf{relax}}$}
\label{eqn:mmotlb-dual-lsip}
\end{align}
Note that \eqref{eqn:mmotlb-dual-lsip} can be seen as a finite-dimensional parametrization of the following dual MMOT problem:
\begin{align}
    \sup_{(h_i\in\CL^1(\CX_i,\mu_i))_{i=1:N}}\Bigg\{\sum_{i=1}^{N}\int_{\CX_i}h_i\DIFFX{\mu_i}: {\textstyle\sum_{i=1}^{N}}h_i\circ\pi_i(\BIx)\le f(\BIx)\; \forall \BIx\in\BCX\Bigg\},
    \tag{$\mathsf{OT}^*$}
    \label{eqn:mmot-dual}
\end{align}
because for any 
$y_0$, $\BIy=(\BIy_1^\TRANSP,\ldots,\BIy_N^\TRANSP)^\TRANSP$ that is a feasible solution of \eqref{eqn:mmotlb-dual-lsip},
it holds that the functions $\big(\frac{y_0}{N}+\langle\BIg_i(\,\cdot\,),\BIy_i\rangle:\CX_i\to\R\big)_{i=1:N}$
constitute a feasible solution of \eqref{eqn:mmot-dual}.

Following the strong duality results in the theory of linear semi-infinite optimization (see, e.g., \citep[Chapter~8]{goberna1998linear}), we can derive the following strong duality tailored to \eqref{eqn:mmotlb} and \eqref{eqn:mmotlb-dual-lsip}.

\begin{theorem}[Strong duality]\label{thm:duality}
Let Assumption~\ref{asp:mmotexistence} hold. 
Let $\CG_i=\big\{g_{i,1},\ldots,g_{i,m_i}\big\}\subset\CL^1(\CX_i,\mu_i)$ where $m_i\in\N$, for $i=1,\ldots,N$, let $m:=\sum_{i=1}^Nm_i$, and let $\BIg(\,\cdot\,)$ and $\bar{\BIg}$ be given by (\ref{eqn:g-vecdef}) and (\ref{eqn:v-vecdef}). 
Then, \widowpenalty-1000
\begin{enumerate}[label=(\roman*),beginpenalty=10000]
\item\label{thms:duality-weak}
the following weak duality between \eqref{eqn:mmotlb-dual-lsip} and \eqref{eqn:mmotlb}  holds:
\begin{align}
\,\,\,\;\sup_{y_0\in\R,\,\BIy\in\R^m}\!\!\Big\{\hspace{-1pt}y_0+\langle\bar{\BIg},\BIy\rangle\!:y_0+\langle\BIg(\BIx),\BIy\rangle\!\le\! f(\BIx) \; \forall \BIx\in\BCX\hspace{-1pt}\Big\}\!\le\!\inf_{\mu\in\Gamma([\mu_1]_{\CG_1},\ldots,[\mu_N]_{\CG_N})}\!\!\bigg\{\!\int_{\BCX} f\DIFFX{\mu}\hspace{-1pt}\bigg\}.
\label{eqn:weak-duality}
\end{align}
\end{enumerate}
Moreover, suppose that the left-hand side of (\ref{eqn:weak-duality}) is not equal to~$-\infty$, that is, the corresponding maximization problem is feasible.
Let the sets $K\subseteq\R^m$ and $C\subset\R^{m+2}$ be defined as follows:  
\begin{align}
\begin{split}
K&:=\conv\Big(\big\{\BIg(\BIx):\BIx\in\BCX\big\}\Big), \\
C&:=\cone\Big(\big\{\big(1,\BIg(\BIx)^\TRANSP,f(\BIx)\big)^\TRANSP:\BIx\in\BCX\big\}\Big), 
\end{split}
\label{eqn:Kset-def}
\end{align}
and let the conditions \ref{thmc:duality1}, \ref{thmc:duality1int}, and \ref{thmc:duality2} be defined as follows:
\begin{enumerate}[label=\normalfont{(SD\arabic*)},leftmargin=3.5em]
\item\label{thmc:duality1}
$\bar{\BIg}\in\relint(K)$;
\item\label{thmc:duality1int}
$\bar{\BIg}\in\inter(K)$;
\item\label{thmc:duality2}
$C$ is closed.
\end{enumerate}
Then, the following statements hold.
\begin{enumerate}[label=(\roman*),beginpenalty=10000]
\setcounter{enumi}{1}
\item\label{thms:duality1} If either \ref{thmc:duality1} or \ref{thmc:duality2} holds, then the following strong duality between \eqref{eqn:mmotlb-dual-lsip} and \eqref{eqn:mmotlb} holds:
\begin{align}
\,\,\,\;\sup_{y_0\in\R,\,\BIy\in\R^m}\!\!\Big\{\hspace{-1pt}y_0+\langle\bar{\BIg},\BIy\rangle\!:y_0+\langle\BIg(\BIx),\BIy\rangle\!\le\! f(\BIx) \; \forall \BIx\in\BCX\hspace{-1pt}\Big\}\!=\!\inf_{\mu\in\Gamma([\mu_1]_{\CG_1},\ldots,[\mu_N]_{\CG_N})}\!\!\bigg\{\!\int_{\BCX} f\DIFFX{\mu}\hspace{-1pt}\bigg\}.
\label{eqn:duality}
\end{align}

\item\label{thms:duality2}
If \ref{thmc:duality1} holds, then the set of optimal solutions of \eqref{eqn:mmotlb-dual-lsip} is non-empty.

\item\label{thms:duality3}
If \ref{thmc:duality1int} holds, then the set of optimal solutions of \eqref{eqn:mmotlb-dual-lsip} is non-empty and bounded.
\end{enumerate}
\end{theorem}

\begin{proof}[Proof of Theorem~\ref{thm:duality}]
See Section~\ref{ssec:proof-duality}.
\end{proof}

\begin{remark}
One can guarantee that the LSIP problem~\eqref{eqn:mmotlb-dual-lsip} has a non-empty feasible set (and thus the left-hand side of (\ref{eqn:weak-duality}) is not equal to $-\infty$) under fairly general assumptions. 
For example, let us suppose that the assumptions of Theorem~\ref{thm:lowerbound} hold,
that is,
there exist 
$L_f>0$, $D\in\CB(\BCX)$, 
Borel measurable functions $(\underline{f}_i:\CX_i\to\R)_{i=1:N}$,
$(\overline{f}_i:\CX_i\to\R)_{i=1:N}$ 
satisfying $\underline{f}_i\in\lspan_1(\CG_i)$, 
$\underline{f}_i\le\overline{f}_i$ for $i=1,\ldots,N$, and 
$(\mu_1,\ldots,\mu_N,f)\in\CA\big(L_f,D,\underline{f}_1,\overline{f}_1,\ldots,\underline{f}_N,\overline{f}_N\big)$,
and that there exist $(\hat{x}_1,\ldots,\hat{x}_N)\in D$, $h_1\in\lspan_1(\CG_1),\ldots,h_N\in\lspan_1(\CG_N)$ satisfying $d_{\CX_i}(\hat{x}_i,\cdot\,)\le h_i(\,\cdot\,)$ for $i=1,\ldots,N$.
Subsequently, 
let $\hat{y}_0:={-\big|f(\hat{x}_1,\ldots,\hat{x}_N)\big|}$,
let $\hat{\BIy}_i\in\R^{m_i}$ satisfy 
$\langle\BIg_i(\,\cdot\,),\hat{\BIy}_i\rangle=\underline{f}_i(\,\cdot\,)-L_fh_i(\,\cdot\,)$ for $i=1,\ldots,N$,
and let $\hat{\BIy}:=(\hat{\BIy}_1^\TRANSP,\ldots,\hat{\BIy}_N^\TRANSP)^\TRANSP$.
Then, $(\hat{y}_0,\hat{\BIy})$ is a feasible solution of
the LSIP problem~\eqref{eqn:mmotlb-dual-lsip}; see (\ref{eqn:lowerbound-proof-boundedness}) in the proof of Theorem~\ref{thm:lowerbound}.
\end{remark}

The next proposition presents sufficient conditions under which the conditions~\ref{thmc:duality1}, \ref{thmc:duality1int}, and \ref{thmc:duality2} in Theorem~\ref{thm:duality} hold. 

\begin{proposition}\label{prop:duality-settings}
Let Assumption~\ref{asp:mmotexistence} hold. 
For $i=1,\ldots,N$, let $\CG_i=\{g_{i,1},\ldots,g_{i,m_i}\}\subset\CL^1(\CX_i,\mu_i)$ where $m_i\in\N$. 
Let $m:=\sum_{i=1}^Nm_i$. 
Then, the following statements hold.
\widowpenalties-1000
\begin{enumerate}[label=(\roman*),beginpenalty=10000]
\item\label{props:duality-setting1}
For $i=1,\ldots,N$, suppose that $\support(\mu_i)=\CX_i$ and $g_{i,1},\ldots,g_{i,m_i}$ are all continuous. Then, the condition~\ref{thmc:duality1} in Theorem~\ref{thm:duality} holds.
\item\label{props:duality-setting2}
For $i=1,\ldots,N$, let $\BIg_i:\CX_i\to\R^{m_i}$ be defined by (\ref{eqn:g-vecdef}). Suppose in addition to the assumptions in statement~\ref{props:duality-setting1} that, for $i=1,\ldots, N$, there exist $m_i+1$ points $x_{i,1},\ldots,x_{i,m_i+1}\in\CX_i$ such that the $m_i+1$ vectors $\BIg_i(x_{i,1}),\ldots,\BIg_i(x_{i,m_i+1})\in\R^{m_i}$ are affinely independent. Then, the condition~\ref{thmc:duality1int} in Theorem~\ref{thm:duality} holds.
\item\label{props:duality-setting3}
For $i=1,\ldots,N$, suppose that $\CX_i$ is compact, $g_{i,1},\ldots,g_{i,m_i}$ are all continuous, and $f$ is continuous. Then, the condition~\ref{thmc:duality2} in Theorem~\ref{thm:duality} holds.
\end{enumerate}
\end{proposition}

\begin{proof}[Proof of Proposition~\ref{prop:duality-settings}]
See Section~\ref{ssec:proof-duality}.
\end{proof}

\begin{remark}
In the case where $\support(\mu_i)\ne\CX_i$, one can replace $\CX_i$ with $\support(\mu_i)$ and subsequently restrict the domain of $f$ to $\bigtimes_{i=1}^N\support(\mu_i)$. 
Then, assuming that all test functions in $\CG_1,\ldots,\CG_N$ are continuous, one may proceed by applying Proposition~\ref{prop:duality-settings}\ref{props:duality-setting1} to show that the condition \ref{thmc:duality1} in Theorem~\ref{thm:duality} holds.
Moreover, if we assume further that for $i=1,\ldots,N$, there exist $m_i+1$ points $x_{i,1},\ldots,x_{i,m_i+1}\in\support(\mu_i)$ such that the $m_i+1$ vectors $\BIg_i(x_{i,1}),\ldots,\BIg_i(x_{i,m_i+1})\in\R^{m_i}$ are affinely independent, then one can show via Proposition~\ref{prop:duality-settings}\ref{props:duality-setting2} that the condition \ref{thmc:duality1int} in Theorem~\ref{thm:duality} holds. 
\end{remark}

\subsection{Computational complexity}\label{ssec:mmot-complexity}

In this subsection, we analyze the computational complexity of the LSIP problem \eqref{eqn:mmotlb-dual-lsip}.
In the subsequent analyses, we assume that the underlying spaces $\CX_1,\ldots,\CX_N$ are all compact and the test functions in $\CG_1,\ldots,\CG_N$ are all continuous, and quantify the computational complexity of \eqref{eqn:mmotlb-dual-lsip} in terms of the number of calls to a global minimization oracle, which is defined as follows.

\begin{definition}[Global minimization oracle for \eqref{eqn:mmotlb-dual-lsip}]\label{def:mmot-oracle}
Let Assumption~\ref{asp:mmotexistence} hold and assume in addition that $\CX_1,\ldots,\CX_N$ are all compact. 
For $i=1,\ldots,N$, let $m_i\in\N$ and $\CG_i:=\big\{g_{i,1},\ldots,g_{i,m_i}\big\}$, where $g_{i,j}:\CX_i\to\R$ is continuous for $j=1,\ldots,m_i$. Let $m:=\sum_{i=1}^Nm_i$ and let $\BIg:\BCX\to\R^m$ be defined by (\ref{eqn:g-vecdef}). 
A procedure $\mathtt{Oracle}(\,\cdot\,)$ is called a global minimization oracle for \eqref{eqn:mmotlb-dual-lsip} if, for every $\BIy\in\R^m$, a call to $\mathtt{Oracle}(\BIy)$ 
returns a tuple $(\BIx^\star,\beta^\star)\in\BCX\times\R$, 
where $\BIx^\star\in\BCX$ is an optimal solution of the global minimization problem $\inf_{\BIx\in\BCX}\big\{f(\BIx)-\langle\BIg(\BIx),\BIy\rangle\big\}$ (which exists due to the compactness of~$\BCX$ and the lower semi-continuity of~$f$) and $\beta^\star:=f(\BIx^\star)-\langle\BIg(\BIx^\star),\BIy\rangle\in\R$ is its corresponding objective value.
\end{definition}

With the global minimization oracle for \eqref{eqn:mmotlb-dual-lsip} defined, the following theorem states 
the existence of an algorithm for approximately solving \eqref{eqn:mmotlb-dual-lsip} 
in which the number of calls to $\mathtt{Oracle}(\,\cdot\,)$ and the number of additional arithmetic operations are both polynomial in $m:=\sum_{i=1}^{N}m_i$.
In our complexity analysis, we denote the computational complexity of the multiplication of two $m\times m$ matrices by $O(m^{\omega})$. For example, when the standard procedure is used, the computational complexity of this operation is $O(m^3)$. However, it is known that $\omega<2.376$; see, e.g., \citep{coppersmith1990matrix}.

\begin{theorem}[Computational complexity of \eqref{eqn:mmotlb-dual-lsip}]\label{thm:mmot-complexity}
Let Assumption~\ref{asp:mmotexistence} hold and assume in addition that $\CX_1,\ldots,\CX_N$ are all compact. 
For $i=1,\ldots,N$, let $m_i\in\N$ and $\CG_i:=\big\{g_{i,1},\ldots,g_{i,m_i}\big\}$, where $g_{i,j}:\CX_i\to\R$ is continuous for $j=1,\ldots,m_i$. Let $m:=\sum_{i=1}^Nm_i$ and let $\BIg_1:\CX_1\to\R^{m_1},\ldots,\BIg_N:\CX_N\to\R^{m_N}$, $\BIg:\BCX\to\R^m$ and $\bar{\BIg}\in\R^m$ be defined by (\ref{eqn:g-vecdef}) and (\ref{eqn:v-vecdef}). 
Let $\mathtt{Oracle}(\,\cdot\,)$ be the global minimization oracle in Definition~\ref{def:mmot-oracle}.
Assume in addition that for $i=1,\ldots,N$, $\big\|\BIg_i(x_i)\big\|_2\le 1$ for all $x_i\in\CX_i$.\footnote{Since $\CX_i$ is compact and $g_{i,1},\ldots,g_{i,m_i}$ are continuous, one may replace $g_{i,j}$ with $\max_{x_i\in\CX_i}\big\{\big\|\BIg_i(x_i)\big\|_2\big\}^{-1}g_{i,j}$ for $j=1,\ldots,m_i$ to guarantee that $\big\|\BIg_i(x_i)\big\|_2\le 1$ for all $x_i\in\CX_i$. 
Observe that this rescaling leaves $\lspan_1(\CG_i)$ unchanged and thus the resulting problem \eqref{eqn:mmotlb-dual-lsip} is equivalent to the problem without rescaling.} 
Suppose that \eqref{eqn:mmotlb-dual-lsip} has an optimal solution $(y_0^\star,\BIy^\star)$ and let $M_{\mathsf{opt}}:=\big\|(y_0^\star,\BIy^{\star\TRANSP})^\TRANSP\big\|_2$. 
Moreover, let $\epsilon_{\mathsf{LSIP}}>0$ be an arbitrary positive tolerance value. 
Then, there exists an algorithm which computes an $\epsilon_{\mathsf{LSIP}}$-optimal solution of \eqref{eqn:mmotlb-dual-lsip} with
$O\big(m\log(mM_{\mathsf{opt}}/\epsilon_{\mathsf{LSIP}})\big)$ calls to $\mathtt{Oracle}(\,\cdot\,)$ 
and $O\big(m^{\omega+1}\log(mM_{\mathsf{opt}}/\epsilon_{\mathsf{LSIP}})\big)$ additional arithmetic operations.
\end{theorem}

\begin{proof}[Proof of Theorem~\ref{thm:mmot-complexity}]
See Section~\ref{ssec:proof-complexity}.
\end{proof}

\begin{remark}
Recall that Proposition~\ref{prop:duality-settings}\ref{props:duality-setting1} has provided a sufficient condition to guarantee the existence of an optimal solution $(y^\star_{0},\BIy^{\star})$ of \eqref{eqn:mmotlb-dual-lsip}.
However, the dependence of the constant $M_{\mathsf{opt}}$ on $N$ and~$m$ is not studied.
The analysis of this dependence will be presented later in Proposition~\ref{prop:mmot-complexity-constant} under more specific assumptions on the spaces $\CX_1,\ldots,\CX_N$, the cost function $f$, and the test functions $\CG_1,\ldots,\CG_N$.
\end{remark}

\section{Explicit construction of reassemblies and moment sets in the Euclidean case}\label{sec:explicitconstruction}

In this section, we address the practical questions from Section~\ref{sec:theory} regarding the explicit construction of a \textit{reassembly} $\tilde{\mu}\in R(\hat{\mu};\mu_1,\ldots,\mu_N)$ in Theorem~\ref{thm:lowerbound}, and the explicit construction of test functions $\CG_1,\ldots,\CG_N$ such that the terms $\big(\specialoverline{W}_{1}(\mu_i,[\mu_i]_{\CG_i})\big)_{i=1:N}$ in the approximation error in Theorem~\ref{thm:lowerbound} can be controlled to be arbitrarily close to~0.
Specifically, in Section~\ref{ssec:reassemblysemidiscrete}, we \textit{explicitly construct} reassemblies both in the case where the marginals are one-dimensional
and in the \textit{semi-discrete} case
where the marginals are multi-dimensional
by adapting existing results from the copula theory and the field of computational geometry.
In Section~\ref{ssec:momentsetdd}, we show that when $\CY$ is a compact subset of a Euclidean space, one can \textit{explicitly construct} finitely many continuous functions $\CG\subset\CL^1(\CY,\mu)$ for any $\mu\in\CP_1(\CY)$ such that $\specialoverline{W}_1(\mu,[\mu]_{\CG})$ can be controlled to be arbitrarily close to~0.

\subsection{Construction of reassemblies in the Euclidean case}\label{ssec:reassemblysemidiscrete}
In this subsection, let us examine the construction of reassemblies in the case where $\CX_i$ is a closed subset of a Euclidean space $\R^{d_i}$ for some $d_i\in\N$, for $i=1,\ldots,N$. 
In Proposition~\ref{prop:reassembly-1d},
we combine Sklar's theorem from the copula theory (see e.g., \citep[Theorem~5.3]{mcneil2005quantitative}) 
and classical results about the optimal couplings of one-dimensional probability measures (see, e.g., \citep[Chapter~3.1]{rachev1998mass})
to characterize reassemblies in the case where $d_1=\cdots=d_N=1$.
Moreover,
Proposition~\ref{prop:reassembly-dd} characterizes reassemblies in the
\textit{semi-discrete} case, that is, when $\hat{\mu}\in\CP_1(\BCX)$ is finitely supported, and for $i=1,\ldots,N$, $\mu_i\in\CP_1(\CX_i)$ is absolutely continuous with respect to the Lebesgue measure on $\R^{d_i}$. 
Optimal couplings in the semi-discrete setting have been previously studied in the field of computational geometry; see, e.g., \citep{aurenhammer1992minkowski, levy2015numerical, levy2018notions, genevay2016stochastic} and \citep[Chapter~5]{peyre2019computational} for related discussions. 
However, these studies only focus on optimal couplings under the \textit{squared} Euclidean distance, which is not directly applicable to our setting since we are interested in finding an optimal coupling where the cost function is a general norm on the underlying space. Therefore, in Proposition~\ref{prop:reassembly-dd}, we provide results about optimal couplings and reassemblies under the assumption that the cost function is given by a norm under which the closed unit ball is strictly convex. 

Before presenting the characterizations, let us first introduce the following lemma which states that the definition of reassemblies is invariant of the underlying spaces of the probability measures $\mu_1,\ldots,\mu_N$.

\begin{lemma}\label{lem:wlogeuclidean}
Let Assumption~\ref{asp:productspace} hold. 
Suppose that for $i=1,\ldots,N$, $(\CX^\dagger_i,d_{\CX^\dagger_i})$ is a Polish space such that $\CX_i\subseteq\CX^\dagger_i$ and $d_{\CX_i}$ is the restriction of $d_{\CX^\dagger_i}$ to $\CX_i$. 
Let $\BCX^\dagger:=\bigtimes_{i=1}^N\CX^\dagger_i$.
For $i=1,\ldots,N$, let $\mu_i^\dagger\in\CP_1(\CX^\dagger_i)$ be defined by $\mu_i^\dagger(E):=\mu_i(E\cap\CX_i)$ $\forall E\in\CB(\CX^\dagger_i)$. 
Similarly, for any $\mu\in\CP_1(\BCX)$, let $\mu^\dagger\in\CP_1(\BCX^\dagger)$ be defined by $\mu^\dagger(E):=\mu(E\cap\BCX)$ $\forall E\in\CB(\BCX^\dagger)$. 
Then, for any $\hat{\mu},\tilde{\mu}\in\CP_1(\BCX)$, $\tilde{\mu}\in R(\hat{\mu};\mu_1,\ldots,\mu_N)$ if and only if $\tilde{\mu}^\dagger\in R(\hat{\mu}^\dagger;\mu_1^\dagger,\ldots,\mu_N^\dagger)$. 
\end{lemma}

\begin{proof}[Proof of Lemma~\ref{lem:wlogeuclidean}]
See Section~\ref{ssec:proof-reassembly}. 
\end{proof}

When $\CX_i$ is a closed subset of $\R^{d_i}$ where $d_i\in\N$ for $i=1,\ldots,N$, Lemma~\ref{lem:wlogeuclidean} shows that one can first extend $\hat{\mu}\in\CP_1(\BCX),\mu_1\in\CP_1(\CX_1),\ldots,\mu_N\in\CP_1(\CX_N)$ to $\hat{\mu}^\dagger\in\CP_1(\R^d)$ (where $d:=\sum_{i=1}^Nd_i$), ${\mu_1^\dagger\in\CP_1(\R^{d_1})},\ldots,\mu_N^\dagger\in\CP_1(\R^{d_N})$ and construct a reassembly $\tilde{\mu}^\dagger\in R(\hat{\mu}^\dagger;\mu_1^\dagger,\ldots,\mu_N^\dagger)$. This can be done via the constructions in Proposition~\ref{prop:reassembly-1d} and Proposition~\ref{prop:reassembly-dd} below under some additional assumptions.
Subsequently, one can define $\tilde{\mu}\in\CP_1(\BCX)$ by $\tilde{\mu}(E):=\tilde{\mu}^\dagger(E)$ $\forall E\in\CB(\BCX)$ and get $\tilde{\mu}\in R(\hat{\mu};\mu_1,\ldots,\mu_N)$.

Proposition~\ref{prop:reassembly-1d} below concerns reassemblies in the case where the marginals $\mu_1,\ldots,\mu_N$ are one-dimensional, i.e., when $d_1=\cdots=d_N=1$.

\begin{proposition}[Reassemblies with one-dimensional marginals]\label{prop:reassembly-1d}
Let Assumption~\ref{asp:productspace} hold and suppose that
$\CX_i=\R$, $d_{\CX_i}(x,z):=|x-z|$ $\forall x,z\in\R$ for $i=1,\ldots,N$.
Let $\hat{\mu}\in\CP_1(\R^N)$, and let $\hat{\mu}_i\in\CP_1(\R)$ denote its $i$-th marginal for $i=1,\ldots,N$.
Moreover, let $F_{\hat{\mu}}:\R^N\to[0,1]$ denote the distribution function of $\hat{\mu}$, 
i.e., 
$F_{\hat{\mu}}(x_1,\ldots,x_N):=\hat{\mu}\big(\!\bigtimes_{i=1}^N(-\infty,x_i]\big)$ $\forall (x_1,\ldots,x_N)\in\R^N$. 
Similarly, for $i=1,\ldots,N$, 
let $F_{\hat{\mu}_i}:\R\to[0,1]$ and $F_{\mu_i}:\R\to[0,1]$ denote the distribution functions of $\hat{\mu}_i$ and~$\mu_i$, respectively, i.e.,
$F_{\hat{\mu}_i}(x):=\hat{\mu}_i\big((-\infty,x]\big)$,
$F_{\mu_i}(x):=\mu_i\big((-\infty,x]\big)$ $\forall x\in\R$.
Furthermore,
for $i=1,\ldots,N$, 
define 
$F^{-1}_{\mu_i}(u):=\inf\big\{x\in\R:F_{\mu_i}(x)\ge u\big\}$ $\forall u\in[0,1]$,
that is, $F^{-1}_{\mu_i}$ is the left-continuous generalized inverse of $F_{\mu_i}$.
Then, the following statements hold.
\begin{enumerate}[label=(\roman*)]

\item\label{props:reassembly-1d-copula}%
Let $C:[0,1]^N\to[0,1]$ be a distribution function with uniform marginals which satisfies
\begin{align*}
F_{\hat{\mu}}(x_1,\ldots,x_N)=C\big(F_{\hat{\mu}_1}(x_1),\ldots,F_{\hat{\mu}_N}(x_N)\big) \hspace{2.8pt} \qquad \forall (x_1,\ldots,x_N)\in\R^N,
\end{align*}
that is, $C$ is a copula with the dependence structure of $\hat{\mu}$; see, e.g., \citep[Definition~5.1 \& Theorem~5.3]{mcneil2005quantitative} for the definition and existence of copulas.
Subsequently, let $F_{\tilde{\mu}}:\R^N\to[0,1]$ be defined as follows:
\begin{align*}
    F_{\tilde{\mu}}(x_1,\ldots,x_N):=C\big(F_{\mu_1}(x_1),\ldots,F_{\mu_N}(x_N)\big) \qquad \forall (x_1,\ldots,x_N)\in\R^N.
\end{align*}
Then, $F_{\tilde{\mu}}$ is the distribution function of a probability measure $\tilde{\mu}\in\CP(\R^N)$ satisfying $\tilde{\mu}\in\nobreak R(\hat{\mu};\mu_1,\ldots,\mu_N)$. 
\end{enumerate}
In the following, let us suppose in addition that 
$\hat{\mu}=\sum_{j=1}^{J}a_j\delta_{\BIx_j}$,
where $J\in\N$,
$(a_j)_{j=1:J}\subset(0,1)$ satisfy $\sum_{j=1}^{J}a_j=\nobreak1$,
and $\big(\BIx_j=(x_{1,j},\ldots,x_{N,j})\big)_{j=1:J}\subset\R^N$ are distinct points.
For $i=1,\ldots,N$, let 
$\sigma_i:\{1,\ldots,J\}\to\{1,\ldots,J\}$ be a bijection that satisfies
$x_{i,\sigma_i(1)}\le x_{i,\sigma_i(2)}\le\cdots\le x_{i,\sigma_i(J)}$.
Subsequently, let us define $(c_{i,j})_{j=1:J,\,i=0:N}$ as follows:
\begin{align*}
    c_{0,j}:=\sum_{l=1}^{j-1}a_l \qquad \forall 1\le j\le J,  \qquad\qquad c_{i,j}:=\sum_{l=1}^{\sigma_i^{-1}(j)-1}a_{\sigma_i(l)} \qquad \forall 1\le j\le J,\; \forall 1\le i\le N.
\end{align*}
Moreover, let $(\Omega,\CF,\PROB)$ be a probability space, 
let $U_0:\Omega\to[0,1]$ be a random variable with uniform distribution, and let 
$\big(U_{i,j}:\Omega\to[0,1]\big)_{j=1:J,\,i=1:N}$ be a collection of random variables satisfying
\begin{align}
    \PROB\big[U_{i,j}\le u \big| c_{0,j}\le U_0 < c_{0,j}+a_j\big]=(u \wedge 1)^+ \qquad \forall u\in\R,\; \forall 1\le j\le J,\; \forall 1\le i\le N.
    \label{eqn:reassembly-1d-rvcollection}
\end{align}
Using these notions, let us define random variables 
$Z_1,\ldots,Z_N:\Omega\to\R$ as follows:
\begin{align}
    Z_i:=F_{\mu_i}^{-1}\Bigg(\sum_{j=1}^{J}(c_{i,j}+a_jU_{i,j})\INDI_{[c_{0,j},c_{0,j}+a_j)}(U_0)\Bigg) \qquad \forall 1\le i\le N.
    \label{eqn:reassembly-1d-invtrans}
\end{align}
Then,
\begin{enumerate}[label=(\roman*),beginpenalty=10000]
    \setcounter{enumi}{1}
    \item\label{props:reassembly-1d-semidisc}
    the law $\tilde{\mu}\in\CP(\R^N)$ of the random variable $(Z_1,\ldots,Z_N):\Omega\to\R^N$ satisfies 
    $\tilde{\mu}\in R(\hat{\mu};\mu_1,\ldots,\mu_N)$.
\end{enumerate}
\end{proposition}

\begin{proof}[Proof of Proposition~\ref{prop:reassembly-1d}]
    See Section~\ref{ssec:proof-reassembly}.
\end{proof}

\begin{remark}\label{rmk:reassembly-1d-uniform-examples}
    The following list presents some examples of $\big(U_{i,j}:\Omega\to[0,1]\big)_{j=1:J,\,i=1:N}$ in Proposition~\ref{prop:reassembly-1d} which satisfy (\ref{eqn:reassembly-1d-rvcollection}).
    \begin{enumerate}[label=(\alph*)]
        \item
        One may let $(U_{i,j})_{j=1:J,\,i=1:N}$ be independent and identically distributed uniform random variables that are independent of~$U_0$.
        
        \item
        One may let $U_{i,j}:=U_{i}$ for $j=1,\ldots,J$, $i=1,\ldots,N$, where 
        $U_1,\ldots,U_N:\Omega\to[0,1]$ are uniform random variables that are independent of~$U_0$.
        In particular, one may let $U_0,U_1,\ldots,U_N$ be jointly independent, 
        or let $U_1=\cdots=U_N$ be independent of $U_0$.

        \item\label{rmks:reassembly-1d-uniform-examples-deterministic}%
        One may define $U_{i,j}:=\frac{1}{a_{j}}(U_0-c_{0,j})\INDI_{\{[c_{0,j},c_{i,j}+a_j)\}}(U_0)$ for $j=1,\ldots,J$, $i=1,\ldots,N$.
    \end{enumerate}
\end{remark}

Proposition~\ref{prop:reassembly-dd} concerns reassemblies in the semi-discrete case with multi-dimensional marginals.

\begin{proposition}[Reassemblies in the semi-discrete case]\label{prop:reassembly-dd}
Let Assumption~\ref{asp:productspace} hold. 
Suppose that for $i=1,\ldots,N$, $\CX_i=\R^{d_i}$ for $d_i\in\N\cap[2,\infty)$ and that $d_{\CX_i}$ is induced by a norm $\|\cdot\|$ on $\R^{d_i}$ under which the closed unit ball $\big\{\BIx\in\R^{d_i}:\|\BIx\|\le 1\big\}$ is a strictly convex set.\footnote{For example, under the $p$-norm, this condition is satisfied for all $1<p<\infty$ by the Minkowski inequality, but fails when $p=1$ or $p=\infty$.} 
Moreover, suppose that for $i=1,\ldots,N$, $\mu_i\in\CP_1(\CX_i)$ is absolutely continuous with respect to the Lebesgue measure on $\R^{d_i}$. 
Let $\hat{\mu}\in\CP_1(\BCX)$ be a finitely supported measure with marginals $\hat{\mu}_1,\ldots,\hat{\mu}_N$. 
For $i=1,\ldots,N$, let $\hat{\mu}_i$ be represented as $\hat{\mu}_i=\sum_{j=1}^{J_i}a_{i,j}\delta_{\BIx_{i,j}}$ 
where 
$J_i\in\N$,
$(a_{i,j})_{j=1:J_i}\subset(0,1)$ satisfy $\sum_{j=1}^{J_i}a_{i,j}=1$,
and
$(\BIx_{i,j})_{j=1:J_i}\subset\CX_i$ are distinct. 
Then, the following statements hold.\widowpenalties-1000
\begin{enumerate}[label=(\roman*),beginpenalty=10000]
\item\label{props:reassembly-dd1} 
For $i=1,\ldots,N$, there exists $\big(\phi^\star_{i,j}\big)_{j=1:J_i}\subset\R$ which is an optimal solution of the following concave maximization problem:
\begin{align}
\sup_{(\phi_{i,j})_{j=1:J_i}\subset\R}\Bigg\{\sum_{j=1}^{J_i}\phi_{i,j}a_{i,j}-\int_{\R^{d_i}}\max_{1\le j\le J_i}\big\{\phi_{i,j}-\|\BIx_{i,j}-\BIz\|\big\}\DIFFM{\mu_i}{\DIFF\BIz}\Bigg\}.
\label{eqn:reassembly-dd-kant}
\end{align}
Moreover, the optimal value of the problem (\ref{eqn:reassembly-dd-kant}) is equal to $W_1(\hat{\mu}_i,\mu_i)$.

\item\label{props:reassembly-dd2} 
For $i=1,\ldots,N$, let $\big(\phi^\star_{i,j}\big)_{j=1:J_i}\subset\R$ be an optimal solution of the problem~(\ref{eqn:reassembly-dd-kant}), and define
\begin{align}
\quad\qquad V_{i,j}:=\bigg\{\BIz\in\R^{d_i}:\phi_{i,j}^\star-\|\BIx_{i,j}-\BIz\|=\max_{1\le j'\le J_i}\big\{\phi_{i,j'}^\star-\|\BIx_{i,j'}-\BIz\|\big\}\bigg\}\quad \forall 1\le j\le J_i.
\label{eqn:reassembly-dd-partition}
\end{align}
Then, it holds that $\mu_i(V_{i,j})=a_{i,j}$ for $j=1,\ldots,J_i$.

\item\label{props:reassembly-dd3} 
Let $(\Omega,\CF,\PROB)$ be a probability space and let $(X_1,\ldots,X_N):\Omega\to\BCX$ be a random variable with law $\hat{\mu}$. 
For $i=1,\ldots,N$, let $(\phi^\star_{i,j})_{j=1:J_i}\subset\R$
be an optimal solution of the problem~(\ref{eqn:reassembly-dd-kant}),
and let the sets $(V_{i,j})_{j=1:J_i}$ be defined by (\ref{eqn:reassembly-dd-partition}).
Let $(Z_i:\Omega\to\R^{d_i})_{i=1:N}$ be random variables which satisfy
\begin{align*}
\;\;\;\;\;\;\;\;\;\;\PROB[Z_i\in E|X_i=\BIx_{i,j}]=\frac{\mu_i\big(E\cap V_{i,j}\big)}{\mu_i(V_{i,j})} \qquad\forall E\in\CB(\R^{d_i}),\;\forall 1\le j\le J_i,\; \forall 1\le i\le N.
%\label{eqn:reassembly-dd-conditional}
\end{align*}
Let $\tilde{\mu}\in\CP(\BCX)$ be the law of the random variable $(Z_1,\ldots,Z_N):\Omega\to\BCX$. 
Then, it holds that $\tilde{\mu}\in R(\hat{\mu};\mu_1,\ldots,\mu_N)$. 
\end{enumerate}
\end{proposition}

\begin{proof}[Proof of Proposition~\ref{prop:reassembly-dd}]
See Section~\ref{ssec:proof-reassembly}.
\end{proof}

\begin{remark}
We would like to remark that the finite support assumption of $\hat{\mu}$ 
in Proposition~\ref{prop:reassembly-1d}\ref{props:reassembly-1d-semidisc} and
in Proposition~\ref{prop:reassembly-dd} are relevant in practice, 
because the numerical method that we use to solve \eqref{eqn:mmotlb} in Section~\ref{sec:numerics} returns an approximately optimal solution of \eqref{eqn:mmotlb} that has finite support (see Algorithm~\ref{alg:cp-mmot} and Proposition~\ref{prop:cpalgo-properties}). 
\end{remark}

\subsection{Controlling supremum $W_1$-metrics within moment sets in the compact Euclidean case}\label{ssec:momentsetdd}
Theorem~\ref{thm:lowerbound} 
shows that, in order to control the approximation error of \eqref{eqn:mmotlb},
one is required to construct test functions $\CG_1,\ldots,\CG_N$ to control 
$\big(\specialoverline{W}_1(\mu_i,[\mu_i]_{\CG_i})\big)_{i=1:N}$ to be arbitrarily close to~0. 
In this subsection, 
we consider the compact Euclidean case,
i.e., when $\CX_1,\ldots,\CX_N$ are compact subsets of Euclidean spaces, 
where we 
\textit{explicitly construct} $\CG_1,\ldots,\CG_N$ 
and develop upper bounds for
$\big(\specialoverline{W}_1(\mu_i,[\mu_i]_{\CG_i})\big)_{i=1:N}$ that can be controlled to be arbitrarily small.

Let us first recall the notions of faces, extreme points, and extreme directions of convex sets,
and introduce the corresponding notations.

\begin{definition}[Faces, extreme points, and extreme directions of convex sets; see {\citep[Section~18]{rockafellar1970convex}}]\label{def:convexanalysisnotions}
Let $d\in\N$. 
A convex subset $C'$ of a convex set $C\subseteq\R^d$ is called a face of $C$ if
\begin{align*}
\forall \BIx_1,\BIx_2\in C,\; \forall \lambda\in(0,1): \qquad\lambda\BIx_1+(1-\lambda)\BIx_2\in C'\quad \Rightarrow \quad \BIx_1\in C',\;\BIx_2\in C'.
\end{align*}
A point $\BIx$ in a convex set $C\subseteq\R^d$ is called an extreme point (or vertex) of $C$ if it is a face of $C$. 
A vector $\BIz\in\R^d$ is called an extreme direction of a convex set $C\subseteq\R^d$ if there exists $\BIx\in C$ such that $\big\{\BIx+\lambda\BIz:\lambda\ge0\big\}$ is a face of $C$. 
Moreover, 
for any polyhedron $C$,
we let 
$\FF(C)$ denote the set of faces of $C$,
let $V(C)$ denote the set of extreme points of $C$,
and let $D(C)$ denote the set of extreme directions of $C$.
Note that $\emptyset\in \FF(C)$ and every non-empty face $F\in\FF(C)$ is a polyhedron with $V(F)\subseteq V(C)$ and $D(F)\subseteq D(C)$; see, e.g., \citep[Theorem~19.1]{rockafellar1970convex}.
\end{definition}

We now state our results on the explicit construction of test functions for controlling the supremum $W_1$-metric of the resulting moment set as follows.

\begin{proposition}[Supremum $W_1$-metric within a moment set in the compact Euclidean case]\label{prop:momentset-simplex}
Let $d\in\N$, let $\CY\subset\R^d$ be compact,
and let $d_{\CY}$ be induced by a norm $\|\cdot\|$ on $\R^d$.
The following statement holds.
\begin{enumerate}[label=(\roman*),beginpenalty=10000]
    \item\label{props:momentset-simplex-bisection}
    Let $(\FC_r)_{r\in\N_0}$ be a sequence of collections of $d$-simplices constructed as follows.\footnote{A set is called a $d$-simplex if it is the convex hull of $d+\nobreak1$~affinely independent points in $\R^d$.}
    \begin{itemize}
        \item Let $C_0$ be a $d$-simplex such that $C_0\supseteq\CY$.
        Set $\FC_0\leftarrow\{C_0\}$.

        \item Repeat the following steps for $r=1,2,\ldots$
        \begin{itemize}
            \item Let $(\BIv,\BIv')\in\argmax_{(\BIw,\BIw')}\big\{\|\BIw-\BIw'\|:\BIw,\BIw'\in V(C),\; C\in \FC_{r-1}\big\}$.
            
            \item Initialize $\FC_r\leftarrow\FC_{r-1}$.
            
            \item For each $C\in\FC_{r}$ which satisfies $\{\BIv,\BIv'\}\subseteq V(C)$, update $\FC_r$ by bisecting the simplex~$C$ at the midpoint of $\BIv$ and $\BIv'$, i.e.,
            $\FC_{r}\leftarrow (\FC_{r}\setminus C)\cup \Big\{\conv\Big(\big(V(C)\setminus\{\BIv\}\big)\cup \big\{{\textstyle\frac{\BIv+\BIv'}{2}}\big\}\Big),\; \conv\Big(\big(V(C)\setminus\{\BIv'\}\big)\cup \big\{{\textstyle\frac{\BIv+\BIv'}{2}}\big\}\Big)\Big\}$.
        \end{itemize}
    \end{itemize}
    Then, for each $r\in\N_0$,
    $\FC_r$ is a finite collection of $d$-simplices where 
    $\bigcup_{C\in\FC_r}C\supseteq\CY$ and 
    $C_1\cap\nobreak C_2\in \FF(C_1)\cap \FF(C_2)$ $\forall C_1,C_2\in\FC_r$.
    Moreover, it holds that
    \begin{align*}
        \lim_{r\to\infty}\max_{C\in\FC_r}\max_{\BIv,\BIv'\in V(C)}\big\{\|\BIv-\BIv'\|\big\}=\nobreak0.
    \end{align*}
\end{enumerate}
In the following, let $\FC$ be a finite collection of $d$-simplices satisfying
$\bigcup_{C\in\FC}C\supseteq\CY$ and 
$C_1\cap C_2\in \FF(C_1)\cap \FF(C_2)$ $\forall C_1,C_2\in\FC$.
Let us 
denote $\FF(\FC):=\bigcup_{C\in\FC}\FF(C)$,
$V(\FC):=\bigcup_{C\in\FC}V(C)$,
and
define the mesh size $\eta(\FC)>0$ of $\FC$ as follows:
\begin{align}
    \eta(\FC):=\max_{C\in\FC}\max_{\BIv,\BIv'\in V(C)}\big\{\|\BIv-\BIv'\|\big\}.
    \label{eqn:momentset-simplex-meshsize}
\end{align}
Then, the following statement holds.
\begin{enumerate}[label=(\roman*),beginpenalty=10000]
    \setcounter{enumi}{1}
    
    \item\label{props:momentset-simplex-uniquerepresentation}
    For every $F\in\FF(\FC)$, there exist functions
    $\big(\lambda_{F,\BIw}:\bigcup_{C\in\FC}C\to\R_+)_{\BIw\in V(F)}$
    which are affine on~$F$
    and satisfy 
    $\lambda_{F,\BIw}(\BIx)=\nobreak0$
    $\forall \BIx\notin F$
    for each $\BIw\in V(F)$,
    as well as
    \begin{align*}
        \sum_{\BIw\in V(F)}\lambda_{F,\BIw}(\BIx)=\nobreak1,\quad \sum_{\BIw\in V(F)}\lambda_{F,\BIw}(\BIx)\BIw=\BIx \qquad \forall \BIx\in\relint(F).
    \end{align*}
    
\end{enumerate}
Next, for every $\BIv\in V(\FC)$, let $g_{\BIv}:\bigcup_{C\in\FC}C\to\R$ be defined as follows:
\begin{align}
    g_{\BIv}(\BIx):=\sum_{F\in\FF(\FC)}\Bigg(\sum_{\BIw\in V(F)}\lambda_{F,\BIw}(\BIx)\INDI_{\{\BIw=\BIv\}}\Bigg)\INDI_{\relint(F)}(\BIx) \qquad \forall \BIx\in\bigcup_{C\in\FC}C.
    \label{eqn:momentset-simplex-testfuncs}
\end{align}
Subsequently, let us define $\CG_{\mathsf{simp}}(\FC):=\big\{g_{\BIv}:\BIv\in V(\FC)\big\}$.
We call $\CG_{\mathsf{simp}}(\FC)$ the simplicial test functions generated by~$\FC$.
Then, the following statements hold.
\begin{enumerate}[label=(\roman*),beginpenalty=10000]
    \setcounter{enumi}{2}
    \item\label{props:momentset-simplex-continuity}
    Every function in $\CG_{\mathsf{simp}}(\FC)$ is non-negative and continuous.
    Moreover, it holds 
    for every $\BIv\in\nobreak V(\FC)$ and every $C\in\FC$ that $g_{\BIv}$ is affine on $C$, and
    that $\sum_{\BIv\in V(\FC)}g_{\BIv}(\BIx)=\nobreak1$ $\forall\BIx\in\bigcup_{C\in\FC}C$.
    
    \item\label{props:momentset-simplex-upperbound}
    For all $\mu,\nu\in\CP(\CY)$ which satisfy
    $\mu\overset{\CG_{\mathsf{simp}}(\FC)}{\scalebox{3.5}[1]{$\sim$}}\nu$,
    it holds that
    $W_1(\mu,\nu)\le 2\eta(\FC)$.
    In particular, $\specialoverline{W}_1\big(\mu,[\mu]_{\CG_{\mathsf{simp}}(\FC)}\big)\le 2\eta(\FC)$ for all $\mu\in\CP(\CY)$.
    
\end{enumerate}
\end{proposition}

\begin{proof}[Proof of Proposition~\ref{prop:momentset-simplex}]
See Section~\ref{ssec:proof-momentsets}.
\end{proof}

In fact, Proposition~\ref{prop:momentset-simplex} is a special case of our general results in Section~\ref{sec:momentset-extended} about controlling the supremum Wasserstein metric within moment sets via finitely many test functions.
Our generalization in Section~\ref{sec:momentset-extended} is three-fold:
\begin{enumerate}[label=(\arabic*),beginpenalty=10000]
    \item we generalize the results in Proposition~\ref{prop:momentset-simplex} to the case where $\CY\subseteq\R^d$ is unbounded;
    
    \item we propose a class of test functions that are more general than those in Proposition~\ref{prop:momentset-simplex};
    
    \item we are able to control the supremum $W_p$-metric for general $p\in[1,\infty)$
    (note that $W_p(\,\cdot\,,\cdot\,)\ge W_1(\,\cdot\,,\cdot\,)$ $\forall p\in[1,\infty)$).
\end{enumerate}
Even though controlling the supremum $W_p$-metric for $p=1$ suffices for controlling the approximation error of \eqref{eqn:mmotlb} in view of Theorem~\ref{thm:lowerbound}, 
our general results cover the $p\in(1,\infty)$ case since that is of independent interest; see Corollary~\ref{cor:breedenlitzenberger} and Remark~\ref{rmk:breedenlitzdd}.

To apply Proposition~\ref{prop:momentset-simplex} to control the approximation error of \eqref{eqn:mmotlb} via Theorem~\ref{thm:lowerbound},
let us introduce the following setting.

\begin{setting}[The compact Euclidean case]\label{sett:compact}
    The following conditions hold.
    \begin{enumerate}[label=\normalfont{(A\arabic*)}]
        \item\label{settc:compact-measures}
        For $i=1,\ldots,N$ where $N\in\N$, $d_i\in\N$, $\CX_i\subset\R^{d_i}$ is compact, and $d_{\CX_i}$ is induced by a norm $\|\cdot\|$ on $\R^{d_i}$.
        Moreover, 
        $\BCX:=\CX_1\times\cdots\times\CX_N$ 
        is equipped with the 1-product metric (\ref{eqn:1prod-metric}),
        and $\mu_i\in\CP(\CX_i)$ for $i=1,\ldots,N$.
        
        \item\label{settc:compact-cost}
        $f:\BCX\to\R$ is $L_f$-Lipschitz continuous for $L_f>0$.
    \end{enumerate}
    Subsequently, let us consider $(\FC_i)_{i=1:N}$ and $(\CG_i)_{i=1:N}$ that are constructed as follows.
    \begin{enumerate}[label=\normalfont{(TF)}]
        \item\label{settc:compact-testfuncs}
        For $i=1,\ldots,N$, 
        let $\FC_i$ be a finite collection of $d_i$-simplices satisfying
        $\bigcup_{C\in\FC_i}C\supseteq \CX_i$ and
        $C_1\cap\nobreak C_2\in \FF(C_1)\cap \FF(C_2)$ $\forall C_1,C_2\in\FC_i$.
        Then, let $\eta(\FC_i)$ be defined by (\ref{eqn:momentset-simplex-meshsize}),
        let $m_i:=|V(\FC_i)|-1$,
        and let $\CG_{\mathsf{simp}}(\FC_i)$ be the simplicial test functions generated by~$\FC_i$ as constructed in Proposition~\ref{prop:momentset-simplex},
        where we denote $\CG_{\mathsf{simp}}(\FC_i)=\{g_{i,\BIv_{i,0}},g_{i,\BIv_{i,1}},\ldots,g_{i,\BIv_{i,m_i}}\}$ with respect to an arbitrary enumeration
        $\{\BIv_{i,0},\BIv_{i,1},\ldots,\BIv_{i,m_i}\}$ of $V(\FC_i)$.
        Let $g_{i,j}:=g_{i,\BIv_{i,j}}$ for $j=0,1,\ldots,m_i$,
        and let $\CG_i:=\{g_{i,1},\ldots,g_{i,m_i}\}$.
        Moreover, 
        let $m:=\sum_{i=1}^{N}m_i$,
        and let $(\BIg_i:\CX_i\to\R^{m_i})_{i=1:N}$,
        $\BIg:\BCX\to\R^{m}$,
        $(\bar{\BIg}_i\in\R^{m_i})_{i=1:N}$,
        and $\bar{\BIg}\in\R^{m}$ be defined by 
        (\ref{eqn:g-vecdef}) and (\ref{eqn:v-vecdef}).
    \end{enumerate}
\end{setting}
Observe that Assumption~\ref{asp:mmotexistence} holds under Setting~\ref{sett:compact}.
Moreover, 
notice that each $\CG_i$ is formed by removing 
one function from $\CG_{\mathsf{simp}}(\FC_i)$.
This is because of the property that 
$\sum_{\BIv\in V(\FC_i)}g_{i,\BIv}(\BIx)=\nobreak 1$ $\forall \BIx\in\bigcup_{C\in\FC_i}C$ in Proposition~\ref{prop:momentset-simplex}\ref{props:momentset-simplex-continuity},
which shows that 
$g_{i,0}\in\lspan_1\big(\{g_{i,1},\ldots,g_{i,m_i}\}\big)$ 
and thus
$\lspan_1(\CG_i)=\lspan_1\big(\CG_{\mathsf{simp}}(\FC_i)\big)$.

Under Setting~\ref{sett:compact},
not only can we explicitly bound the total number of test functions $m:=\sum_{i=1}^{N}|\CG_i|$ in \eqref{eqn:mmotlb} in order to control its approximation error to any $\epsilon>\nobreak0$,
but we can also explicitly bound the constant 
$M_{\mathsf{opt}}$ in the computational complexity of \eqref{eqn:mmotlb-dual-lsip} in Theorem~\ref{thm:mmot-complexity}.
These results are presented in the two following propositions.

\begin{proposition}[Number of test functions to control the approximation error in Theorem~\ref{thm:lowerbound} under Setting~\ref{sett:compact}]\label{prop:momentset-errorcontrol}
    Under the conditions \ref{settc:compact-measures} and \ref{settc:compact-cost} in Setting~\ref{sett:compact},
    let $\epsilon>0$, $\epsilon_{\mathsf{LSIP}}\in(0,\epsilon)$ be arbitrary.
    For $i=1,\ldots,N$, 
    let $(-\infty<\underline{M}_{i,j}<\overline{M}_{i,j}<\infty)_{j=1:d_i}$ satisfy
    $\bigtimes_{j=1}^{d_i}[\underline{M}_{i,j},\overline{M}_{i,j}]\supseteq\CX_i$.
    Moreover, let $C_{i,\|\cdot\|}\ge 1$ satisfy $\|\BIx_i\|\le C_{i,\|\cdot\|}\|\BIx_i\|_1$ $\forall\BIx_i\in\CX_i$.
    Then, one can explicitly construct $(\FC_i)_{i=1:N}$ and $(\CG_i)_{i=1:N}$ to satisfy the conditions in \ref{settc:compact-testfuncs} as well as 
    \begin{align*}
        \eta(\FC_i)\le \frac{\epsilon-\epsilon_{\mathsf{LSIP}}}{2NL_f},\quad |\CG_i|=\Bigg[\prod_{j=1}^{d_i}\bigg(1+\bigg\lceil\frac{2NL_f(\overline{M}_{i,j}-\underline{M}_{i,j})C_{i,\|\cdot\|}d_i}{\epsilon-\epsilon_{\mathsf{LSIP}}}\bigg\rceil\bigg)\Bigg]-1 \qquad \forall 1\le i\le N.
    \end{align*}
    We refer to Corollary~\ref{cor:momentset-scalability}
    with 
    $p\leftarrow1$, 
    $d\leftarrow d_i$,
    $\CY\leftarrow\CX_i$,
    $(\underline{M}_i,\overline{M}_i)_{i=1:d}\leftarrow(\underline{M}_{i,j},\overline{M}_{i,j})_{j=1:d_i}$,
    $\epsilon\leftarrow\nobreak\frac{\epsilon-\epsilon_{\mathsf{LSIP}}}{NL_f}$,
    $C_{\|\cdot\|}\leftarrow C_{i,\|\cdot\|}$
    for the explicit construction.
    With this choice of $(\CG_i)_{i=1:N}$,
    for every $\epsilon_{\mathsf{LSIP}}$-optimal solution 
    $\hat{\mu}\in\Gamma\big([\mu_1]_{\CG_1},\ldots,[\mu_N]_{\CG_N}\big)$ of \eqref{eqn:mmotlb},
    it holds that every $\tilde{\mu}\in R(\hat{\mu};\mu_1,\ldots,\mu_N)$ is an $\epsilon$-optimal solution of \eqref{eqn:mmot}.
\end{proposition}

\begin{proof}[Proof of Proposition~\ref{prop:momentset-errorcontrol}]
    See Section~\ref{ssec:proof-momentsets}.
\end{proof}

Observe in Proposition~\ref{prop:momentset-errorcontrol} that $m:=\sum_{i=1}^N|\CG_i|$ is exponential in the dimensions $d_1,\ldots,d_N$ of the spaces $\CX_1,\ldots,\CX_N$ while it is polynomial in $N$ when the dimensions $d_1,\ldots,d_N$ are fixed as constants.
Moreover, recall that the support sparsity result about \eqref{eqn:mmotlb} in Proposition~\ref{prop:mmotlb-sparsity} is linear in~$m$.

\begin{proposition}[Explicit expression of $M_{\mathsf{opt}}$ in Theorem~\ref{thm:mmot-complexity} under Setting~\ref{sett:compact}]\label{prop:mmot-complexity-constant}
Under Setting~\ref{sett:compact},
let $D(\BCX):=\sum_{i=1}^N\max_{\BIx_i,\BIx'_i\in\CX_i}\big\{\|\BIx_i-\BIx'_i\|\big\}$,
and assume in addition that 
$\max_{\BIx\in\BCX}\big\{f(\BIx)\big\}=0$,\footnote{This assumption can be satisfied by subtracting a constant from the cost function~$f$, i.e., $f\leftarrow f-\max_{\BIx\in\BCX}\big\{f(\BIx)\big\}$.}
$\bigcup_{C\in\FC_i}C=\CX_i$,
and
$\int_{\CX_i}g_{i,j}\DIFFX{\mu_i}>0$ 
for $j=0,1,\ldots,m_i$,
for $i=1,\ldots,N$.
Then, the following statements hold.
\begin{enumerate}[label=(\roman*),beginpenalty=10000]
    \item\label{props:mmot-complexity-constant-optimizer}
    \eqref{eqn:mmotlb-dual-lsip} admits an optimal solution $(y_0^\star,\BIy^\star)$ that satisfies 
    $\big\|(y_0^\star,\BIy^{\star\TRANSP})^\TRANSP\big\|_2\le 2L_fD(\BCX)(m+1)$.
    
    \item\label{props:mmot-complexity-constant-complexity}
    As a consequence of statement~\ref{props:mmot-complexity-constant-optimizer},
    for any $\epsilon_{\mathsf{LSIP}}>0$,
    there exists an algorithm which computes an $\epsilon_{\mathsf{LSIP}}$-optimal solution of \eqref{eqn:mmotlb-dual-lsip} with
    $O\big(m\log(L_fD(\BCX)m/\epsilon_{\mathsf{LSIP}})\big)$ calls to $\mathtt{Oracle}(\,\cdot\,)$ 
    and $O\big(m^{\omega+1}\log(L_fD(\BCX)m/\epsilon_{\mathsf{LSIP}})\big)$ additional arithmetic operations.
\end{enumerate}
\end{proposition}

\begin{proof}[Proof of Proposition~\ref{prop:mmot-complexity-constant}]
    See Section~\ref{ssec:proof-momentsets}.
\end{proof}

In Proposition~\ref{prop:mmot-complexity-constant},
both the number of calls to $\mathtt{Oracle}(\,\cdot\,)$
and the number of additional arithmetic operations have a polynomial dependence on $m:=\sum_{i=1}^{N}|\CG_i|$.
Note, however, that the computational cost of $\mathtt{Oracle}(\,\cdot\,)$ also depends on~$m$.
Hence, 
combining Proposition~\ref{prop:momentset-errorcontrol} and Proposition~\ref{prop:mmot-complexity-constant} shows that,
for classes of MMOT problems 
in which the dimensions of the underlying spaces $\CX_1,\ldots,\CX_N$ are fixed and in which $\mathtt{Oracle}(\,\cdot\,)$ admits a computationally efficient implementation whose complexity depends polynomially on~$m$, the overall computational complexity of \eqref{eqn:mmotlb-dual-lsip} has a polynomial dependence on the number~$N$ of marginals. 
This is in line with the results of \citet{altschuler2023polynomial} about the computational complexity of MMOT problems with discrete marginals.

\section{Numerical methods}\label{sec:numerics}

In this section, we present our numerical method for approximately solving \eqref{eqn:mmot}. 
Specifically, we first develop a cutting-plane discretization algorithm (i.e., Algorithm~\ref{alg:cp-mmot}) inspired by Conceptual Algorithm~11.4.1 of \citet{goberna1998linear}, 
which, 
for any $\epsilon_{\mathsf{LSIP}}>0$, 
can provide $\epsilon_{\mathsf{LSIP}}$-optimal solutions of both \eqref{eqn:mmotlb-dual-lsip} and \eqref{eqn:mmotlb}. 
Subsequently, we develop an algorithm (i.e., Algorithm~\ref{alg:mmot})
that is able to compute an $\epsilon$-optimal solution of \eqref{eqn:mmot} for any $\epsilon>0$. 
Moreover, it computes both an upper bound and a lower bound for the optimal value of \eqref{eqn:mmot} that are at most~$\epsilon$~apart.

We work under the compact Euclidean case presented in Setting~\ref{sett:compact}.
We develop two sufficient conditions \ref{settc:algo-fullsupport} and \ref{settc:algo-vertices}
in Setting~\ref{sett:algo} below,
which 
ensure the existence of a finite set $\BCX^{\dagger(0)}\subseteq\BCX$ such that 
the linear programming (LP) relaxation of \eqref{eqn:mmotlb-dual-lsip} formed by replacing $\BCX$ with $\BCX^{\dagger(0)}$ has non-empty and bounded superlevel sets.
This is a crucial result to guarantee the convergence of Algorithm~\ref{alg:cp-mmot}, and is
presented in 
Proposition~\ref{prop:boundedness-mmotlb-dual-lsip}\ref{props:boundedness-mmotlb-dual-lsip-boundedness}.
Moreover, 
under the condition \ref{settc:algo-vertices},
Proposition~\ref{prop:boundedness-mmotlb-dual-lsip}\ref{props:boundedness-mmotlb-dual-lsip-explicit} offers a practical method presented in Algorithm~\ref{alg:cp-mmot-initialization} to explicit construct $\BCX^{\dagger(0)}$.

\begin{setting}\label{sett:algo}
    In addition to the conditions \ref{settc:compact-measures}, \ref{settc:compact-cost}, and \ref{settc:compact-testfuncs} in Setting~\ref{sett:compact},
    at least one of the two following conditions \ref{settc:algo-fullsupport} and \ref{settc:algo-vertices} holds.
    \begin{enumerate}[label=\normalfont{(C\arabic*)},beginpenalty=10000]
        \item\label{settc:algo-fullsupport}%
        For $i=1,\ldots,N$, $\support(\mu_i)=\CX_i$ and $\aff(\CX_i\cap C)=\R^{d_i}$ $\forall C\in\FC_i$; 
        in particular, $\aff(\CX_i\cap\nobreak C)=\nobreak\R^{d_i}$ holds if $\inter(C)\cap\inter(\CX_i)\neq\emptyset$.
        
        \item\label{settc:algo-vertices}%
        For $i=1,\ldots,N$, $V(\FC_i)\subseteq\CX_i$ and $\int_{\CX_i}g_{i,j}\DIFFX{\mu_i}>0$ $\forall 0\le j\le m_i$.
    \end{enumerate}
\end{setting}

\setcounter{algocf}{-1}
\begin{algorithm}[t]
\caption{{\bf Explicit construction of $\BCX^{\dagger(0)}$ in Proposition~\ref{prop:boundedness-mmotlb-dual-lsip}}}\label{alg:cp-mmot-initialization}
\KwIn{$\{\BIv_{i,j}\}_{j=0:m_i,\,i=1:N}$, $\big(\int_{\CX_i}g_{i,j}\DIFFX{\mu_i}\big)_{j=0:m_i,\,i=1:N}$.}
\KwOut{$\hat{\mu}^{(0)}$, $\BCX^{\dagger(0)}$, $\mathtt{flag}$.}
\nl$\CQ\leftarrow \emptyset$, $\mathtt{flag}\leftarrow 0$. \\
\nl \For{$i=1,\ldots,N$}{
    \nl $r_i\leftarrow 0$.\\
    \nl \For{$j=0,1,\ldots,m_i$}{
        \nl $\eta_{i,j}\leftarrow \int_{\CX_i}g_{i,j}\DIFFX{\mu_i}$. \\
    }
}
\nl\label{alglin:cpinit-whileloop}\While{$r_i\le m_i$ $\forall 1\le i\le N$}{
    \nl\label{alglin:cpinit-minprob}$\eta_{\min}\leftarrow \bigwedge_{i=1}^{N}\eta_{i,r_i}$. \\
    \nl \For{$i=1,\ldots,N$}{
        \nl\label{alglin:cpinit-subtraction}$\BIx_i\leftarrow\BIv_{i,r_i}$, $\eta_{i,r_i}\leftarrow \eta_{i,r_i}-\eta_{\min}$.\\
    }
    \nl\label{alglin:cpinit-zeroindices}$\BIx\leftarrow(\BIx_1,\ldots,\BIx_N)$, $\CQ\leftarrow\CQ \cup \big\{(\BIx,\eta_{\min})\big\}$, $\CI\leftarrow \big\{i\in\{1,\ldots,N\}:\eta_{i,r_i}=0\big\}$. \\
    \nl\label{alglin:cpinit-flagcond}\If{$|\CI|>1$ and $\bigvee_{i=1}^{N}(m_i-r_i)>0$}{
        \nl $\mathtt{flag}\leftarrow 1$.\\
    }
    \nl\For{each $i\in\CI$}{
        \nl\label{alglin:cpinit-advance}$r_i\leftarrow r_i+1$. \\
    }
}
\nl\label{alglin:cpinit-output}$\hat{\mu}^{(0)}\leftarrow \sum_{(\BIx, \eta)\in \CQ}\eta\delta_{\BIx}$, $\BCX^{\dagger(0)}\leftarrow \support(\hat{\mu}^{(0)})$. \\
\nl \Return $\hat{\mu}^{(0)}$, $\BCX^{\dagger(0)}$, $\mathtt{flag}$. \\
\end{algorithm}

\begin{proposition}[LP relaxations of \eqref{eqn:mmotlb-dual-lsip}]\label{prop:boundedness-mmotlb-dual-lsip}
    Under Setting~\ref{sett:algo},
    the following statements hold.
    \begin{enumerate}[label=(\roman*),beginpenalty=10000]
        \item\label{props:boundedness-mmotlb-dual-lsip-boundedness}%
        The set of optimal solutions of the LSIP problem~\eqref{eqn:mmotlb-dual-lsip} is non-empty and bounded, and there exists a finite set $\BCX^{\dagger(0)}\subseteq\BCX$ such that the following LP relaxation of \eqref{eqn:mmotlb-dual-lsip} has non-empty and bounded superlevel sets: 
        \begin{align}
        \begin{split}
        \maximize_{y_0,\,\BIy}\quad & y_0+\langle\bar{\BIg},\BIy\rangle\\
        \text{subject to}\quad & y_0+\langle\BIg(\BIx),\BIy\rangle\le f(\BIx) \qquad \forall \BIx\in\BCX^{\dagger(0)},\\
        & y_0\in\R,\qquad \BIy\in\R^m.
        \end{split}
        \label{eqn:boundedness-mmotlb-dual-lsip}
        \end{align}
        This means the set $\big\{(y_0,\BIy^\TRANSP)^\TRANSP\in\R^{m+1}:y_0+\langle\BIg(\BIx),\BIy\rangle\le f(\BIx)\;\forall \BIx\in\BCX^{\dagger(0)},\;y_0+\langle\bar{\BIg},\BIy\rangle\ge\alpha\big\}$
        is non-empty and bounded for all $\alpha$ that is less than or equal to the optimal value of (\ref{eqn:boundedness-mmotlb-dual-lsip}).

        \item\label{props:boundedness-mmotlb-dual-lsip-explicit}%
        If the condition \ref{settc:algo-vertices} holds,
        let $(\hat{\mu}^{(0)},\BCX^{\dagger(0)},\mathtt{flag})$ be the outputs of Algorithm~\ref{alg:cp-mmot-initialization}.
        Then, it holds that
        $\hat{\mu}^{(0)}\in\Gamma\big([\mu_1]_{\CG_1},\ldots,[\mu_N]_{\CG_N}\big)$ and the LP problem (\ref{eqn:boundedness-mmotlb-dual-lsip}) has an optimal solution.
        Morevoer, if $\mathtt{flag}=0$, then $\big|\BCX^{\dagger(0)}\big|=m+1$ and the LP problem (\ref{eqn:boundedness-mmotlb-dual-lsip}) with this choice of $\BCX^{\dagger(0)}$ has non-empty and bounded superlevel sets.

    \end{enumerate}
\end{proposition}

\begin{proof}[Proof of Proposition~\ref{prop:boundedness-mmotlb-dual-lsip}]
See Section~\ref{ssec:proof-numerics}.
\end{proof}

\begin{remark}\label{rmk:numerics-assumptions-superlevel-vertices}
    If the condition $V(\FC_i)\subseteq\CX_i$ for all $i=1,\ldots,N$ in \ref{settc:algo-vertices} is not satisfied, then one may extend $\CX_i$ to $\widetilde{\CX}_i:=\bigcup_{C\in\FC_i}C$ for $i=1,\ldots,N$, and then extend the definition of the cost function $f:\BCX\to\R$ to $\overtilde{\BCX}:=\bigtimes_{i=1}^N\widetilde{\CX}_i$ via Proposition~\ref{prop:extension-lip} below.
\end{remark}

\begin{proposition}\label{prop:extension-lip}
    Let $(\CY,d_{\CY})$ be a Polish space and let $D\subseteq\widetilde{D}\subseteq\CY$. Let $f:D\to\R$ be an $L_f$-Lipschitz continuous function for $L_f>0$ and let $\tilde{f}:\widetilde{D}\to\R$ be defined as
    \begin{align*}
    \tilde{f}(x):=\inf_{x'\in D}\big\{f(x')+L_fd_{\CY}(x, x')\big\} \qquad \forall x\in \widetilde{D}.
    \end{align*}
    Then, $\tilde{f}$ is $L_f$-Lipschitz continuous and $\tilde{f}(x)=f(x)$ for all $x\in D$.
 \end{proposition}
    
\begin{proof}[Proof of Proposition~\ref{prop:extension-lip}]
See Section~\ref{ssec:proof-numerics}.
\end{proof}

Algorithm~\ref{alg:cp-mmot} shows our cutting-plane discretization algorithm for solving \eqref{eqn:mmotlb-dual-lsip} and \eqref{eqn:mmotlb}, which is inspired by Conceptual Algorithm~11.4.1 of \citet{goberna1998linear}. Remark~\ref{rmk:cpalgo} explains the assumptions and details of Algorithm~\ref{alg:cp-mmot}. The properties of Algorithm~\ref{alg:cp-mmot} are presented in Proposition~\ref{prop:cpalgo-properties}.

\begin{algorithm}[t]
\caption{{\bf Cutting-plane discretization algorithm for solving \eqref{eqn:mmotlb-dual-lsip} and \eqref{eqn:mmotlb}}}\label{alg:cp-mmot}
\KwIn{$(\CX_i)_{i=1:N}$, $f:\BCX\to\R$, $\BIg:\BCX\to\R^m$, $\bar{\BIg}\in\R^m$, $\BCX^{\dagger(0)}\subseteq\BCX$, $\mathtt{Oracle}(\,\cdot\,)$, $\epsilon_{\mathsf{LSIP}}>0$.}
\KwOut{$\alpha_{\mathsf{relax}}^{\SFU\SFB}$, $\alpha_{\mathsf{relax}}^{\SFL\SFB}$, $\hat{y}_0$, $\hat{\BIy}$, $\hat{\mu}$.}
\nl $r\leftarrow 0$. \\
\nl \While{true}{
\nl Solve~the~LP~problem:~$\alpha^{(r)}\leftarrow\displaystyle\max_{y_0\in\R,\,\BIy\in\R^m}\big\{y_0+\langle\bar{\BIg},\BIy\rangle:y_0+\langle\BIg(\BIx),\BIy\rangle\le f(\BIx)\;\forall \BIx{\in\BCX^{\dagger(r)}}\big\}$, denote the computed primal and dual optimal solutions as $\big(y_0^{(r)},\BIy^{(r)}\big)$ and $\big(\mu^{(r)}_{\BIx}\big)_{\BIx\in\BCX^{\dagger(r)}}$.\label{alglin:cp-lp}\\
\nl Call $\mathtt{Oracle}(\BIy^{(r)})$ and denote the outputs by $(\BIx^\star,s^{(r)})$.\label{alglin:cp-global}\\
\nl\label{alglin:cp-termination1}\If{$y_0^{(r)}-{s}^{(r)}\le\epsilon_{\mathsf{LSIP}}$}{
\nl\label{alglin:cp-termination2}Skip to Line~\ref{alglin:cp-bounds}.\\
}
\nl Let $\BCX^\star\subseteq\BCX$ be a finite set such that $\BIx^\star\in\BCX^\star$.\label{alglin:cp-global-optimizers} \\
\nl $\BCX^{\dagger(r+1)}\leftarrow \BCX^{\dagger(r)}\cup\BCX^\star$.\label{alglin:cp-aggregate}\\ 
\nl $r\leftarrow r+1$.\label{alglin:cp-update}\\
}
\nl $\alpha_{\mathsf{relax}}^{\SFU\SFB}\leftarrow\alpha^{(r)}$, $\alpha_{\mathsf{relax}}^{\SFL\SFB}\leftarrow\alpha^{(r)}-y_0^{(r)}+{s}^{(r)}$.\label{alglin:cp-bounds}\\
\nl $\hat{y}_0\leftarrow {s}^{(r)}$, $\hat{\BIy}\leftarrow\BIy^{(r)}$.\label{alglin:cp-dual}\\
\nl $\hat{\mu}\leftarrow\sum_{\BIx\in\BCX^{\dagger(r)}}\mu^{(r)}_{\BIx}\delta_{\BIx}$.\label{alglin:cp-primal}\\
\nl \Return $\alpha_{\mathsf{relax}}^{\SFU\SFB}$, $\alpha_{\mathsf{relax}}^{\SFL\SFB}$, $\hat{y}_0$, $\hat{\BIy}$, $\hat{\mu}$. \\
\end{algorithm}

\begin{algorithm}[t]
\caption{{\bf Algorithm for solving \eqref{eqn:mmot}}}\label{alg:mmot}
\KwIn{$(\CX_i)_{i=1:N}$, $f:\BCX\to\R$, $\BIg:\BCX\to\R^m$, $\bar{\BIg}\in\R^m$, $\mathtt{Oracle}(\,\cdot\,)$, $\epsilon_{\mathsf{LSIP}}>0$.}
\KwOut{$\alpha^{\SFU\SFB}$, $\alpha^{\SFL\SFB}$, $\tilde{\mu}$, $(\tilde{h}_i)_{i=1:N}$, $\tilde{\epsilon}_{\mathsf{sub}}$.}
\nl\label{alglin:mmot-initcuts}Construct a finite set $\BCX^{\dagger(0)}\subseteq \BCX$ such that the LP relaxation (\ref{eqn:boundedness-mmotlb-dual-lsip}) of \eqref{eqn:mmotlb-dual-lsip} has non-empty and bounded superlevel sets. \\
\nl\label{alglin:mmot-cpalgo}$\big(\alpha_{\mathsf{relax}}^{\SFU\SFB},\alpha_{\mathsf{relax}}^{\SFL\SFB},\hat{y}_0,\hat{\BIy},\hat{\mu}\big)\leftarrow$ the outputs of Algorithm~\ref{alg:cp-mmot} with inputs $\big((\CX_i)_{i=1:N}, f, \BIg, \bar{\BIg}, \BCX^{\dagger(0)}, \mathtt{Oracle}(\,\cdot\,), \epsilon_{\mathsf{LSIP}}\big)$. \\
\nl\label{alglin:mmot-reassembly}Let $\tilde{\mu}\in R(\hat{\mu};\mu_1,\ldots,\mu_N)$. \\
\nl\label{alglin:mmot-bounds}$\alpha^{\SFL\SFB}\leftarrow\alpha_{\mathsf{relax}}^{\SFL\SFB}$, $\alpha^{\SFU\SFB}\leftarrow\int_{\BCX}f\DIFFX{\tilde{\mu}}$, $\tilde{\epsilon}_{\mathsf{sub}}\leftarrow \alpha^{\SFU\SFB} - \alpha^{\SFL\SFB}$. \\
\nl\label{alglin:mmot-dual-prepare}Express $\hat{\BIy}=(\hat{\BIy}_1^\TRANSP,\ldots,\hat{\BIy}_N^\TRANSP)^\TRANSP$, where $\hat{\BIy}_i\in\R^{m_i}$ for $i=1,\ldots,N$.\\
\nl\label{alglin:mmot-dual-forloop}\For{$i=1,\ldots,N$}{
\nl\label{alglin:mmot-dual}$\tilde{h}_i(\,\cdot\,)\leftarrow \frac{\hat{y}_0}{N}+\langle\BIg_i(\,\cdot\,),\hat{\BIy}_i\rangle$.\\
}
\nl\label{alglin:mmot-return}\Return $\alpha^{\SFU\SFB}$, $\alpha^{\SFL\SFB}$, $\tilde{\mu}$, $(\tilde{h}_i)_{i=1:N}$, $\tilde{\epsilon}_{\mathsf{sub}}$.
\end{algorithm}

\begin{remark}[Details of Algorithm~\ref{alg:cp-mmot}]\label{rmk:cpalgo}%
In Algorithm~\ref{alg:cp-mmot}, we 
work under Setting~\ref{sett:algo}.
Below is a list explaining the inputs to Algorithm~\ref{alg:cp-mmot}. 
\begin{itemize}[beginpenalty=10000]
\item%
$(\CX_i)_{i=1:N}$, $f:\BCX\to\R$, $\BIg:\BCX\to\R^{m}$, and $\bar{\BIg}\in\R^m$ are given by \ref{settc:compact-measures}, \ref{settc:compact-cost}, and \ref{settc:compact-testfuncs} in Setting~\ref{sett:compact}.

\item%
$\BCX^{\dagger(0)}\subseteq\BCX$ is a finite set that guarantees the non-emptiness and boundedness of the superlevel sets of the LP relaxation (\ref{eqn:boundedness-mmotlb-dual-lsip}) of \eqref{eqn:mmotlb-dual-lsip};
see Proposition~\ref{prop:boundedness-mmotlb-dual-lsip}.

\item%
$\mathtt{Oracle}(\,\cdot\,)$ is the global minimization oracle in Definition~\ref{def:mmot-oracle}.

\item%
$\epsilon_{\mathsf{LSIP}}>0$ is a pre-specified numerical tolerance value (see Proposition~\ref{prop:cpalgo-properties}).%
\end{itemize}%

In the following, we provide explanations of some lines in Algorithm~\ref{alg:cp-mmot}.
\begin{itemize}[beginpenalty=10000]%
\item%
Line~\ref{alglin:cp-lp} solves an LP relaxation of \eqref{eqn:mmotlb-dual-lsip} where the semi-infinite constraint is replaced by finitely many constraints each corresponding to an element of $\BCX^{\dagger(r)}$. When solving the LP relaxation in Line~\ref{alglin:cp-lp} by the dual simplex algorithm (see, e.g., \citep[Chapter~6.4]{vanderbei2020linear}) or the interior point algorithm (see, e.g., \citep[Chapter~18]{vanderbei2020linear}), one can obtain the corresponding optimal solution of the dual LP problem from the outputs of these algorithms.

\item%
Line~\ref{alglin:cp-global-optimizers} allows more than one constraint to be generated in each iteration. 
The finite set $\BCX^\star\subseteq\BCX$ can be thought of as a set of approximately optimal solutions of the global minimization problem solved by $\mathtt{Oracle}(\BIy^{(r)})$.
\end{itemize}%
\end{remark}%

\begin{proposition}[Properties of Algorithm~\ref{alg:cp-mmot}]\label{prop:cpalgo-properties}
Under Setting~\ref{sett:algo}, the following statements hold. 
\begin{enumerate}[label=(\roman*),beginpenalty=10000]
\item\label{props:cpalgo-termination}%
Algorithm~\ref{alg:cp-mmot} terminates after finitely many iterations. 

\item\label{props:cpalgo-bounds}%
$\alpha_{\mathsf{relax}}^{\SFL\SFB}\le$ \eqref{eqn:mmotlb-dual-lsip} $\le\alpha_{\mathsf{relax}}^{\SFU\SFB}$ where $\alpha_{\mathsf{relax}}^{\SFU\SFB}-\alpha_{\mathsf{relax}}^{\SFL\SFB}\le\epsilon_{\mathsf{LSIP}}$.

\item\label{props:cpalgo-dual}%
$(\hat{y}_0,\hat{\BIy})$ is an $\epsilon_{\mathsf{LSIP}}$-optimal solution of \eqref{eqn:mmotlb-dual-lsip} with $\hat{y}_0+\langle\bar{\BIg},\hat{\BIy}\rangle=\alpha_{\mathsf{relax}}^{\SFL\SFB}$. 

\item\label{props:cpalgo-primal}%
$\hat{\mu}$ is an $\epsilon_{\mathsf{LSIP}}$-optimal solution of \eqref{eqn:mmotlb} with $\int_{\BCX}f\DIFFX{\hat{\mu}}=\alpha_{\mathsf{relax}}^{\SFU\SFB}$. 
Moreover, $\support(\hat{\mu})$ is finite.
\end{enumerate}
\end{proposition}

\begin{proof}[Proof of Proposition~\ref{prop:cpalgo-properties}]
See Section~\ref{ssec:proof-numerics}.
\end{proof}

\begin{remark}
    Our cutting-plane discretization algorithm, i.e., Algorithm~\ref{alg:cp-mmot}, uses the idea of adaptive constraint generation, which is also present in column generation algorithms for large-scale linear programming problems.
    Column generation algorithms have been used for solving optimal transport and related problems involving discrete probability measures, see, e.g.,
    \citet{friesecke2022genetic, bogwardt2022column}.
    The difference between Algorithm~\ref{alg:cp-mmot} and column generation algorithms is that the LSIP problem \eqref{eqn:mmotlb-dual-lsip} has infinitely many constraints in general, while column generation algorithms are designed for solving large-scale linear programming problems involving only finitely many decision variables and constraints. 
    Therefore, the justification for the convergence of Algorithm~\ref{alg:cp-mmot} requires more delicate analyses. 
\end{remark}

The concrete procedure for computing an $\epsilon$-optimal solution of \eqref{eqn:mmot} is presented in Algorithm~\ref{alg:mmot}. 
Theorem~\ref{thm:mmotalgo} shows the properties of Algorithm~\ref{alg:mmot}.

%\pagebreak

\begin{theorem}[Properties of Algorithm~\ref{alg:mmot}]\label{thm:mmotalgo}
Under Setting~\ref{sett:algo}, 
let $(\rho_i)_{i=1:N}$ satisfy $\rho_i\ge \specialoverline{W}_{1}(\mu_i,[\mu_i]_{\CG_i})$ for $i=1,\ldots,N$, and let $\epsilon_{\mathsf{theo}}:=\epsilon_{\mathsf{LSIP}}+L_f\sum_{i=1}^N\rho_i$.
Then, the following statements hold.
\begin{enumerate}[label=(\roman*), beginpenalty=10000]
\item\label{thms:mmotalgo-bounds}%
$\alpha^{\SFL\SFB}\le$ \eqref{eqn:mmot} $\le\alpha^{\SFU\SFB}$
where $\alpha^{\mathsf{UB}}-\alpha^{\mathsf{LB}}=\tilde{\epsilon}_{\mathsf{sub}}\le \epsilon_{\mathsf{theo}}$.

\item\label{thms:mmotalgo-primal}%
$\tilde{\mu}$ is an $\tilde{\epsilon}_{\mathsf{sub}}$-optimal solution of \eqref{eqn:mmot} with $\int_{\BCX}f\DIFFX{\tilde{\mu}}=\alpha^{\mathsf{UB}}$.

\item\label{thms:mmotalgo-dual}%
$(\tilde{h}_i)_{i=1:N}$ is an $\tilde{\epsilon}_{\mathsf{sub}}$-optimal solution of \eqref{eqn:mmot-dual} with $\sum_{i=1}^{N}\int_{\CX_i}\tilde{h}_i\DIFFX{\mu_i}=\alpha^{\mathsf{LB}}$.

\item\label{thms:mmotalgo-control}%
Let us assume further that, for $i=1,\ldots,N$, $\support(\mu_i)=\CX_i$
and that
there exists a finite collection $\widetilde{\FC}_i$ of $d_i$-simplices satisfying
$\bigcup_{C\in\widetilde{\FC}_i}C=\CX_i$ and $C_1\cap C_2\in \FF(C_1)\cap \FF(C_2)$ $\forall C_1,C_2\in\nobreak\widetilde{\FC}_i$.
Then, for any $\epsilon>0$ and any $\epsilon_{\mathsf{LSIP}}\in(0,\epsilon)$, 
one can explicitly construct $(\FC_i)_{i=1:N}$ 
via the iterative bisection procedure in Proposition~\ref{prop:momentset-simplex}\ref{props:momentset-simplex-bisection}
to satisfy the conditions in \ref{settc:compact-testfuncs} 
as well as $\eta(\FC_i)\le\frac{\epsilon-\epsilon_{\mathsf{LSIP}}}{2NL_f}$ for $i=1,\ldots,N$.
Subsequently, using $\BIg(\,\cdot\,)$, $\bar{\BIg}$ constructed by \ref{settc:compact-testfuncs} in Algorithm~\ref{alg:mmot} guarantees $\tilde{\epsilon}_{\mathsf{sub}}\le \epsilon$.
\end{enumerate}
\end{theorem}

\begin{proof}[Proof of Theorem~\ref{thm:mmotalgo}]
See Section~\ref{ssec:proof-numerics}.
\end{proof}

\begin{remark}[Sub-optimality estimate in Algorithm~\ref{alg:mmot} and its a priori upper bound]\label{rmk:mmotalgo-control-gap}
    Theorem~\ref{thm:mmotalgo}\ref{thms:mmotalgo-control} is a theoretical statement which says that, for any given $\epsilon>0$, one can explicitly choose the inputs of Algorithm~\ref{alg:mmot} such that an $\epsilon$-optimal solution of \eqref{eqn:mmot} can be computed. 
    However, from a numerical viewpoint, it is more practical to specify the inputs 
    $\epsilon_{\mathsf{LSIP}}>0$,
    $\BIg:\BCX\to\R^m$,
    $\bar{\BIg}\in\R^m$ of Algorithm~\ref{alg:mmot} based on, for example, the available budget of computation, and subsequently evaluate the sub-optimality of the computed solution $\tilde{\mu}$ of \eqref{eqn:mmot} from the computed output $\tilde{\epsilon}_{\mathsf{sub}}$; see Theorem~\ref{thm:mmotalgo}\ref{thms:mmotalgo-primal}. 
    The term $\epsilon_{\mathsf{theo}}$ in Theorem~\ref{thm:mmotalgo} is a theoretical upper bound for the sub-optimality estimate $\tilde{\epsilon}_{\mathsf{sub}}$ computed by Algorithm~\ref{alg:mmot} that is based on the upper estimates $(\rho_i)_{i=1:N}$ of $\big(\specialoverline{W}_{1}(\mu_i,[\mu_i]_{\CG_i})\big)_{i=1:N}$. 
    It is therefore called an 
    a priori error bound or an
    a priori upper bound for~$\tilde{\epsilon}_{\mathsf{sub}}$ and it can be computed independent of Algorithm~\ref{alg:mmot}. 
    The computed value of $\tilde{\epsilon}_{\mathsf{sub}}$ is typically much less conservative compared to its a priori upper bound $\epsilon_{\mathsf{theo}}$, as we will demonstrate in the numerical experiments in Section~\ref{sec:experiments}. 
\end{remark}

\section{Numerical experiments}\label{sec:experiments}

In this section, we demonstrate Algorithm~\ref{alg:mmot} in three numerical experiments.
In Section~\ref{ssec:experiments-fluid}, we apply Algorithm~\ref{alg:mmot} to an MMOT problem that stems from fluid dynamics, in which the marginals are one-dimensional and the cost function admits a graphical structure. 
In Section~\ref{ssec:experiments-barycenter}, we examine the well-known MMOT formulation of the Wasserstein barycenter problem where we use Algorithm~\ref{alg:mmot} to approximately compute the Wasserstein barycenter of two-dimensional probability measures. 
In Section~\ref{ssec:experiments-CPWA}, we showcase the performance of Algorithm~\ref{alg:mmot} in an MMOT problem with a continuous piece-wise affine cost function and $N=100$ one-dimensional marginals. 
The code used in this work is available on our GitHub repository.\footnote{URL: \url{https://github.com/qikunxiang/MultiMarginalOptimalTransport}.}

\subsubsection*{Comparison with existing methods}
In Experiment~1 and Experiment~2 (Section~\ref{ssec:experiments-fluid} and Section~\ref{ssec:experiments-barycenter}),
we compare Algorithm~\ref{alg:mmot} with two methods for numerically solving MMOT problems with non-discrete marginals.
The first method is
developed by
\citet*{genevay2016stochastic} 
and utilizes reproducing kernel Hilbert space (RKHS) parametrizations.
The second method based on neural networks (NN) is developed by \citet*{eckstein2019computation}.%
\footnote{The code that implements these two methods are adapted from the code available on the GitHub repository of \citet{eckstein2019computation}: \url{https://github.com/stephaneckstein/OT_Comparison}.}
Both methods parametrize and solve the following regularized version of 
\eqref{eqn:mmot-dual}:
\begin{align}
    \sup_{(h_i\in\CC_b(\CX_i))_{i=1:N}}\Bigg\{\Bigg(\sum_{i=1}^{N}\int_{\CX_i}h_i\DIFFX{\mu_i}\Bigg)-\int_{\BCX}\frac{1}{\gamma}\rho\Big(\gamma\Big(\big({\textstyle\sum_{i=1}^{N}}h_i\circ\pi_i(\BIx)\big)-f(\BIx)\Big)\Big)\DIFFM{\theta}{\DIFF\BIx}\Bigg\},
    \tag{$\mathsf{OT}^*_{\mathsf{reg}}$}
    \label{eqn:mmot-dual-regularized}
\end{align}
where 
$\CC_b(\CX_i)$ denotes the set of continuous and bounded functions on $\CX_i$ for $i=1,\ldots,N$,
$\rho:\R\to\R_+$ is a differentiable, strictly convex, and increasing function called the regularization function,
$\gamma>0$ is called the inverse regularization parameter,
and $\theta\in\CP(\BCX)$ is called the sampling measure.
For $i=1,\ldots,N$,
the RKHS-based method of \citet{genevay2016stochastic} parametrizes $h_i\in\CC_b(\CX_i)$
by an RKHS $\big(\CH_i,\langle\,\cdot\,,\cdot\,\rangle_{\CH_i}\big)$ with respect to a positive definite kernel function $\kappa_i:\CX_i\times\CX_i\to\R$,
whereas the NN-based method of \citet{eckstein2019computation} parametrizes $h_i\in\CC_b(\CX_i)$ 
by a neural network.
Each method solves the parametrized maximization problem via a variant of the stochastic gradient descent (SGD) algorithm.
Moreover, \citep[Theorem~2.2]{eckstein2019computation} shows that 
if $(\hat{h}_i)_{i=1:N}$ is an optimal solution of \eqref{eqn:mmot-dual-regularized},
then $\hat{\mu}_{\mathsf{NN}}\in\CP(\BCX)$ defined by
\begin{align}
    \frac{\DIFF\hat{\mu}_{\mathsf{NN}}}{\DIFF\theta}:=\rho'\Big(\gamma\Big(\big({\textstyle\sum_{i=1}^N}\hat{h}_i\circ\pi_i\big)-f\Big)\Big)
    \label{eqn:mmot-dual-regularized-primalopt}
\end{align}
is an optimal solution of the following dual optimization problem of \eqref{eqn:mmot-dual-regularized}:
\begin{align}
    \inf_{\mu\in\Gamma(\mu_1,\ldots,\mu_N)}\bigg\{\int_{\BCX}f\DIFFX{\mu} + \int_{\BCX}\frac{1}{\gamma}\rho^*\bigg(\frac{\DIFF \mu}{\DIFF\theta}\bigg)\DIFFX{\theta}\bigg\},
    \tag{$\mathsf{OT}_{\mathsf{reg}}$}
    \label{eqn:mmot-primal-regularized}
\end{align}
where $\rho^*(z):=\sup_{v\in\R}\big\{zv-\rho(v)\big\}$ $\forall z\in\R$ denotes the convex conjugate of $\rho$.
However, we would like to remark that neither method is able to compute feasible solutions of \eqref{eqn:mmot-dual} nor \eqref{eqn:mmot} in practice.
Despite that the optimal value of \eqref{eqn:mmot-dual-regularized} is an upper bound for the optimal value of \eqref{eqn:mmot-dual}, 
it is impractical to obtain exact optimal solutions of \eqref{eqn:mmot-dual-regularized} in general, 
and the objective value of any approximately optimal solution of \eqref{eqn:mmot-dual-regularized} is not guaranteed to be an upper bound nor a lower bound for the optimal value of \eqref{eqn:mmot-dual}.
Moreover, $\hat{\mu}_{\mathsf{NN}}$ constructed via (\ref{eqn:mmot-dual-regularized-primalopt}) with respect to an approximately optimal solution $(\tilde{h}_i)_{i=1:N}$ of \eqref{eqn:mmot-dual-regularized} is generally infeasible for \eqref{eqn:mmot-primal-regularized} and thus also infeasible for~\eqref{eqn:mmot}.

In the following, let us present the concrete settings of these two methods in our numerial experiments.
In \eqref{eqn:mmot-dual-regularized}, we set $\theta=\mu_1\otimes\cdots\otimes\mu_N$ to be the product measure of $\mu_1,\ldots,\mu_N$,
set $\rho(v):=\big((v)^+\big)^{2}$ $\forall v\in\R$,
and vary the inverse regularization parameter $\gamma>0$ among a number of preselected values in each experiment.
We choose to use the quadratic regularization function rather than the exponential regularization function due to its stability, as reported by \citet{eckstein2019computation}.
For any candidate solution $(\tilde{h}_i)_{i=1:N}$ computed by either method,
we approximate its objective value with respect to \eqref{eqn:mmot-dual-regularized} via Monte Carlo integration with 50000 independent samples. 
Moreover, we run each SGD algorithm 20 times for each problem instance in order to examine its stability. 

In the RKHS-based method, 
we use the same Laplace kernel 
$\kappa_i(\BIx,\BIx'):=\exp\left({-\frac{\|\BIx-\BIx'\|_2}{\varsigma}}\right)$ $\forall\BIx,\BIx'\in\nobreak\CX_i$
for $i=1,\ldots,N$,
where $\varsigma>0$ is called the bandwidth parameter.
We use $\varsigma=\frac{1}{2}$ in Experiment~1 and $\varsigma=8$ in Experiment~2.
The implementation of the SGD algorithm is adapted from Algorithm~3 in \citep{genevay2016stochastic}, with two differences:
we use the quadratic regularization function rather than the exponential regularization function,
and we adjust the step size by an exponential decay scheme.
Concretely, the step size in the SGD algorithm decays from $6\times 10^{-5}$ to $10^{-7}$ over $6\times 10^{5}$ iterations.
These hyperparameters are selected based on initial trial runs.

In the NN-based method, we parametrize each $h_i$ with a \mbox{5-layer}
ReLU network with $64d$ neurons per layer, where $d=1$ in Experiment~1 and $d=2$ in Experiment~2 are the dimensions of the underlying spaces $\CX_1,\ldots,\CX_N$.
We use the Adam optimizer with parameters $\beta_1=0.99$, $\beta_2=0.995$, where the step size decreases from $10^{-4}$ to $5\times 10^{-7}$ over $50000+10000Nd$ iterations via a second-order polynomial decay scheme,
and each iteration utilizes 2048 independent samples from the marginals $\mu_1,\ldots,\mu_N$ and 4096 independent samples from the sampling measure $\theta$ to approximate the objective function of \eqref{eqn:mmot-dual-regularized}.
These hyperparameter settings follow the source code of \citet{eckstein2019computation}.
Moreover, for any candidate solution $(\tilde{h}_i)_{i=1:N}$ of \eqref{eqn:mmot-dual-regularized},
we follow the source code of \citet{eckstein2019computation} to construct $\hat{\mu}_{\mathsf{NN}}$ by (\ref{eqn:mmot-dual-regularized-primalopt}) and generate independent random samples $\hat{\mu}_{\mathsf{NN}}$ via rejection sampling.

\subsection{Experiment~1: fluid dynamics}\label{ssec:experiments-fluid}

In the first numerical experiment, we consider the optimization problem proposed by 
\citet{brenier1989least, brenier1993dual, brenier1999minimal, brenier2008generalized}
as a relaxation of the Euler equation of incompressible fluid with given initial (i.e., at time $t=0$) and final distributions (i.e., at time $t=1$), expressed by a volume preserving map $\Xi:[0,1]^d\to[0,1]^d$ satisfying
$\FL_{[0,1]^d}\circ \Xi^{-1}=\FL_{[0,1]^d}$,
where $\FL_{[0,1]^d}$ denotes the Lebesgue measure on $[0,1]^d$.
\citet[Section~4.3]{benamou2015iterative} proposed to discretize the original problem of Brenier both in time and in space, which results in a discrete MMOT problem with $N\in\N$ marginals representing $N$ discrete time points where each marginal is a uniform discrete measure supported on a grid in $[0,1]^d$. 
We consider the $d=1$ variant of the problem, which can be seen as considering the projection of the fluid in $[0,1]^3$ onto one of the axes. 
This discrete MMOT problem 
is presented as follows:
\begin{align*}
    \inf_{\mu\in\Gamma(\hat{\mu}_{\CCCD},\ldots,\hat{\mu}_{\CCCD})}\bigg\{\int_{\CCCD^N}\big(x_N-\sigma(x_1)\big)^2+{\textstyle\sum_{i=1}^{N-1}}(x_{i+1}-x_i)^2\DIFFM{\mu}{\DIFF x_1,\ldots,\DIFF x_N}\bigg\},
\end{align*}
where $\CCCD$ contains $|\CCCD|\in\N$ equally-spaced points in $[0,1]$, 
$\hat{\mu}_{\CCCD}:=\frac{1}{|\CCCD|}\sum_{x\in\CCCD}\delta_{x}$,
and $\sigma:\CCCD\to\CCCD$ is a permutation of the points in $\CCCD$ which can be seen as a discretization of the volume preserving map $\Xi:[0,1]\to[0,1]$.
Here, $x_1$ represents the initial position of a fluid particle at time $t=0$, $x_i$ represents the position of the particle at time $t=\frac{i-1}{N}$ for $i=2,\ldots,N$, and $\sigma(x_1)$ represents the final position of the particle at time $t=1$. 
The cost function in the above problem exhibits a graphical structure 
corresponding to the circular graph:
$(\CV,\CE)$ where 
$\CV:=\{1,\ldots,N\}$,
$\CE:=\big\{\{1,2\},\ldots,\{N-\nobreak1,N\}, \{1,N\}\big\}$.
This discrete MMOT problem has been numerically studied by \citet{benamou2015iterative, ba2022accelerating}, and \citet{altschuler2023polynomial}. 

We adopt an alternative approach with $N$ discrete time points but \textit{without spatial discretization}. 
Specifically, we consider the following MMOT problem:
\begin{align}
    \inf_{\mu\in\Gamma(\mu_1,\ldots,\mu_N)}\bigg\{\int_{[0,1]^N}\big(x_N-\Xi(x_1)\big)^2+{\textstyle\sum_{i=1}^{N-1}}(x_{i+1}-x_{i})^2\DIFFM{\mu}{\DIFF x_1,\ldots,\DIFF x_N}\bigg\},
    \label{eqn:experiments-fluid-OT}
\end{align}
where ${\CX_1=\cdots=\CX_N}=[0,1]$ and $\mu_1=\cdots=\mu_N=\FL_{[0,1]}$.
Subsequently, we can use Algorithm~\ref{alg:mmot} to compute 
lower and upper bounds for the optimal value of (\ref{eqn:experiments-fluid-OT}), 
an approximately optimal solution $\tilde{\mu}$ of (\ref{eqn:experiments-fluid-OT}), and a sub-optimality estimate.
The following proposition covers the details when applying Algorithm~\ref{alg:mmot} to approximately solve (\ref{eqn:experiments-fluid-OT}).

\begin{proposition}[Experiment~1]\label{prop:experiments-fluid}
    For $i=1,\ldots,N$, 
    let $\CX_i=[0,1]$, $d_{\CX_i}(x,z):=|x-z|$ $\forall x,z\in[0,1]$,
    and let $\mu_i:=\FL_{[0,1]}$.
    Let $\Xi:[0,1]\to[0,1]$ be an $L_{\Xi}$-Lipschitz continuous function for $L_{\Xi}>0$ which satisfies $\FL_{[0,1]}\circ \Xi^{-1}=\FL_{[0,1]}$.
    Moreover, let $\BCX:=[0,1]^N$ and let $f:\BCX\to\R$ be defined as follows:
    \begin{align}
        f(x_1,\ldots,x_N):= -2x_N\Xi(x_1) - 2\sum_{i=1}^{N-1}x_ix_{i+1} \qquad \forall (x_1,\ldots,x_N)\in[0,1]^N.
        \label{eqn:experiments-fluid-cost}
    \end{align}
    Let us consider the problem \eqref{eqn:mmot} and its dual optimization problem \eqref{eqn:mmot-dual} with respect to the cost function $f$ in (\ref{eqn:experiments-fluid-cost}).
    Then, the following statements hold.
    \begin{enumerate}[label=(\roman*),beginpenalty=10000]
        \item\label{props:experiments-fluid-setting}%
        The conditions \ref{settc:compact-measures} and \ref{settc:compact-cost} in Setting~\ref{sett:compact} are satisfied with respect to $L_f\leftarrow 2L_{\Xi}+2$.
        
        \item\label{props:experiments-fluid-shift}%
        Let $\tilde{\alpha}^\star$ denote the optimal value of \eqref{eqn:mmot}.
        Then, the optimal value of (\ref{eqn:experiments-fluid-OT}) is equal to $\tilde{\alpha}^\star+\frac{2N}{3}$.
        Likewise, if $\tilde{\alpha}^{\mathsf{LB}}$, $\tilde{\alpha}^{\mathsf{UB}}$ are lower and upper bounds for the optimal value of 
        \eqref{eqn:mmot},
        then 
        $\alpha^{\mathsf{LB}}:=\tilde{\alpha}^{\mathsf{LB}}+\frac{2N}{3}$, 
        $\alpha^{\mathsf{UB}}:=\tilde{\alpha}^{\mathsf{UB}}+\frac{2N}{3}$
        are lower and upper bounds for the optimal value of~(\ref{eqn:experiments-fluid-OT}).

        \item\label{props:experiments-fluid-primaldual-solutions}%
        Let $\epsilon>0$,
        let $\tilde{\mu}$ be an $\epsilon$-optimal solution of \eqref{eqn:mmot}, and let $(\tilde{h}^\dagger_i)_{i=1:N}$ be an $\epsilon$-optimal solution of \eqref{eqn:mmot-dual}.
        Let
        $\tilde{h}_1(x):=\tilde{h}^\dagger_1(x)+x^2+\Xi(x)^2$ $\forall x\in[0,1]$,
        and let 
        $\tilde{h}_i(x):=\tilde{h}^\dagger_i(x)+2x^2$ $\forall x\in[0,1]$ for $i=2,\ldots,N$.
        Then, $\tilde{\mu}$ is an $\epsilon$-optimal solution of (\ref{eqn:experiments-fluid-OT}),
        and $(\tilde{h}_i)_{i=1:N}$ is an $\epsilon$-optimal solution of the dual optimization problem of (\ref{eqn:experiments-fluid-OT}).
    \end{enumerate}
    In the following, let $\hat{\mu}\in\CP\big([0,1]^N\big)$ be given by $\hat{\mu}:=\sum_{j=1}^J a_j \delta_{\BIx_j}$, 
    where $J\in\N$, 
    $(a_j)_{j=1:J}\subset(0,1)$, 
    $\sum_{j=1}^Ja_j=1$, and $\big\{\BIx_j=(x_{1,j},\ldots,x_{N,j}):1\le j\le J\big\}$ 
    are distinct points satisfying $x_{1,1}\le x_{1,2}\le\nobreak \cdots\le x_{1,J}$. 
    For $i=2,\ldots,N$, 
    let $\sigma_i:\{1,\ldots,J\}\to\{1,\ldots,J\}$ be a bijection that satisfies $x_{i,\sigma_i(1)}\le x_{i,\sigma_i(2)}\le\nobreak \cdots \le x_{i,\sigma_i(J)}$.
    Moreover, let $(c_{i,j})_{j=1:J,\,i=1:N}$ be defined as follows:
    \begin{align*}
        c_{1,j}:=\sum_{l=1}^{j-1}a_l \qquad \forall 1\le j\le J, \qquad\qquad c_{i,j}:=\sum_{l=1}^{\sigma_i^{-1}(j)-1}a_{\sigma_i(l)} \qquad \forall 1\le j\le J,\; \forall 2\le i\le N.
    \end{align*}
    Furthermore, let $\Xi_1:[0,1]\to[0,1]$ be the identity map, and let $\big(\Xi_i:[0,1]\to\R\big)_{i=2:N}$ be defined as follows:
    \begin{align*}
        \Xi_i(x):= \begin{cases}
            c_{i,j} + x - c_{1,j} & \forall x\in [c_{1,j}, c_{1,j+1}), \; \forall 1\le j\le J-1, \\ 
            c_{i,J} + x - c_{1,J} & \forall x\in[c_{1,J},1],
        \end{cases}\qquad \forall 2\le i\le N.
    \end{align*}
    \noindent Then, \widowpenalties-1000
    \begin{enumerate}[label=(\roman*),beginpenalty=10000]
        \setcounter{enumi}{3}
        \item\label{props:experiments-fluid-reassembly}%
        $\Xi_i\big([0,1]\big)=[0,1]$ for $i=1,\ldots,N$, and
        $\tilde{\mu}:=\mu_1\circ (\Xi_1,\ldots,\Xi_N)^{-1}\in\CP\big([0,1]^N\big)$ satisfies $\tilde{\mu}\in R(\hat{\mu};\mu_1,\ldots,\mu_N)$
        as well as 
        \begin{align*}
            \int_{\BCX}f\DIFFX{\tilde{\mu}}=\int_{0}^{1}-2\Xi_N(x)\Xi(x)-2{\textstyle\sum_{i=1}^{N-1}}\Xi_i(x)\Xi_{i+1}(x) \DIFFX{x}.
        \end{align*}
    \end{enumerate}
\end{proposition}

\begin{proof}[Proof of Proposition~\ref{prop:experiments-fluid}]
    See Section~\ref{ssec:proof-experiments}.
\end{proof}

Using the results in Proposition~\ref{prop:experiments-fluid},
we can use Algorithm~\ref{alg:mmot} to approximately solve 
\eqref{eqn:mmot} with respect to the cost function $f$ defined in (\ref{eqn:experiments-fluid-cost}),
where the construction of the reassembly~$\tilde{\mu}$ in Line~\ref{alglin:mmot-reassembly}
and the computation of $\int_{\BCX}f\DIFFX{\tilde{\mu}}$ in Line~\ref{alglin:mmot-bounds}
are carried out via Proposition~\ref{prop:experiments-fluid}\ref{props:experiments-fluid-reassembly}.
We let $\big(\tilde{\alpha}^{\mathsf{UB}},\tilde{\alpha}^{\mathsf{LB}},\tilde{\mu},(\tilde{h}^\dagger_i)_{i=1:N},\tilde{\epsilon}_{\mathsf{sub}}\big)$
denote the outputs of Algorithm~\ref{alg:mmot},
and set
$\alpha^{\mathsf{LB}}:=\tilde{\alpha}^{\mathsf{LB}}+\frac{2N}{3}$,
$\alpha^{\mathsf{UB}}:=\tilde{\alpha}^{\mathsf{UB}}+\frac{2N}{3}$,
$\tilde{h}_1(x):=\tilde{h}^\dagger_1(x)+x^2+\Xi(x)^2$ $\forall x\in[0,1]$,
and 
$\tilde{h}_i(x):=\tilde{h}^\dagger_i(x)+2x^2$ $\forall x\in[0,1]$
for $i=2,\ldots,N$.
Theorem~\ref{thm:mmotalgo},
Proposition~\ref{prop:experiments-fluid}\ref{props:experiments-fluid-shift}, and Proposition~\ref{prop:experiments-fluid}\ref{props:experiments-fluid-primaldual-solutions} then
guarantee that
$\alpha^{\mathsf{LB}}$, $\alpha^{\mathsf{UB}}$
are lower and upper bounds for the optimal value of~(\ref{eqn:experiments-fluid-OT}),
$\tilde{\mu}$ is an $\epsilon$-optimal solution of (\ref{eqn:experiments-fluid-OT}),
and $(\tilde{h}_i)_{i=1:N}$ is an $\epsilon$-optimal solution of the dual optimization problem of (\ref{eqn:experiments-fluid-OT}).
Moreover, the functions $(\Xi_i)_{i=1:N}$ constructed by Proposition~\ref{prop:experiments-fluid}\ref{props:experiments-fluid-reassembly}
are volume preserving maps, 
where $\Xi_2(x),\ldots,\Xi_N(x),\Xi(x)$
can be interpreted as 
the positions of a fluid particle at time $t=\frac{1}{N},\frac{2}{N},\ldots,\frac{N-1}{N},1$, respectively, given its initial position $x\in[0,1]$ at $t=0$.

\subsubsection*{Experimental settings}
We consider the two volume preserving maps $\Xi^{(1)},\Xi^{(2)}:[0,1]\to[0,1]$ below:
\begin{align*}
    \Xi^{(1)}(x):=\begin{cases}
        2x & \textstyle \forall x\in\big[0,\frac{1}{2}\big], \\
        2-2x & \forall \textstyle x\in\big(\frac{1}{2},1\big],
    \end{cases}
    \qquad 
    \Xi^{(2)}(x):=\begin{cases}
        1-4x & \textstyle \forall x\in\big[0,\frac{1}{4}\big], \\
        4x-1 & \textstyle \forall x\in\big(\frac{1}{4},\frac{1}{2}\big], \\
        3-4x & \textstyle \forall x\in\big(\frac{1}{2},\frac{3}{4}\big], \\
        4x-3 & \textstyle \forall x\in\big(\frac{3}{4},1\big].
    \end{cases}
\end{align*}
For each volume preserving map, we perform time discretization with respect to $N=20$ discrete time points.
Moreover, we let $\FC=\big\{\big[0,\frac{1}{m_0}\big],\big[\frac{1}{m_0},\frac{2}{m_0}\big],\ldots,\big[\frac{m_0-1}{m_0},1\big]\big\}$
where $m_0\in\N$ is varied between 4~and~128,
and we construct
$\CG_1=\cdots=\CG_N$
as well as $\BIg(\,\cdot\,)$, $\bar{\BIg}$
by \ref{settc:compact-testfuncs}
with respect to $\FC_i\leftarrow \FC$ $\forall 1\le\nobreak i\le\nobreak N$.
Thus, we have $|\CG_1|=\cdots=|\CG_N|=m_0$, that is, we use 4~to~128 test functions for each marginal.
Concretely, the test functions $(g_{i,j})_{j=1:m_0,\,i=1:N}$ are defined as follows:
\begin{align*}
g_{i,j}(x)&:=\textstyle {\big(m_0x-j+1\big)^+}\wedge \big(j+1-m_0 x\big)^+ \qquad\forall x\in[0,1],\; \forall 1\le j\le m_0,\; \forall 1\le i\le N.
%\label{eqn:experiments-fluid-basis}
\end{align*}
Furthermore, we set $\epsilon_{\mathsf{LSIP}}=10^{-5}$ for all values of~$m_0$ in Algorithm~\ref{alg:mmot}.
We remark that the graphical structure in the cost function (\ref{eqn:experiments-fluid-cost}) guarantees that $\mathtt{Oracle}(\,\cdot\,)$ in Definition~\ref{def:mmot-oracle} can be tractably implemented with high efficiency; see the discussion of \citet[Section~5]{altschuler2023polynomial}.

\begin{figure}[t]
    \includegraphics[width=0.99\linewidth]{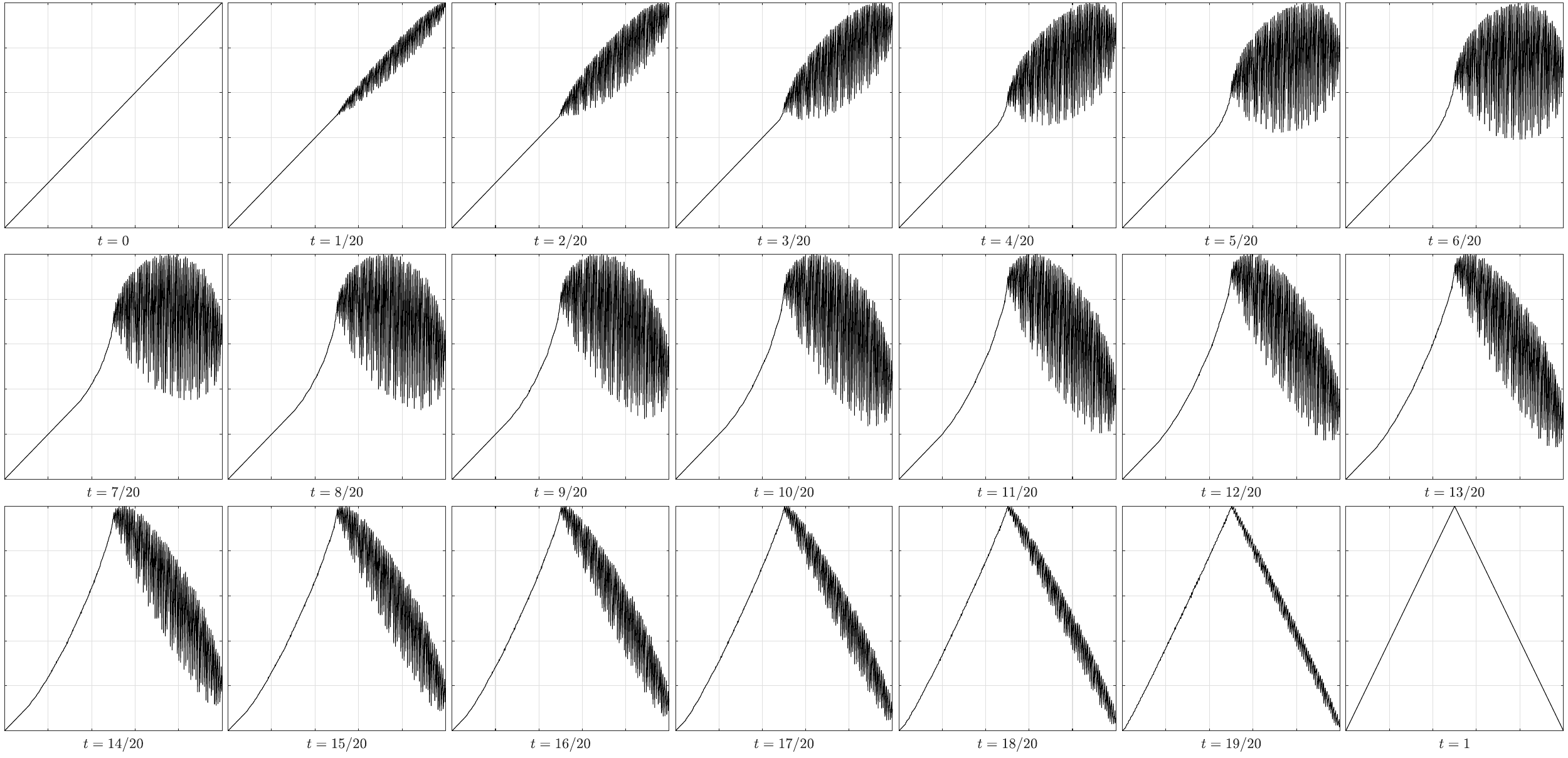}
    
    \caption{\textbf{Experiment~1} -- Volume preserving maps $\Xi^{(1)}_1,\ldots,\Xi^{(1)}_N$.}\label{fig:fluid-maps1}
\end{figure}

\begin{figure}[t]
    \includegraphics[width=0.99\linewidth]{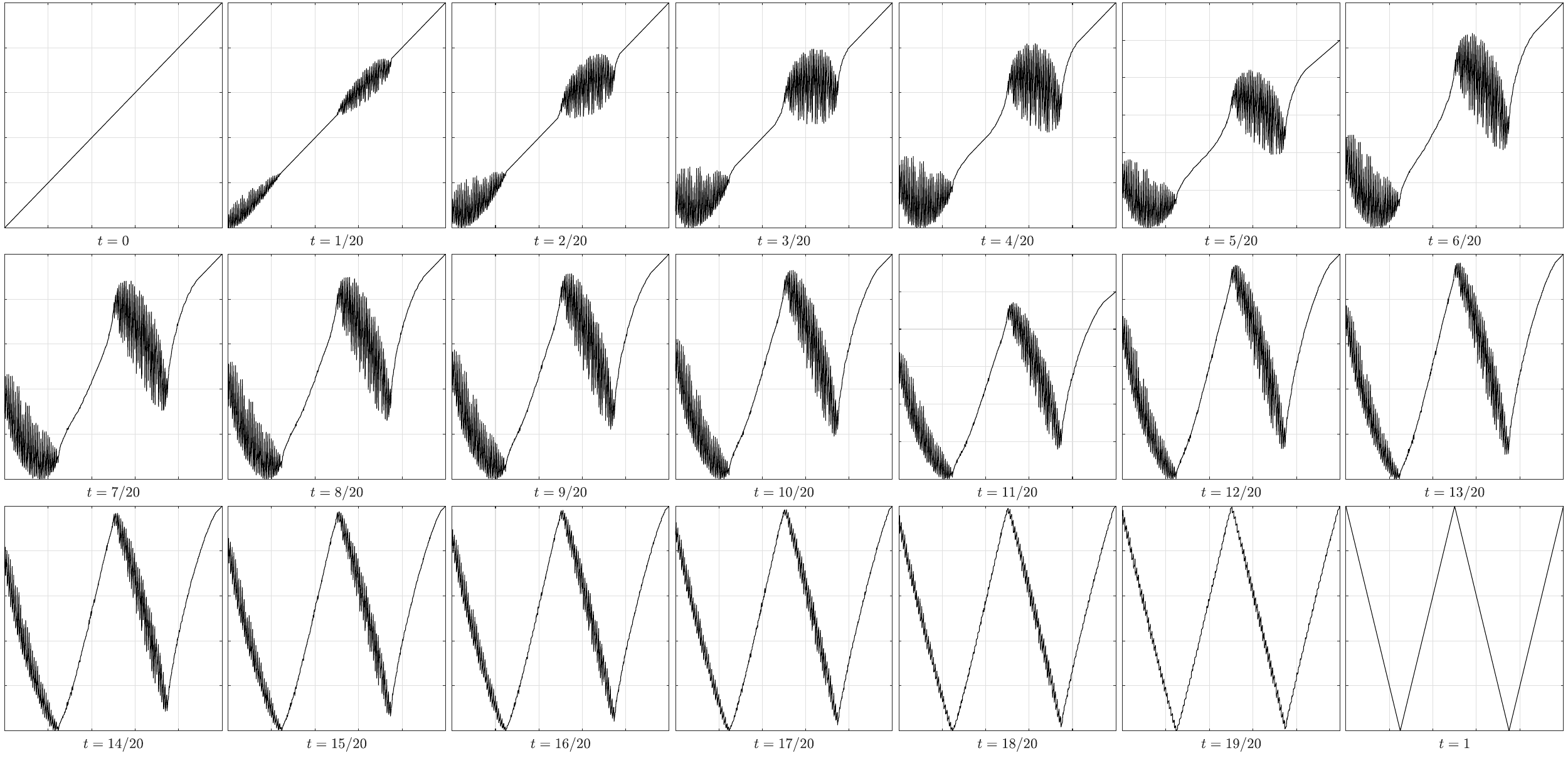}
    
    \caption{\textbf{Experiment~1} -- Volume preserving maps $\Xi^{(2)}_1,\ldots,\Xi^{(2)}_N$.}\label{fig:fluid-maps2}
\end{figure}

\begin{figure}[t]
    \includegraphics[width=0.45\linewidth]{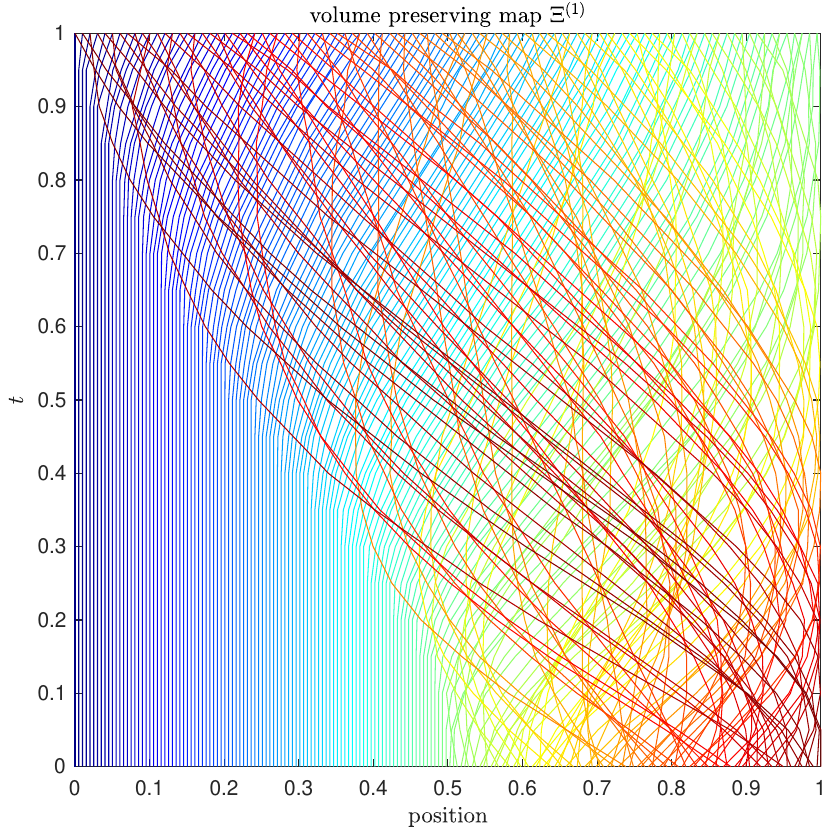}~
    \includegraphics[width=0.45\linewidth]{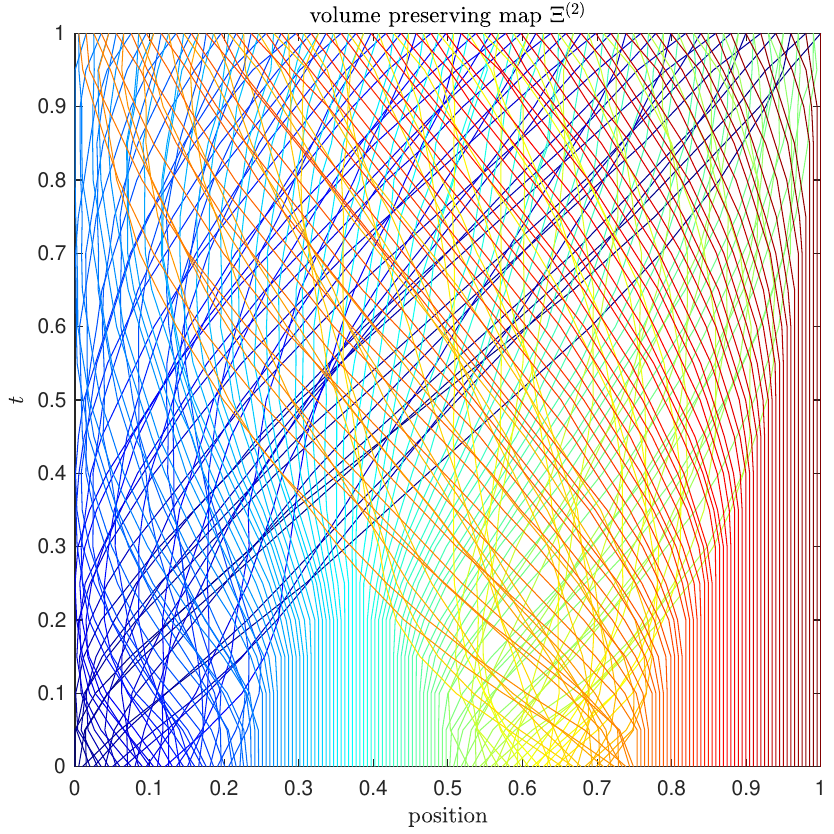}
    
    \caption{\textbf{Experiment~1} -- Trajectories of the fluid particles from $t=0$ to $t=1$ implied by the computed volume preserving maps $\Xi^{(1)}_1,\ldots,\Xi^{(1)}_N$ and $\Xi^{(2)}_1,\ldots,\Xi^{(2)}_N$.}\label{fig:fluid-trajectories}
\end{figure}

\subsubsection*{Results and discussions}
Figure~\ref{fig:fluid-maps1} and Figure~\ref{fig:fluid-maps2} show the volume preserving maps constructed via Proposition~\ref{prop:experiments-fluid}\ref{props:experiments-fluid-reassembly}, that is, 
$\Xi^{(1)}_1,\ldots,\Xi^{(1)}_N$ are constructed with respect to $\Xi\leftarrow\Xi^{(1)}$
and $\Xi^{(2)}_1,\ldots,\Xi^{(2)}_N$ are constructed with respect to $\Xi\leftarrow\Xi^{(2)}$.
The left (resp., right) panel of Figure~\ref{fig:fluid-trajectories} shows the trajectories of the fluid particles implied by the computed volume preserving maps, that is, it plots $\Xi^{(1)}_1(x),\Xi^{(1)}_2(x),\ldots,\Xi^{(1)}_N(x),\Xi^{(1)}(x)$ (resp., $\Xi^{(2)}_1(x),\Xi^{(2)}_2(x),\ldots,\Xi^{(2)}_N(x),\Xi^{(2)}(x)$) to represent the positions of a fluid particle at time $t=0,\frac{1}{N},\frac{2}{N}, \ldots,\frac{N-1}{N},1$ given the initial position $x\in[0,1]$ of the particle.
We use different colors for the trajectories depending on their initial positions for better visualization. 
Similar numerical results have been produced by \citet[Figures~1--6]{brenier2008generalized}, \citet[Figures~5--7]{benamou2015iterative}, and \citet[Figures~8 \& 9]{ba2022accelerating}. 
In particular, Figure~\ref{fig:fluid-maps1} and the left panel of Figure~\ref{fig:fluid-trajectories} are very similar to Figure~1 and Figure~4 of \citep{brenier2008generalized}. 
In Figure~\ref{fig:fluid-maps1} and Figure~\ref{fig:fluid-maps2}, oscillations of the volume preserving maps are observed in certain intervals of the domain; specifically, $[1/2, 1]$ in Figure~\ref{fig:fluid-maps1} and $[0,1/4]$ and $[1/2, 3/4]$ in Figure~\ref{fig:fluid-maps2}.
This is an indication that mass tends to be split in these regions and that particles with approximately the same initial positions will have diverging trajectories. 
Indeed, the divergence of trajectories in these regions are observed in Figure~\ref{fig:fluid-trajectories}. 
We remark that the crossings of the trajectories do not hinder their physical interpretation because they are projections of three-dimensional trajectories into a single dimension; see the discussion of \citet{brenier2008generalized}.
Using the terminologies of \citet{brenier2008generalized}, 
potential flows are observed 
in the interval $[0, 1/2]$ 
in the left panel of Figure~\ref{fig:fluid-trajectories} 
and in the intervals $[1/4, 1/2]$ and $[3/4, 1]$ 
in the right panel of Figure~\ref{fig:fluid-trajectories}, 
while vortical flows are observed 
in the interval $[1/2, 1]$ 
in the left panel of Figure~\ref{fig:fluid-trajectories} 
and in the intervals 
$[0, 1/4]$ and $[1/2, 3/4]$ in the right panel of Figure~\ref{fig:fluid-trajectories}.
In contrast, Figure~5 in \citep{benamou2015iterative} and Figure~9 in \citep{ba2022accelerating} resemble Figure~\ref{fig:fluid-maps1}, but are much more blurry due to the effect of regularization. 
Compared to the spatial discretization approaches, one advantage of our approach is that the volume preserving maps constructed via Algorithm~\ref{alg:cp-mmot} and Proposition~\ref{prop:experiments-fluid}\ref{props:experiments-fluid-reassembly} allow one to approximate the trajectory of a particle at any initial position in $[0,1]$, rather than being restricted to initial positions in the grid. 
Moreover, our algorithm provides a sub-optimality estimate of the computed solution $\tilde{\mu}$ for the MMOT problem; see our results and discussions in the next paragraph. 

\begin{figure}[t]
    \includegraphics[width=0.45\linewidth]{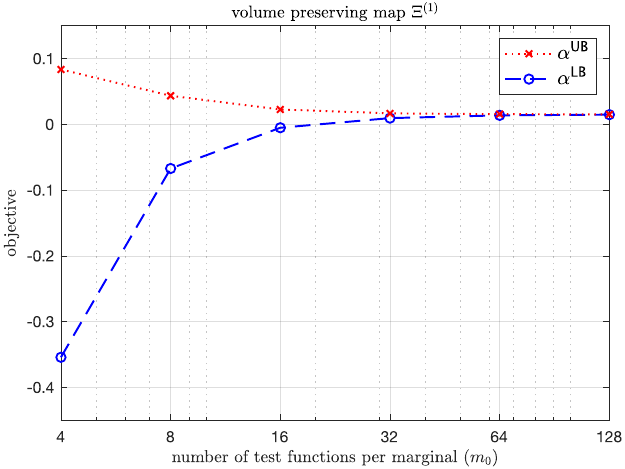}~
    \includegraphics[width=0.45\linewidth]{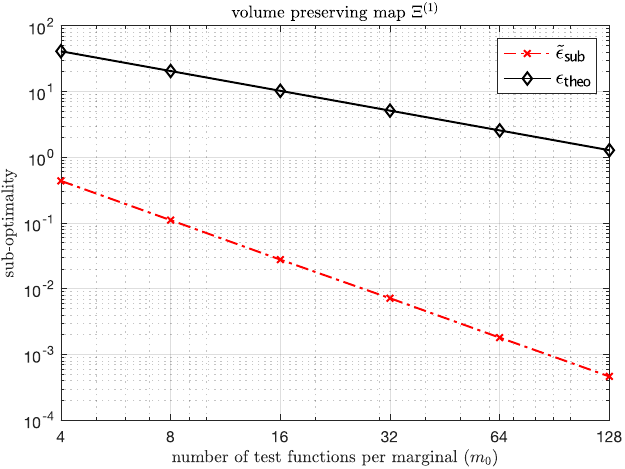}

    \vspace{4pt}

    \includegraphics[width=0.45\linewidth]{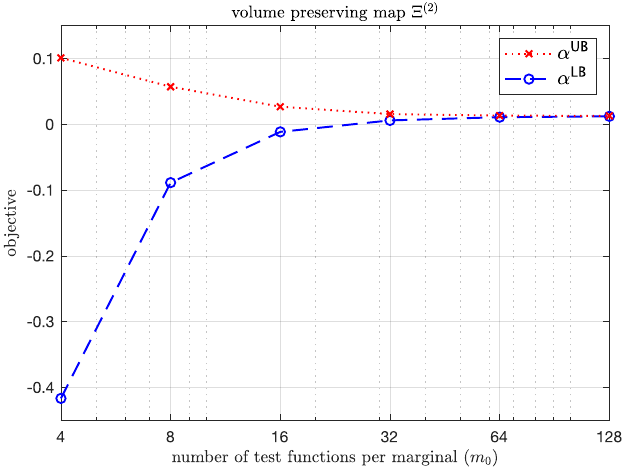}~
    \includegraphics[width=0.45\linewidth]{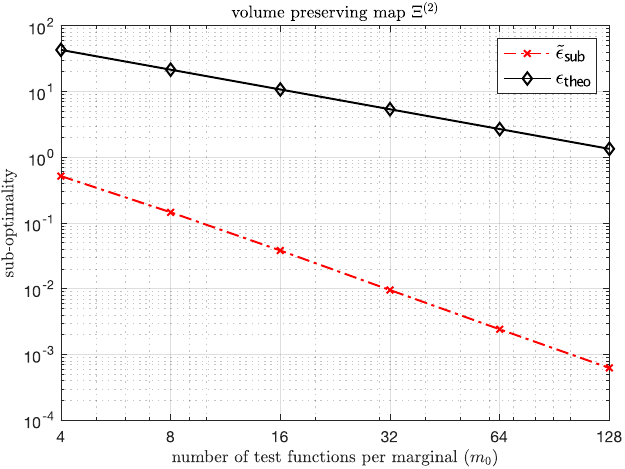}
    
    \caption{\textbf{Experiment~1} -- Left column: the lower bound $\alpha^{\SFL\SFB}$ and the upper bound~$\alpha^{\SFU\SFB}$. Right column: the sub-optimality estimate $\tilde{\epsilon}_{\mathsf{sub}}$ and its a priori upper bound $\epsilon_{\mathsf{theo}}$ on the log-scale.}\label{fig:fluid-bounds}
\end{figure}

Figure~\ref{fig:fluid-bounds} shows the values of the lower bound $\alpha^{\SFL\SFB}$, the upper bound $\alpha^{\SFU\SFB}$, and the sub-optimality estimate $\tilde{\epsilon}_{\mathsf{sub}}$ computed by Algorithm~\ref{alg:mmot}, along with the a priori error bound $\epsilon_{\mathsf{theo}}$ derived via Theorem~\ref{thm:mmotalgo} with respect to $\rho_i\leftarrow 2\eta(\FC)= \frac{2}{m_0}$ $\forall 1\le i\le N$. 
It can be observed from the left column of Figure~\ref{fig:fluid-bounds} that, for both volume preserving maps $\Xi^{(1)}$ and $\Xi^{(2)}$, the lower bound $\alpha^{\SFL\SFB}$ and the upper bound $\alpha^{\SFU\SFB}$ are initially far apart when $m_0=4$ test functions are used for each marginal. 
When $m_0=64$ and $m_0=128$ test functions are used for each marginal, the differences between the lower bound $\alpha^{\SFL\SFB}$ and the upper bound $\alpha^{\SFU\SFB}$ become small. 
This is confirmed by the sub-optimality estimates plotted on the log-scale in the right column of Figure~\ref{fig:fluid-bounds}. 
Indeed, when $m_0=128$ test functions are used for each marginal, 
we get $\tilde{\epsilon}_{\mathsf{sub}}=4.6400\times 10^{-4}$ with the volume preserving map $\Xi^{(1)}$ and $\tilde{\epsilon}_{\mathsf{sub}}=6.2476\times 10^{-4}$ with the volume preserving map~$\Xi^{(2)}$. 
This indicates that the approximately optimal solutions $\tilde{\mu}$ for \eqref{eqn:mmot} computed by Algorithm~\ref{alg:mmot} are close to being optimal. 
Moreover, observe from the right column of Figure~\ref{fig:fluid-bounds} that the computed sub-optimality estimates $\tilde{\epsilon}_{\mathsf{sub}}$ are two to three orders of magnitude smaller than their a priori upper bounds~$\epsilon_{\mathsf{theo}}$, and they seem to be decreasing at a faster rate than $\epsilon_{\mathsf{theo}}$. 
This highlights an important practical advantage of our numerical method, which is that our algorithm produces practically meaningful sub-optimality estimates that are not over-conservative. 
If one uses approximation methods that do not compute both lower and upper bounds for \eqref{eqn:mmot} in this problem instance (such as by discretization of the marginals, or by a regularization-based method) and relies on a theoretical estimate of the approximation error, then one will end up with an error estimate that is orders of magnitude too conservative, which limits the practicality of such approximation methods;
see also the subsequent results comparing our algorithm with two existing algorithms.

\subsubsection*{Comparison with existing methods}
\begin{figure}[t]
    \includegraphics[height=200pt]{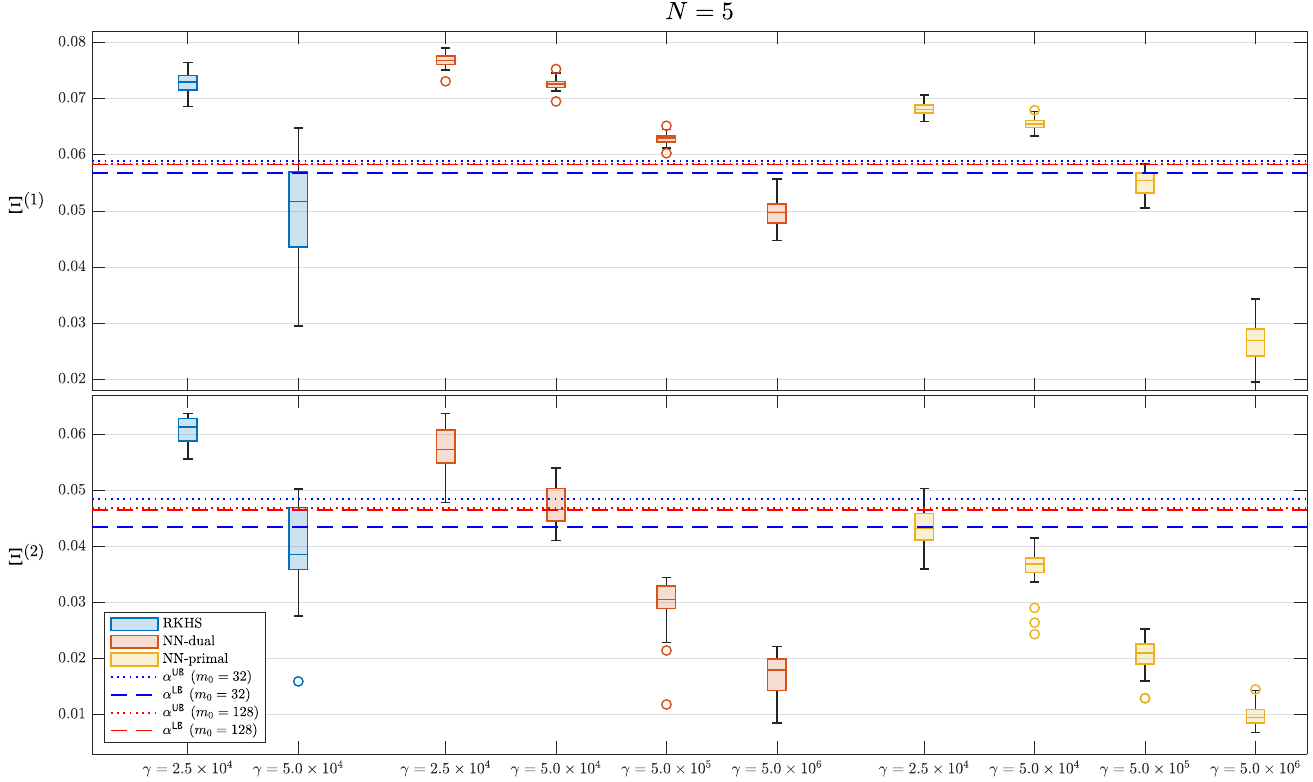}

    \vspace{8pt}

    \hspace{4.4pt}\includegraphics[height=200pt]{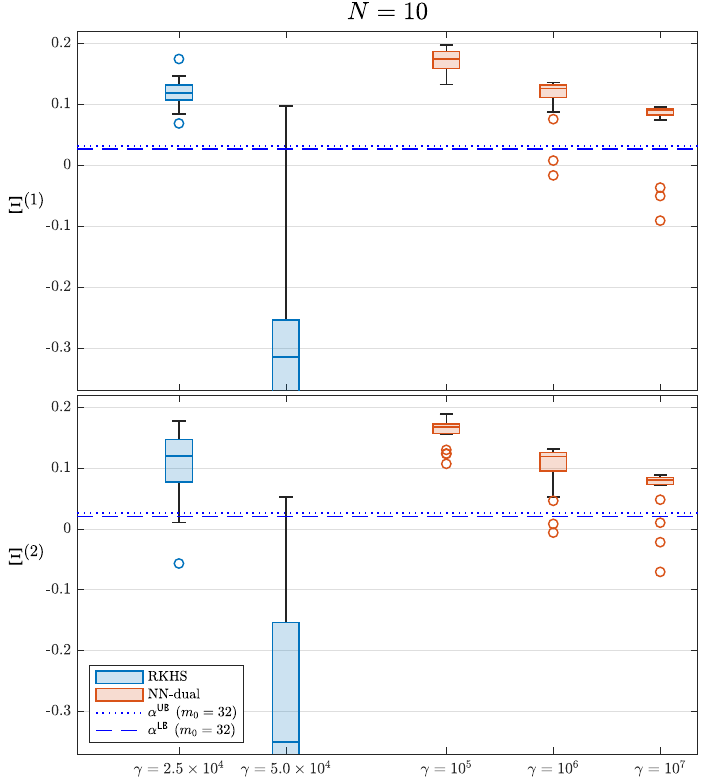}~\hspace{3.6pt}
    \includegraphics[height=200pt]{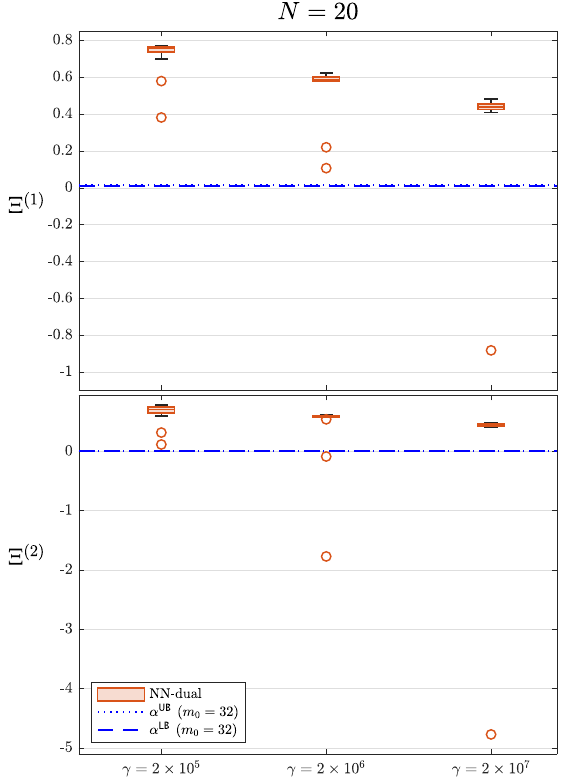}
    
    \caption{\textbf{Experiment~1} -- Comparison of the objective values computed by Algorithm~\ref{alg:mmot}, the RKHS-based algorithm, and the NN-based algorithm.}\label{fig:fluid-comparisons}
\end{figure}

In order to compare our algorithm with the two regularization-based algorithms,
we consider instances of the MMOT problem in 
(\ref{eqn:experiments-fluid-OT})
where $N\in\{5,10,20\}$.
That is, 
in addition to the time discretized version of the problem of \citet{brenier1989least, brenier1993dual, brenier1999minimal, brenier2008generalized} with $N=20$ discrete time points
studied above,
we also examine coarser discretizations with $N=5$ and $N=10$ discrete time points 
which contain fewer marginals.
Figure~\ref{fig:fluid-comparisons} compares the bounds $\alpha^{\mathsf{LB}}$, $\alpha^{\mathsf{UB}}$ computed by Algorithm~\ref{alg:mmot} 
with the outputs of the RKHS-based algorithm and the NN-based algorithm,
where box plots are used to visualize the variability among 20 independent runs of each SGD algorithm.
Specifically, 
the bounds $\alpha^{\mathsf{LB}}$, $\alpha^{\mathsf{UB}}$ computed by Algorithm~\ref{alg:mmot} 
when $m_0=32$ and when $m_0=128$
are shown as horizontal lines,
the objective values with respect to \eqref{eqn:mmot-dual-regularized} computed by the RKHS-based algorithm are shown as box plots with blue color, and
the objective values with respect to \eqref{eqn:mmot-dual-regularized} computed by the NN-based algorithm are shown as box plots with red color.
We did not test the performance of the RKHS-based algorithm when $N=20$ because it did not converge when $N=10$ and $\gamma=5.0\times 10^{4}$.
Similar patterns
can be observed in Figure~\ref{fig:fluid-comparisons}
from the outputs of both regularization-based algorithms.
On the one hand, when the inverse regularization parameter~$\gamma$ is small,
the convergence of both algorithms seem to be stable, while the regularization term introduced large biases.
This means that even if the RKHS-based algorithm or the NN-based algorithm produces a candidate solution of \eqref{eqn:mmot-dual-regularized} with small sub-optimality, the difference between the optimal value of \eqref{eqn:mmot-dual-regularized} and the optimal value of \eqref{eqn:mmot} is large.
In addition, observe that this bias seems to become larger when $N$ is large.
On the other hand, for large values of $\gamma$, both algorithms struggled to converge, especially when $N$~is large.
When $N=20$, observe that the NN-based algorithm did not compute objective values that reasonably approximate the optimal value of \eqref{eqn:mmot}, with all three values of~$\gamma$.
One potential way to obtain results better than those shown in Figure~\ref{fig:fluid-comparisons} from the RKHS-based algorithm or the NN-based algorithm is to 
reduce the step sizes in the SGD algorithm and increase the number of iterations at the expense of longer computation time.

Moreover, Figure~\ref{fig:fluid-comparisons} also shows the primal objective values computed by the NN-based algorithm,
i.e., $\int_{\BCX}f\DIFFX{\hat{\mu}_{\mathsf{NN}}}$ where $\hat{\mu}_{\mathsf{NN}}$ is defined by (\ref{eqn:mmot-dual-regularized-primalopt}), 
as box plots with yellow color.
As discussed earlier, $\hat{\mu}_{\mathsf{NN}}$ constructed this way is not feasible for \eqref{eqn:mmot}.
To examine how much the marginals of $\hat{\mu}_{\mathsf{NN}}$ deviate from $\FL_{[0,1]}$, we perform chi-squared goodness-of-fit tests with 50000 samples and 10~bins.
The maximum p-value among the results of all tests was 
$2.3937\times 10^{-24}$, which indicates that $\hat{\mu}_{\mathsf{NN}}$ constructed via the NN-based algorithm severely violates the marginal constraints in \eqref{eqn:mmot}.
Furthermore, it becomes numerically challenging to generate samples from $\hat{\mu}_{\mathsf{NN}}$ via rejection sampling when $N$~is large.
Because of this, we did not compute the primal objective values when $N=10$ and when $N=20$.

\begin{figure}[t]
    \includegraphics[width=\linewidth]{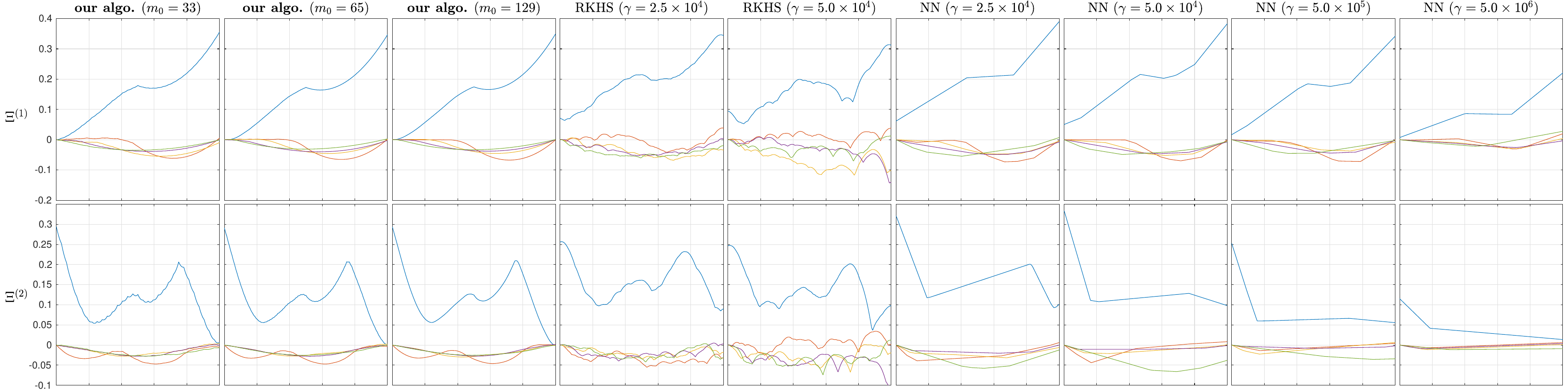}
    
    \caption{\textbf{Experiment~1} -- Approximate dual solutions computed by Algorithm~\ref{alg:mmot}, the RKHS-based algorithm, and the NN-based algorithm when $N=5$.}\label{fig:fluid-dualfuncs-5}
\end{figure}

Figure~\ref{fig:fluid-dualfuncs-5}
presents the approximate solutions of \eqref{eqn:mmot-dual} computed by Algorithm~\ref{alg:mmot} and the approximate solutions of \eqref{eqn:mmot-dual-regularized} computed by the RKHS-based algorithm and the NN-based algorithm when $N=\nobreak5$.
Note that the decision variables in \eqref{eqn:mmot-dual} and \eqref{eqn:mmot-dual-regularized} are invariant under addition of constants that sum up to~0.
Hence, for the purpose of visual comparison, we have shifted every approximate solution $(\tilde{h}_i)_{i=1:N}$ such that $\tilde{h}_i(0)=0$ for $i=2,\ldots,N$.
From Figure~\ref{fig:fluid-dualfuncs-5},
one can observe that the solutions computed by the RKHS-based algorithm and the NN-based algorithm exhibit some similar structures compared to the high-quality solutions computed by Algorithm~\ref{alg:mmot}.
Moreover, one can observe that the solutions computed by the RKHS-based algorithm and the NN-based algorithm are infeasible.
For example, one checks that 
$\big(x_N-\Xi^{(1)}(x_1)\big)^2+\sum_{i=1}^{N-1}(x_{i+1}-x_{i})^2=0$ when
$x_1=\cdots=x_N=0$, 
but the values of
$\sum_{i=1}^{N}\tilde{h}_i(0)$
computed by the RKHS-based algorithm and the NN-based algorithm in the top row of Figure~\ref{fig:fluid-dualfuncs-5} are all positive.

\subsection{Experiment~2: Wasserstein barycenter}\label{ssec:experiments-barycenter}

In the second numerical experiment, we use Algorithm~\ref{alg:mmot} to study the Wasserstein barycenter problem introduced by \citet{agueh2011barycenters}.
In this problem, one is given $N\in\N$ probability measures 
$\mu_1\in\CP_2(\CX_1),\ldots,\mu_N\in\CP_2(\CX_N)$ 
where $\CX_i\subseteq \R^d$ is equipped with the Euclidean metric for $i=1,\ldots,N$, 
and one seeks to find a barycenter~$\bar{\nu}\in\CP_2(\R^d)$ of $\mu_1,\ldots,\mu_N$ in the second order Wasserstein space, defined to be an optimal solution of the following minimization problem:
\begin{align}
    \inf_{\nu\in\CP_2(\R^d)}\Bigg\{\frac{1}{N}\sum_{i=1}^N W_2(\mu_i,\nu)^2\Bigg\}.
    \label{eqn:experiments-barycenter-definition}
\end{align}
Wasserstein barycenter ``lifts'' the notion of barycenter from the Euclidean space to the space of probability measures while preserving the geometric properties of the underlying Euclidean space.
Thus, it can be seen as an average or a summary of the input probability measures $\mu_1,\ldots,\mu_N$. 
Due to this property, it has been widely applied to various fields such as statistical inference \citep{srivastava2015wasp,srivastava2018scalable, li2020continuous, bigot2019penalization}, unsupervised clustering \citep{ye2014scaling, ye2017fast,puccetti2020on}, geometric shape interpolation \citep{solomon2015convolutional, lindheim2023simple}, etc.
There have also been an abundance of studies about the computation of Wasserstein barycenter, see, for example, \citep{altschuler2023polynomial, alvarez2016fixed, anderes2016discrete, borgwardt2020lp, borgwardt2022integer, claici2018stochastic, li2020continuous, luise2019sinkhorn, puccetti2020on, korotin2022wasserstein, staib2017parallel, yang2021fast} and the references therein.

It is well-known that (\ref{eqn:experiments-barycenter-definition}) can be formulated into an MMOT problem; see, e.g., \citep[Section~4]{agueh2011barycenters}.
In the following proposition, let us state the properties of this formulation under Setting~\ref{sett:compact} as well as some details when utilizing Algorithm~\ref{alg:mmot} to compute an approximate Wasserstein barycenter, i.e., an approximately optimal solution of (\ref{eqn:experiments-barycenter-definition}).
%\pagebreak
\begin{proposition}[Experiment~2]\label{prop:experiments-barycenter-mmot}
    For $i=1,\ldots,N$, let $\CX_i\subset\R^d$ be a compact set equipped with the Euclidean metric, and let $\mu_i\in\CP(\CX_i)$.
    Let $\BCX:=\bigtimes_{i=1}^{N}\CX_i$. 
    Then,\widowpenalties-1000
    \begin{enumerate}[label=(\roman*),beginpenalty=10000]
        \item\label{props:experiments-barycenter-mmot-formulation}%
        the optimal value of (\ref{eqn:experiments-barycenter-definition}) is equal to the optimal value of the following MMOT problem:
        \begin{align}
            \inf_{\mu\in\Gamma(\mu_1,\ldots,\mu_N)}\bigg\{\int_{\BCX}\min_{\BIz\in\CZ}\Big\{{\textstyle\frac{1}{N}\sum_{i=1}^N}\|\BIx_i-\BIz\|_2^2\Big\} \DIFFM{\mu}{\DIFF\BIx_1,\ldots,\DIFF\BIx_N}\bigg\}.
            \label{eqn:experiments-barycenter-mmot-formulation}
        \end{align}
    \end{enumerate}
    In the following,
    let $f:\BCX\to\R$ be defined as follows:
    \begin{align}
        f(\BIx_1,\ldots,\BIx_N):=-\frac{1}{N^2}\sum_{i=1}^N\sum_{j=1}^N\langle\BIx_i,\BIx_j\rangle \qquad \forall (\BIx_1,\ldots,\BIx_N)\in\BCX.
        \label{eqn:experiments-barycenter-cost}
    \end{align}
    Let us consider the problem \eqref{eqn:mmot} and its dual optimization problem \eqref{eqn:mmot-dual} with respect to the cost function $f$ in (\ref{eqn:experiments-barycenter-cost}).
    Then, 
    \begin{enumerate}[label=(\roman*),beginpenalty=10000]
        \setcounter{enumi}{1}
        \item\label{props:experiments-barycenter-setting}%
        the conditions \ref{settc:compact-measures} and \ref{settc:compact-cost} in Setting~\ref{sett:compact} hold with respect to $L_f\leftarrow \frac{2}{N}\sup_{\BIz\in\CZ}\big\{\|\BIz\|_2\big\}$.
    \end{enumerate}
    Now, 
    let $\CZ:=\frac{1}{N}\sum_{i=1}^{N}\CX_i$ (here the addition of sets in $\R^d$ denotes the Minkowski sum),
    let $\bar{\BIz}(\BIx_1,\ldots,\BIx_N):=\frac{1}{N}\sum_{i=1}^N\BIx_i$ $\forall (\BIx_1,\ldots,\BIx_N)\in\BCX$,
    and let
    $C_{\mathsf{quad}}:=\frac{1}{N}\sum_{i=1}^N \int_{\CX_i}\|\BIx_i\|_2^2\DIFFM{\mu_i}{\DIFF\BIx_i}$.
    Moreover, let $\epsilon_{\mathsf{LSIP}}>0$,
    let $(\CG_i)_{i=1:N}$, $\BIg(\,\cdot\,)$, $\bar{\BIg}$ be constructed by \ref{settc:compact-testfuncs},
    let $(\rho_i)_{i=1:N}$ satisfy $\rho_i\ge \specialoverline{W}_{1}(\mu_i,[\mu_i]_{\CG_i})$ $\forall 1\le i\le N$,
    and let $\epsilon_{\mathsf{theo}}:=\epsilon_{\mathsf{LSIP}}+\frac{2}{N}\sup_{\BIz\in\CZ}\big\{\|\BIz\|_2\big\}\sum_{i=1}^N \rho_i$.
    Let $\big(\tilde{\alpha}^{\mathsf{UB}},\tilde{\alpha}^{\mathsf{LB}},\tilde{\mu},(\tilde{h}^\dagger_i)_{i=1:N},\tilde{\epsilon}_{\mathsf{sub}}\big)$ be the outputs of Algorithm~\ref{alg:mmot},
    and let 
    $\hat{\mu}$ denote the probability measure computed by Algorithm~\ref{alg:cp-mmot} in Line~\ref{alglin:mmot-cpalgo} of Algorithm~\ref{alg:mmot}.
    Then, the following statements hold.
    \begin{enumerate}[label=(\roman*),beginpenalty=10000]
        \setcounter{enumi}{2}
        \item\label{props:experiments-barycenter-bounds}%
        $\alpha^{\mathsf{LB}}:=\tilde{\alpha}^{\mathsf{LB}}+C_{\mathsf{quad}}$ and $\alpha^{\mathsf{UB}}:=\tilde{\alpha}^{\mathsf{UB}}+C_{\mathsf{quad}}$ are lower and upper bounds for the optimal value of (\ref{eqn:experiments-barycenter-definition}),
        and $\alpha^{\mathsf{UB}}-\alpha^{\mathsf{LB}}=\tilde{\epsilon}_{\mathsf{sub}}\le \epsilon_{\mathsf{theo}}$.

        \item\label{props:experiments-barycenter-dual-solution}%
        For $i=1,\ldots,N$, let $\tilde{h}_i(\BIx):=\tilde{h}^\dagger_i(\BIx)+\frac{1}{N}\|\BIx\|_2^2$ $\forall \BIx\in\CX_i$.
        Then, $(\tilde{h}_i)_{i=1:N}$ is an $\tilde{\epsilon}_{\mathsf{sub}}$-optimal solution of the dual optimization problem of (\ref{eqn:experiments-barycenter-mmot-formulation}).

        \item\label{props:experiments-barycenter-primal-solution-W1}%
        $\tilde{\nu}:=\tilde{\mu}\circ \bar{\BIz}^{-1}$ is an $\tilde{\epsilon}_{\mathsf{sub}}$-optimal solution of (\ref{eqn:experiments-barycenter-definition}).
        
        \item\label{props:experiments-barycenter-primal-solution-W2}%
        Let $\hat{\nu}:=\hat{\mu}\circ \bar{\BIz}^{-1}\in\CP(\CZ)$ and let $\hat{\gamma}_i\in\Gamma(\mu_i,\hat{\nu})$ satisfy $\int_{\CX_i\times\CZ}\|\BIx_i-\BIz\|_2^2\DIFFM{\hat{\gamma}_i}{\DIFF\BIx_i,\DIFF\BIz} = W_2(\mu_i,\hat{\nu})^2$ for $i=1,\ldots,N$. 
        Then, there exists $\hat{\gamma}\in\CP(\BCX\times\CZ)$ such that, for $i=1,\ldots,N$, the marginal of $\hat{\gamma}$ on $\CX_i\times\CZ$ is equal to $\hat{\gamma}_i$.
        Moreover, let $\breve{\mu}$ be the marginal of $\hat{\gamma}$ on $\BCX$, 
        and let $\breve{\nu}:=\breve{\mu}\circ \bar{\BIz}^{-1}$, 
        $\beta^{\mathsf{UB}}:=\int_{\BCX}f\DIFFX{\breve{\mu}}+C_{\mathsf{quad}}$,
        $\tilde{\xi}_{\mathsf{sub}}:=\beta^{\mathsf{UB}} - \tilde{\alpha}^{\mathsf{LB}} - C_{\mathsf{quad}}$. 
        Then, $\breve{\nu}$ is a $\tilde{\xi}_{\mathsf{sub}}$-optimal solution of (\ref{eqn:experiments-barycenter-definition}),
        $\beta^{\mathsf{UB}}$ is an upper bound for the optimal value of (\ref{eqn:experiments-barycenter-definition}),
        and $\tilde{\xi}_{\mathsf{sub}}\le \epsilon_{\mathsf{theo}}$.
    \end{enumerate}
\end{proposition}

\begin{proof}[Proof of Proposition~\ref{prop:experiments-barycenter-mmot}]
    See Section~\ref{ssec:proof-experiments}.
\end{proof}

\begin{figure}[t]
    \includegraphics[width=\linewidth]{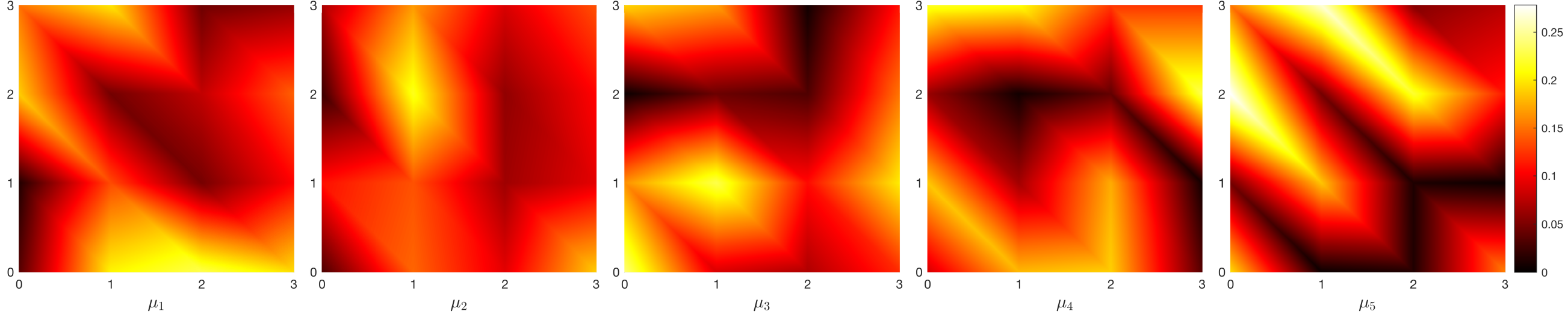}
    
    \caption{\textbf{Experiment~2} -- Probability density functions of $\mu_1,\ldots,\mu_N$.}\label{fig:barycenter-inputs}
\end{figure}

\subsubsection*{Experimental settings}
In this experiment, we use Algorithm~\ref{alg:mmot} and Proposition~\ref{prop:experiments-barycenter-mmot} to approximately compute the Wasserstein barycenter of $N=5$ absolutely continuous probability measures supported on $\CX_1=\cdots=\CX_{5}=[0,3]^2$ with continuous piece-wise affine density functions. 
Figure~\ref{fig:barycenter-inputs} shows the probability density functions of $\mu_1,\ldots,\mu_5\in\CP([0,3]^2)$ as color plots. 
Moreover, we let 
$\FC$ be a finite collection of 2-simplices (i.e., triangles) that satisfy
$\bigcup_{C\in\FC}C=[0,1]^2$ and $C_1\cap C_2\in\FF(C_1)\cap\FF(C_2)$ $\forall C_1,C_2\in\FC$.
Subsequently, we construct $\CG_1=\cdots=\CG_N$ 
as well as $\BIg(\,\cdot\,)$, $\bar{\BIg}$
by \ref{settc:compact-testfuncs} with respect to $\FC_i\leftarrow \FC$ $\forall 1\le i\le N$.
Let $m_0:=|V(\FC)|-1=|\CG_1|=\cdots=|\CG_N|$, which varies between 15 and 5775, that is, we use 15~to~5775 test functions per marginal. 
Furthermore, for each value of~$m_0$, we set $\epsilon_{\mathsf{LSIP}}=10^{-4}$ and use 
Algorithm~\ref{alg:mmot} and 
Proposition~\ref{prop:experiments-barycenter-mmot} 
to compute 
a lower bound $\alpha^{\mathsf{LB}}$ and
two upper bounds $\alpha^{\mathsf{UB}}$, $\beta^{\mathsf{UB}}$ for the optimal value of (\ref{eqn:experiments-barycenter-definition}),
as well as an $\tilde{\epsilon}_{\mathsf{sub}}$-optimal solution $(\tilde{h}_i)_{i=1:N}$ of the dual optimization problem of (\ref{eqn:experiments-barycenter-mmot-formulation}),  
an $\tilde{\epsilon}_{\mathsf{sub}}$-approximate Wasserstein barycenters~$\tilde{\nu}$,
and a $\tilde{\xi}_{\mathsf{sub}}$-approximate Wasserstein barycenter $\breve{\nu}$.
Proposition~\ref{prop:experiments-barycenter-mmot} guarantees that $\tilde{\epsilon}_{\mathsf{sub}}\le \epsilon_{\mathsf{theo}}$ and
$\tilde{\xi}_{\mathsf{sub}}\le \epsilon_{\mathsf{theo}}$,
where $\epsilon_{\mathsf{theo}}:=\epsilon_{\mathsf{LSIP}}+4\sup_{\BIz\in\CZ}\big\{\|\BIz\|_2\big\}\eta(\FC)$ is derived from Theorem~\ref{thm:mmotalgo} and Proposition~\ref{prop:experiments-barycenter-mmot}
with respect to $\rho_i\leftarrow 2\eta(\FC)$ $\forall 1\le i\le N$.
We remark that we are able to tractably implement $\mathtt{Oracle}(\,\cdot\,)$ in Definition~\ref{def:mmot-oracle} using the methods introduced by \citet[Section~3.2]{altschuler2021wasserstein}, and the resulting implementation is highly efficient.\footnote{Our implementation utilizes the Computational Geometry Algorithms Library (CGAL) \citep{CGAL} as well as the C++ Kd-tree library implemented by \citet{dalitz2009kd}.}
The upper bounds $\alpha^{\mathsf{UB}}$ and $\beta^{\mathsf{UB}}$ are approximately computed by Monte Carlo integration with $10^8$ independent random samples, and we repeat this 100 times to examine the Monte Carlo errors.

\begin{figure}[t]
    \includegraphics[width=0.9\linewidth]{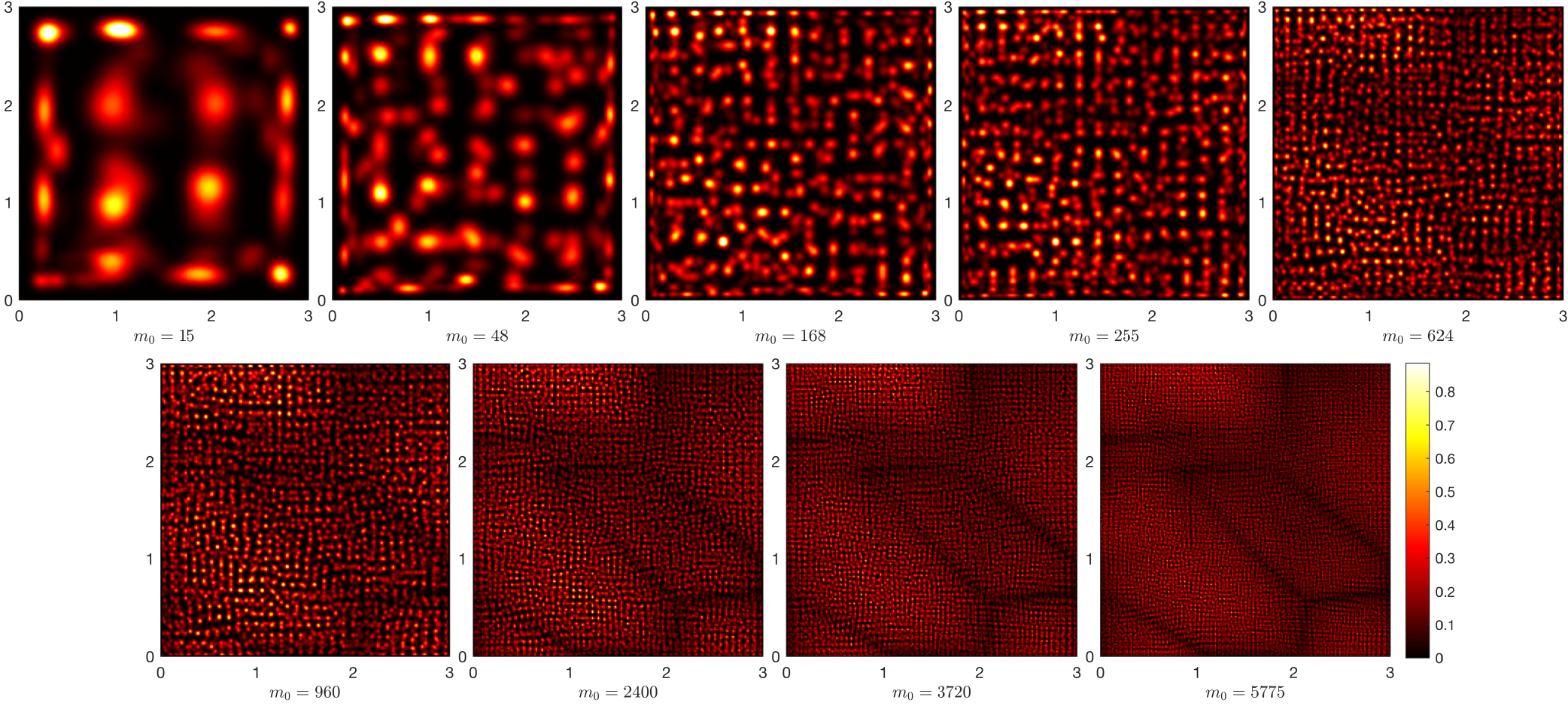}
    \caption{\textbf{Experiment~2} -- Histograms of the approximate Wasserstein barycenter~$\tilde{\nu}$.}\label{fig:barycenter-outputs}
\end{figure}

\begin{figure}[t]
    \includegraphics[width=0.9\linewidth]{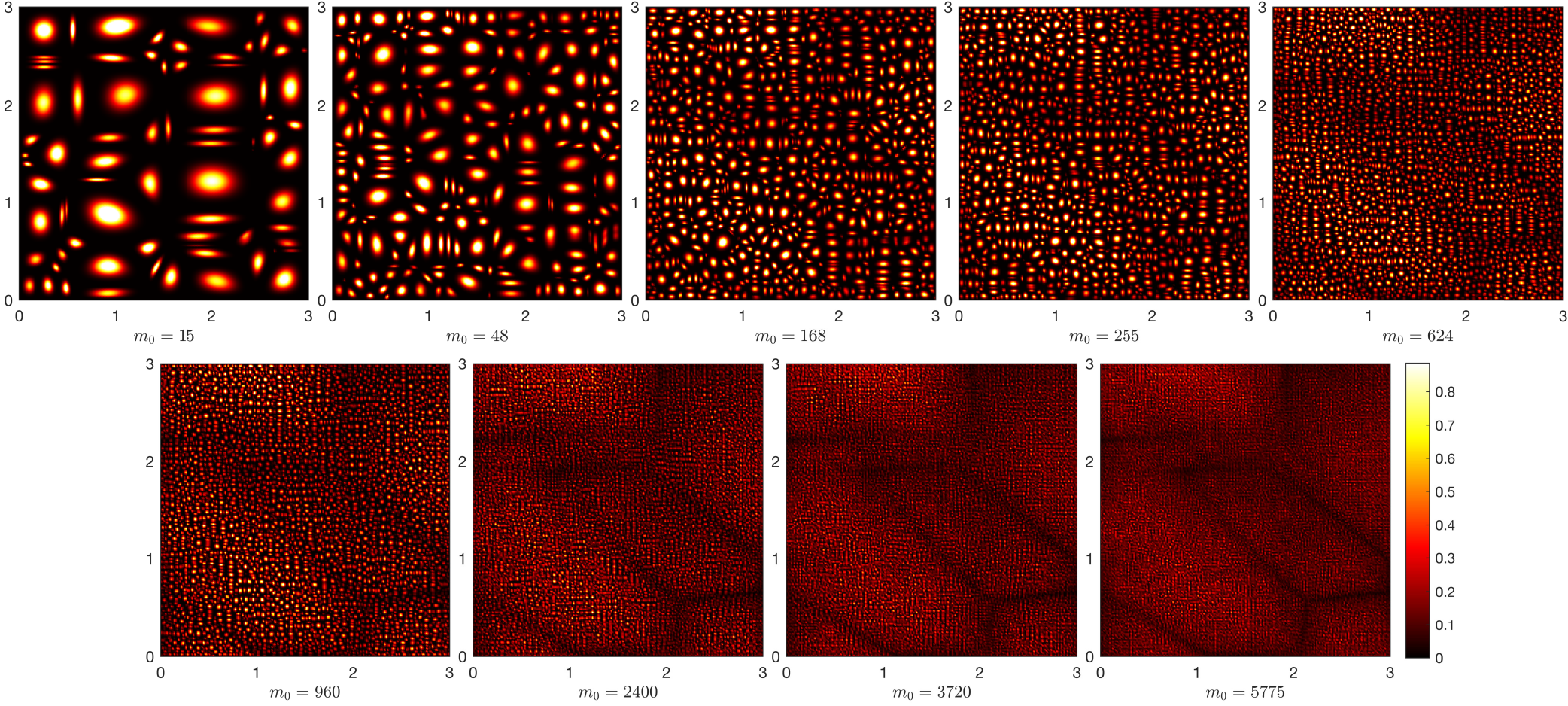}
    \caption{\textbf{Experiment~2} -- Histograms of the approximate Wasserstein barycenter~$\breve{\nu}$.}\label{fig:barycenter-W2OToutputs}
\end{figure}

\subsubsection*{Results and discussions}
In order to visualize $\tilde{\nu}$ and $\breve{\nu}$, we generate $10^{10}$ random samples from each of them and show the corresponding histograms as color plots in Figure~\ref{fig:barycenter-outputs} and Figure~\ref{fig:barycenter-W2OToutputs}.
Note that the histograms have been normalized into probability density functions (i.e., they integrate to~1 on $[0,3]^2$), and the color bars in Figure~\ref{fig:barycenter-outputs} and Figure~\ref{fig:barycenter-W2OToutputs} show the density values corresponding to the colors. 
Observe that the approximate Wasserstein barycenters~$\tilde{\nu}$ in Figure~\ref{fig:barycenter-outputs} are absolutely continuous probability measures on $[0,3]^2$ that are mixtures of ``blob''-shaped components, where the number of blobs increases with the number~$m_0$ of test functions. 
As the value of~$m_0$ becomes large, the sizes of these blobs shrink, and they begin to form a spatial pattern with a noticeable structure that approximates the true Wasserstein barycenter of $\mu_1,\ldots,\mu_5$. 
This pattern contains regions with high and low probabilities with irregular boundaries. 
The approximate Wasserstein barycenters~$\breve{\nu}$ in Figure~\ref{fig:barycenter-W2OToutputs} are also absolutely continuous probability measures formed with blob-shaped components. 
With small values of~$m_0$, the approximate Wasserstein barycenters $\breve{\nu}$ and $\tilde{\nu}$ look different as $\breve{\nu}$ is formed with a larger number of blobs where each one is more concentrated. 
However, when $m_0\ge2400$, $\breve{\nu}$ and $\tilde{\nu}$ begin to look indistinguishable as they both exhibit the same spatial pattern. 
This is an indication that both $\tilde{\nu}$ and $\breve{\nu}$ are accurate approximations of the true Wasserstein barycenter of $\mu_1,\ldots,\mu_5$.

\begin{figure}[t]
    \includegraphics[width=0.33\linewidth]{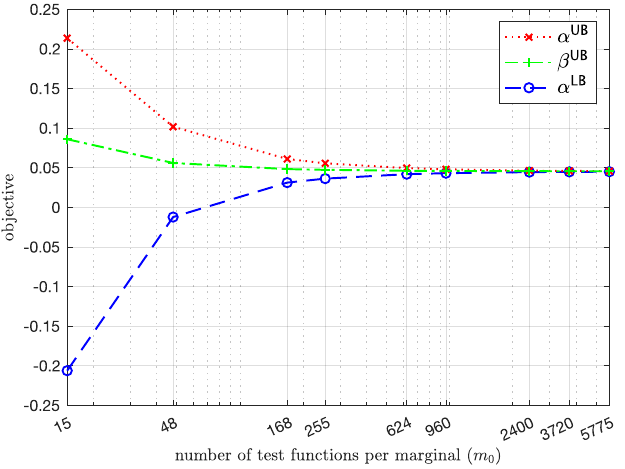}~
    \includegraphics[width=0.33\linewidth]{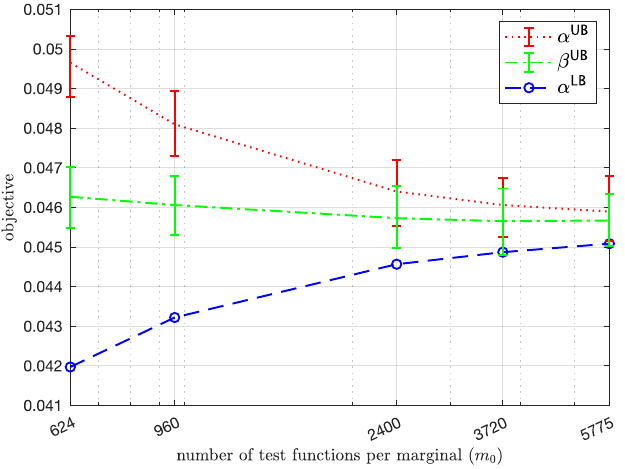}~
    \includegraphics[width=0.33\linewidth]{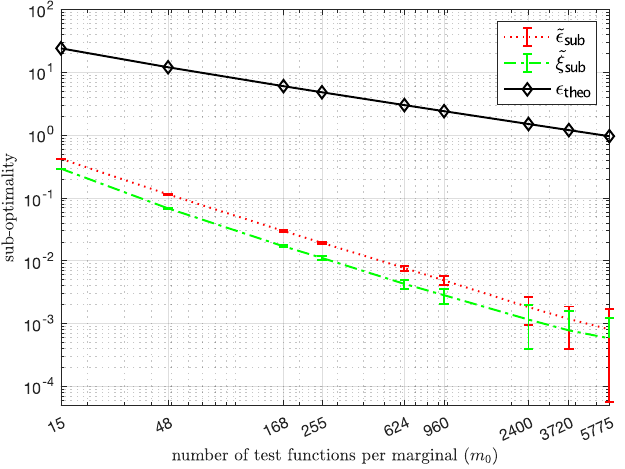}
    
    \caption{\textbf{Experiment~2} -- Left: the lower bound $\alpha^{\SFL\SFB}$ and the upper bounds $\alpha^{\SFU\SFB}$, $\beta^{\SFU\SFB}$. 
    Center: magnification of the right half of the left panel. The error bars indicate the Monte Carlo errors when computing the upper bounds. 
    Right: the sub-optimality estimates $\tilde{\epsilon}_{\mathsf{sub}}$, $\tilde{\xi}_{\mathsf{sub}}$ and their a priori upper bound $\epsilon_{\mathsf{theo}}$ on the log-scale.}\label{fig:barycenter-bounds}
\end{figure}

Figure~\ref{fig:barycenter-bounds} shows the values of $\alpha^{\mathsf{LB}}$, $\alpha^{\mathsf{UB}}$, $\beta^{\mathsf{UB}}$, $\tilde{\epsilon}_{\mathsf{sub}}$, $\tilde{\xi}_{\mathsf{sub}}$, and $\epsilon_{\mathsf{theo}}$ plotted against the number~$m_0$ of test functions per marginal. 
The left panel of Figure~\ref{fig:barycenter-bounds} shows the values of $\alpha^{\mathsf{LB}}$, $\alpha^{\mathsf{UB}}$, and $\beta^{\mathsf{UB}}$. 
When the value of~$m_0$ is small, the lower bound $\alpha^{\mathsf{LB}}$ is far from the upper bounds $\alpha^{\mathsf{UB}}$, $\beta^{\mathsf{UB}}$. 
As $m_0$ increases, the differences between the bounds shrink. 
We provide a magnified version of the $m_0\ge 624$ part of the left panel in the center panel of Figure~\ref{fig:barycenter-bounds}. 
The error bars indicate the 95\% Monte Carlo intervals of the upper bounds. 
When $m_0=5775$, we have $\tilde{\epsilon}_{\mathsf{sub}}:=\alpha^{\mathsf{UB}}-\alpha^{\mathsf{LB}}=8.1107\times 10^{-4}$ and $\tilde{\xi}_{\mathsf{sub}}:=\beta^{\mathsf{UB}}-\alpha^{\mathsf{LB}}=5.8522\times 10^{-4}$ (here $\alpha^{\mathsf{UB}}$ and $\beta^{\mathsf{UB}}$ are averages of the 100 Monte Carlo integrals), indicating that both $\tilde{\nu}$ and $\breve{\nu}$ are close to the true Wasserstein barycenter of $\mu_1,\ldots,\mu_5$. 
In addition, we show the values of the sub-optimality estimates $\tilde{\epsilon}_{\mathsf{sub}}$, $\tilde{\xi}_{\mathsf{sub}}$ and their a priori upper bound $\epsilon_{\mathsf{theo}}$ on the log-scale in the right panel of Figure~\ref{fig:barycenter-bounds}.
Whenever the lower bound $\alpha^{\mathsf{LB}}$ falls within the 95\% Monte Carlo intervals of the upper bound $\alpha^{\mathsf{UB}}$ or $\beta^{\mathsf{UB}}$, the lower branch of the corresponding error bar is omitted.
Observe that the computed sub-optimality estimates $\tilde{\epsilon}_{\mathsf{sub}}$, $\tilde{\xi}_{\mathsf{sub}}$ are around two to three orders of magnitude smaller than their a priori upper bound $\epsilon_{\mathsf{theo}}$, and they seem to decrease at a faster rate compared to $\epsilon_{\mathsf{theo}}$. 
This coincides with our observations in Experiment~1 in Section~\ref{ssec:experiments-fluid}. 
Moreover, $\beta^{\mathsf{UB}}< \alpha^{\mathsf{UB}}$ holds for all values of~$m_0$, indicating that $\breve{\nu}$ is more accurate than $\tilde{\nu}$ as an approximate Wasserstein barycenter.
Furthermore, the difference between $\alpha^{\mathsf{UB}}$ and $\beta^{\mathsf{UB}}$ is more pronounced for small values of~$m_0$ and becomes insignificant for $m_0\ge 2400$. 
This is in agreement with our qualitative observations from Figure~\ref{fig:barycenter-outputs} and Figure~\ref{fig:barycenter-W2OToutputs}.

\subsubsection*{Comparison with existing methods}
Figure~\ref{fig:barycenter-comparisons} compares the bounds $\alpha^{\mathsf{LB}}$, $\alpha^{\mathsf{UB}}$ computed by Algorithm~\ref{alg:mmot}
when $m_0=624$
with the outputs of the RKHS-based algorithm and the NN-based algorithm,
where box plots are used to visualize the variability among 20 independent runs of each SGD algorithm.
One can observe the same dilemma discussed in Experiment~1, 
that is, 
a small value of $\gamma$ leads to large biases in the objective values,
and a large value of~$\gamma$ negatively affects convergence.
Overall, neither the RKHS-based algorithm nor the NN-based algorithm was able to reliably approximate the optimal value of \eqref{eqn:mmot}, 
while the solutions computed by our algorithm are not only feasible but also close to being optimal.

\begin{figure}[t]
    \includegraphics[width=0.45\linewidth]{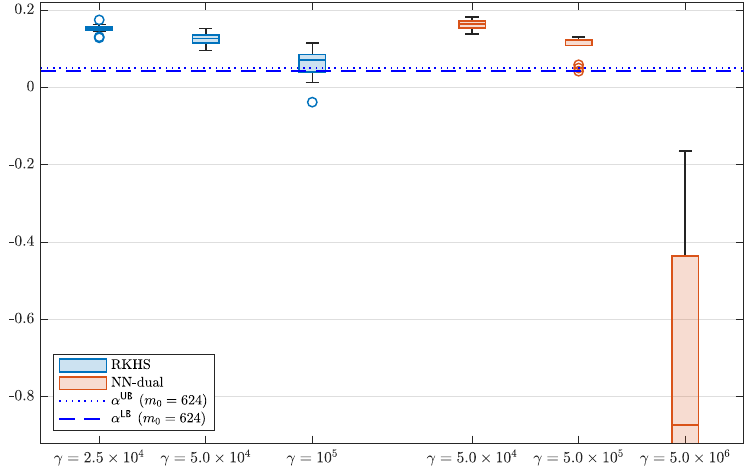}
    
    \caption{\textbf{Experiment~2} -- Comparison of the objective values computed by Algorithm~\ref{alg:mmot}, the RKHS-based algorithm, and the NN-based algorithm.}\label{fig:barycenter-comparisons}
\end{figure}

\subsection{Experiment~3: continuous piece-wise affine cost function}\label{ssec:experiments-CPWA}

In the third numerical experiment, we consider a large-scale MMOT problem with $N=100$ marginals.
Specifically, for $i=1,\ldots,N$, we let $\CX_i:=[-10, 10]$ and let $\mu_i\in\CP(\CX_i)$ be a mixture of normal distributions truncated to $\CX_i$, where the number of mixture components is randomly generated between 3~and~5 and the parameters of each mixture component are also randomly generated. 
Moreover, we let $f:\BCX\to\R$ be the following continuous piece-wise affine (CPWA) function:
\begin{align*}
\begin{split}
f(\BIx):=\left(\sum_{k=1}^{2}\big|\langle\BIs^+_{k},\BIx\rangle-t^+_{k}\big|\right)-\left(\sum_{k=1}^{2}\big|\langle\BIs^-_{k},\BIx\rangle-t^-_{k}\big|\right),
\end{split}
%\label{eqn:numerics-exp-f}
\end{align*}
where $\BIs^+_{1},\BIs^+_{2},\BIs^-_{1},\BIs^-_{2}\in\R^N$ are randomly generated from the unit sphere in $\R^N$, and $t^+_{1},t^+_{2},t^-_{1},t^-_{2}\in\nobreak\R$ are randomly generated real numbers. 
Notice that $f$ is neither convex nor concave, and that $f(\BIx)$ cannot be separated into a sum of functions involving disjoint components of $\BIx$ (otherwise \eqref{eqn:mmot} can be decomposed into independent sub-problems). 
We chose this $f$ in order to demonstrate the performance of Algorithm~\ref{alg:mmot} in a large-scale setting when $N$ is large. 

\subsubsection*{Experimental settings}
To approximately solve \eqref{eqn:mmot} with respect to the above cost function~$f$,
for $i=1,\ldots,N$,
we define
$\FC_i=\big\{[\kappa_{i,0},\kappa_{i,1}],\ldots,\allowbreak[\kappa_{i,m_0-1},\kappa_{i,m_0}]\big\}$,
where $m_0\in\N$ is varied among four values: 4, 8, 16, and 32,
and 
${-10}=\kappa_{i,0}<\kappa_{i,1}<\cdots<\kappa_{i,m_0}=10$ 
satisfy 
$\mu_i\big([\kappa_{i,j-1},\kappa_{i,j}]\big)=\frac{1}{m_0}$ $\forall 1\le
\nobreak j\le \nobreak m_0$. 
Subsequently, we construct 
$\CG_1,\ldots,\CG_N$, $\BIg(\,\cdot\,)$, $\bar{\BIg}$ by \ref{settc:compact-testfuncs},
which yields
\begin{align*}
g_{i,j}(x)&:=\frac{(x-\kappa_{i,j-1})^+}{\kappa_{i,j}-\kappa_{i,j-1}}\wedge \frac{(\kappa_{i,j+1}-x)^+}{\kappa_{i,j+1}-\kappa_{i,j}} && \forall x\in[-10,10],\;\forall 1\le j\le m_0-1,\; \forall 1\le i\le N,\\
g_{i,m_0}(x)&:=\frac{(x-\kappa_{i,m_0-1})^+}{\kappa_{i,m_0}-\kappa_{i,m_0-1}} && \forall x\in[-10,10],\; \forall 1\le i\le N.
%\label{eqn:numerics-exp-ifb}
\end{align*}
Moreover, for each value of $m_0$, we fix $\epsilon_{\mathsf{LSIP}}=10^{-4}$,
and derive 
the a priori error bound $\epsilon_{\mathsf{theo}}$ via Theorem~\ref{thm:mmotalgo} with respect to $\rho_i\leftarrow 2\eta(\FC_i)= 2\max_{1\le j\le m_0}\{\kappa_{i,j}-\kappa_{i,j-1}\}$ $\forall 1\le i\le N$.
We remark that $\mathtt{Oracle}(\,\cdot\,)$ in Definition~\ref{def:mmot-oracle} can be formulated into a mixed-integer linear programming problem and solved with Gurobi~\citep{gurobi}, which is a state-of-the-art mixed-integer solver. 
Furthermore, since $\CX_1,\ldots,\CX_N$ are all one-dimensional, the reassembly $\tilde{\mu}\in R(\hat{\mu};\mu_1,\ldots,\mu_N)$
in Line~\ref{alglin:mmot-reassembly} of Algorithm~\ref{alg:mmot}
is constructed by Proposition~\ref{prop:reassembly-1d}\ref{props:reassembly-1d-semidisc}. 
Subsequently, the approximate computation of $\alpha^{\SFU\SFB}$ in Line~\ref{alglin:mmot-bounds} of Algorithm~\ref{alg:mmot} is done via Monte Carlo integration with $10^{6}$ independent samples. 
The Monte Carlo step is repeated 100 times in order to examine the Monte Carlo errors.

\begin{figure}[t]
    \includegraphics[width=0.45\linewidth]{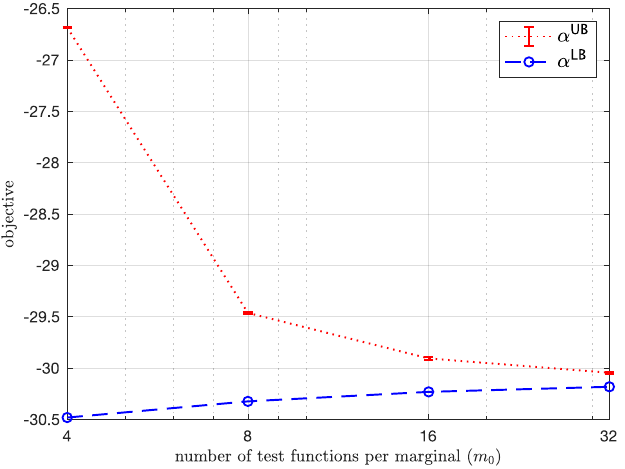}~
    \includegraphics[width=0.45\linewidth]{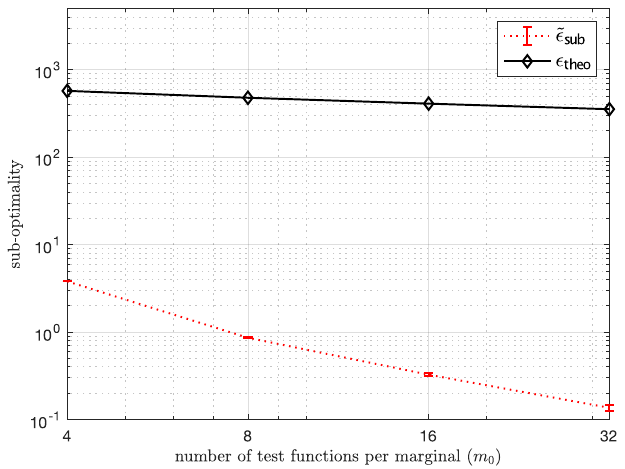}
    
    \caption{\textbf{Experiment~3} -- Left: the lower bound $\alpha^{\SFL\SFB}$ and the upper bound $\alpha^{\SFU\SFB}$. The error bars indicate the Monte Carlo errors when computing the upper bound. 
    Right: the sub-optimality estimate $\tilde{\epsilon}_{\mathsf{sub}}$ and its a priori upper bound $\epsilon_{\mathsf{theo}}$ on the log-scale.}\label{fig:CPWA-bounds}
\end{figure}

\subsubsection*{Results and discussions}
The results in this experiment are shown in Figure~\ref{fig:CPWA-bounds}. 
The error bars in Figure~\ref{fig:CPWA-bounds} indicate the 95\% Monte Carlo intervals of the upper bounds. 
The left panel of Figure~\ref{fig:CPWA-bounds} shows the computed values of the lower bound $\alpha^{\SFL\SFB}$ and the upper bound $\alpha^{\SFU\SFB}$.
Similar to the results in Experiment~1 and Experiment~2, the bounds are far apart when $m_0=4$ and get closer when $m_0$ increases. 
When $m_0=32$, we get $\alpha^{\SFL\SFB}=-30.1813$, $\alpha^{\SFU\SFB}=-30.0451$, and their difference $\tilde{\epsilon}_{\mathsf{sub}}=0.1362$. 
This shows that when $m_0=32$ test functions are used for each marginal, the computed approximately optimal solution $\tilde{\mu}$ has a decently small sub-optimality. 
Moreover, the sub-optimality estimates~$\tilde{\epsilon}_{\mathsf{sub}}$ in the right panel of Figure~\ref{fig:CPWA-bounds} are two to three orders of magnitude smaller than their a priori upper bounds~$\epsilon_{\mathsf{theo}}$. 
Note that in this experiment, $\epsilon_{\mathsf{theo}}$ decreases slowly as $m_0$ increases due to the non-uniform choices of $\kappa_{i,0},\kappa_{i,1},\ldots,\kappa_{i,m_0}$. 
This experiment demonstrates that even when the cost function $f$ has a non-trivial structure and when the number $N$ of marginals is large, our algorithm is capable of computing approximately optimal solutions of \eqref{eqn:mmot} with high accuracy. 
Since the RKHS-based algorithm and the NN-based algorithm have 
both experienced difficulties in converging to approximately optimal solutions in Experiment~1 when $N$ is large, 
we did not test their performances in Experiment~3.

\section{Extended results: construction of moment sets with controlled supremum $W_p$-metrics in the possibly unbounded Euclidean case}\label{sec:momentset-extended}

In this section, we present our extended results on the explicit construction of moment sets that generalize our results in Proposition~\ref{prop:momentset-simplex} to the possibly unbounded Euclidean case, and to the control of the supremum $W_p$-metric for general $p\in[1,\infty)$.

Let us first introduce the following notions of polyhedral cover, vertex interpolation function set, and $p$-radial function set.

\begin{definition}[Polyhedral cover]\label{def:polycover}
Let $d\in\N$, let $\R^d$ be equipped with a norm $\|\cdot\|$, and let $\CY\subseteq\R^d$ be closed. 
A collection $\FC$ of subsets of $\R^d$ is called a polyhedral cover of $\CY$ if 
$\FC$ is a finite collection of polyhedra 
such that $V(C)\neq\emptyset$ $\forall C\in\FC$,
$\bigcup_{C\in\FC}C\supseteq\CY$,
and $C_1\cap C_2\in\FF(C_1)\cap \FF(C_2)$ $\forall C_1,C_2\in\FC$.
A polyhedral cover $\FC$ of $\CY$ is called bounded if every $C\in\FC$ is bounded (which only exists if $\CY$ is bounded). 
Let $\FF(\FC):=\bigcup_{C\in\FC}\FF(C)$,
$V(\FC):=\bigcup_{C\in\FC}V(C)$, and
$D(\FC):=\bigcup_{C\in\FC}D(C)$. 
Moreover, let $\eta(\FC):=\max_{C\in\FC}\max_{\BIv,\BIv'\in V(C)}\big\{\|\BIv-\BIv'\|\big\}$ denote the mesh size of a polyhedral cover $\FC$.
\end{definition}
Notice that the finite collection $\FC$ of $d$-simplices in Proposition~\ref{prop:momentset-simplex} is an example of a bounded polyhedral cover.

\begin{definition}[Vertex interpolation function set and $p$-radial function set]\label{def:cover-interp}
Let $d\in\N$, let $\R^d$ be equipped with a norm $\|\cdot\|$, and let $\FC$ be a polyhedral cover of a closed set $\CY\subseteq\R^{d}$.
\begin{itemize}[leftmargin=15pt,beginpenalty=10000]
    \item A set of continuous and non-negative functions $\big\{(g_{\BIv}:\bigcup_{C\in\FC}C\to\R):\BIv\in V(\FC)\big\}$ is called a 
    vertex interpolation function set 
    (abbreviated to VIFS)
    for $\FC$ if it satisfies the properties \ref{defs:cover-vif-normalize} and \ref{defs:cover-vif-disjoint} below.\widowpenalty-1000
    \begin{enumerate}[label=\normalfont{(VIFS\arabic*)},leftmargin=4em,beginpenalty=20000]
        \item\label{defs:cover-vif-normalize}%
        For every $F\in\FF(\FC)$ and every $\BIx\in F$, $\sum_{\BIv\in V(F)}g_{\BIv}(\BIx)=1$.

        \item\label{defs:cover-vif-disjoint}%
        For every $F\in\FF(\FC)$, every $\BIx\in F$, and every $\BIv\in V(\FC)\setminus V(F)$, $g_{\BIv}(\BIx)=0$. 
    \end{enumerate}
    In particular, \ref{defs:cover-vif-normalize} and \ref{defs:cover-vif-disjoint} imply that $\big\{g_{\BIv}:\BIv\in V(\FC)\big\}$ satisfies the following orthonormality property: 
    \begin{align}
        g_{\BIv}(\BIv')=\INDI_{\{\BIv=\BIv'\}} \qquad \forall \BIv,\BIv'\in V(\FC).
        \label{eqn:cover-vif-orthonormality}
    \end{align}

    \item When $\FC$ is not bounded (hence $D(\FC)\ne\emptyset$), a set of continuous and non-negative functions ${\big\{(\overline{g}_{\BIu}:\bigcup_{C\in\FC}C\to\R):\BIu\in D(\FC)\big\}}$ is called a 
    $p$-radial function set 
    (abbreviated to $p$-RFS)
    for $\FC$ with respect to some $p\in[1,\infty)$ if it satisfies the property \ref{defs:cover-rf-bound} below.
    \begin{enumerate}[label=\normalfont{($p$-RFS)},leftmargin=4em]
        \item\label{defs:cover-rf-bound}%
        For every unbounded $F\in\FF(\FC)$ and every $\BIx\in F$, 
        \begin{align*}
        \left(\min_{\BIy\in\conv(V(F))}\big\{\|\BIx-\BIy\|\big\}\right)^p\le \sum_{\BIu\in D(F)}\overline{g}_{\BIu}(\BIx).
        \end{align*}
    \end{enumerate}
\end{itemize}
\end{definition}

The above definitions can be interpreted as follows. 
Let us consider a polyhedral cover $\FC$ of $\CY$ and a VIFS $\big\{g_{\BIv}:\BIv\in V(\FC)\big\}$ for $\FC$. 
The property \ref{defs:cover-vif-normalize} requires that the functions $\big\{g_{\BIv}:\BIv\in V(\FC)\big\}$ form a non-negative partition of unity on each face $F\in\FF(\FC)$.
The property \ref{defs:cover-vif-disjoint} requires each $g_{\BIv}(\,\cdot\,)$ to be local to the faces adjacent to the vertex $\BIv$, 
i.e., 
$g_{\BIv}(\,\cdot\,)$ is local to
$\{F\in\FF(\FC):\BIv\in V(F)\}$.
In the case where $\FC$ is not bounded, let us consider a $p$-RFS $\big\{\overline{g}_{\BIu}:\BIu\in D(\FC)\big\}$ for~$\FC$. 
As required by the property \ref{defs:cover-rf-bound}, for an unbounded face $F\in\FF(\FC)$, the functions $\big\{\overline{g}_{\BIu}:\BIu\in D(F)\big\}$ control the $p$-th power of the distance traveled when transporting a point $\BIx\in F$ to the bounded set $\conv(V(F))$.
Informally speaking, the functions $\big\{\overline{g}_{\BIu}:\BIu\in D(\FC)\big\}$ are used to control the rate at which probability mass can ``escape to infinity'' in each possible direction $\BIu\in D(\FC)$.

The following proposition shows that the simplicial test functions $\CG_{\mathsf{simp}}(\FC)$ in Proposition~\ref{prop:momentset-simplex} constitute a VIFS for $\FC$.

\begin{proposition}[Simplicial test functions are VIFS]\label{prop:momentset-simplex-VIFS}
    Let $d\in\R$, let $\CY\subset\R^d$ be compact, and let $d_{\CY}$ be a metric induced by a norm $\|\cdot\|$ on $\R^d$.
    Let $\FC$ be a finite collection of $d$-simplices which satisfy 
    $\bigcup_{C\in\FC}C\supseteq\CY$ and 
    $C_1\cap C_2\in \FF(C_1)\cap \FF(C_2)$ $\forall C_1,C_2\in\FC$.
    Then, $\FC$ is a bounded polyhedral cover of~$\CY$.
    Moreover, the simplitial test functions $\CG_{\mathsf{simp}}(\FC)$ defined in Proposition~\ref{prop:momentset-simplex} is a VIFS for $\FC$.
\end{proposition}

\begin{proof}[Proof of Proposition~\ref{prop:momentset-simplex-VIFS}]
    See Section~\ref{ssec:proof-momentset-extended}.
\end{proof}

In the following, let us demonstrate the explicit construction of a finite set of functions whose linear combinations 
contain a VIFS as well as a $p$-RFS for a polyhedral cover consisting of hyperrectangles instead of simplices.

\begin{proposition}[VIFS and $p$-RFS for hyperrectangles]\label{prop:cover-cube}
Let $d\in\N$ and let $\R^d$ be equipped with a norm $\|\cdot\|$. 
For $i=1,\ldots,d$, let $n_i\in\N$, $\beta_i>0$, $\underline{\kappa}_i\in\R$, and let $\kappa_{i,j}:=\underline{\kappa}_i+j\beta_i$ for $j=0,\ldots,n_i$. 
Moreover, for $i=1,\ldots,d$, let $\CI_i:=\big\{[\kappa_{i,0},\kappa_{i,1}],\ldots,[\kappa_{i,n_i-1},\kappa_{i,n_i}]\big\}$, let $\overline{\CI}_i:=\CI_i \cup \big\{(-\infty,\kappa_{i,0}],[\kappa_{i,n_i},\infty)\big\}$, and let $\FC_0$, $\FC$, $\CG_0$, and $\CG_p$ for $p\in[1,\infty)$ be defined as follows:
\begin{align*}
%\begin{split}
    \FC_0&:=\big\{I_1\times\cdots\times I_d : I_i\in\CI_i\;\forall 1\le i\le d\big\},\\
    \FC&:=\big\{I_1\times\cdots\times I_d : I_i\in\overline{\CI}_i\;\forall 1\le i\le d\big\},\allowdisplaybreaks\\
    \CG_0&:=\bigg\{\R^d\ni(x_1,\ldots,x_d)^\TRANSP\mapsto \Big(\max_{i\in L}\big\{\beta_i^{-1}(x_i-\kappa_{i,j_i})^+\big\}\Big)^+\in\R : \\
    &\hspace{14em} j_i\in\{0,1,\ldots,n_i\}\;\forall i\in L,\; L\subseteq\{1,\ldots,d\}\bigg\},\allowdisplaybreaks\\
    \CG_p&:=\begin{cases}
        \CG_0 \cup \big\{\R^d\ni(x_1,\ldots,x_d)^\TRANSP\mapsto x_i\in\R:1\le i\le d\big\} & \text{if }p=1,\\
        \CG_0 \cup \big\{\R^d\ni(x_1,\ldots,x_d)^\TRANSP\mapsto \big((\kappa_{i,0}-x_i)^+\big)^p\in\R:1\le i\le d\big\}\\
        \phantom{\CG_0}\cup \big\{\R^d\ni(x_1,\ldots,x_d)^\TRANSP\mapsto \big((x_i-\kappa_{i,n_i})^+\big)^p\in\R:1\le i\le d\big\} & \text{if }p>1.
    \end{cases}
%\end{split}
%\label{eqn:cover-cube-basis}
\end{align*}
Then, the following statements hold.
\begin{enumerate}[label=(\roman*)]
    \item\label{props:cover-cube-bounded}%
    $\FC_0$ is a polyhedral cover of $\bigtimes_{i=1}^d[\kappa_{i,0},\kappa_{i,n_i}]$ 
    (and so also of any closed subset $\CY\subseteq \bigtimes_{i=1}^d[\kappa_{i,0},\kappa_{i,n_i}]$)
    and
    one can explicitly construct $\CG_{\mathsf{VIFS}}\subset\lspan_1(\CG_0)$ which is a VIFS for~$\FC_0$.

    \item\label{props:cover-cube-unbounded}%
    $\FC$ is a polyhedral cover of $\R^d$ (and so also of any closed subset $\CY\subseteq\R^d$) and, 
    for all $p\in[1,\infty)$, 
    one can explicitly construct $\CG_{\mathsf{VIFS}}\subset\lspan_1(\CG_p)$ which is a VIFS for $\FC$.
    Moreover, for all $p\in[1,\infty)$, 
    let $C_{\|\cdot\|}\ge 1$ be a constant such that $\|\BIx\|\le C_{\|\cdot\|}\|\BIx\|_p$ for all $\BIx\in\R^d$.
    Then,
    \begin{align*}
        \CG_{p\text{-}\mathsf{RFS}}&:=\Big\{\R^d\ni(x_1,\ldots,x_d)^\TRANSP\mapsto \big(C_{\|\cdot\|}(\kappa_{i,0}-x_i)^+\big)^p\in\R:1\le i\le d\Big\} \\
        &\qquad \cup \Big\{\R^d\ni(x_1,\ldots,x_d)^\TRANSP\mapsto \big(C_{\|\cdot\|}(x_i-\kappa_{i,n_i})^+\big)^p\in\R:1\le i\le d\Big\}
    \end{align*}
    is a $p$-RFS for $\FC$
    and $\CG_{p\text{-}\mathsf{RFS}}\subset\lspan_1(\CG_p)$.
\end{enumerate}
\end{proposition}

\begin{proof}[Proof of Proposition~\ref{prop:cover-cube}]
See Section~\ref{ssec:proof-momentset-extended}.
\end{proof}

For $p\in[1,\infty)$ and for a polyhedral cover $\FC$ of a closed set $\CY\subseteq\R^d$, 
the following theorem establishes an upper bound for $W_p(\mu,\nu)$ for all $\mu,\nu\in\CP_p(\CY;\CG)$ satisfying $\mu\overset{\CG}{\sim}\nu$
whenever $\lspan_1(\CG)$ contains a VIFS and a $p$-RFS for~$\FC$.

\begin{theorem}[$W_p$ upper bound for a moment set in the Euclidean case]\label{thm:momentset-euclidean}
Let $p\in[1,\infty)$, $d\in\N$, let $\CY\subseteq\R^d$ be closed, and let $d_{\CY}$ be a metric induced by a norm $\|\cdot\|$ on $\R^d$. 
Let $\FC$ be a polyhedral cover of~$\CY$ with mesh size $\eta(\FC)$.
Then, the following statements hold.
\begin{enumerate}[label=(\roman*)]
    \item\label{thms:momentset-euclidean-bounded}%
    If $\FC$ is bounded
    (which means $\CY$ is compact)
    and $\CG$ is 
    a collection of $\R$-valued Borel measurable functions where $\lspan_1(\CG)$ contains a VIFS for $\FC$, 
    then it holds that
    \begin{align*}
    \forall \mu,\nu\in\CP_p(\CY;\CG):\qquad \mu\overset{\CG}{\sim}\nu \quad \Rightarrow \quad W_p(\mu,\nu)\le 2\eta(\FC).
    \end{align*}
    
    \item\label{thms:momentset-euclidean-unbounded}%
    If $\FC$ is not bounded and $\CG$ 
    is a collection of $\R$-valued Borel measurable functions 
    where $\lspan_1(\CG)$ contains both a VIFS and a $p$-RFS $\big\{\overline{g}_{\BIu}:\BIu\in D(\FC)\big\}$ for $\FC$, 
    then it holds that
    \begin{align*}
        \qquad\qquad\forall \mu,\nu\in\CP_p(\CY;\CG): \qquad \mu\overset{\CG}{\sim}\nu \quad \Rightarrow \quad W_p(\mu,\nu)\le 2\eta(\FC)+2\left(\sum_{\BIu\in D(\FC)}\int_{\CY}\overline{g}_{\BIu}\DIFFX{\mu}\right)^{\frac{1}{p}}.
    \end{align*}
\end{enumerate}
\end{theorem}

\begin{proof}[Proof of Theorem~\ref{thm:momentset-euclidean}]
See Section~\ref{ssec:proof-momentset-extended}.
\end{proof}

Theorem~\ref{thm:momentset-euclidean}\ref{thms:momentset-euclidean-bounded} and Proposition~\ref{prop:momentset-simplex-VIFS} 
provide us with an explicit estimate of the number of test functions in $\CG$ needed in order to control 
$W_p(\mu,\nu)\le\nobreak\epsilon$ for all pairs of $\mu,\nu\in\CP(\CY)$ satisfying $\mu\overset{\CG}{\sim}\nu$ in the compact Euclidean case.
This is presented below in Corollary~\ref{cor:momentset-scalability}.
Note that Proposition~\ref{prop:momentset-errorcontrol} is a consequence of Corollary~\ref{cor:momentset-scalability}.

\begin{corollary}[Number of test functions in $\CG$ to control $\sup_{\mu\overset{\CG}{\sim}\nu}\{W_p(\mu,\nu)\}$]\label{cor:momentset-scalability}
Let $p\in[1,\infty)$, $d\in\N$, let $\CY\subseteq\bigtimes_{i=1}^d[\underline{M}_i,\overline{M}_i]$ be closed, where $-\infty<\underline{M}_i<\overline{M}_i<\infty$ for $i=1,\ldots,d$, and let $d_{\CY}$ be a metric induced by a norm $\|\cdot\|$ on $\R^d$. 
Let $\epsilon>0$ be arbitrary, let $C_{\|\cdot\|}\ge 1$ be a constant such that $\|\BIx\|\le C_{\|\cdot\|}\|\BIx\|_p$ for all $\BIx\in\R^d$, and let $n_i:=\Big\lceil\frac{2(\overline{M}_i-\underline{M}_i)C_{\|\cdot\|} d^{1/p}}{\epsilon}\Big\rceil$ for $i=1,\ldots,d$. 
Next, let 
\begin{align*}
    \widehat{V}:=\Big\{\big({\textstyle \underline{M}_{1}+\frac{j_{1}}{n_{1}}(\overline{M}_{1}-\underline{M}_{1}), \ldots, \underline{M}_{d}+\frac{j_{d}}{n_{d}}(\overline{M}_{d}-\underline{M}_{d})}\big)^\TRANSP: 
    (j_1,\ldots,j_d)\in\textstyle\bigtimes_{i=1}^{d}\{0,1,\ldots,n_i\}\Big\},
\end{align*}
and let 
$\FC$ be a Delaunay triangulation of $\widehat{V}$,
that is,
$\FC$ is a polyhedral cover of 
$\bigtimes_{i=1}^d[\underline{M}_i,\overline{M}_i]$
where every $C\in\FC$ is a $d$-simplex, 
$V(\FC)=\widehat{V}$,
and the Delaunay condition holds:
for each $C\in\FC$, its unique circum-hypersphere $S_{C}$ satisfies
$\inter\big(\conv(S_C)\big)\cap V(\FC)=\nobreak \emptyset$.
See, for example, \citep{elshakhs2024comprehensive}
for a review of explicit procedures for constructing a Delaunay triangulation.
Note that Delaunay triangulation of $\widehat{V}$ is not unique here.
Subsequently, let $\CG_{\mathsf{simp}}(\FC)$ be the simplicial test functions generated by $\FC$ defined in Proposition~\ref{prop:momentset-simplex}.
Then, it holds that 
$\big|\CG_{\mathsf{simp}}(\FC)\big|=\prod_{i=1}^d\Big(1+\Big\lceil\frac{2(\overline{M}_i-\underline{M}_i)C_{\|\cdot\|} d^{1/p}}{\epsilon}\Big\rceil\Big)$ and $W_p(\mu,\nu)\le\epsilon$ for any $\mu,\nu\in\CP_p\big(\CY;\CG_{\mathsf{simp}}(\FC)\big)$ satisfying $\mu\overset{\CG_{\mathsf{simp}}(\FC)}{\scalebox{3.5}[1]{$\sim$}}\nu$.
\end{corollary}

\begin{proof}[Proof of Corollary~\ref{cor:momentset-scalability}]
    See Section~\ref{ssec:proof-momentset-extended}.
\end{proof}

The following corollary is a consequence of Proposition~\ref{prop:cover-cube} and Theorem~\ref{thm:momentset-euclidean}\ref{thms:momentset-euclidean-unbounded} which has a natural interpretation in the context of mathematical finance, as discussed in Remark~\ref{rmk:breedenlitzdd}. 
\begin{corollary}\label{cor:breedenlitzenberger}
Let $p\in[1,\infty)$, $d\in\N$, let $\CY\subseteq\R^d$ be closed, and let $d_{\CY}$ be a metric induced by a norm $\|\cdot\|$ on $\R^d$. 
Let $\mu\in\CP_p(\CY)$ and let $\mu_i\in\CP_p(\R)$ denote the $i$-th marginal of $\mu$ for $i=1,\ldots,d$. 
Let $\beta>0$ and let $C_{\|\cdot\|}\ge1$ be a constant such that $\|\BIx\|\le C_{\|\cdot\|}\|\BIx\|_p$ for all $\BIx\in\R^d$.
For $i=1,\ldots,d$, let $n_i\in\N$, $\underline{\kappa}_i\in\R$, and let $\kappa_{i,j}:=\underline{\kappa}_i+j\beta$ for $j=0,\ldots,n_i$. 
Moreover, let $\CG_p$ be a finite collection of functions given by
\begin{align*}
%\begin{split}
\CG_0&:=\bigg\{\R^d\ni(x_1,\ldots,x_d)^\TRANSP\mapsto \left(\max_{i\in L}\big\{(x_i-\kappa_{i,j_i})^+\big\}\right)^+\in\R : \\
&\hspace{10.5em} j_i\in\{0,1,\ldots,n_i\}\;\forall i\in L,\; L\subseteq\{1,\ldots,d\}\bigg\},\\
\CG_p&:=\begin{cases}
    \CG_0 \cup \big\{\R^d\ni(x_1,\ldots,x_d)^\TRANSP\mapsto x_i\in\R:1\le i\le d\big\} & \text{if }p=1,\\
    \CG_0 \cup \big\{\R^d\ni(x_1,\ldots,x_d)^\TRANSP\mapsto \big((\kappa_{i,0}-x_i)^+\big)^p\in\R:1\le i\le d\big\} \\
    \hspace{13.8pt} \cup\hspace{2.3pt} \big\{\R^d\ni(x_1,\ldots,x_d)^\TRANSP\mapsto \big((x_i-\kappa_{i,n_i})^+\big)^p\in\R:1\le i\le d\big\} & \text{if }p>1.
\end{cases}
%\end{split}
%\label{eqn:best-of-call}
\end{align*}
Then, 
\begin{align}
\begin{split}
\specialoverline{W}_{p}(\mu,[\mu]_{\CG_p})&\le 2C_{\|\cdot\|} d^{1/p}\beta+2C_{\|\cdot\|}\left(\sum_{i=1}^d\int_{\R}\big((\kappa_{i,0}-x_i)^+\big)^p+\big((x_i-\kappa_{i,n_i})^+\big)^p\DIFFM{\mu_i}{\DIFF x_i}\right)^{\frac{1}{p}}.
\end{split}
\label{eqn:breedenlitzenberger-bound}
\end{align} 
In particular, for any $\epsilon>0$, one can choose $\beta>0$, $(n_i)_{i=1:d}\subset\N$, and $(\underline{\kappa}_i)_{i=1:d}\subset\R$ such that
$\CG_p$ satisfies
$\specialoverline{W}_{p}(\mu,[\mu]_{\CG_p})\le\epsilon$.
\end{corollary}

\begin{proof}[Proof of Corollary~\ref{cor:breedenlitzenberger}]
See Section~\ref{ssec:proof-momentset-extended}.
\end{proof}

\begin{remark}[Financial interpretation of Corollary~\ref{cor:breedenlitzenberger}]\label{rmk:breedenlitzdd}
Corollary~\ref{cor:breedenlitzenberger} has a natural interpretation in mathematical finance. 
Consider a financial market where $d\in\N$ risky assets are traded. Let $\CY\subseteq\R^d$ (typically $\CY=\R^d_+$) be a closed set that corresponds to the possible prices of these assets at a fixed future time, called the maturity. 
Then, for $i=1,\ldots,d$, the function
\begin{align*}
    \CY\ni(x_1,\ldots,x_d)^\TRANSP \quad \mapsto\quad x_i\in\R
\end{align*}
corresponds to the payoff at maturity when investing into a single unit of asset $i$. 
Moreover, for $i=1,\ldots,d$, the functions
\begin{align*}
    \CY\ni(x_1,\ldots,x_d)^\TRANSP \quad \mapsto\quad \left((x_i-\kappa_i)^+\right)^p\in\R
\end{align*}
corresponds to an exotic option whose payoff is equal to that of a European call option with strike price~$\kappa_i$ raised to the $p$-th power, and the function
\begin{align*}
    \CY\ni(x_1,\ldots,x_d)^\TRANSP \quad \mapsto\quad \left((\kappa_i-x_i)^+\right)^p\in\R
\end{align*}
corresponds to an exotic option whose payoff is equal to that of a European put option with strike price~$\kappa_i$ raised to the $p$-th power.
Furthermore, for any non-empty set $L\subseteq\{1,\ldots,d\}$ and any $(\kappa_{i})_{i\in L}\subset\R$, the function
\begin{align*}
 \CY\ni(x_1,\ldots,x_d)^\TRANSP \quad \mapsto\quad \left(\max_{i\in L}\big\{(x_i-\kappa_{i})^+\big\}\right)^+\in\R
\end{align*}
corresponds to the payoff of a best-of-call option (a type of financial derivative) written on the assets in the set $L$ with strike prices $(\kappa_{i})_{i\in L}$. 
Let us consider $\mu\in\CP_p(\CY)$ as a risk-neutral pricing measure for this financial market, and let $\CG_p$ be defined by Corollary~\ref{cor:breedenlitzenberger}. 
In the case where $p=1$, the set $[\mu]_{\CG_p}$ corresponds to the set of risk-neutral pricing measures that produce the same forward prices (for each of the assets) as well as the same prices of best-of-call options written on any non-empty subset $L$ of the assets with strikes 
$\big\{\underline{\kappa}_i+j\beta:j\in\{0,1,\ldots, n_i\},\;i\in L\big\}$. 
In the case where $p>1$, the set $[\mu]_{\CG_p}$ corresponds to the set of risk-neutral pricing measures that produce the same prices of the exotic options with power payoffs with strikes $(\kappa_{i,0})_{i=1:d},(\kappa_{i,n_i})_{i=1:d}$ as well as the same prices of best-of-call options written on any non-empty subset $L$ of the assets with strikes 
$\big\{\underline{\kappa}_i+j\beta:j\in\{0,1,\ldots, n_i\},\;i\in L\big\}$.
$\specialoverline{W}_{p}(\mu,[\mu]_{\CG_p})$ is thus the supremum model risk measured via the $W_p$-metric when we only assume the knowledge of forward prices (or the aforementioned exotic option prices in the case where $p>1$) and the aforementioned best-of-call option prices. 

Corollary~\ref{cor:breedenlitzenberger} states that, for any $\epsilon>0$, one can select finitely many best-of-call options to control the supremum model risk to $\specialoverline{W}_{p}(\mu,[\mu]_{\CG_p})\le\epsilon$.
This is related to the classical result of \citet{breeden1978prices}, which states that: for $\mu\in\CP_1(\R)$ that is absolutely continuous with respect to the Lebesgue measure, if the function
\begin{align*}
\R\ni \kappa \quad\mapsto\quad \int_{\R}(x-\kappa)^+\DIFFM{\mu}{\DIFF x}\in\R
\end{align*}
is twice continuously differentiable, then it uniquely characterizes the density of $\mu$. 
\citet{talponen2014note} later generalized this result to the multi-dimensional case. 
Theorem~2.1 of \citep{talponen2014note} states that: for $\mu\in\CP_1(\R^d_+)$ that is absolutely continuous with respect to the Lebesgue measure, the density of $\mu$ is uniquely characterized by the function
\begin{align*}
\R^d_+\ni(\kappa_1,\ldots,\kappa_d) \quad\mapsto\quad \int_{\R^m_+}\left(\max_{1\le i\le d}\big\{(x_i-\kappa_{i})^+\big\}\right)^+\DIFFM{\mu}{\DIFF x_1,\ldots,\DIFF x_d}\in\R.
\end{align*}
Corollary~\ref{cor:breedenlitzenberger} can therefore be seen as a non-asymptotic generalization of \citep[Theorem~2.1]{talponen2014note}. 
\end{remark}

\begin{remark}\label{rmk:comparison-with-Alphonsi-2}
\citet{alfonsi2021approximation} have also developed results on controlling $\sup_{\mu\overset{\CG}{\sim}\nu}\big\{W_p(\mu,\nu)\big\}$ with specific classes of test functions $\CG$.
However, they have only considered the case where the underlying space $\CY$ is a compact interval in $\R$, and they have only constructed classes of discontinuous test functions, while their convergence results (see \citep[Theorem~4.1 \& Proposition~4.2]{alfonsi2021approximation}) rely on the assumption that the test functions are all continuous.
Moreover, discontinuity of the test functions not only complicates the duality results but also makes it hard to treat the relaxed problem numerically; recall that Proposition~\ref{prop:duality-settings} and Theorem~\ref{thm:mmot-complexity} both require the continuity of the test functions in $\CG_1,\ldots,\CG_N$. 
In contrast, the test functions we have constructed in this section to control $\sup_{\mu\overset{\CG}{\sim}\nu}\big\{W_p(\mu,\nu)\big\}$ are all continuous, and our construction can account for $d\ge2$ dimensions as well as the case where $\CY$ is unbounded. 
\end{remark}

\section{Proofs of theoretical results}\label{sec:proof-theory}

\subsection{Proofs of results in Section~\ref{ssec:relaxation}}\label{ssec:proof-approximation}

%% Proof of Proposition (existence of a sparse optimizer)
\begin{proof}[Proof of Proposition~\ref{prop:mmotlb-sparsity}]
Let us first prove statement~\ref{props:mmotlb-sparsity-disc}.
Let us fix an arbitrary $\epsilon_0>0$ as well as an arbitrary $\epsilon_0$-optimal solution $\mu_{\epsilon_0}\in\Gamma\big([\mu_1]_{\CG_1},\ldots,[\mu_N]_{\CG_N}\big)$ of \eqref{eqn:mmotlb},
i.e., $\int_{\BCX}f\DIFFX{\mu_{\epsilon_0}}\le \inf_{\mu\in\Gamma([\mu_1]_{\CG_1},\ldots,[\mu_N]_{\CG_N})}\big\{\int_{\BCX}f\DIFFX{\mu}\big\} + \epsilon_0$. 
For $i=1,\ldots,N$, denote the $i$-th marginal of $\mu_{\epsilon_0}$ by $\mu_{i,\epsilon_0}\in\CP(\CX_i)$.
It thus holds that $\mu_{i,\epsilon_0}\in[\mu_i]_{\CG_i}$.
Define $\Bphi:\BCX\to\R^{m+2}$ by 
\begin{align*}
\Bphi(x_1,\ldots,x_N)&:=\big(1, g_{1,1}(x_1),\ldots, g_{1,m_1}(x_1),\ldots,g_{N,1}(x_N), \ldots, g_{N,m_N}(x_N),f(x_1,\ldots,x_N)\big)^\TRANSP\\
&\hspace{290pt}\forall (x_1,\ldots,x_N)\in\BCX.
\end{align*}
By an application of Tchakaloff's theorem in \citep[Corollary~2]{bayer2006proof}, there exist $1\le q\le m+2$, $(a_{k})_{k=1:q}\subset(0,1)$, and 
$(\BIx_k)_{k=1:q}\subseteq\BCX$ such that
\begin{align}
\sum_{k=1}^qa_k&=\int_{\BCX}1\DIFFX{\mu_{\epsilon_0}}=1,\label{eqn:mmotlb-sparsity-proof-normalize}  \\
\sum_{k=1}^qa_kg_{i,j}(\pi_i(\BIx_k))&=\int_{\BCX}g_{i,j}\DIFFX{\mu_{i,\epsilon_0}}=\int_{\CX_i}g_{i,j}\DIFFX{\mu_i} \qquad\forall 1\le j\le m_i,\;\forall 1\le i\le N,\label{eqn:mmotlb-sparsity-proof-moment}\\
\sum_{k=1}^qa_k f(\BIx_k)&=\int_{\BCX}f\DIFFX{\mu_{\epsilon_0}}\le \inf_{\mu\in\Gamma([\mu_1]_{\CG_1},\ldots,[\mu_N]_{\CG_N})}\bigg\{\int_{\BCX}f\DIFFX{\mu}\bigg\} + \epsilon_0.\label{eqn:mmotlb-sparsity-proof-optimal}
\end{align}
Let $\hat{\mu}:=\sum_{k=1}^qa_k\delta_{\BIx_k}$. 
Then, it follows from (\ref{eqn:mmotlb-sparsity-proof-normalize}) that $\hat{\mu}\in\CP(\BCX)$. 
For $i=1,\ldots,N$, let us denote the marginal of $\hat{\mu}$ on $\CX_i$ by $\hat{\mu}_i$.
Subsequently, (\ref{eqn:mmotlb-sparsity-proof-moment}) guarantees that
\begin{align*}
    \int_{\CX_i}g_{i,j}\DIFFX{\hat{\mu}_i}=\int_{\BCX}{g_{i,j}\circ\pi_i}\DIFFX{\hat{\mu}}=\sum_{k=1}^qa_kg_{i,j}(\pi_i(\BIx_k))=\int_{\CX_i}g_{i,j}\DIFFX{\mu_i} \qquad \forall 1\le j\le m_i,\; \forall 1\le i\le N,
\end{align*}
and it hence holds that $\hat{\mu}\in\Gamma\big([\mu_1]_{\CG_1},\ldots,[\mu_N]_{\CG_N}\big)$. 
Finally, (\ref{eqn:mmotlb-sparsity-proof-optimal}) implies that
\begin{align*}
\int_{\BCX}f\DIFFX{\hat{\mu}}=\sum_{k=1}^qa_kf(\BIx_k)\le \inf_{\mu\in\Gamma([\mu_1]_{\CG_1},\ldots,[\mu_N]_{\CG_N})}\bigg\{\int_{\BCX}f\DIFFX{\mu}\bigg\} + \epsilon_0,
\end{align*}
which shows that $\hat{\mu}$ is an $\epsilon_0$-optimal solution of \eqref{eqn:mmotlb}. This proves statement~\ref{props:mmotlb-sparsity-disc}.
To prove statement~\ref{props:mmotlb-sparsity-cont}, observe that when $\CX_1,\ldots,\CX_N$ are compact and all test functions $(g_{i,j})_{j=1:m_i,\,i=1:N}$ are continuous, an optimal solution $\mu^\star$ of \eqref{eqn:mmotlb} is attained since $\Gamma\big([\mu_1]_{\CG_1},\ldots,[\mu_N]_{\CG_N}\big)$ is a closed subset of the compact metric space $\big(\CP(\BCX),W_1\big)$ (see, e.g., \citep[Remark~6.19]{villani2008optimal}) and the mapping $\CP(\BCX)\ni\mu\mapsto\int_{\BCX}f\DIFFX{\mu}\in\R$ is lower semi-continuous (see, e.g., \citep[Lemma~4.3]{villani2008optimal}). The statement then follows from the same argument used in the proof of statement~\ref{props:mmotlb-sparsity-disc} with $\mu_{\epsilon_0}$ replaced by $\mu^\star$ and with $\epsilon_0$ replaced by~0. 
The proof is now complete. 
\end{proof}

%% Proof of Lemma (reassembly)
\begin{proof}[Proof of Lemma~\ref{lem:reassembly}]
Let $\hat{\mu}_1,\ldots,\hat{\mu}_N$ denote the marginals of $\hat{\mu}$ on $\CX_1,\ldots,\CX_N$, respectively. 
Since $\hat{\mu}\in\CP_1(\BCX)$, we have $\hat{\mu}_i\in\CP_1(\CX_i)$ for $i=1,\ldots,N$ by (\ref{eqn:1prod-metric}). 
Moreover, for $i=1,\ldots,N$, the existence of an optimal coupling $\gamma_i$ of $\hat{\mu}_i$ and $\mu_i$ under the cost function $d_{\CX_i}$ follows from the non-negativity and continuity of $d_{\CX_i}$ and \citep[Theorem~4.1]{villani2008optimal}. The existence of a probability measure $\gamma\in\CP(\CX_1\times\cdots\times\CX_N\times\allowbreak\bar{\CX}_1\times\nobreak\cdots\times\nobreak\bar{\CX}_N)$ that satisfies the conditions in Definition~\ref{def:reassembly} follows from the following inductive argument that repeatedly applies Lemma~\ref{lem:gluing}. 
Specifically, 
let $\gamma^{(0)}:=\hat{\mu}\in\Gamma(\hat{\mu}_1,\ldots,\hat{\mu}_N)\subseteq\CP(\CX_1\times\cdots\times\CX_N)$.
Then, for $i=1,\ldots,N$, one applies Lemma~\ref{lem:gluing} with
$\CY_1\leftarrow \big(\bigtimes_{1\le j\le N,\,j\ne i}\CX_j\big)\times\big(\bigtimes_{1\le k\le i-1}\bar{\CX}_k\big)$,
$\CY_2\leftarrow\CX_i$,
$\CY_3\leftarrow\bar{\CX}_i$,
$\gamma_{12}\leftarrow \gamma^{(i-1)}$,
$\gamma_{23}\leftarrow \gamma_i$
to ``glue together'' 
$\gamma^{(i-1)}\in\Gamma(\hat{\mu}_1,\ldots,\hat{\mu}_{N},\mu_1,\ldots,\mu_{i-1})$ and 
$\gamma_i\in\Gamma(\hat{\mu}_i,\mu_i)$ and obtain 
$\gamma^{(i)}\in\Gamma(\hat{\mu}_1,\ldots,\hat{\mu}_{N},\mu_1,\ldots,\mu_{i})\subseteq\CP(\CX_1\times\cdots\times\CX_N\times\bar{\CX}_1\times\cdots\times\nobreak\bar{\CX}_i)$.
One may check that $\gamma^{(N)}$ satisfies all the required properties of $\gamma$ and thus letting $\gamma:=\gamma^{(N)}$ completes the construction. 
Finally, one checks that the marginal $\tilde{\mu}$ of $\gamma$ on $\bar{\CX}_1\times\cdots\times\bar{\CX}_N$ satisfies $\tilde{\mu}\in R(\hat{\mu};\mu_1,\ldots,\mu_N)$. 
The proof is now complete.
\end{proof}

%% Proof of Theorem (MMOT approximation)
\begin{proof}[Proof of Theorem~\ref{thm:lowerbound}]
To prove statement~\ref{thms:lowerbound-control}, 
let us first denote $\hat{\BIx}:=(\hat{x}_1,\ldots,\hat{x}_N)$
and prove that $f\in\CL^1(\BCX,\tilde{\mu})$.
Observe that since 
$(\mu_1,\ldots,\mu_N,f)\in\CA(L_f,D,\underline{f}_1,\overline{f}_1,\ldots,\underline{f}_N,\overline{f}_N)$ and
$\hat{\BIx}\in D$, 
it holds that
$f(\BIx)\ge f(\hat{\BIx})-L_fd_{\BCX}(\hat{\BIx},\BIx)\ge -\big|f(\hat{\BIx})\big|-L_fd_{\BCX}(\hat{\BIx},\BIx)$
$\forall\BIx\in D$,
and that
$f(\BIx)\INDI_{\BCX\setminus D}(\BIx)\ge \sum_{i=1}^{N}\underline{f}_i\circ\pi_i(\BIx)$ $\forall\BIx\in\BCX$.
Thus, combining these two inequalities and
using the assumption
$d_{\CX_i}(\hat{x}_i,\cdot\,)\le h_i(\,\cdot\,)$ for $i=\nobreak1,\ldots,N$ yield
\begin{align}
\begin{split}
f(\BIx)&=f(\BIx)\INDI_D(\BIx) + f(\BIx)\INDI_{\BCX\setminus D}(\BIx) \\
&\ge {-\big|f(\hat{\BIx})\big|}-L_f\left(\sum_{i=1}^Nd_{\CX_i}(\hat{x}_i,x_i)\right)+\left(\sum_{i=1}^N\underline{f}_i(x_i)\right)\\
&\ge -\big|f(\hat{\BIx})\big|+\left(\sum_{i=1}^N \underline{f}_i(x_i) - L_fh_i(x_i)\right) \qquad \forall \BIx=(x_1,\ldots,x_N)\in\BCX.
\end{split}
\label{eqn:lowerbound-proof-boundedness}
\end{align}
Moreover, we similarly get
\begin{align}
    f(\BIx)\le \big|f(\hat{\BIx})\big|+\left(\sum_{i=1}^N \overline{f}_i(x_i) + L_fh_i(x_i)\right) \qquad \forall \BIx=(x_1,\ldots,x_N)\in\BCX.
    \label{eqn:lowerbound-proof-boundedness-above}
\end{align}
Since $\underline{f}_i,\overline{f}_i\in\CL^1(\CX_i,\mu_i)$, 
$h_i\in\lspan_1(\CG_i)\subseteq\CL^1(\CX_i,\mu_i)$
for $i=1,\ldots,N$ by assumption,
and since $\tilde{\mu}\in R(\hat{\mu};\mu_1,\ldots,\mu_N)\subseteq\Gamma(\mu_1,\ldots,\mu_N)$,
combining (\ref{eqn:lowerbound-proof-boundedness}) and (\ref{eqn:lowerbound-proof-boundedness-above}) proves that $f\in\CL^1(\BCX,\tilde{\mu})$.
Next,
let us split $\int_{\BCX}f\DIFFX{\tilde{\mu}}-\int_{\BCX}f\DIFFX{\hat{\mu}}$ into two parts:
\begin{align}
\int_{\BCX}f\DIFFX{\tilde{\mu}}-\int_{\BCX}f\DIFFX{\hat{\mu}}=\left(\int_{D}f\DIFFX{\tilde{\mu}}-\int_{D}f\DIFFX{\hat{\mu}}\right)+\left(\int_{\BCX\setminus D}f\DIFFX{\tilde{\mu}}-\int_{\BCX\setminus D}f\DIFFX{\hat{\mu}}\right),
\label{eqn:lowerbound-step1}
\end{align}
and control them separately. 
By the assumption that $\tilde{\mu}\in R(\hat{\mu};\mu_1,\ldots,\mu_N)$ and Definition~\ref{def:reassembly}, there exists a probability measure $\gamma\in\CP(\CX_1\times\cdots\times\CX_N\times\bar{\CX}_1\times\cdots\times\bar{\CX}_N)$
where $\bar{\CX}_i:=\CX_i$ for $i=1,\ldots,N$, 
such that the marginal of $\gamma$ on $\CX_1\times\cdots\times\CX_N$ is $\hat{\mu}$, the marginal $\gamma_i\in\Gamma(\hat{\mu}_i,\mu_i)$ of $\gamma$ on $\CX_i\times\bar{\CX}_i$ satisfies $\int_{\CX_i\times\bar{\CX}_i}d_{\CX_i}(x,\bar{x})\DIFFM{\gamma_i}{\DIFF x,\DIFF \bar{x}}=W_1(\hat{\mu}_i,\mu_i)$ for $i=1,\ldots,N$, and the marginal of $\gamma$ on $\bar{\CX}_1\times\cdots\times\bar{\CX}_N$ is~$\tilde{\mu}$. 
Thus, we have by 
the definition of $\CA(L_f,D,\underline{f}_1,\overline{f}_1,\ldots,\underline{f}_N,\overline{f}_N)$ in Definition~\ref{def:controlsets} and the definition of $d_{\BCX}(\,\cdot\,,\cdot\,)$ in Assumption~\ref{asp:productspace} that
\begin{align}
\begin{split}
\int_{D}f\DIFFX{\tilde{\mu}}-\int_{D}f\DIFFX{\hat{\mu}}&= \int_{D\times D}f(\bar{\BIx})-f(\BIx)\DIFFM{\gamma}{\DIFF \BIx, \DIFF \bar{\BIx}}\\
&\le \int_{D\times D}L_fd_{\BCX}(\BIx,\bar{\BIx})\DIFFM{\gamma}{\DIFF\BIx,\DIFF\bar{\BIx}}\\
&\le \int_{\BCX\times \BCX}L_fd_{\BCX}(\BIx,\bar{\BIx})\DIFFM{\gamma}{\DIFF\BIx,\DIFF\bar{\BIx}}\\
&=\sum_{i=1}^N L_f\int_{\CX_i\times\CX_i}d_{\CX_i}(x_i,\bar{x}_i)\DIFFM{\gamma_i}{\DIFF x_i,\DIFF\bar{x}_i}=\sum_{i=1}^N L_fW_1(\hat{\mu}_i,\mu_i).
\end{split}
\label{eqn:lowerbound-step1-1}
\end{align}
On the other hand, 
since $\tilde{\mu}\in \Gamma(\mu_1,\ldots,\mu_N)$, 
and since $\hat{\mu}_i\in[\mu_i]_{\CG_i}$, $\underline{f}_i\in\lspan_1(\CG_i)$ for $i=1,\ldots,N$, it follows that
\begin{align}
\begin{split}
\int_{\BCX\setminus D}f\DIFFX{\tilde{\mu}}-\int_{\BCX\setminus D}f\DIFFX{\hat{\mu}}&\le\int_{\BCX}{\textstyle\sum_{i=1}^N\overline{f}_i\circ\pi_i}\DIFFX{\tilde{\mu}}-\int_{\BCX}{\textstyle\sum_{i=1}^N\underline{f}_i\circ\pi_i}\DIFFX{\hat{\mu}}\\
&=\sum_{i=1}^N\int_{\CX_i}\overline{f}_i\DIFFX{\mu_i}-\int_{\CX_i}\underline{f}_i\DIFFX{\hat{\mu}_i}=\sum_{i=1}^N\int_{\CX_i}\overline{f}_i-\underline{f}_i\DIFFX{\mu_i}.
\end{split}
\label{eqn:lowerbound-step1-2}
\end{align}
Subsequently, combining (\ref{eqn:lowerbound-step1}), (\ref{eqn:lowerbound-step1-1}), and (\ref{eqn:lowerbound-step1-2}) completes the proof of statement~\ref{thms:lowerbound-control}.

Furthermore, 
notice that
integrating (\ref{eqn:lowerbound-proof-boundedness}) with respect to any $\hat{\mu}\in\Gamma\big([\mu_1]_{\CG_1},\ldots,[\mu_N]_{\CG_N}\big)$ 
and 
using the assumption
$\underline{f}_i,h_i\in\lspan_1(\CG_i)\subseteq\CL^1(\CX_i,\mu_i)$ for $i=\nobreak1,\ldots,N$ result in
\begin{align*}
\int_{\BCX}f\DIFFX{\hat{\mu}}&\ge -\big|f(\hat{\BIx})\big|+\sum_{i=1}^N\int_{\BCX}\underline{f}_i\circ\pi_i - L_f h_i\circ \pi_i\DIFFX{\hat{\mu}}= -\big|f(\hat{\BIx})\big|+\sum_{i=1}^N\int_{\CX_i}\underline{f}_i-L_fh_i\DIFFX{\mu_i}>-\infty.
\end{align*}
This proves statement~\ref{thms:lowerbound-boundedness}.

To prove statement~\ref{thms:lowerbound-w1radfinite},
observe that the assumptions $d_{\CX_i}(\hat{x}_i,\cdot\,)\le h_i(\,\cdot\,)$ and $h_i\in\lspan_1(\CG_i)\subseteq\CL^1(\CX_i,\mu_i)$ for $i=\nobreak1,\ldots,N$
guarantee that
\begin{align*}
W_1(\mu_i,\nu_i)&\le  W_1(\mu_i,\delta_{\hat{x}_i})+W_1(\nu_i,\delta_{\hat{x}_i})\\
&=W_1(\mu_i,\delta_{\hat{x}_i})+\int_{\CX_i}d_{\CX_i}(\hat{x}_i,x)\DIFFM{\nu_i}{\DIFF x}\\
&\le W_1(\mu_i,\delta_{\hat{x}_i})+\int_{\CX_i}h_i(x)\DIFFM{\nu_i}{\DIFF x}\\
&=W_1(\mu_i,\delta_{\hat{x}_i})+\int_{\CX_i}h_i(x)\DIFFM{\mu_i}{\DIFF x}<\infty \qquad \forall \nu_i\in[\mu_i]_{\CG_i},\;\forall 1\le i\le N.
\end{align*}
It thus holds for $i=1,\ldots,N$ that $\specialoverline{W}_{1}(\mu_i,[\mu_i]_{\CG_i})=\sup_{\nu_i\in[\mu_i]_{\CG_i}}\big\{W_1(\mu_i,\nu_i)\big\}\le W_1(\mu_i,\delta_{\hat{x}_i})+\int_{\CX_i}h_i(x)\DIFFM{\mu_i}{\DIFF x}<\infty$, which proves statement~\ref{thms:lowerbound-w1radfinite}. 

Let us now prove statement~\ref{thms:lowerbound-epsilonoptimal}. 
For $i=1,\ldots,N$, let $\hat{\mu}_i\in\CP_1(\CX_i)$ denote the marginal of $\hat{\mu}$ on $\CX_i$. We have $\hat{\mu}_i\in[\mu_i]_{\CG_i}$ by the definition of $\Gamma\big([\mu_1]_{\CG_1},\ldots,[\mu_N]_{\CG_N}\big)$. 
By statement~\ref{thms:lowerbound-control}, we have
\begin{align}
\begin{split}
\int_{\BCX}f\DIFFX{\tilde{\mu}}-\int_{\BCX}f\DIFFX{\hat{\mu}}&\le \sum_{i=1}^N L_f W_1(\mu_i,\hat{\mu}_i)+\int_{\CX_i}\overline{f}_i-\underline{f}_i \DIFFX{\mu_i}\\
&\le \sum_{i=1}^N L_f \specialoverline{W}_{1}(\mu_i,[\mu_i]_{\CG_i})+\int_{\CX_i}\overline{f}_i-\underline{f}_i \DIFFX{\mu_i}.
\end{split}
\label{eqn:lowerbound-proof-part3}
\end{align}
Moreover, since it holds by definition that $\mu_i\in[\mu_i]_{\CG_i}$ for $i=1,\ldots,N$, we have \sloppy{$\Gamma(\mu_1,\ldots,\mu_N)\subseteq\Gamma\big([\mu_1]_{\CG_1},\ldots,[\mu_N]_{\CG_N}\big)$}. Hence,
\begin{align}
\inf_{\mu\in\Gamma([\mu_1]_{\CG_1},\ldots,[\mu_N]_{\CG_N})}\bigg\{\int_{\BCX}f\DIFFX{\mu}\bigg\}\le\inf_{\mu\in\Gamma(\mu_1,\ldots,\mu_N)}\bigg\{\int_{\BCX}f\DIFFX{\mu}\bigg\}.
\label{eqn:lowerbound-proof-relaxation}
\end{align}
We then combine (\ref{eqn:lowerbound-proof-part3}), (\ref{eqn:lowerbound-epsilon-optimizer}), (\ref{eqn:lowerbound-proof-relaxation}), and the definition of $\epsilon_{\mathsf{apx}}$ to finish the proof of statement~\ref{thms:lowerbound-epsilonoptimal}. 

Finally, let us prove statement~\ref{thms:lowerbound-mmotcontrol}. 
The first inequality in (\ref{eqn:lowerbound}) follows from (\ref{eqn:lowerbound-proof-relaxation}). 
Moreover, for every $\hat{\mu}\in\Gamma\big([\mu_1]_{\CG_1},\ldots,[\mu_N]_{\CG_N}\big)$ and every $\tilde{\mu}\in R(\hat{\mu};\mu_1,\ldots,\mu_N)\subseteq\Gamma(\mu_1,\ldots,\mu_N)$, (\ref{eqn:lowerbound-proof-part3}) implies that
\begin{align*}
\inf_{\mu\in\Gamma(\mu_1,\ldots,\mu_N)}\bigg\{\int_{\BCX}f\DIFFX{\mu}\bigg\}-\int_{\BCX}f\DIFFX{\hat{\mu}} &\le \int_{\BCX}f\DIFFX{\tilde{\mu}}-\int_{\BCX}f\DIFFX{\hat{\mu}}\\
&\le \sum_{i=1}^N \left(L_f \specialoverline{W}_{1}(\mu_i,[\mu_i]_{\CG_i})+\int_{\CX_i}\overline{f}_i-\underline{f}_i \DIFFX{\mu_i}\right)<\infty.
\end{align*} 
Taking the infimum over $\hat{\mu}\in\Gamma\big([\mu_1]_{\CG_1},\ldots,[\mu_N]_{\CG_N}\big)$ in the inequality above proves the second inequality in (\ref{eqn:lowerbound}). 
The proof is now complete. 
\end{proof}

%% Proof of Theorem (MMOT optimizer convergence)
\begin{proof}[Proof of Theorem~\ref{thm:lowerboundconverge}]
By a multi-marginal extension of \citep[Lemma~4.4]{villani2008optimal}, one can show that the set of probability measures $\Gamma(\mu_1,\ldots,\mu_N)$ is weakly precompact. 
Hence, $\big(\tilde{\mu}^{(l)}\big)_{l\in\N}\subseteq\Gamma(\mu_1,\ldots,\mu_N)$ has at least one weakly convergent subsequence. Now, assume without loss of generality that $\tilde{\mu}^{(l)}$ converges weakly to ${\tilde{\mu}\in\CP(\BCX)}$ as $l\to\infty$. 
For $i=1,\ldots,N$ and for any bounded continuous function $h:\CX_i\to\R$, we have
\begin{align*}
\int_{\BCX}h\circ\pi_i\DIFFX{\tilde{\mu}}=\lim_{l\to\infty}\int_{\BCX}h\circ\pi_i\DIFFX{\tilde{\mu}^{(l)}}=\int_{\CX_i}h\DIFFX{\mu_i}.
\end{align*}
Thus, $\tilde{\mu}\in\Gamma(\mu_1,\ldots,\mu_N)$. Moreover, for any $\hat{\BIx}=(\hat{x}_1,\ldots,\hat{x}_N)\in\BCX$, we have
by the definition of $d_{\BCX}(\,\cdot\,,\cdot\,)$ in Assumption~\ref{asp:productspace} that
\begin{align*}
\lim_{l\to\infty}\int_{\BCX}d_{\BCX}(\hat{\BIx},\BIx)\DIFFM{\tilde{\mu}^{(l)}}{\DIFF \BIx}=\sum_{i=1}^N\int_{\CX_i}d_{\CX_i}(\hat{x}_i,x)\DIFFM{\mu_i}{\DIFF x}=\int_{\BCX}d_{\BCX}(\hat{\BIx},\BIx)\DIFFM{\tilde{\mu}}{\DIFF \BIx}. 
\end{align*}
Therefore, it follows from \citep[Definition~6.8]{villani2008optimal} and \citep[Theorem~6.9]{villani2008optimal} that $\tilde{\mu}^{(l)}\to\tilde{\mu}$ in $\big(\CP_1(\BCX),W_1\big)$ as $l\to\infty$. 
For every $l\in\N$, we have by (\ref{eqn:relaxed-epsilon-optimizer}) and Theorem~\ref{thm:lowerbound}\ref{thms:lowerbound-epsilonoptimal} that 
\begin{align*}
\inf_{\mu\in\Gamma(\mu_1,\ldots,\mu_N)}\bigg\{\int_{\BCX}f\DIFFX{\mu}\bigg\}&\ge \int_{\BCX}f\DIFFX{\tilde{\mu}^{(l)}}-\left(\sum_{i=1}^N L^{(l)}_f \specialoverline{W}_{1}\big(\mu_i,[\mu_i]_{\CG^{(l)}_i}\big)+\int_{\CX_i}\overline{f}^{(l)}_i-\underline{f}^{(l)}_i \DIFFX{\mu_i}\right)-\epsilon_0^{(l)}.
\end{align*}
Moreover, it follows from Assumption~\ref{asp:mmotexistence} and a multi-marginal extension of  \citep[Lemma~4.3]{villani2008optimal} that $\liminf_{l\to\infty}\big\{\int_{\BCX}f\DIFFX{\tilde{\mu}^{(l)}}\big\}\ge\int_{\BCX}f\DIFFX{\tilde{\mu}}$. 
Hence,
\begin{align*}
\int_{\BCX}f\DIFFX{\tilde{\mu}}&\ge \inf_{\mu\in\Gamma(\mu_1,\ldots,\mu_N)}\bigg\{\int_{\BCX}f\DIFFX{\mu}\bigg\}\\
&\ge \liminf_{l\to\infty}\left[\int_{\BCX}f\DIFFX{\tilde{\mu}^{(l)}}-\left(\sum_{i=1}^N L^{(l)}_f \specialoverline{W}_{1}\big(\mu_i,[\mu_i]_{\CG^{(l)}_i}\big)+\int_{\CX_i}\overline{f}^{(l)}_i-\underline{f}^{(l)}_i \DIFFX{\mu_i}\right)-\epsilon_0^{(l)}\right]\\
&\ge \int_{\BCX}f\DIFFX{\tilde{\mu}}.
\end{align*}
This shows that $\tilde{\mu}\in\Gamma(\mu_1,\ldots,\mu_N)$ is an optimal solution of \eqref{eqn:mmot}. The proof is now complete. 
\end{proof}

\subsection{Proofs of results in Section~\ref{ssec:duality}}\label{ssec:proof-duality}

%% Proof of Theorem (duality)
\begin{proof}[Proof of Theorem~\ref{thm:duality}]
For any $y_0\in\R$ and $\BIy\in\R^m$ that satisfy $y_0+\langle\BIg(\BIx),\BIy\rangle\le f(\BIx)$ $\forall\BIx\in\BCX$, and for any $\mu\in\Gamma\big([\mu_1]_{\CG_1},\ldots,[\mu_N]_{\CG_N}\big)$, it holds that
\begin{align*}
y_0+\langle\bar{\BIg},\BIy\rangle=\int_{\BCX}y_0+\langle\BIg(\BIx),\BIy\rangle\DIFFM{\mu}{\DIFF\BIx}\le\int_{\BCX}f(\BIx)\DIFFM{\mu}{\DIFF\BIx}.
%\label{eqn:duality-proof-weak}
\end{align*}
Taking the supremum over all such $y_0\in\R$ and $\BIy\in\R^m$ and taking the infimum over all $\mu\in\Gamma\big([\mu_1]_{\CG_1},\ldots,[\mu_N]_{\CG_N}\big)$ in the above inequality yields the weak duality (\ref{eqn:weak-duality}). This proves statement~\ref{thms:duality-weak}.

Now, to establish the strong duality, we assume that the left-hand side of (\ref{eqn:weak-duality}) is not equal to~$-\infty$. 
We first show that $\bar{\BIg}\in\clos(K)$. Suppose for the sake of contradiction that $\bar{\BIg}\notin\clos(K)$. 
Then, 
due to the strong separation of the closed convex set $\clos(K)$ and the compact and convex set $\{\bar{\BIg}\}$
(see, e.g., \citep[Corollary~11.4.2]{rockafellar1970convex}), there exist $\BIy\in\R^m$ and $\alpha>0$ such that $\langle\BIw,\BIy\rangle-\langle\bar{\BIg},\BIy\rangle\ge\alpha$ for all $\BIw\in\clos(K)$. In particular, we have $\langle\BIg(\BIx),\BIy\rangle-\langle\bar{\BIg},\BIy\rangle\ge\alpha$ for all $\BIx\in\BCX$. 
However, this implies that 
\begin{align*}
0=\langle\bar{\BIg},\BIy\rangle-\langle\bar{\BIg},\BIy\rangle=\int_{\BCX}\langle\BIg(\BIx),\BIy\rangle-\langle\bar{\BIg},\BIy\rangle\DIFFM{\mu}{\DIFF\BIx}\ge\alpha>0 \qquad \forall \mu\in\Gamma(\mu_1,\ldots,\mu_N),
\end{align*}
which is a contradiction. This shows that $\bar{\BIg}\in\clos(K)$. 

Next, to prove statement~\ref{thms:duality1}, let us first suppose that the condition~\ref{thmc:duality1} holds, i.e., $\bar{\BIg}\in\relint(K)$. 
Let $U:=\cone\big(\big\{(1,\BIg(\BIx)^\TRANSP)^\TRANSP:\BIx\in\BCX\big\}\big)=\cone\big(\big\{(1,\BIu^\TRANSP)^\TRANSP:\BIu\in K\big\}\big)\subset\R^{m+1}$. 
By \citep[Corollary~6.8.1]{rockafellar1970convex}, it holds that
\begin{align}
\relint(U)=\big\{(\lambda,\lambda\BIu^\TRANSP)^\TRANSP:\lambda>0,\;\BIu\in\relint(K)\big\}.
\label{eqn:lsip-duality-relint-cone}
\end{align}
Under the assumption that $\bar{\BIg}\in\relint(K)$, we have $(1,\bar{\BIg}^\TRANSP)^\TRANSP\in\relint(U)$, and thus by \citep[Theorem~8.2]{goberna1998linear} (see the fifth case in \citep[Table~8.1]{goberna1998linear}), with $M\leftarrow U$, $N\leftarrow C$, $K\leftarrow\cone\big(C\cup\big\{(\veczero_{m+1}^\TRANSP,1)^\TRANSP\big\}\big)$, $c\leftarrow(1,\bar{\BIg}^\TRANSP)^\TRANSP$ in the notation of \citep{goberna1998linear} (see also \citep[p.~81 \& p.~49]{goberna1998linear}), the left-hand side of (\ref{eqn:duality}) is equal to the optimal value of the following minimization problem:
\begin{align}
\begin{split}
\minimize_{(a_l,\,\BIx_l)}\quad & \sum_{l=1}^k a_l f(\BIx_l)\\
\text{subject to} \quad & \sum_{l=1}^k a_l=1,\qquad \sum_{l=1}^k a_l\BIg(\BIx_l)=\bar{\BIg},\\
& k\in\N,\qquad (a_l)_{l=1:k}\subset\R_+,\qquad (\BIx_l)_{l=1:k}\subset\BCX.
\end{split}
\label{eqn:lsip-duality}
\end{align}
Notice that for any $(a_l)_{l=1:k}\subset\R_+$, $(\BIx_l)_{l=1:k}\subseteq\BCX$ that is feasible for (\ref{eqn:lsip-duality}), 
we have $\hat{\mu}:=\sum_{l=1}^ka_l\delta_{\BIx_l}\in\CP(\BCX)$.
Subsequently, it holds by (\ref{eqn:g-vecdef}) and (\ref{eqn:v-vecdef}) that $\hat{\mu}$ satisfies 
\begin{align*}
\int_{\BCX}g_{i,j}\circ\pi_i\DIFFX{\hat{\mu}}&=\sum_{l=1}^ka_l g_{i,j}\circ\pi_i(\BIx_l)=\int_{\CX_i}g_{i,j}\DIFFX{\mu_i} \qquad\forall 1\le j\le m_i,\;\forall 1\le i\le N.
\end{align*}
This shows that $\hat{\mu}\in\Gamma\big([\mu_1]_{\CG_1},\ldots,[\mu_N]_{\CG_N}\big)$ and thus $\hat{\mu}$ is feasible for the right-hand side of (\ref{eqn:duality}). 
Moreover, since $\int_{\BCX}f\DIFFX{\hat{\mu}}=\sum_{l=1}^ka_lf(\BIx_l)$, it holds that the right-hand side of (\ref{eqn:duality}) is less than or equal to the optimal value of (\ref{eqn:lsip-duality}),
which is in turn equal to the left-hand side of (\ref{eqn:duality}). Consequently, (\ref{eqn:duality}) follows from the weak duality in statement~\ref{thms:duality-weak}. 
Furthermore, note that statement~\ref{thms:duality2} follows directly from \citep[Theorem~8.1(v)]{goberna1998linear} and the property $(1,\bar{\BIg}^\TRANSP)^\TRANSP\in\relint(U)$.

In the case where the condition~\ref{thmc:duality1} does not hold, 
i.e., when $\bar{\BIg}\notin\relint(K)$,
let us establish (\ref{eqn:duality}) under the condition~\ref{thmc:duality2}, i.e., when $C$ is closed.
While we have $(1,\bar{\BIg}^\TRANSP)^\TRANSP\in\{1\}\times \clos(K)\subset \clos(U)$, we have by (\ref{eqn:lsip-duality-relint-cone}) that $(1,\bar{\BIg}^\TRANSP)^\TRANSP\notin\relint(U)$.
It hence holds that $(1,\bar{\BIg}^\TRANSP)^\TRANSP\in\relbd(U)$.
By the assumption that the left-hand side of (\ref{eqn:weak-duality}) is not equal to~$-\infty$, we have by \citep[Theorem~4.5]{goberna1998linear}, again with $M\leftarrow U$, $N\leftarrow C$, $K\leftarrow\cone\big(C\cup\big\{(\veczero_{m+1}^\TRANSP,1)^\TRANSP\big\}\big)$, $c\leftarrow(1,\bar{\BIg}^\TRANSP)^\TRANSP$ in the notation of \citep{goberna1998linear}, that $\cone\big(C\cup\big\{(\veczero_{m+1}^\TRANSP,1)^\TRANSP\big\}\big)$ is also closed. Thus, (\ref{eqn:duality}) follows from \citep[Theorem~8.2]{goberna1998linear} (see the sixth case in \citep[Table~8.1]{goberna1998linear}) and a similar argument as above. 
We have thus proved statement~\ref{thms:duality1}. 

Finally, if the condition~\ref{thmc:duality1int} holds, then $K\subseteq\R^m$ has a non-empty interior, and thus contains $m+1$ affinely independent points, say $\BIg_1,\ldots,\BIg_{m+1}$. 
Consequently, one checks that $(0,\veczero_m^\TRANSP)^\TRANSP,\allowbreak(1,\BIg_1^\TRANSP)^\TRANSP,\ldots,(1,\BIg_{m+1}^\TRANSP)^\TRANSP\in U$ are $m+2$ affinely independent points in $\R^{m+1}$, which implies that $\aff(U)=\R^{m+1}$ and $(1,\bar{\BIg}^\TRANSP)^\TRANSP\in\relint(U)=\inter(U)$. 
Statement~\ref{thms:duality3} then follows from \citep[Theorem~8.1(vi)]{goberna1998linear}, with $M\leftarrow U$, $c\leftarrow(1,\bar{\BIg}^\TRANSP)^\TRANSP$ in the notation of \citep{goberna1998linear}. 
The proof is now complete. 
\end{proof}

% Proof of Proposition (sufficient conditions for the duality)
\begin{proof}[Proof of Proposition~\ref{prop:duality-settings}]
Let us first prove statement~\ref{props:duality-setting1}. 
Suppose for the sake of contradiction that $\bar{\BIg}\notin\relint(K)$. 
By \citep[Theorem~20.2]{rockafellar1970convex}, there exists a hyperplane
\begin{align*}
H:=\big\{\BIw\in\R^m:\langle\BIw,\BIy\rangle=\alpha\big\},
\end{align*}
where $\BIy=\big(y_{1,1},\ldots,y_{1,m_1},\ldots,y_{N,1},\ldots,y_{N,m_N}\big)\in\R^m$, $\alpha\in\R$, and $\BIy\ne\veczero$,
that properly separates the convex set $K$ and the polyhedral convex set $\{\bar{\BIg}\}$ and that $K\nsubseteq H$. 
Suppose without loss of generality that $\bar{\BIg}$ is contained in the closed half-space $\big\{\BIw\in\R^m:\langle\BIw,\BIy\rangle\le\alpha\big\}$. Then, we have $\langle\BIg(\BIx),\BIy\rangle\ge \alpha\ge\langle\bar{\BIg},\BIy\rangle$ for all $\BIx\in\BCX$. 
This and the definitions of $\BIg(\,\cdot\,)$ and $\bar{\BIg}$ in (\ref{eqn:g-vecdef}) and (\ref{eqn:v-vecdef}) imply that
\begin{align}
\sum_{i=1}^N\inf_{x_i\in\CX_i}\Bigg\{\sum_{j=1}^{m_i}y_{i,j}\bigg(g_{i,j}(x_i)-\int_{\CX_i}g_{i,j}\DIFFX{\mu_i}\bigg)\Bigg\}\ge0.
\label{eqn:duality-settings-proof1}
\end{align}
We will prove the following claim.

\noindent\textbf{Claim}~(C): whenever $\inf_{x_i\in\CX_i}\Big\{\sum_{j=1}^{m_i}y_{i,j}\big(g_{i,j}(x_i)-\int_{\CX_i}g_{i,j}\DIFFX{\mu_i}\big)\Big\}\ge0$ for some $i\in\{1,\ldots,N\}$,
it holds that 
$\sum_{j=1}^{m_i}y_{i,j}\big(g_{i,j}(x_i)-\int_{\CX_i}g_{i,j}\DIFFX{\mu_i}\big)=0$ for all $x_i\in\CX_i$.

If Claim~(C) holds, then we can conclude by (\ref{eqn:duality-settings-proof1}) that $\sum_{i=1}^N\sum_{j=1}^{m_i}y_{i,j}\big(g_{i,j}(x_i)-\int_{\CX_i}g_{i,j}\DIFFX{\mu_i}\big)=0$ for all $x_1\in\CX_1,\ldots,x_N\in\CX_N$. 

Let us prove Claim~(C) now. 
Suppose for the sake of contradiction that there exist $i\in\{1,\ldots,N\}$ and $\hat{x}_i\in\CX_i$ such that 
$\inf_{x_i\in\CX_i}\Big\{\sum_{j=1}^{m_i}y_{i,j}\big(g_{i,j}(x_i)-{\textstyle\int_{\CX_i}}g_{i,j}\DIFFX{\mu_i}\big)\Big\}\ge0$ 
and $\beta:=\sum_{j=1}^{m_i}y_{i,j}\big(g_{i,j}(\hat{x}_i)-{\textstyle\int_{\CX_i}}g_{i,j}\DIFFX{\mu_i}\big)>0$. 
Then, by the continuity of $g_{i,1},\ldots,g_{i,m_i}$, there exists an open set $E\subseteq\CX_i$ such that $\hat{x}_i\in E$ and
\begin{align*}
\sum_{j=1}^{m_i}y_{i,j}\bigg(g_{i,j}(x_i)-\int_{\CX_i}g_{i,j}\DIFFX{\mu_i}\bigg)>\frac{\beta}{2}\qquad\forall x_i\in E.
\end{align*}
By the assumption that $\support(\mu_i)=\CX_i$, we have $\mu_i(E)>0$. Thus, it follows that
\begin{align*}
0&=\int_{\CX_i}{\textstyle\sum_{j=1}^{m_i}}y_{i,j}\big(g_{i,j}(x_i)-{\textstyle\int_{\CX_i}}g_{i,j}\DIFFX{\mu_i}\big)\DIFFM{\mu_i}{\DIFF x_i}\\
&\ge \int_{E}{\textstyle\sum_{j=1}^{m_i}}y_{i,j}\big(g_{i,j}(x_i)-{\textstyle\int_{\CX_i}}g_{i,j}\DIFFX{\mu_i}\big)\DIFFM{\mu_i}{\DIFF x_i} \ge \frac{\beta\mu_i(E)}{2}>0,
\end{align*}
which is a contradiction. Hence, Claim~(C) holds. 

Therefore, we have shown that indeed $\sum_{i=1}^N\sum_{j=1}^{m_i}y_{i,j}\big(g_{i,j}(x_i)-{\textstyle\int_{\CX_i}}g_{i,j}\DIFFX{\mu_i}\big)=0$ holds for all $x_1\in\CX_1,\ldots,x_N\in\CX_N$. 
This implies that $\langle\BIg(\BIx),\BIy\rangle=\langle\bar{\BIg},\BIy\rangle=\alpha$ for all $\BIx\in\BCX$, 
and thus $\BIg(\BIx)\in H$ for all $\BIx\in \BCX$. 
This yields $K=\conv\big(\big\{\BIg(\BIx):\BIx\in\BCX\big\}\big)\subseteq H$, which contradicts $K\nsubseteq H$. 
The proof of statement~\ref{props:duality-setting1} is now complete. 

To prove statement~\ref{props:duality-setting2}, let $K_i:=\conv\big(\big\{\BIg_i(x_i):x_i\in\CX_i\big\}\big)$ for $i=1,\ldots,N$. It follows from (\ref{eqn:g-vecdef}) and (\ref{eqn:Kset-def}) that $K=\bigtimes_{i=1}^NK_i$. Since $K_i$ contains $\conv\big(\big\{\BIg_i(x_{i,1}),\ldots,\BIg_i(x_{i,m_i+1})\big\}\big)$, which is an $m_i$-simplex, we have $\aff(K_i)=\R^{m_i}$. Therefore, $\aff(K)=\bigtimes_{i=1}^N\R^{m_i}=\R^m$ and $\relint(K)=\inter(K)$. Consequently, statement~\ref{props:duality-setting2} follows from statement~\ref{props:duality-setting1}. 

Finally, to prove statement~\ref{props:duality-setting3}, notice that if $\CX_i$ is compact, $g_{i,1},\ldots,g_{i,m_i}$ are all continuous for $i=1,\ldots,N$, and $f$ is continuous, then the set 
$V:=\Big\{\big(1,\BIg(\BIx)^\TRANSP,f(\BIx)\big)^\TRANSP:\BIx\in\BCX\Big\}$ is compact.
By \citep[Theorem~17.2]{rockafellar1970convex}, it holds that $\conv(V)$ is also compact. 
Since $C:=\cone(V)=\cone(\conv(V))$ and $\veczero\notin\conv(V)$, the condition~\ref{thmc:duality2} in Theorem~\ref{thm:duality} follows from \citep[Corollary~9.6.1]{rockafellar1970convex}. 
The proof is now complete.
\end{proof}

\subsection{Proofs of results in Section~\ref{ssec:mmot-complexity}}\label{ssec:proof-complexity}
%% Proof of Theorem (LSIP complexity)
\begin{proof}[Proof of Theorem~\ref{thm:mmot-complexity}]
Throughout this proof, let us fix an arbitrary tolerance value $\epsilon_{\mathsf{LSIP}}>0$.
Let $\alpha^\star\in\R$ denote the optimal value of \eqref{eqn:mmotlb-dual-lsip}, let $S\subset\R^{m+1}$ denote the feasible set of \eqref{eqn:mmotlb-dual-lsip}, i.e., $S:=\big\{(y_0,\BIy^\TRANSP)^\TRANSP\in\R^{m+1}:y_{0}+\langle\BIg(\BIx),\BIy\rangle\le f(\BIx)\;\forall \BIx\in\BCX\big\}$, and let $S_{\alpha}\subseteq S$ denote the \mbox{$\alpha$-superlevel} set of \eqref{eqn:mmotlb-dual-lsip} for any $\alpha\in\R$, i.e., $S_{\alpha}:=\big\{(y_0,\BIy^\TRANSP)^\TRANSP\in S:y_0+\langle\bar{\BIg},\BIy\rangle\ge\alpha\big\}$. 
Moreover, for any $r>0$, let $B(r)\subset\R^{m+1}$ denote the closed Euclidean ball with radius~$r$ centered at the origin,
i.e., $B(r):=\big\{\BIw\in\R^{m+1}:\|\BIw\|_2\le r\big\}$.
In this proof, we apply the \textit{volumetric center} algorithm of \citet{vaidya1996new}, where we consider the maximization of the linear objective function
$c\big((y_0,\BIy^\TRANSP)^\TRANSP\big):=y_0+\langle\bar{\BIg},\BIy\rangle$ 
$\forall (y_0,\BIy^\TRANSP)^\TRANSP\in\R^{m+1}$
over the feasible set $S\,\cap\, B(M_{\mathsf{opt}}+\epsilon_{\mathsf{LSIP}})$.
By assumption, restricting the feasible set of \eqref{eqn:mmotlb-dual-lsip} to $S\,\cap\, B(M_{\mathsf{opt}}+\epsilon_{\mathsf{LSIP}})$ does not affect its optimal value. 
In order to apply the theory of \citet{vaidya1996new}, we need to establish the two following claims.
\begin{itemize}[leftmargin=15pt]
\item \textbf{Claim}~(I): 
the set $S_{\alpha^\star-\epsilon_{\mathsf{LSIP}}}\cap B(M_{\mathsf{opt}}+\epsilon_{\mathsf{LSIP}})$ contains a closed Euclidean ball with radius $\frac{\epsilon_{\mathsf{LSIP}}}{2\sqrt{N+1}}$.
\item \textbf{Claim}~(II):
one can implement a so-called \textit{separation oracle} using the global minimization oracle $\mathtt{Oracle}(\,\cdot\,)$, which takes any $\hat{y}_0\in\R$, $\hat{\BIy}\in\R^m$ as inputs, 
and returns two outputs: 
a one-bit output indicating whether $(\hat{y}_0,\hat{\BIy}^\TRANSP)^\TRANSP\in S\,\cap\, B(M_{\mathsf{opt}}+\epsilon_{\mathsf{LSIP}})$ holds,
and a vector-valued output $(\hat{g}_0,\hat{\BIg}^\TRANSP)^\TRANSP\in\R^{m+1}$.
In the case where $(\hat{y}_0,\hat{\BIy}^\TRANSP)^\TRANSP\notin S\,\cap\, B(M_{\mathsf{opt}}+\epsilon_{\mathsf{LSIP}})$,
the output $(\hat{g}_0,\hat{\BIg}^\TRANSP)^\TRANSP$ satisfies
\begin{align*}
    \hat{g}_0y_0+\langle\hat{\BIg},\BIy\rangle\le \hat{g}_0\hat{y}_0+\langle\hat{\BIg},\hat{\BIy}\rangle \qquad \forall (y_0,\BIy^\TRANSP)^\TRANSP\in S\cap B(M_{\mathsf{opt}}+\epsilon_{\mathsf{LSIP}}).
\end{align*}
In the case where $(\hat{y}_0,\hat{\BIy}^\TRANSP)^\TRANSP\in S\,\cap\, B(M_{\mathsf{opt}}+\epsilon_{\mathsf{LSIP}})$,
the output $(\hat{g}_0,\hat{\BIg}^\TRANSP)^\TRANSP$ satisfies
\begin{align*}
    \qquad\quad\hat{g}_0y_0+\langle\hat{\BIg},\BIy\rangle\le \hat{g}_0\hat{y}_0+\langle\hat{\BIg},\hat{\BIy}\rangle \qquad \forall (y_0,\BIy^\TRANSP)^\TRANSP\in\R^{m+1} \text{ where }c\big((y_0,\BIy^\TRANSP)^\TRANSP\big)\ge c\big((\hat{y}_0,\hat{\BIy}^\TRANSP)^\TRANSP\big).
\end{align*}
Moreover, each call to this separation oracle involves at most a single call to $\mathtt{Oracle}(\,\cdot\,)$, as well as $O(m)$ additional arithmetic operations.
\end{itemize}
To prove Claim~(I), let $(y_0^\star,\BIy^\star)$ be the optimal solution of \eqref{eqn:mmotlb-dual-lsip} in the statement of the theorem.
Note that $(y^\star_0,\BIy^{\star\TRANSP})^\TRANSP\in S$, 
$y^\star_0+\langle\bar{\BIg},\BIy^\star\rangle=\alpha^\star$, and
$\big\|(y^\star_0,\BIy^{\star\TRANSP})^\TRANSP\big\|_2=M_{\mathsf{opt}}$.
Next, let $\hat{y}_0:=y^\star_0-\frac{\epsilon_{\mathsf{LSIP}}}{2}$, $\hat{\BIy}:=\BIy^\star$,
let $(u_0,\BIu^\TRANSP)^\TRANSP\in B(1)$ be arbitrary, 
and let 
$(y^{\circ}_0,\BIy^{\circ\TRANSP})^\TRANSP:=(\hat{y}_0,\hat{\BIy}^{\TRANSP})^\TRANSP+\frac{\epsilon_{\mathsf{LSIP}}}{2\sqrt{N+1}}(u_0,\BIu^\TRANSP)^\TRANSP$.
Hence, $(y^{\circ}_0,\BIy^{\circ\TRANSP})^\TRANSP\in\R^{m+1}$ is an arbitrary point in the closed Euclidean ball with radius $\frac{\epsilon_{\mathsf{LSIP}}}{2\sqrt{N+1}}$ centered at 
$(\hat{y}_0,\hat{\BIy}^{\TRANSP})^\TRANSP$.
In the following, we will prove Claim~(I) by showing that 
$(y^\circ_0,\BIy^{\circ\TRANSP})^\TRANSP\in S$, 
$y^\circ_0+\langle\bar{\BIg},\BIy^\circ\rangle\ge\alpha^\star-\epsilon_{\mathsf{LSIP}}$,
and
$\big\|(y^\circ_0,\BIy^{\circ\TRANSP})^\TRANSP\big\|_2\le M_{\mathsf{opt}}+\epsilon_{\mathsf{LSIP}}$.

First, we show that $(y^\circ_0,\BIy^{\circ\TRANSP})^\TRANSP\in S$.
It follows from the property that $y^\star_0+\langle\BIg(\BIx),\BIy^\star\rangle\le f(\BIx)$ $\forall \BIx\in\BCX$ and the Cauchy--Schwarz inequality that
\begin{align*}
y^\circ_0 + \langle\BIg(\BIx),\BIy^\circ\rangle&=y^\star_0-\frac{\epsilon_{\mathsf{LSIP}}}{2}+\langle\BIg(\BIx),\BIy^\star\rangle+\frac{\epsilon_{\mathsf{LSIP}}}{2\sqrt{N+1}}u_0+\frac{\epsilon_{\mathsf{LSIP}}}{2\sqrt{N+1}}\langle\BIg(\BIx),\BIu\rangle\\
&\le f(\BIx)-\frac{\epsilon_{\mathsf{LSIP}}}{2}+\frac{\epsilon_{\mathsf{LSIP}}}{2\sqrt{N+1}}\big(u_0+\langle\BIg(\BIx),\BIu\rangle\big)\\
&\le f(\BIx)-\frac{\epsilon_{\mathsf{LSIP}}}{2}+\frac{\epsilon_{\mathsf{LSIP}}}{2\sqrt{N+1}}\big\|(1,\BIg(\BIx)^\TRANSP)^\TRANSP\big\|_2 \big\|(u_0,\BIu^\TRANSP)^\TRANSP\big\|_2\\
&\le f(\BIx)-\frac{\epsilon_{\mathsf{LSIP}}}{2}+\frac{\epsilon_{\mathsf{LSIP}}}{2\sqrt{N+1}}\Bigg(1+\sum_{i=1}^N\big\|\BIg_i(x_i)\big\|_2^2\Bigg)^{\frac{1}{2}} \qquad\! \forall \BIx=(x_1,\ldots,x_N)\in\BCX.
\end{align*}
Using the assumption that $\big\|\BIg_i(x_i)\big\|_2\le 1$ $\forall x_i\in\CX_i$, for $i=1,\ldots,N$,
we get $y^\circ_0 + \langle\BIg(\BIx),\BIy^\circ\rangle\le f(\BIx)$ $\forall\BIx\in\BCX$ and thus $(y^\circ_0,\BIy^{\circ\TRANSP})^\TRANSP\in S$.
Second,
to show that $y^\circ_0+\langle\bar{\BIg},\BIy^\circ\rangle\ge\alpha^\star-\epsilon_{\mathsf{LSIP}}$,
we let $\bar{\BIg}_1,\ldots,\bar{\BIg}_N$ be defined by (\ref{eqn:v-vecdef}).
For $i=1,\ldots,N$, it holds by the assumption that $\big\|\BIg_i(x_i)\big\|_2\le 1$ $\forall x_i\in\CX_i$ and Jensen's inequality that $\|\bar{\BIg}_i\|_2\le 1$. 
Subsequently, using the property that $y^\star_0+\langle\bar{\BIg},\BIy^\star\rangle=\alpha^\star$ as well as the Cauchy--Schwarz inequality, we get
\begin{align*}
y^\circ_0+\langle\bar{\BIg},\BIy^\circ\rangle&= y^\star_0-\frac{\epsilon_{\mathsf{LSIP}}}{2}+\langle\bar{\BIg},\BIy^\star\rangle+\frac{\epsilon_{\mathsf{LSIP}}}{2\sqrt{N+1}}u_0+\frac{\epsilon_{\mathsf{LSIP}}}{2\sqrt{N+1}}\langle\bar{\BIg},\BIu\rangle\\
&=\alpha^\star-\frac{\epsilon_{\mathsf{LSIP}}}{2}+\frac{\epsilon_{\mathsf{LSIP}}}{2\sqrt{N+1}}\big(u_0+\langle\bar{\BIg},\BIu\rangle\big)\\
&\ge \alpha^\star-\frac{\epsilon_{\mathsf{LSIP}}}{2}-\frac{\epsilon_{\mathsf{LSIP}}}{2\sqrt{N+1}} \big\|(1,\bar{\BIg}^\TRANSP)^\TRANSP\big\|_2 \big\|(u_0,\BIu^\TRANSP)^\TRANSP\big\|_2\\
&\ge \alpha^\star-\frac{\epsilon_{\mathsf{LSIP}}}{2}-\frac{\epsilon_{\mathsf{LSIP}}}{2\sqrt{N+1}} \Bigg(1+\sum_{i=1}^N\|\bar{\BIg}_i\|_2^2\Bigg)^{\frac{1}{2}}\ge\alpha^\star-\epsilon_{\mathsf{LSIP}}.
\end{align*}
Third, the triangle inequality guarantees that
\begin{align*}
\big\|(y^\circ_0,\BIy^{\circ\TRANSP})^\TRANSP\big\|_2&\le\big\|(\hat{y}_0,\hat{\BIy}^{\TRANSP})^\TRANSP\big\|_2+\frac{\epsilon_{\mathsf{LSIP}}}{2\sqrt{N+1}}\big\|(u_0,\BIu^\TRANSP)^\TRANSP\big\|_2\\
&\le\big\|(y^\star_0,\BIy^{\star\TRANSP})^\TRANSP\big\|_2+\frac{\epsilon_{\mathsf{LSIP}}}{2}+\frac{\epsilon_{\mathsf{LSIP}}}{2\sqrt{N+1}}\\
&=M_{\mathsf{opt}}+\frac{\epsilon_{\mathsf{LSIP}}}{2}+\frac{\epsilon_{\mathsf{LSIP}}}{2\sqrt{N+1}}<M_{\mathsf{opt}}+\epsilon_{\mathsf{LSIP}}.
\end{align*}
Therefore, we conclude that $(y^{\circ}_0,\BIy^{\circ\TRANSP})^\TRANSP\in S_{\alpha^\star-\epsilon_{\mathsf{LSIP}}}\cap B(M_{\mathsf{opt}}+\epsilon_{\mathsf{LSIP}})$
and thus 
$S_{\alpha^\star-\epsilon_{\mathsf{LSIP}}}\cap B(M_{\mathsf{opt}}+\epsilon_{\mathsf{LSIP}})$ contains a closed Euclidean ball with radius $\frac{\epsilon_{\mathsf{LSIP}}}{2\sqrt{N+1}}$ centered at $(\hat{y}_0,\hat{\BIy}^\TRANSP)^\TRANSP$.
The proof of Claim~(I) is complete.

To prove Claim~(II), let us fix arbitrary inputs $\hat{y}_0\in\R$ and $\hat{\BIy}\in\R^m$ 
and implement the separation oracle via the following procedure.
\begin{itemize}
    \item Step~1: computing $\big\|(\hat{y}_0,\hat{\BIy}^\TRANSP)^\TRANSP\big\|_2$ and checking whether 
    $\big\|(\hat{y}_0,\hat{\BIy}^\TRANSP)^\TRANSP\big\|_2\le M_{\mathsf{opt}}+\epsilon_{\mathsf{LSIP}}$.
    If $\big\|(\hat{y}_0,\hat{\BIy}^\TRANSP)^\TRANSP\big\|_2> M_{\mathsf{opt}}+\epsilon_{\mathsf{LSIP}}$, skip Step~2 and return
    $(\hat{y}_0,\hat{\BIy}^\TRANSP)^\TRANSP\notin S\,\cap\, B(M_{\mathsf{opt}}+\epsilon_{\mathsf{LSIP}})$ as well as 
    $(\hat{g}_0,\hat{\BIg}^\TRANSP)^\TRANSP\leftarrow\big(\hat{y}_0,\hat{\BIy}^\TRANSP)^\TRANSP$ as outputs.
    Otherwise, proceed to Step~2.

    \item Step~2: calling $\mathtt{Oracle}(\hat{\BIy})$ and denoting the outputs by 
    $(\BIx^\star,\beta^\star)\in\BCX\times\R$.
    If $\beta^\star< \hat{y}_0$, 
    return 
    $(\hat{y}_0,\hat{\BIy}^\TRANSP)^\TRANSP\notin S\,\cap\, B(M_{\mathsf{opt}}+\epsilon_{\mathsf{LSIP}})$ as well as 
    $(\hat{g}_0,\hat{\BIg}^\TRANSP)^\TRANSP\leftarrow(1,\BIg(\BIx^\star)^\TRANSP)^\TRANSP$ as outputs.
    Otherwise, 
    return
    $(\hat{y}_0,\hat{\BIy}^\TRANSP)^\TRANSP\in S\,\cap\, B(M_{\mathsf{opt}}+\epsilon_{\mathsf{LSIP}})$ as well as 
    $(\hat{g}_0,\hat{\BIg}^\TRANSP)^\TRANSP\leftarrow(-1,-\bar{\BIg}^\TRANSP)^\TRANSP$ as outputs.
\end{itemize}
One may check that the above procedure incurs at most a single call to $\mathtt{Oracle}(\,\cdot\,)$ and $O(m)$ additional arithmetic operations.
First, in the case where $\big\|(\hat{y}_0,\hat{\BIy}^\TRANSP)^\TRANSP\big\|_2> M_{\mathsf{opt}}+\epsilon_{\mathsf{LSIP}}$ holds in Step~1,
it holds that $(\hat{y}_0,\hat{\BIy}^\TRANSP)^\TRANSP\notin S\,\cap\, B(M_{\mathsf{opt}}+\epsilon_{\mathsf{LSIP}})$, and
the Cauchy--Schwarz inequality implies that
the output $(\hat{g}_0,\hat{\BIg}^\TRANSP)^\TRANSP\leftarrow\big(\hat{y}_0,\hat{\BIy}^\TRANSP)^\TRANSP$ satisfies
\begin{align*}
    \hat{g}_0y_0+\langle\hat{\BIg},\BIy\rangle
    &\le \big\|(\hat{y}_0,\hat{\BIy}^\TRANSP)^\TRANSP\big\|_2\big\|(y_0,\BIy^\TRANSP)^\TRANSP\big\|_2\\
    &\le (M_{\mathsf{opt}}+\epsilon_{\mathsf{LSIP}}) \big\|(\hat{y}_0,\hat{\BIy}^\TRANSP)^\TRANSP\big\|_2\\
    &<\big\|(\hat{y}_0,\hat{\BIy}^\TRANSP)^\TRANSP\big\|_2^2\\
    &= \hat{g}_0\hat{y}_0+\langle\hat{\BIg},\hat{\BIy}\rangle \qquad \forall (y_0,\BIy^\TRANSP)^\TRANSP \in S\cap B(M_{\mathsf{opt}}+\epsilon_{\mathsf{LSIP}}).
\end{align*}
Second, in the case where 
$\big\|(\hat{y}_0,\hat{\BIy}^\TRANSP)^\TRANSP\big\|_2\le M_{\mathsf{opt}}+\epsilon_{\mathsf{LSIP}}$ holds in Step~1 and
$\beta^\star<\hat{y}_0$ holds in Step~2,
it holds by Definition~\ref{def:mmot-oracle} that
$\hat{y}_0>\beta^\star=f(\BIx^\star)-\langle\BIg(\BIx^\star),\hat{\BIy}\rangle$
and thus $(\hat{y}_0,\hat{\BIy}^\TRANSP)^\TRANSP\notin S\,\cap\, B(M_{\mathsf{opt}}+\epsilon_{\mathsf{LSIP}})$.
In this case, the output $(\hat{g}_0,\hat{\BIg}^\TRANSP)^\TRANSP\leftarrow(1,\BIg(\BIx^\star)^\TRANSP)^\TRANSP$ satisfies
\begin{align*}
    \hat{g}_0y_0+\langle\hat{\BIg},\BIy\rangle
    &= y_0 + \langle\BIg(\BIx^\star),\BIy\rangle\\
    &\le f(\BIx^\star)\\
    &<\hat{y}_0+\langle\BIg(\BIx^\star),\hat{\BIy}\rangle\\
    &=\hat{g}_0\hat{y}_0+\langle\hat{\BIg},\hat{\BIy}\rangle \qquad \forall (y_0,\BIy^\TRANSP)^\TRANSP \in S\cap B(M_{\mathsf{opt}}+\epsilon_{\mathsf{LSIP}}).
\end{align*}
Third, in the case where 
$\big\|(\hat{y}_0,\hat{\BIy}^\TRANSP)^\TRANSP\big\|_2\le M_{\mathsf{opt}}+\epsilon_{\mathsf{LSIP}}$ holds in Step~1 and
$\beta^\star\ge\hat{y}_0$ holds in Step~2,
we have by Definition~\ref{def:mmot-oracle} that
$\hat{y}_0\le\beta^\star=\inf_{\BIx\in\BCX}\big\{f(\BIx)-\langle\BIg(\BIx),\hat{\BIy}\rangle\big\}$,
which shows that 
$\hat{y}_0+\langle\BIg(\BIx),\hat{\BIy}\rangle\le f(\BIx)$ $\forall\BIx\in\BCX$ 
and thus 
$(\hat{y}_0,\hat{\BIy}^\TRANSP)^\TRANSP\in S\,\cap\, B(M_{\mathsf{opt}}+\epsilon_{\mathsf{LSIP}})$.
Moreover, in this case, 
the output $(\hat{g}_0,\hat{\BIg}^\TRANSP)^\TRANSP\leftarrow(-1,-\bar{\BIg}^\TRANSP)^\TRANSP$ satisfies
\begin{align*}
    \hat{g}_0y_0+\langle\hat{\BIg},\BIy\rangle
    &= -y_0 - \langle\bar{\BIg},\BIy\rangle=-c\big((y_0,\BIy^\TRANSP)^\TRANSP\big)\\
    &\le -c\big((\hat{y}_0,\hat{\BIy}^\TRANSP)^\TRANSP\big)\\
    &=\hat{g}_0\hat{y}_0+\langle\hat{\BIg},\hat{\BIy}\rangle \qquad \forall (y_0,\BIy^\TRANSP)^\TRANSP \text{ where }c\big((y_0,\BIy^\TRANSP)^\TRANSP\big)\ge c\big((\hat{y}_0,\hat{\BIy}^\TRANSP)^\TRANSP\big).
\end{align*}
This completes the proof of Claim~(II).

Now that both Claim~(I) and Claim~(II) have been established, one can apply the volumetric center algorithm of \citet{vaidya1996new} to compute an $\epsilon_{\mathsf{LSIP}}$-optimal solution of the concave maximization problem with objective function
$c(\,\cdot\,)$ and feasible set $S\,\cap\,B(M_{\mathsf{opt}}+\epsilon_{\mathsf{LSIP}})$,
which is also an $\epsilon_{\mathsf{LSIP}}$-optimal solution
of \eqref{eqn:mmotlb-dual-lsip}.
Let us briefly describe Vaidya's algorithm as follows.
The algorithm begins with a polytope $P_0\subset\R^{m+1}$ such that $B(M_{\mathsf{opt}}+\epsilon_{\mathsf{LSIP}})\subset P_0$, 
and sets $\hat{\BIv}_0\in\R^{m+1}$ to be the volumetric center of $P_0$;
see \citep[p.~294--295]{vaidya1996new} for the definition and the properties of the volumetric center.
Subsequently, each iteration $r=1,2,\ldots$ of the algorithm performs the following operations.
\begin{itemize}
    \item Choose one out of the two following operations:
    \begin{itemize}
        \item either add a closed half-space to those characterizing $P_{r-1}$ to form an updated polytope $P_{r}\subset P_{r-1}$ where the added closed half-space is constructed from the output of the separation oracle in Claim~(II) using $\hat{\BIv}_{r-1}$ as the input;
        
        \item or remove a closed half-space from those characterizing $P_{r-1}$ to form an updated polytope $P_{r}\supset P_{r-1}$.
    \end{itemize}
    As a consequence of Claim~(II), this step incurs at most a single call to $\mathtt{Oracle}(\,\cdot\,)$ and $O(m)$ additional arithmetic operations.

    \item Perform a constant number of Newton-type updates to $\hat{\BIv}_{r-1}$ to compute $\hat{\BIv}_{r}\in\R^{m+1}$ as an approximation of the volumetric center of $P_{r}$.
    This step incurs $O(m^\omega)$ arithmetic operations.
\end{itemize}
The algorithm terminates whenever the volume of 
the polytope $P_{r}$ is less that the volume of $B\big(\frac{\epsilon_{\mathsf{LSIP}}}{2\sqrt{N+1}}\big)$.
Claim~(I) guarantees that when the algorithm terminates it outputs an $\epsilon_{\mathsf{LSIP}}$-optimal solution of the concave maximization problem with objective function
$c(\,\cdot\,)$ and feasible set $S\,\cap\,B(M_{\mathsf{opt}}+\epsilon_{\mathsf{LSIP}})$.
Moreover, the result of \citet{vaidya1996new} with
$n\leftarrow m+1$, 
$L\leftarrow \log_2\Big(\frac{M_{\mathsf{opt}}+\epsilon_{\mathsf{LSIP}}}{\epsilon_{\mathsf{LSIP}}/(2\sqrt{N+1})}\Big)$
in the notation of \citep{vaidya1996new}
states that this algorithm is guaranteed to terminate within
$O\big((m+1)\log_2\big(2(m+\nobreak1)\sqrt{N+1}(M_{\mathsf{opt}}+\epsilon_{\mathsf{LSIP}})/\epsilon_{\mathsf{LSIP}}\big)\big)=O\big(m\log(mM_{\mathsf{opt}}/\epsilon_{\mathsf{LSIP}})\big)$ 
iterations.
Consequently, 
combining this with the per-iteration cost analyses above,
we conclude that an $\epsilon_{\mathsf{LSIP}}$-optimal solution of \eqref{eqn:mmotlb-dual-lsip} can be computed with
$O\big(m\log(mM_{\mathsf{opt}}/\epsilon_{\mathsf{LSIP}})\big)$
calls to $\mathtt{Oracle}(\,\cdot\,)$
and 
$O\big(m^{\omega+1}\log(mM_{\mathsf{opt}}/\epsilon_{\mathsf{LSIP}})\big)$
additional arithmetic operations.
The proof is now complete. 
\end{proof}

\subsection{Proof of results in Section~\ref{ssec:reassemblysemidiscrete}}\label{ssec:proof-reassembly}

% Proof of Lemma (reassembly subset -> Euclidean)
\begin{proof}[Proofs of Lemma~\ref{lem:wlogeuclidean}]
In this proof, we let $\bar{\CX}_i:=\CX_i$, $\bar{\CX}^\dagger_i:=\CX^\dagger_i$ for $i=1,\ldots,N$, and let $\overbar{\BCX}:=\bigtimes_{i=1}^N\bar{\CX}_i$, $\overbar{\BCX}^\dagger:=\bigtimes_{i=1}^N\bar{\CX}^\dagger_i$ in order to differentiate copies of the same space. 
Let us first assume that $\tilde{\mu}\in R(\hat{\mu};\mu_1,\ldots,\mu_N)$ for some $\hat{\mu},\tilde{\mu}\in\CP_1(\BCX)$
and show that $\tilde{\mu}^{\dagger}\in R(\hat{\mu}^\dagger;\mu_1^\dagger,\ldots,\mu_N^\dagger)$. 
For $i=1,\ldots,N$, let $\hat{\mu}_i$ denote the $i$-th marginal of $\hat{\mu}$ and let $\hat{\mu}_i^\dagger$ denote the $i$-th marginal of $\hat{\mu}^\dagger$. 
By Definition~\ref{def:reassembly}, $\tilde{\mu}\in R(\hat{\mu};\mu_1,\ldots,\mu_N)$ implies that there exists $\gamma\in\CP(\BCX\times\overbar{\BCX})$ such that the marginal of $\gamma$ on $\BCX$ is $\hat{\mu}$, the marginal of $\gamma$ on $\overbar{\BCX}$ is $\tilde{\mu}$, 
and the marginal $\gamma_i\in\Gamma(\hat{\mu}_i,\mu_i)$ of $\gamma$ on $\CX_i\times\bar{\CX}_i$ satisfies $\int_{\CX_i\times\bar{\CX}_i}d_{\CX_i}(x,z)\DIFFM{\gamma_i}{\DIFF x,\DIFF z}=W_1(\hat{\mu}_i,\mu_i)$ for $i=1,\ldots,N$. 
Let us define $\gamma^\dagger\in\CP(\BCX^\dagger\times\overbar{\BCX}^\dagger)$ by $\gamma^\dagger(E):=\gamma\big(E\cap(\BCX\times\overbar{\BCX})\big)$ $\forall E\in\CB(\BCX^\dagger\times\overbar{\BCX}^\dagger)$. 
Then, the marginal of $\gamma^\dagger$ on $\BCX^\dagger$ is exactly~$\hat{\mu}^\dagger$ and the marginal of $\gamma^\dagger$ on $\overbar{\BCX}^\dagger$ is exactly~$\tilde{\mu}^\dagger$. 
For $i=1,\ldots,N$, let us denote the marginal of $\gamma^\dagger$ on $\CX^\dagger_i\times\bar{\CX}^\dagger_i$ by $\gamma_i^\dagger$. 
Then, it holds by construction that, for $i=1,\ldots,N$, $\gamma_i^\dagger(E)=\gamma_i(E)$ $\forall E\in\CB(\CX_i\times\bar{\CX}_i)$ and in particular $\gamma_i^\dagger(\CX_i\times\bar{\CX}_i)=1$. 
We thus get
\begin{align}
\int_{\CX^\dagger_i\times\bar{\CX}^\dagger_i}d_{\CX^\dagger_i}(x,z)\DIFFM{\gamma_i^\dagger}{\DIFF x,\DIFF z}&=\int_{\CX_i\times\bar{\CX}_i}d_{\CX_i}(x,z)\DIFFM{\gamma_i}{\DIFF x,\DIFF z}=W_1(\hat{\mu}_i,\mu_i).
\label{eqn:wlogeuclidean-proof-1-1}
\end{align}
Next, let us fix an arbitrary $\theta_i^\dagger\in\Gamma(\hat{\mu}_i^\dagger,\mu_i^\dagger)$.
We have $\theta_i^\dagger\big((\CX^\dagger_i\times\bar{\CX}^\dagger_i)\setminus(\CX_i\times\bar{\CX}_i)\big)\le\hat{\mu}_i^\dagger(\CX^\dagger_i\setminus\CX_i)+\mu_i^\dagger(\bar{\CX}^\dagger_i\setminus\nobreak\bar{\CX}_i)=0$ and thus $\theta_i^\dagger(\CX_i\times\bar{\CX}_i)=1$. Let us define $\theta_i\in\CP(\CX_i\times\bar{\CX}_i)$ by $\theta_i(E):=\theta_i^\dagger(E)$ $\forall E\in\CB(\CX_i\times\bar{\CX}_i)$. 
Then, since $\theta_i\in\Gamma(\hat{\mu}_i,\mu_i)$, we have
\begin{align}
\int_{\CX^\dagger_i\times\bar{\CX}^\dagger_i} d_{\CX^\dagger_i}(x,z) \DIFFM{\theta_i^\dagger}{\DIFF x,\DIFF z}&=\int_{\CX_i\times\bar{\CX}_i}d_{\CX_i}(x,z)\DIFFM{\theta_i}{\DIFF x,\DIFF z}\ge W_1(\hat{\mu}_i,\mu_i).
\label{eqn:wlogeuclidean-proof-1-2}
\end{align}
Combining (\ref{eqn:wlogeuclidean-proof-1-1}) and (\ref{eqn:wlogeuclidean-proof-1-2}) 
shows that $\gamma_i^\dagger$ is an optimal coupling between $\hat{\mu}_i^\dagger$ and $\mu_i^\dagger$ with respect to the cost function $d_{\CX^\dagger_i}$. Consequently, it holds by Definition~\ref{def:reassembly} that $\tilde{\mu}^\dagger\in R(\hat{\mu}^\dagger;\mu_1^\dagger,\ldots,\mu_N^\dagger)$. 

Conversely, let us assume that $\tilde{\mu}^\dagger\in R(\hat{\mu}^\dagger;\mu_1^\dagger,\ldots,\mu_N^\dagger)$ for some $\hat{\mu},\tilde{\mu}\in\CP_1(\BCX)$
and show that $\tilde{\mu}\in R(\hat{\mu};\mu_1,\ldots,\mu_N)$. 
Again, for $i=1,\ldots,N$, let $\hat{\mu}_i$ denote the $i$-th marginal of $\hat{\mu}$ and let $\hat{\mu}_i^\dagger$ denote the $i$-th marginal of $\hat{\mu}^\dagger$. 
Definition~\ref{def:reassembly} implies that there exists $\gamma^\dagger\in\CP(\BCX^\dagger\times\overbar{\BCX}^\dagger)$ such that the marginal of $\gamma^\dagger$ on $\BCX^\dagger$ is $\hat{\mu}^\dagger$, the marginal of $\gamma^\dagger$ on $\overbar{\BCX}^\dagger$ is $\tilde{\mu}^\dagger$, 
and the marginal $\gamma_i^\dagger\in\Gamma(\hat{\mu}_i^\dagger,\mu_i^\dagger)$ of $\gamma^\dagger$ on $\CX^\dagger_i\times\bar{\CX}^\dagger_i$ satisfies $\int_{\CX^\dagger_i\times\bar{\CX}^\dagger_i}d_{\CX^\dagger_i}(x,z)\DIFFM{\gamma_i^\dagger}{\DIFF x,\DIFF z}=W_1(\hat{\mu}_i^\dagger,\mu_i^\dagger)$ for $i=1,\ldots,N$. 
Since $\gamma^\dagger(\BCX\times\overbar{\BCX})=1$ by assumption, let us define $\gamma\in\CP(\BCX\times\overbar{\BCX})$ by $\gamma(E):=\gamma^\dagger(E)$ $\forall E\in\CB(\BCX\times\overbar{\BCX})$. 
Thus, the marginal of $\gamma$ on $\BCX$ is exactly $\hat{\mu}$ and the marginal of $\gamma$ on $\overbar{\BCX}$ is exactly $\tilde{\mu}$. For $i=1,\ldots,N$, let us denote the marginal of $\gamma$ on $\CX_i\times\bar{\CX}_i$ by $\gamma_i$. For $i=1,\ldots,N$, it holds by construction that $\gamma_i(E)=\gamma_i^\dagger(E)$ $\forall E\in\CB(\CX_i\times\bar{\CX}_i)$. We hence get
\begin{align}
\int_{\CX_i\times\bar{\CX}_i}d_{\CX_i}(x,z)\DIFFM{\gamma_i}{\DIFF x,\DIFF z}&=\int_{\CX^\dagger_i\times\bar{\CX}^\dagger_i}d_{\CX^\dagger_i}(x,z)\DIFFM{\gamma_i^\dagger}{\DIFF x,\DIFF z}=W_1(\hat{\mu}_i^\dagger,\mu_i^\dagger).
\label{eqn:wlogeuclidean-proof-2-1}
\end{align}
Let us now fix an arbitrary $\theta_i\in\Gamma(\hat{\mu}_i,\mu_i)$, and define $\theta_i^\dagger\in\CP(\CX^\dagger_i\times\bar{\CX}^\dagger_i)$ by $\theta_i^\dagger(E):=\theta_i\big(E\cap(\CX_i\times\bar{\CX}_i)\big)$ $\forall E\in\CB(\CX^\dagger_i\times\bar{\CX}^\dagger_i)$. 
Subsequently, it holds that $\theta_i^\dagger\in\Gamma(\hat{\mu}_i^\dagger,\mu_i^\dagger)$ and hence
\begin{align}
\int_{\CX_i\times\bar{\CX}_i}d_{\CX_i}(x,z)\DIFFM{\theta_i}{\DIFF x,\DIFF z}&=\int_{\CX^\dagger_i\times\bar{\CX}^\dagger_i}d_{\CX^\dagger_i}(x,z)\DIFFM{\theta_i^\dagger}{\DIFF x,\DIFF z}
\ge W_1(\hat{\mu}_i^\dagger,\mu_i^\dagger).
\label{eqn:wlogeuclidean-proof-2-2}
\end{align}
Combining (\ref{eqn:wlogeuclidean-proof-2-1}) and (\ref{eqn:wlogeuclidean-proof-2-2}) shows that $\gamma_i$ is an optimal coupling between $\hat{\mu}_i$ and $\mu_i$ with respect to the cost function $d_{\CX_i}$. Consequently, it holds by Definition~\ref{def:reassembly} that $\tilde{\mu}\in R(\hat{\mu};\mu_1,\ldots,\mu_N)$. The proof is now complete. 
\end{proof}

Before we prove Proposition~\ref{prop:reassembly-1d}, let us first establish the following lemma.
\begin{lemma}\label{lem:optimal-coupling-1d}
    Under the settings of Proposition~\ref{prop:reassembly-1d},
    let $(\Omega,\CF,\PROB)$ be a probability space and let
    $V_1,\ldots,V_N:\Omega\to[0,1]$ be uniformly distributed random variables.
    For $i=1,\ldots,N$,
    let $F^{-1}_{\hat{\mu}_i}(u):=\inf\big\{x\in\R:F_{\hat{\mu}_i}(x)\ge u\big\}$ for all $u\in[0,1]$,
    that is, $F^{-1}_{\hat{\mu}_i}$ is the left-continuous generalized inverse of $F_{\hat{\mu}_i}$,
    and let $X_i,Z_i:\Omega\to\R$ be random variables satisfying 
    $X_i=F_{\hat{\mu}_i}^{-1}(V_i)$ and 
    $Z_i=F_{\mu_i}^{-1}(V_i)$ $\PROB$-almost surely.
    Then, the following statements hold.
    \begin{enumerate}[beginpenalty=10000,label=(\roman*)]
        \item\label{lems:optimal-coupling-1d-equivalence}
        For $i=1,\ldots,N$ and for any $x\in\R$, 
        $\INDI_{\{X_i\le x\}}=\INDI_{\{V_i\le F_{\hat{\mu}_i}(x)\}}$ 
        holds $\PROB$-almost surely.
        Similarly, for $i=1,\ldots,N$ and for any $z\in\R$,
        $\INDI_{\{Z_i\le z\}}=\INDI_{\{V_i\le F_{\mu_i}(z)\}}$
        holds $\PROB$-almost surely.
        
        \item\label{lems:optimal-coupling-1d-optimality}
        For $i=1,\ldots,N$, the law $\gamma_i\in\CP(\R\times\R)$ of the random variable $(X_i,Z_i):\Omega\to\R\times\R$ satisfies 
        $\gamma_i\in\Gamma(\hat{\mu}_i,\mu_i)$ and 
        $\int_{\R\times\R}|x-z|\DIFFM{\gamma_i}{\DIFF x,\DIFF z}=W_1(\hat{\mu}_i,\mu_i)$.
    \end{enumerate}
\end{lemma}

\begin{proof}[Proof of Lemma~\ref{lem:optimal-coupling-1d}]
    Let us fix an arbitrary $i\in\{1,\ldots,N\}$.
    It follows from \citep[Proposition~A.3(v)--(vi)]{mcneil2005quantitative}
    that 
    $V_i\le F_{\hat{\mu}_i}\big(F^{-1}_{\hat{\mu}_i}(V_i)\big)$ $\PROB$-almost surely, 
    and that $F^{-1}_{\hat{\mu}_i}\big(F_{\hat{\mu}_i}(x)\big)\le x$ for all $x\in\R$.
    Subsequently, combining these two inequalities with the non-decreasing properties of 
    $F_{\hat{\mu}_i}$ and $F^{-1}_{\hat{\mu}_i}$
    yields that
    \begin{align*}
    X_i\le x &\quad\Rightarrow\quad V_i\le F_{\hat{\mu}_i}\big(F^{-1}_{\hat{\mu}_i}(V_i)\big)= F_{\hat{\mu}_i}(X_i)\le F_{\hat{\mu}_i}(x) \quad \PROB\text{-a.s.} && \forall x\in\R,\; \forall 1\le i\le N,\\
    V_i\le F_{\hat{\mu}_i}(x) &\quad\Rightarrow\quad X_i=F^{-1}_{\hat{\mu}_i}(V_i)\le F^{-1}_{\hat{\mu}_i}(F_{\hat{\mu}_i}(x))\le x \quad \PROB\text{-a.s.} && \forall x\in\R,\; \forall 1\le i\le N.
    \end{align*}
    It thus holds that 
    $\INDI_{\{X_i\le x\}}=\INDI_{\{V_i\le F_{\hat{\mu}_i}(x)\}}$ 
    $\PROB$-almost surely for all $x\in\R$.
    The same argument applied to 
    $Z_i$, $F_{\mu_i}$, and $F_{\mu_i}^{-1}$ proves that
    $\INDI_{\{Z_i\le z\}}=\INDI_{\{V_i\le F_{\mu_i}(z)\}}$
    $\PROB$-almost surely for all $z\in\R$.
    The proof of statement~\ref{lems:optimal-coupling-1d-equivalence} is complete.

    Since $V_i$ is uniformly distributed on $[0,1]$, statement~\ref{lems:optimal-coupling-1d-equivalence} implies that 
    $\PROB[X_i\le x]=\PROB[V_i\le F_{\hat{\mu}_i}(x)]=F_{\hat{\mu}_i}(x)$ for all $x\in\R$, and $\PROB[Z_i\le z]=\PROB[V_i\le F_{\mu_i}(z)]=F_{\mu_i}(z)$ for all $z\in\R$. 
    Consequently, it holds that $\gamma_i\in\Gamma(\hat{\mu}_i,\mu_i)$, and
    \begin{align*}
    \int_{\R\times\R}|x-z|\DIFFM{\gamma_i}{\DIFF x,\DIFF z}=\EXP\big[|X_i-Z_i|\big]
    =\EXP\big[\big|F^{-1}_{\hat{\mu}_i}(V_i)-F^{-1}_{\mu_i}(V_i)\big|\big]
    =\int_{[0,1]}\big|F^{-1}_{\hat{\mu}_i}(v)-F^{-1}_{\mu_i}(v)\big|\DIFFX{v}.
    \end{align*}
    Applying \citep[Eq.~(3.1.6)]{rachev1998mass} then shows that
    $\int_{\R\times\R}|x-z|\DIFFM{\gamma_i}{\DIFF x,\DIFF z}=W_1(\hat{\mu}_i,\mu_i)$.
    The proof is now complete.
\end{proof}

\begin{proof}[Proof of Proposition~\ref{prop:reassembly-1d}]
Throughout the proof, 
let $F^{-1}_{\hat{\mu}_i}(u):=\inf\big\{x\in\R:F_{\hat{\mu}_i}(x)\ge u\big\}$ $\forall u\in[0,1]$
be the left-continuous generalized inverse of $F_{\hat{\mu}_i}$,
for $i=1,\ldots,N$.
To prove statement~\ref{props:reassembly-1d-copula}, 
let $(\Omega,\CF,\PROB)$ be a probability space and 
let $(V_1,\ldots,V_N):\Omega\to[0,1]^N$ be a random variable with distribution function $C$,
i.e., $\PROB\big[V_1\le\nobreak v_1,\;\ldots,\;V_N\le v_N\big]=C(v_1,\ldots,v_N)$ for all $(v_1,\ldots,v_N)\in[0,1]^N$.
Then, the properties of $C$ imply that $V_1,\ldots,V_N$ are uniformly distributed.
Moreover, for $i=1,\ldots,N$, 
let the random variables $X_i,Z_i:\Omega\to\R$ be defined by
$X_i:=F^{-1}_{\hat{\mu}_i}(V_i)$ and $Z_i:=F^{-1}_{\mu_i}(V_i)$. 
Lemma~\ref{lem:optimal-coupling-1d}\ref{lems:optimal-coupling-1d-optimality} then implies that
the law $\gamma_i\in\CP(\R\times\R)$ of the random variable
$(X_i,Z_i):\Omega\to\R\times\R$ satisfies 
$\gamma_i\in\Gamma(\hat{\mu}_i,\mu_i)$ and 
$\int_{\R\times\R}|x-z|\DIFFM{\gamma_i}{\DIFF x,\DIFF z}=W_1(\hat{\mu}_i,\mu_i)$.
Moreover, Lemma~\ref{lem:optimal-coupling-1d}\ref{lems:optimal-coupling-1d-equivalence} and
the properties of $C$
imply the following identities:
\begin{align*}
\PROB[X_1\le x_1,\ldots,X_N\le x_N]&=\PROB\big[V_1\le F_{\hat{\mu}_1}(x_1),\ldots,V_N\le F_{\hat{\mu}_N}(x_N)\big]\\
&=C\big(F_{\hat{\mu}_1}(x_1),\ldots,F_{\hat{\mu}_N}(x_N)\big)\\
&=F_{\hat{\mu}}(x_1,\ldots,x_N) && \forall (x_1,\ldots,x_N)\in\R^N,\\
\PROB[Z_1\le z_1,\ldots,Z_N\le z_N]&=\PROB\big[V_1\le F_{\mu_1}(z_1),\ldots,V_N\le F_{\mu_N}(z_N)\big]\\
&=C\big(F_{\mu_1}(z_1),\ldots,F_{\mu_N}(z_N)\big)\\
&=F_{\tilde{\mu}}(z_1,\ldots,z_N) && \forall (z_1,\ldots,z_N)\in\R^N.
\end{align*}
Therefore, the law of the random variable $(X_1,\ldots,X_N):\Omega\to\R^N$
is $\hat{\mu}\in\CP(\R^N)$,
and $F_{\tilde{\mu}}$ is the distribution function of the random variable
$(Z_1,\ldots,Z_N):\Omega\to\R^N$.
Subsequently, 
let $\tilde{\mu}\in\CP(\R^N)$ denote the law of the random variable
$(Z_1,\ldots,Z_N):\Omega\to\R^N$ and let $\gamma$ denote the law of the random variable $(X_1,\ldots,X_N,Z_1,\ldots,Z_N):\Omega\to\R^{2N}$. 
One checks that $\gamma$ satisfies all the required properties stated in Definition~\ref{def:reassembly}, and it hence holds that $\tilde{\mu}\in R(\hat{\mu};\mu_1,\ldots,\mu_N)$. 
This completes the proof of statement~\ref{props:reassembly-1d-copula}.

To prove statement~\ref{props:reassembly-1d-semidisc},
let $(\sigma_i)_{i=1:N}$, $(c_{i,j})_{j=1:J,\,i=0:N}$ be given by the statement of the proposition, and let 
$U_0$, $(U_{i,j})_{j=1:J,\,i=1:N}$ satisfy the properties in
(\ref{eqn:reassembly-1d-rvcollection}).
Subsequently, let us define the random variables $V_1,\ldots,V_N:\Omega\to[0,1]$ and $X_1,\ldots,X_N:\Omega\to\R$ as follows:
\begin{align*}
    V_i&:=\sum_{j=1}^{J}(c_{i,j}+a_jU_{i,j})\INDI_{[c_{0,j},c_{0,j}+a_j)}(U_0) && \forall 1\le i\le N, \\
    X_i&:=\sum_{j=1}^{J}x_{i,j}\INDI_{[c_{0,j},c_{0,j}+a_j)}(U_0) && \forall 1\le i\le N.
\end{align*}
Moreover, let us define $C:[0,1]^N\to[0,1]$ to be the distribution function of the random variable $(V_1,\ldots,V_N):\Omega\to[0,1]^N$.
The proof of statement~\ref{props:reassembly-1d-semidisc} is divided into the following five steps.
\begin{itemize}[beginpenalty=10000]
    \item \textit{Step 1: showing that $0=c_{i,\sigma_i(1)}<c_{i,\sigma_i(2)}<\cdots<c_{i,\sigma_i(J)}<1$, $c_{i,\sigma_i(j+1)}-c_{i,\sigma_i(j)}=a_{\sigma_i(j)}$ for $j=1,\ldots,J-1$, $i=1,\ldots,N$, and that
    $1-c_{i,\sigma_i(J)}=a_{\sigma_i(J)}$, for $i=1,\ldots,N$.}
    
    \item \textit{Step 2: showing that $X_i=F_{\hat{\mu}_i}^{-1}(V_i)$ $\PROB$-almost surely for $i=1,\ldots,N$.}
    
    \item \textit{Step 3: showing that $V_i$ has uniform distribution for $i=1,\ldots,N$.}
    
    \item \textit{Step 4: showing that the random variable $(X_1,\ldots,X_N):\Omega\to\R^N$ has law~$\hat{\mu}$.}
    
    \item \textit{Step 5: using statement~\ref{props:reassembly-1d-copula} to show that statement~\ref{props:reassembly-1d-semidisc} holds.}
\end{itemize}

\textit{Step 1.}
Observe that since $(a_j)_{j=1:J}\subset(0,1)$ and $\sum_{j=1}^{J}a_{\sigma_i(j)}=\sum_{j=1}^{J}a_j=1$ for $i=1,\ldots,N$,
we have
\begin{align*}
    c_{i,\sigma_i(j+1)}-c_{i,\sigma_i(j)}&=\Bigg(\sum_{l=1}^{\sigma_i^{-1}(\sigma_i(j+1))-1}a_{\sigma_i(l)}\Bigg) - \Bigg(\sum_{l=1}^{\sigma_i^{-1}(\sigma_i(j))-1}a_{\sigma_i(l)}\Bigg)\\
    &=a_{\sigma_i(j)}>0  \qquad\qquad \forall 1\le j\le J-1,\; \forall 1\le i\le N,\allowdisplaybreaks\\
    1-c_{i,\sigma_i(J)}&=1 - \Bigg(\sum_{l=1}^{\sigma_i^{-1}(\sigma_i(J))-1}a_{\sigma_i(l)}\Bigg)\\
    &=a_{\sigma_i(J)}>0  \hspace{122.5pt} \forall 1\le i\le N.
\end{align*}

\textit{Step 2.}
Let us fix arbitrary $i\in\{1,\ldots,N\}$, $j\in\{1,\ldots,J\}$, and $v\in\big(c_{i,\sigma_i(j)},c_{i,\sigma_i(j)}+a_{\sigma_i(j)}\big)$.
On the one hand,
it holds that
\begin{align*}
    F_{\hat{\mu}_i}(x_{i,\sigma_i(j)})=\hat{\mu}_i\big((-\infty,x_{i,\sigma_i(j)}]\big)=\sum_{l=1}^{J}a_{\sigma_i(l)}\INDI_{\{x_{i,\sigma_i(l)}\le x_{i,\sigma_i(j)}\}}\ge \sum_{l=1}^{j}a_{\sigma_i(l)}=c_{i,\sigma_i(j)}+a_{\sigma_i(j)}>v,
\end{align*}
which shows that $F_{\hat{\mu}_i}^{-1}(v)\le x_{i,\sigma_i(j)}$.
On the other hand, it holds for any $\epsilon>0$ that
\begin{align*}
    F_{\hat{\mu}_i}(x_{i,\sigma_i(j)}-\epsilon)\le \hat{\mu}_i\big((-\infty,x_{i,\sigma_i(j)})\big)=\sum_{l=1}^{J}a_{\sigma_i(l)}\INDI_{\{x_{i,\sigma_i(l)}<x_{i,\sigma_i(j)}\}}\le \sum_{l=1}^{j-1}a_{\sigma_i(l)}=c_{i,\sigma_i(j)}<v.
\end{align*}
This shows that $F_{\hat{\mu}_i}^{-1}(v)\ge x_{i,\sigma_i(j)}-\epsilon$ for all $\epsilon>0$,
and thus $F_{\hat{\mu}_i}^{-1}(v)= x_{i,\sigma_i(j)}$.
Because $\sigma_i$ is a bijection, the above argument has shown that
\begin{align}
    \label{eqn:reassembly-1d-proof-step2-1}
    \{c_{i,j}<V_i<c_{i,j}+a_j\}\subseteq \big\{F_{\hat{\mu}_i}^{-1}(V_i)=x_{i,j}\big\} \qquad \forall 1\le j\le J,\; \forall 1\le i\le N.
\end{align}
Moreover, the definitions of $V_i$ and $X_i$ along with the property of $(U_{i,j})_{j=1:J,\,i=1:N}$ in (\ref{eqn:reassembly-1d-rvcollection}) imply that
\begin{align}
    \begin{split}
        \PROB\big[X_i=x_{i,j}\big| c_{0,j}\le U_0 < c_{0,j}+a_j\big]&=1 \hspace{51.8pt} \qquad \forall 1\le j\le J,\; \forall 1\le i\le N,\\
        \!\!\!\PROB\big[c_{i,j}<V_{i}<c_{i,j}+a_j \big| c_{0,j}\le U_0 < c_{0,j}+a_j\big] &= \PROB\big[0<U_{i,j}<1 \big| c_{0,j}\le U_0 < c_{0,j}+a_j\big]=1\\
        & \hspace{94pt} \forall 1\le j\le J,\; \forall 1\le i\le N.
    \end{split}
    \label{eqn:reassembly-1d-proof-step2-2}
\end{align}
Combining (\ref{eqn:reassembly-1d-proof-step2-1}) and (\ref{eqn:reassembly-1d-proof-step2-2}) leads to
\begin{align*}
    \PROB\big[X_i=F_{\hat{\mu}_i}^{-1}(V_i)\big] 
    &\ge \sum_{j=1}^{J}\PROB\big[X_i=F_{\hat{\mu}_i}^{-1}(V_i)=x_{i,j} \big| c_{0,j}\le U_0 < c_{0,j}+a_j \big] \PROB[c_{0,j}\le U_0 < c_{0,j}+a_j]\\
    &\ge \sum_{j=1}^{J} a_j \PROB\big[X_i=x_{i,j},\; c_{i,j}<V_i<c_{i,j}+a_j \big| c_{0,j}\le U_0 < c_{0,j}+a_j \big]\\
    &=\sum_{j=1}^{J}a_j=1 \hspace{214pt} \qquad \forall 1\le i\le N.
\end{align*}
This completes Step~2.

\textit{Step 3.}
Let us fix an arbitrary $i\in\{1,\ldots,N\}$.
It follows from the definition of $V_i$ and the property of $(U_{i,j})_{j=1:J,\,i=1:N}$ in (\ref{eqn:reassembly-1d-rvcollection}) that
\begin{align*}
    \PROB[V_i\le v,\; c_{0,j}\le U_0 < c_{0,j}+a_j]
    &= \PROB\Big[U_{i,j}\le {\textstyle\frac{v-c_{i,j}}{a_j}} \Big| c_{0,j}\le U_0 < c_{0,j}+a_j\Big] \PROB[c_{0,j}\le U_0 < c_{0,j}+a_j]\\
    &=a_j \bigg(\frac{(v-c_{i,j})^+}{a_j} \wedge 1\bigg)\\
    &= (v-c_{i,j})^+ \wedge a_j \hspace{89pt} \qquad \forall v\in[0,1],\; \forall 1\le i\le J.
\end{align*}
Summing the above identity over $j=1,\ldots,J$ and then using the conclusion of Step~1 yields
\begin{align*}
    \PROB[V_i\le v]&= \sum_{j=1}^{J}\PROB[V_i\le v,\; c_{0,j}\le U_0 < c_{0,j}+a_j]\\
    &= \sum_{j=1}^{J} (v-c_{i,j})^+ \wedge a_j = \sum_{j=1}^{J} (v-c_{i,\sigma_i(j)})^+ \wedge a_{\sigma_i(j)}
    =v \qquad \forall v\in[0,1].
\end{align*}
Step~3 is now complete.

\textit{Step 4.}
For $j=1,\ldots,J$, one checks that on the event $\{c_{0,j}\le U_0<c_{0,j}+a_j\}$, it holds $\PROB$-almost surely that
$(X_1,\ldots,X_N)=(x_{1,j},\ldots,x_{N,j})=\BIx_j$.
Hence, because $(\BIx_j)_{j=1:J}$ are distinct points by assumption,
we get
$\PROB\big[(X_1,\ldots,X_N)=\BIx_j\big]=\PROB[c_{0,j}\le U_0<c_{0,j}+a_j]=a_j$ for $j=1,\ldots,J$.

\textit{Step 5.}
Finally, observe that Step~3 has shown that $C$ is a copula.
Combining Lemma~\ref{lem:optimal-coupling-1d}\ref{lems:optimal-coupling-1d-equivalence} with the conclusions of Step~2 and Step~4 leads to
\begin{align*}
    C\big(F_{\hat{\mu}_1}(x_1),\ldots,F_{\hat{\mu}_N}(x_N)\big)&=\PROB\big[V_1\le F_{\hat{\mu}_1}(x_1),\ldots,V_N\le F_{\hat{\mu}_N}(x_N)\big]\\
    &=\PROB[X_1\le x_1,\ldots,X_N\le x_N]\\
    &= F_{\hat{\mu}}(x_1,\ldots,x_N)  \qquad \forall (x_1,\ldots,x_N)\in\R^N.
\end{align*}
Consequently, $C$ satisfies the properties in statement~\ref{props:reassembly-1d-copula}.
Moreover, Lemma~\ref{lem:optimal-coupling-1d}\ref{lems:optimal-coupling-1d-equivalence} also implies that
\begin{align*}
    C\big(F_{{\mu}_1}(z_1),\ldots,F_{{\mu}_N}(z_N)\big)&=\PROB\big[V_1\le F_{{\mu}_1}(z_1),\ldots,V_N\le F_{{\mu}_N}(z_N)\big]\\
    &=\PROB[Z_1\le z_1,\ldots,Z_N\le z_N] \qquad \forall (z_1,\ldots,z_N)\in\R^N.
\end{align*}
Now, statement~\ref{props:reassembly-1d-semidisc} follows from statement~\ref{props:reassembly-1d-copula}.
The proof is complete.
\end{proof}

% Proof of Proposition (reassembly in $\R^d$)
The proof of Proposition~\ref{prop:reassembly-dd}
is built upon the following property of Euclidean spaces with $d\ge 2$ dimensions equipped with a norm under which the closed unit ball is strictly convex.

\begin{lemma}\label{lem:strict-convexity}
    Let $d\in\N\cap[2,\infty)$ and let $\R^d$ be equipped with a norm $\|\cdot\|$ under which the closed unit ball is strictly convex.
    Then, for any $\BIx_0,\BIx_1\in\R^d$ with $\BIx_0\ne\BIx_1$
    and for any $\beta\in\R$,
    the set $\big\{\BIz\in\R^d:\|\BIx_0-\BIz\|-\|\BIx_1-\BIz\|=\beta\big\}$ has Lebesgue measure~0.
\end{lemma}

\begin{proof}[Proof of Lemma~\ref{lem:strict-convexity}]
    The assumption that the closed unit ball $\big\{\BIz\in\R^d:\|\BIz\|\le 1\big\}$ is strictly convex implies 
    that $\|\cdot\|$ possesses the following property:
    \begin{align}
        \forall \BIv_0,\BIv_1\in\R^{d},\; \forall \xi\in(0,1): \quad \|\BIv_0\|=\|\BIv_1\|=\big\|\xi\BIv_0+(1-\xi)\BIv_1\big\|=1 \quad \Rightarrow \quad \BIv_0=\BIv_1.
        \label{eqn:strict-convexity-property1}
    \end{align}
    Moreover, $\|\cdot\|$ also has the following property:
    \begin{align}
        \forall \BIw_0,\BIw_1\in\R^d: \quad \|\BIw_0\|+\|\BIw_1\|=\|\BIw_0+\BIw_1\| \quad \Rightarrow \quad \|\BIw_1\|\BIw_0=\|\BIw_0\|\BIw_1.
        \label{eqn:strict-convexity-property2}
    \end{align}
    To see this, notice that 
    $\|\BIw_1\|\BIw_0=\|\BIw_0\|\BIw_1$ holds trivially if 
    $\BIw_0=\veczero$ or $\BIw_1=\veczero$.
    If $\BIw_0\ne\veczero$ and $\BIw_1\ne\veczero$
    satisfy $\|\BIw_0\|+\|\BIw_1\|=\|\BIw_0+\BIw_1\|$,
    letting 
    $\BIv_0:=\frac{\BIw_0}{\|\BIw_0\|}$,
    $\BIv_1:=\frac{\BIw_1}{\|\BIw_1\|}$,
    and $\xi:=\frac{\|\BIw_0\|}{\|\BIw_0\|+\|\BIw_1\|}\in(0,1)$
    yields
    $\|\BIv_0\|=\|\BIv_1\|=1$
    as well as
    $\big\|\xi\BIv_0+(1-\xi)\BIv_1\big\|=\frac{\|\BIw_0+\BIw_1\|}{\|\BIw_0\|+\|\BIw_1\|}=1$,
    and hence the property (\ref{eqn:strict-convexity-property1}) guarantees that
    $\|\BIw_1\|\BIw_0=\|\BIw_0\|\BIw_1$.
    In this proof, let us denote 
    $\mathrm{Hyp}(\BIv_0,\BIv_1,\beta):=\big\{\BIz\in\R^d:\|\BIv_0-\BIz\|-\|\BIv_1-\BIz\|=\beta\big\}$ 
    for any $\BIv_0,\BIv_1\in\R^{d}$ and any $\beta\in\R$,
    and let us denote the Lebesgue measure on $\R^d$ by 
    $\FL$.
    Notice that under the Euclidean norm $\|\cdot\|\leftarrow\|\cdot\|_2$,
    the set $\mathrm{Hyp}(\BIv_0,\BIv_1,\beta)$ coincides with one branch of a $d$-dimensional hyperboloid if it is non-empty.
    Let us fix arbitrary 
    $\BIx_0,\BIx_1\in\R^d$ with $\BIx_0\ne\BIx_1$,
    fix an arbitrary $\beta\in\R$, and denote
    $\BIx_{\lambda}:=(1-\lambda)\BIx_0+\lambda\BIx_1$ for all $\lambda\in\R$.
    We will use the property~(\ref{eqn:strict-convexity-property2}) of $\|\cdot\|$ to prove that $\FL\big(\mathrm{Hyp}(\BIx_0,\BIx_1,\beta)\big)=\nobreak0$.
    This will be carried out via the following three steps.
    \begin{itemize}
        \item\textit{Step~1: showing that the following properties hold:}
        \begin{align}
            \begin{split}
                \beta=-\|\BIx_0-\BIx_1\| \quad &\Leftrightarrow \quad \mathrm{Hyp}(\BIx_0,\BIx_1,\beta)=\big\{\BIx_{\lambda}:\lambda\in(-\infty,0]\big\} \\
                &\Leftrightarrow \quad \mathrm{Hyp}(\BIx_0,\BIx_1,\beta)\cap\big\{\BIx_{\lambda}:\lambda\in(-\infty,0]\big\}\ne\emptyset,\\
                \beta=\|\BIx_0-\BIx_1\| \quad &\Leftrightarrow \quad \mathrm{Hyp}(\BIx_0,\BIx_1,\beta)=\big\{\BIx_{\lambda}:\lambda\in[1,\infty)\big\} \\
                &\Leftrightarrow \quad \mathrm{Hyp}(\BIx_0,\BIx_1,\beta)\cap\big\{\BIx_{\lambda}:\lambda\in[1,\infty)\big\}\ne\emptyset.
            \end{split}\
            \label{eqn:strict-convexity-proof-degenerate}
        \end{align}
        
        \item\textit{Step~2: showing that 
        $\mathrm{Hyp}(\BIx_{\lambda},\BIx_{1+\lambda},\beta)\cap\mathrm{Hyp}(\BIx_{\lambda'},\BIx_{1+\lambda'},\beta) =\emptyset$ whenever $|\beta|\ne\|\BIx_0-\BIx_1\|$, $\lambda,\lambda'\in[0,1)$, and $\lambda\ne\lambda'$.}
        
        \item\textit{Step~3: showing that $\FL\big(\mathrm{Hyp}(\BIx_0,\BIx_1,\beta)\big)=0$.}
    \end{itemize}

    \textit{Step~1.}
    It is clear that 
    \begin{align}
        \begin{split}
        \mathrm{Hyp}(\BIx_0,\BIx_1,\beta)=\big\{\BIx_{\lambda}:\lambda\in(-\infty,0]\big\}\quad&\Rightarrow\quad \mathrm{Hyp}(\BIx_0,\BIx_1,\beta)\cap\big\{\BIx_{\lambda}:\lambda\in(-\infty,0]\big\}\ne\emptyset,\\
        \mathrm{Hyp}(\BIx_0,\BIx_1,\beta)=\big\{\BIx_{\lambda}:\lambda\in[1,\infty)\big\}\quad&\Rightarrow\quad\mathrm{Hyp}(\BIx_0,\BIx_1,\beta)\cap\big\{\BIx_{\lambda}:\lambda\in[1,\infty)\big\}\ne\emptyset.
        \end{split}
        \label{eqn:strict-convexity-proof-degenerate-1}
    \end{align}
    Observe that $\BIx_{\lambda}=\BIx_0+\lambda(\BIx_1-\BIx_0)$ $\forall\lambda\in\R$,
    and we thus have
    \begin{align}
        \|\BIx_0-\BIx_{\lambda}\|-\|\BIx_{1}-\BIx_{\lambda}\|=\big(|\lambda|-|1-\nobreak\lambda|\big)\|\BIx_0-\BIx_1\|\qquad \forall \lambda\in\R.
        \label{eqn:strict-convexity-proof-degenerate-extra}
    \end{align}
    From (\ref{eqn:strict-convexity-proof-degenerate-extra}), one verifies that
    \begin{align}
        \begin{split}
        \mathrm{Hyp}(\BIx_0,\BIx_1,\beta)\cap\big\{\BIx_{\lambda}:\lambda\in(-\infty,0]\big\}\ne\emptyset\quad&\Rightarrow\quad\beta=-\|\BIx_0-\BIx_1\|,\\
        \mathrm{Hyp}(\BIx_0,\BIx_1,\beta)\cap\big\{\BIx_{\lambda}:\lambda\in[1,\infty)\big\}\ne\emptyset \quad&\Rightarrow\quad \beta=\|\BIx_0-\BIx_1\|.
        \end{split}
        \label{eqn:strict-convexity-proof-degenerate-2}
    \end{align}
    Next, let us assume that 
    $\beta=-\|\BIx_0-\BIx_1\|$
    and fix an arbitrary $\BIz\in \mathrm{Hyp}(\BIx_0,\BIx_1,\beta)$.
    It thus holds that 
    $\|\BIx_0-\BIz\|-\|\BIx_1-\BIz\|=-\|\BIx_0-\BIx_1\|$.
    Applying (\ref{eqn:strict-convexity-property2}) with respect to 
    $\BIw_0\leftarrow \BIx_1-\BIx_0$, $\BIw_1\leftarrow \BIx_0-\BIz$
    yields
    $\|\BIx_0-\BIz\|(\BIx_1-\BIx_0)=\|\BIx_0-\BIx_1\|(\BIx_0-\BIz)$,
    which implies that
    $\BIz=\big(1+\frac{\|\BIx_0-\BIz\|}{\|\BIx_0-\BIx_1\|}\big)\BIx_0-\frac{\|\BIx_0-\BIz\|}{\|\BIx_0-\BIx_1\|}\BIx_1
    =\BIx_{\lambda_{-}}$,
    where $\lambda_{-}:=-\frac{\|\BIx_0-\BIz\|}{\|\BIx_0-\BIx_1\|}\in(-\infty,0]$.
    Consequently, we get 
    $\mathrm{Hyp}(\BIx_0,\BIx_1,\beta)\subseteq\big\{\BIx_{\lambda}:\lambda\in(-\infty,0]\big\}$.
    Moreover, $\big\{\BIx_{\lambda}:\lambda\in(-\infty,0]\big\}\subseteq \mathrm{Hyp}(\BIx_0,\BIx_1,\beta)$ 
    follows directly from (\ref{eqn:strict-convexity-proof-degenerate-extra}).
    We have thus shown that 
    \begin{align}
        \begin{split}
        \beta=-\|\BIx_0-\BIx_1\|\quad&\Rightarrow\quad\mathrm{Hyp}(\BIx_0,\BIx_1,\beta)=\big\{\BIx_{\lambda}:\lambda\in(-\infty,0]\big\}.
        \end{split}
        \label{eqn:strict-convexity-proof-degenerate-3-1}
    \end{align}
    Lastly, let us assume that
    $\beta=\|\BIx_0-\BIx_1\|$
    and fix an arbitrary $\BIz\in \mathrm{Hyp}(\BIx_0,\BIx_1,\beta)$.
    It holds that 
    $\|\BIx_0-\BIz\|-\|\BIx_1-\BIz\|=\|\BIx_0-\BIx_1\|$.
    Applying (\ref{eqn:strict-convexity-property2}) with respect to 
    $\BIw_0\leftarrow \BIx_0-\BIx_1$, $\BIw_1\leftarrow \BIx_1-\BIz$
    yields
    $\|\BIx_1-\BIz\|(\BIx_0-\BIx_1)=\|\BIx_0-\BIx_1\|(\BIx_1-\BIz)$,
    which implies that
    $\BIz=-\frac{\|\BIx_1-\BIz\|}{\|\BIx_0-\BIx_1\|}\BIx_0+\big(1+\frac{\|\BIx_1-\BIz\|}{\|\BIx_0-\BIx_1\|}\big)\BIx_1
    =\BIx_{\lambda_{+}}$,
    where $\lambda_{+}:=1+\frac{\|\BIx_0-\BIz\|}{\|\BIx_0-\BIx_1\|}\in[1,\infty)$.
    Consequently, we get 
    $\mathrm{Hyp}(\BIx_0,\BIx_1,\beta)\subseteq\big\{\BIx_{\lambda}:\lambda\in[1,\infty)\big\}$.
    Moreover, $\big\{\BIx_{\lambda}:\lambda\in[1,\infty)\big\}\subseteq \mathrm{Hyp}(\BIx_0,\BIx_1,\beta)$ 
    follows directly from (\ref{eqn:strict-convexity-proof-degenerate-extra}).
    We have thus shown that 
    \begin{align}
        \begin{split}
        \beta=\|\BIx_0-\BIx_1\|\quad&\Rightarrow\quad\mathrm{Hyp}(\BIx_0,\BIx_1,\beta)=\big\{\BIx_{\lambda}:\lambda\in[1,\infty)\big\}.
        \end{split}
        \label{eqn:strict-convexity-proof-degenerate-3-2}
    \end{align}
    Combining (\ref{eqn:strict-convexity-proof-degenerate-1}), (\ref{eqn:strict-convexity-proof-degenerate-2}),
    (\ref{eqn:strict-convexity-proof-degenerate-3-1}), and (\ref{eqn:strict-convexity-proof-degenerate-3-2}) proves 
    (\ref{eqn:strict-convexity-proof-degenerate}).

    \textit{Step~2.}
    In this step, let us assume that $|\beta|\ne\|\BIx_0-\BIx_1\|$.
    Observe that
    \begin{align}
        \mathrm{Hyp}(\BIx_{\lambda},\BIx_{1+\lambda},\beta)=\mathrm{Hyp}(\BIx_{0},\BIx_{1},\beta) + \lambda (\BIx_1-\BIx_0) \qquad \forall \lambda\in[0,1).
        \label{eqn:strict-convexity-proof-shift-invariance}
    \end{align}
    Let us fix arbitrary $\lambda,\lambda'\in[0,1)$ where $\lambda>\lambda'$.
    It holds by (\ref{eqn:strict-convexity-proof-shift-invariance}) that
    \begin{align*}
        &\mathrm{Hyp}(\BIx_{\lambda},\BIx_{1+\lambda},\beta)\cap \mathrm{Hyp}(\BIx_{\lambda'},\BIx_{1+\lambda'},\beta) \\
        &\qquad= \Big(\big[\mathrm{Hyp}(\BIx_0,\BIx_1,\beta) + (\lambda-\lambda')(\BIx_1-\BIx_0)\big]\cap \mathrm{Hyp}(\BIx_0,\BIx_1,\beta)\Big) + \lambda'(\BIx_1-\BIx_0)\\
        &\qquad=\big(\mathrm{Hyp}(\BIx_{\lambda-\lambda'},\BIx_{1+\lambda-\lambda'},\beta)\cap \mathrm{Hyp}(\BIx_0,\BIx_1,\beta)\big) + \lambda'(\BIx_1-\BIx_0).
    \end{align*}
    Therefore, in order to show that 
    $\mathrm{Hyp}(\BIx_{\lambda},\BIx_{1+\lambda},\beta)\cap \mathrm{Hyp}(\BIx_{\lambda'},\BIx_{1+\lambda'},\beta)=\nobreak\emptyset$,
    we can assume without loss of generality that $\lambda'=0$.
    In the following, let us take an arbitrary $\BIz\in\mathrm{Hyp}(\BIx_0,\BIx_1,\beta)$ 
    and show that
    $\BIz\notin \mathrm{Hyp}(\BIx_{\lambda},\BIx_{1+\lambda},\beta)$.
    Note that because $\lambda\in(0,1)$,
    we have 
    $\BIx_1=\frac{1}{1+\lambda}\BIx_{1+\lambda}+\frac{\lambda}{1+\lambda}\BIx_0$,
    and the triangle inequality implies that
    \begin{align}
        \|\BIx_\lambda-\BIz\|&=\big\|(1-\lambda)(\BIx_0-\BIz)+\lambda(\BIx_1-\BIz)\big\|
        \le (1-\lambda)\|\BIx_0-\BIz\| + \lambda\|\BIx_1-\BIz\|,
        \label{eqn:strict-convexity-proof-step2-ineq1}\\
        \|\BIx_1-\BIz\|&= \big\|{\textstyle\frac{1}{1+\lambda}}(\BIx_{1+\lambda}-\BIz)+{\textstyle\frac{\lambda}{1+\lambda}}(\BIx_0-\BIz)\big\|
        \le \frac{1}{1+\lambda}\|\BIx_{1+\lambda}-\BIz\|+\frac{\lambda}{1+\lambda}\|\BIx_0-\BIz\|.
        \label{eqn:strict-convexity-proof-step2-ineq2}
    \end{align}
    We will show that both (\ref{eqn:strict-convexity-proof-step2-ineq1}) and (\ref{eqn:strict-convexity-proof-step2-ineq2}) are strict inequalities.
    First, let us suppose for the sake of contradiction that the inequality in (\ref{eqn:strict-convexity-proof-step2-ineq1}) is an equality.
    Applying (\ref{eqn:strict-convexity-property2}) with respect to
    $\BIw_0\leftarrow (1-\nobreak\lambda)(\BIx_0-\BIz)$, $\BIw_1\leftarrow \lambda(\BIx_1-\BIz)$ leads to
    $\lambda(1-\lambda)\|\BIx_1-\BIz\|(\BIx_0-\BIz)=\lambda(1-\lambda)\|\BIx_0-\BIz\|(\BIx_1-\BIz)$,
    and hence
    \begin{align}
        \big(\|\BIx_0-\BIz\|-\|\BIx_1-\BIz\|\big)\BIz=\|\BIx_0-\BIz\|\BIx_1-\|\BIx_1-\BIz\|\BIx_0.
        \label{eqn:strict-convexity-proof-step2-contra1}
    \end{align}
    If $\|\BIx_0-\BIz\|-\|\BIx_1-\BIz\|=0$, then (\ref{eqn:strict-convexity-proof-step2-contra1}) implies that $\BIx_0=\BIx_1$, 
    which contradicts the assumption that $\BIx_0\ne\BIx_1$.
    Thus, $\|\BIx_0-\BIz\|-\|\BIx_1-\BIz\|\ne0$ and (\ref{eqn:strict-convexity-proof-step2-contra1})
    implies that
    $\BIz=-\frac{\|\BIx_1-\BIz\|}{\|\BIx_0-\BIz\|-\|\BIx_1-\BIz\|}\BIx_0+\frac{\|\BIx_0-\BIz\|}{\|\BIx_0-\BIz\|-\|\BIx_1-\BIz\|}\BIx_1=\BIx_{\hat{\lambda}}$,
    where $\hat{\lambda}:=\frac{\|\BIx_0-\BIz\|}{\|\BIx_0-\BIz\|-\|\BIx_1-\BIz\|}$.
    One checks that $\hat{\lambda}(1-\hat{\lambda})\le 0$ and thus $\hat{\lambda}\in(-\infty,0]\cup[1,\infty)$.
    This shows that $\big(\mathrm{Hyp}(\BIx_0,\BIx_1,\beta)\cap \big\{\BIx_{\bar{\lambda}}:\bar{\lambda}\in(-\infty,0]\big\}\big) \cup \big(\mathrm{Hyp}(\BIx_0,\BIx_1,\beta)\cap \big\{\BIx_{\bar{\lambda}}:\bar{\lambda}\in[1,\infty)\big\}\big)\neq\nobreak\emptyset$.
    Subsequently, applying (\ref{eqn:strict-convexity-proof-degenerate}) yields
    $|\beta|=\|\BIx_0-\BIx_1\|$, which contradicts the assumption that $|\beta|\ne\|\BIx_0-\BIx_1\|$.
    We thus conclude that the inequality in (\ref{eqn:strict-convexity-proof-step2-ineq1}) is strict.
    Next, let us suppose for the sake of contradiction that the inequality in (\ref{eqn:strict-convexity-proof-step2-ineq2}) is an equality.
    Applying (\ref{eqn:strict-convexity-property2}) with respect to
    $\BIw_0\leftarrow \frac{1}{1+\lambda}(\BIx_{1+\lambda}-\BIz)$, $\BIw_1\leftarrow \frac{\lambda}{1+\lambda}(\BIx_0-\BIz)$ yields
    $\frac{\lambda}{(1+\lambda)^2}\|\BIx_0-\BIz\|(\BIx_{1+\lambda}-\BIz)=\frac{\lambda}{(1+\lambda)^2}\|\BIx_{1+\lambda}-\BIz\|(\BIx_0-\BIz)$,
    and hence
    \begin{align}
        \begin{split}
            \big(\|\BIx_{1+\lambda}-\BIz\|-\|\BIx_0-\BIz\|\big)\BIz &= \|\BIx_{1+\lambda}-\BIz\|\BIx_0 - \|\BIx_0-\BIz\|\BIx_{1+\lambda}\\
            &=\big(\|\BIx_{1+\lambda}-\BIz\|+\lambda\|\BIx_0-\BIz\|\big)\BIx_0 - (1+\lambda)\|\BIx_0-\BIz\|\BIx_1.
        \end{split}
        \label{eqn:strict-convexity-proof-step2-contra2}
    \end{align}
    If $\|\BIx_{1+\lambda}-\BIz\|-\|\BIx_0-\BIz\|=0$, 
    then (\ref{eqn:strict-convexity-proof-step2-contra2}) implies that 
    $\BIx_0=\BIx_{1+\lambda}=-\lambda\BIx_{0}+(1+\lambda)\BIx_1$ and thus 
    $\BIx_0=\BIx_1$, which contradicts the assumption that $\BIx_0\ne\BIx_1$.
    Thus, $\|\BIx_{1+\lambda}-\BIz\|-\|\BIx_0-\BIz\|\ne 0$ and (\ref{eqn:strict-convexity-proof-step2-contra2}) implies that
    $\BIz=\frac{\|\BIx_{1+\lambda}-\BIz\|+\lambda\|\BIx_0-\BIz\|}{\|\BIx_{1+\lambda}-\BIz\|-\|\BIx_0-\BIz\|}\BIx_0-\frac{(1+\lambda)\|\BIx_0-\BIz\|}{\|\BIx_{1+\lambda}-\BIz\|-\|\BIx_0-\BIz\|}\BIx_{1}=\BIx_{\tilde{\lambda}}$,
    where $\tilde{\lambda}:=-\frac{(1+\lambda)\|\BIx_0-\BIz\|}{\|\BIx_{1+\lambda}-\BIz\|-\|\BIx_0-\BIz\|}$.
    One checks that $\tilde{\lambda}(1-\tilde{\lambda})\le 0$ and hence $\tilde{\lambda}\in(-\infty,0]\cup[1,\infty)$.
    This again shows that $\big(\mathrm{Hyp}(\BIx_0,\BIx_1,\beta)\cap \big\{\BIx_{\bar{\lambda}}:\bar{\lambda}\in(-\infty,0]\big\}\big) \cup \big(\mathrm{Hyp}(\BIx_0,\BIx_1,\beta)\cap \big\{\BIx_{\bar{\lambda}}:\bar{\lambda}\in[1,\infty)\big\}\big)\neq\nobreak\emptyset$.
    Subsequently, applying (\ref{eqn:strict-convexity-proof-degenerate}) yields
    $|\beta|=\|\BIx_0-\BIx_1\|$, which contradicts the assumption that $|\beta|\ne\|\BIx_0-\BIx_1\|$.
    Hence, we conclude that the inequality in (\ref{eqn:strict-convexity-proof-step2-ineq2}) is strict.
    Now, combining (\ref{eqn:strict-convexity-proof-step2-ineq1}) and (\ref{eqn:strict-convexity-proof-step2-ineq2}) results in
    \begin{align*}
        \|\BIx_{\lambda}-\BIz\|-\|\BIx_{1+\lambda}-\BIz\| &< \big((1-\lambda)\|\BIx_0-\BIz\|+\lambda\|\BIx_1-\BIz\|\big) + \big(\lambda\|\BIx_0-\BIz\|-(1+\lambda)\|\BIx_1-\BIz\|\big)\\
        &=\|\BIx_0-\BIz\|-\|\BIx_1-\BIz\|=\beta,
    \end{align*}
    which shows that $\BIz\notin\mathrm{Hyp}(\BIx_{\lambda},\BIx_{1+\lambda},\beta)$.
    This completes Step~2.

    \textit{Step~3.}
    In this step, let $B(r):=\big\{\BIz\in\R^d:\|\BIz\|\le r\big\}$ and 
    $\mathrm{Hyp}_{r}(\BIx_0,\BIx_1,\beta):=\mathrm{Hyp}(\BIx_0,\BIx_1,\beta)\cap B(r)$ for all $r>0$.
    In the case where $|\beta|=\|\BIx_0-\BIx_0\|$, 
    it follows from Step~1 that 
    $\mathrm{Hyp}(\BIx_0,\BIx_1,\beta)\subset\big\{\BIx_{\lambda}:\lambda\in\R\big\}$.
    Since $d\ge 2$ and $\big\{\BIx_{\lambda}:\lambda\in\R\big\}$ is a one-dimensional affine subset in $\R^d$,
    we get
    $\FL\big(\mathrm{Hyp}(\BIx_0,\BIx_1,\beta)\big)\le \FL\big(\big\{\BIx_{\lambda}:\lambda\in\R\big\}\big)=\nobreak0$.
    In the case where $|\beta|\ne\|\BIx_0-\BIx_1\|$, we have
    \begin{align}
        \bigcup_{\lambda\in[0,1)\,\cap\,\Q} \big(\mathrm{Hyp}_r(\BIx_0,\BIx_1,\beta) + \lambda (\BIx_1 - \BIx_0)\big)\subseteq B\big(r+\|\BIx_0-\BIx_1\|\big) \qquad \forall r>0.
        \label{eqn:strict-convexity-proof-countable-union}
    \end{align}
    Moreover, it follows from (\ref{eqn:strict-convexity-proof-shift-invariance}) that
    \begin{align*}
        \mathrm{Hyp}_{r}(\BIx_0,\BIx_1,\beta)+\lambda(\BIx_1-\BIx_0)& \subseteq \mathrm{Hyp}(\BIx_0,\BIx_1,\beta) + \lambda(\BIx_1-\BIx_0)\\
        &= \mathrm{Hyp}(\BIx_{\lambda},\BIx_{1+\lambda},\beta) \qquad \forall \lambda\in[0,1)\cap\Q,\; \forall r>0,
    \end{align*}
    and thus Step~2 implies that the sets 
    $\big(\mathrm{Hyp}_{r}(\BIx_0,\BIx_1,\beta)+\lambda(\BIx_1-\BIx_0)\big)_{\lambda\in[0,1)\,\cap\,\Q}$ are pair-wise disjoint for all $r>0$.
    Consequently, 
    (\ref{eqn:strict-convexity-proof-countable-union}) and the translation invariance property of $\FL$ imply that
    \begin{align*}
        \sum_{\lambda\in[0,1)\,\cap\,\Q}\FL\big(\mathrm{Hyp}_{r}(\BIx_0,\BIx_1,\beta)\big)
        &= \sum_{\lambda\in[0,1)\,\cap\,\Q}\FL\big(\mathrm{Hyp}_{r}(\BIx_0,\BIx_1,\beta) + \lambda(\BIx_1-\BIx_0)\big)\\
        &=\FL\Bigg(\bigcup_{\lambda\in[0,1)\,\cap\,\Q}\big[\mathrm{Hyp}_{r}(\BIx_0,\BIx_1,\beta) + \lambda(\BIx_1-\BIx_0)\big]\Bigg)\\
        &\le \FL\big(B\big(r+\|\BIx_0-\BIx_1\|\big)\big)<\infty \qquad\qquad\qquad\qquad \forall r>0.
    \end{align*}
    This shows that $\FL\big(\mathrm{Hyp}_{r}(\BIx_0,\BIx_1,\beta)\big)=0$ for all $r>0$.
    Finally, 
    we get $\FL\big(\mathrm{Hyp}(\BIx_0,\BIx_1,\beta)\big)=\FL\big(\bigcup_{n\in\N}\mathrm{Hyp}_{n}(\BIx_0,\BIx_1,\beta)\big)=\lim_{n\to\infty}\FL\big(\mathrm{Hyp}_{n}(\BIx_0,\BIx_1,\beta)\big)=\nobreak0$.
    The proof is now complete.
\end{proof}

\begin{proof}[Proof of Proposition~\ref{prop:reassembly-dd}]
Let us first fix an arbitrary $i\in\{1,\ldots,N\}$ and prove statement~\ref{props:reassembly-dd1}. 
Notice that since $d_{\CX_i}$ is induced by the norm $\|\cdot\|$ on $\R^{d_i}$,
it is continuous and non-negative.
Moreover, 
for any $\BIv\in\R^{d_i}$, it holds by the triangle inequality that
$d_{\CX_i}(\BIx,\BIz)\le d_{\CX_i}(\BIx,\BIv)+d_{\CX_i}(\BIz,\BIv)$ for all $\BIx,\BIz\in\R^{d_i}$,
where $d_{\CX_i}(\,\cdot\,,\BIv)$ is continuous and belongs to both $\CL^1(\R^{d_i},\hat{\mu}_i)$ and $\CL^1(\R^{d_i},\mu_i)$.
Consequently, it follows from 
\citep[Theorem~5.10(iii)]{villani2008optimal} that
\begin{align}
\begin{split}
W_1(\hat{\mu}_i,\mu_i)&=\inf_{\gamma_i\in\Gamma(\hat{\mu}_i,\mu_i)}\bigg\{\int_{\R^{d_i}\times\R^{d_i}}d_{\CX_i}(\BIx,\BIz)\DIFFM{\gamma_i}{\DIFF\BIx,\DIFF\BIz}\bigg\}\\
&=\max_{\phi\in\CL^1(\R^{d_i}\!,\,\hat{\mu}_i)}\left\{\int_{\R^{d_i}}\phi(\BIx)\DIFFM{\hat{\mu}_i}{\DIFF\BIx}-\int_{\R^{d_i}}\sup_{\BIx\in\R^{d_i}}\big\{\phi(\BIx)-\|\BIx-\BIz\|\big\}\DIFFM{\mu_i}{\DIFF\BIz}\right\},
\end{split}
\label{eqn:reassembly-dd-proof-kant}
\end{align}
where the maximum is attained.
Let $\phi^\star\in\CL^1(\R^{d_i},\hat{\mu}_i)$ be an arbitrary optimal solution of the maximization problem in (\ref{eqn:reassembly-dd-proof-kant}).
We will show that
\begin{align}
\int_{\R^{d_i}}\sup_{\BIx\in\R^{d_i}}\big\{\phi^\star(\BIx)-\|\BIx-\BIz\|\big\}\DIFFM{\mu_i}{\DIFF\BIz}=\int_{\R^{d_i}}\max_{1\le j\le J_i}\big\{\phi^\star(\BIx_{i,j})-\|\BIx_{i,j}-\BIz\|\big\}\DIFFM{\mu_i}{\DIFF\BIz}.
\label{eqn:reassembly-dd-proof-supmax}
\end{align}
Suppose for the sake of contradiction that (\ref{eqn:reassembly-dd-proof-supmax}) does not hold. Then, because $\sup_{\BIx\in\R^{d_i}}\big\{\phi^\star(\BIx)-\|\BIx-\nobreak\BIz\|\big\}\ge\max_{1\le j\le J_i}\big\{\phi^\star(\BIx_{i,j})-\|\BIx_{i,j}-\BIz\|\big\}$ for all $\BIz\in\R^{d_i}$, there exist $\beta>0$ and a set 
\begin{align*}
E_\beta:=\Bigg\{\BIz\in\R^{d_i}:\sup_{\BIx\in\R^{d_i}}\big\{\phi^\star(\BIx)-\|\BIx-\BIz\|\big\}-\max_{1\le j\le J_i}\big\{\phi^\star(\BIx_{i,j})-\|\BIx_{i,j}-\BIz\|\big\}>\beta\Bigg\}\in \CB(\R^{d_i})
\end{align*}
with $\mu_i(E_\beta)>0$. 
Let us define $\tilde{\phi}_{\beta}\in\CL^1(\R^{d_i},\hat{\mu}_i)$ as follows:
\begin{align*}
\tilde{\phi}_{\beta}(\BIx):=\begin{cases}
\phi^\star(\BIx) & \forall \BIx\in\{\BIx_{i,j}:1\le j\le J_i\},\\
\phi^\star(\BIx)-\beta & \forall \BIx\in\R^{d_i}\setminus \{\BIx_{i,j}:1\le j\le J_i\}. 
\end{cases}
\end{align*}
It holds that $\tilde{\phi}_{\beta}(\BIx)\le \phi^\star(\BIx)$ for all $\BIx\in\R^{d_i}$.
Subsequently, let us fix an arbitrary $\tilde{\BIz}\in E_{\beta}$.
It holds that
\begin{align}
    \begin{split}
    \sup_{\BIx\in\R^{d_i}}\big\{\phi^\star(\BIx)-\|\BIx-\tilde{\BIz}\|-\beta\big\}
    &>\max_{1\le j\le J_i}\big\{\phi^\star(\BIx_{i,j})-\|\BIx_{i,j}-\tilde{\BIz}\|\big\}\\
    &=\max_{1\le j\le J_i}\big\{\tilde{\phi}_{\beta}(\BIx_{i,j})-\|\BIx_{i,j}-\BIz\|\big\},
    \end{split}
    \label{eqn:reassembly-dd-proof-step1-property1}\\
    \sup_{\BIx\in\R^{d_i}}\big\{\phi^\star(\BIx)-\|\BIx-\tilde{\BIz}\|\big\}
    &=\sup_{\BIx\in\R^{d_i}\setminus\{\BIx_{i,j}:1\le j\le J_i\}}\big\{\phi^\star(\BIx)-\|\BIx-\tilde{\BIz}\|\big\}.
    \label{eqn:reassembly-dd-proof-step1-property2}
\end{align}
Then, (\ref{eqn:reassembly-dd-proof-step1-property1}) implies that there exists 
$\tilde{\BIx}\in\R^{d_i}\setminus\{\BIx_{i,j}:1\le j\le J_i\}$
which satisfies 
$\tilde{\phi}_{\beta}(\tilde{\BIx})-\|\tilde{\BIx}-\tilde{\BIz}\|=\phi^\star(\tilde{\BIx})-\|\tilde{\BIx}-\tilde{\BIz}\|-\beta > \max_{1\le j\le J_i}\big\{\tilde{\phi}_{\beta}(\BIx_{i,j})-\|\BIx_{i,j}-\BIz\|\big\}$.
It hence holds by (\ref{eqn:reassembly-dd-proof-step1-property2}) that
\begin{align*}
    \sup_{\BIx\in\R^{d_i}}\big\{\tilde{\phi}_{\beta}(\BIx)-\|\BIx-\tilde{\BIz}\|\big\}
    &= \sup_{\BIx\in\R^{d_i}\setminus\{\BIx_{i,j}:1\le j\le J_i\}}\big\{\tilde{\phi}_{\beta}(\BIx)-\|\BIx-\tilde{\BIz}\|\big\}\\
    &= \sup_{\BIx\in\R^{d_i}\setminus\{\BIx_{i,j}:1\le j\le J_i\}}\big\{\phi^\star(\BIx)-\|\BIx-\tilde{\BIz}\|\big\}-\beta\\
    &= \sup_{\BIx\in\R^{d_i}}\big\{\phi^\star(\BIx)-\|\BIx-\tilde{\BIz}\|\big\}-\beta.
\end{align*}
Since $\tilde{\BIz}\in E_\beta$ is arbitrary, we obtain
\begin{align*}
\int_{\R^{d_i}}\sup_{\BIx\in\R^{d_i}}\big\{\phi^\star(\BIx)-\|\BIx-\BIz\|\big\}\DIFFM{\mu_i}{\DIFF\BIz}-\int_{\R^{d_i}}\sup_{\BIx\in\R^{d_i}}\big\{\tilde{\phi}_{\beta}(\BIx)-\|\BIx-\BIz\|\big\}\DIFFM{\mu_i}{\DIFF\BIz}\ge\beta\mu_i(E_\beta)>0.
\end{align*}
Moreover, because $\int_{\R^{d_i}}\tilde{\phi}_{\beta}(\BIx)\DIFFM{\hat{\mu}_i}{\DIFF\BIx}=\int_{\R^{d_i}}\phi^\star(\BIx)\DIFFM{\hat{\mu}_i}{\DIFF\BIx}$,
we get
\begin{align*}
    &\int_{\R^{d_i}}\tilde{\phi}_{\beta}(\BIx)\DIFFM{\hat{\mu}_i}{\DIFF\BIx}-\int_{\R^{d_i}}\sup_{\BIx\in\R^{d_i}}\big\{\tilde{\phi}_{\beta}(\BIx)-\|\BIx-\BIz\|\big\}\DIFFM{\mu_i}{\DIFF\BIz}\\
    &\qquad >\int_{\R^{d_i}}\phi^\star(\BIx)\DIFFM{\hat{\mu}_i}{\DIFF\BIx}-\int_{\R^{d_i}}\sup_{\BIx\in\R^{d_i}}\big\{\phi^\star(\BIx)-\|\BIx-\nobreak\BIz\|\big\}\DIFFM{\mu_i}{\DIFF\BIz},
\end{align*}
which contradicts the optimality of $\phi^\star$ for the maximization problem in 
(\ref{eqn:reassembly-dd-proof-kant}).
Therefore, we conclude that (\ref{eqn:reassembly-dd-proof-supmax}) holds.
Combining (\ref{eqn:reassembly-dd-proof-kant}) and (\ref{eqn:reassembly-dd-proof-supmax})
leads to 
\begin{align}
    \begin{split}
        W_1(\hat{\mu}_i,\mu_i)&=\int_{\R^{d_i}}\phi^\star(\BIx)\DIFFM{\hat{\mu}_i}{\DIFF\BIx}-\int_{\R^{d_i}}\max_{1\le j\le J_i}\big\{\phi^\star(\BIx_{i,j})-\|\BIx_{i,j}-\nobreak\BIz\|\big\}\DIFFM{\mu_i}{\DIFF\BIz}\\
        &=\sum_{j=1}^{J_i}a_{i,j}\phi^\star(\BIx_{i,j})-\int_{\R^{d_i}}\max_{1\le j\le J_i}\big\{\phi^\star(\BIx_{i,j})-\|\BIx_{i,j}-\nobreak\BIz\|\big\}\DIFFM{\mu_i}{\DIFF\BIz}\\
        &\le \sup_{(\phi_{i,j})_{j=1:J_i}\subset\R}\Bigg\{\sum_{j=1}^{J_i}a_{i,j}\phi_{i,j}-\int_{\R^{d_i}}\max_{1\le j\le J_i}\big\{\phi_{i,j}-\|\BIx_{i,j}-\nobreak\BIz\|\big\}\DIFFM{\mu_i}{\DIFF\BIz}\Bigg\}.
    \end{split}
    \label{eqn:reassembly-dd-proof-kant-updated}
\end{align}
It remains to show that
\begin{align}
    \sup_{(\phi_{i,j})_{j=1:J_i}\subset\R}\Bigg\{\sum_{j=1}^{J_i}a_{i,j}\phi_{i,j}-\int_{\R^{d_i}}\max_{1\le j\le J_i}\big\{\phi_{i,j}-\|\BIx_{i,j}-\nobreak\BIz\|\big\}\DIFFM{\mu_i}{\DIFF\BIz}\Bigg\}\le W_1(\hat{\mu}_i,\mu_i).
    \label{eqn:reassembly-dd-proof-kant-converse}
\end{align}
To that end, let us fix arbitrary $(\phi_{i,j})_{j=1:J_i}\subset\R$ and define 
$\tilde{\phi}:\R^{d_i}\to\R$ as follows:
\begin{align*}
    \tilde{\phi}(\BIx):=\begin{cases}
        \tilde{\phi}_{i,j} & \text{if }\BIx=\BIx_{i,j},\; \forall 1\le j\le J_i, \\
        \tilde{\phi}_{i,1} - \|\BIx-\BIx_{i,1}\| & \forall \BIx \in \R^{d_i}\setminus \{\BIx_{i,j}:1\le j\le J_i\}.
    \end{cases}
\end{align*}
One checks that $\tilde{\phi}\in\CL^1(\R^{d_i},\hat{\mu}_i)$.
Moreover, it holds that
\begin{align*}
    \tilde{\phi}(\BIx)-\|\BIx-\BIz\|&=\tilde{\phi}_{i,j}-\|\BIx-\BIx_{i,1}\| - \|\BIx-\BIz\|\\
    &\le \tilde{\phi}(\BIx_{i,1})-\|\BIx_{i,1}-\BIz\|\\
    &\le \max_{1\le j\le J_i}\big\{\tilde{\phi}_{i,j}-\|\BIx_{i,j}-\BIz\|\big\} \qquad \forall \BIx\in\R^{d_i}\setminus \{\BIx_{i,j}:1\le j\le J_i\},\; \forall \BIz\in\R^{d_i}.
\end{align*}
It follows that
$\sup_{\BIx\in\R^{d_i}}\big\{\tilde{\phi}(\BIx)-\|\BIx-\BIz\|\big\}=\max_{1\le j\le J_i}\big\{\tilde{\phi}_{i,j}-\|\BIx_{i,j}-\BIz\|\big\}$ for all $\BIz\in\R^{d_i}$,
and thus (\ref{eqn:reassembly-dd-proof-kant}) implies that 
\begin{align*}
    W_1(\hat{\mu}_i,\mu_i)&\ge \int_{\R^{d_i}}\tilde{\phi}(\BIx)\DIFFM{\hat{\mu}_i}{\DIFF\BIx}-\int_{\R^{d_i}}\sup_{\BIx\in\R^{d_i}}\big\{\tilde{\phi}(\BIx)-\|\BIx-\BIz\|\big\}\DIFFM{\mu_i}{\DIFF\BIz}\\
    &= \sum_{j=1}^{J_i}a_{i,j}\tilde{\phi}_{i,j} - \int_{\R^{d_i}}\max_{1\le j\le J_i}\big\{\tilde{\phi}_{i,j}-\|\BIx_{i,j}-\BIz\|\big\}\DIFFM{\mu_i}{\DIFF\BIz}.
\end{align*}
Taking the supremum over all $(\tilde{\phi}_{i,j})_{j=1:J_i}\subset\R$ in the above inequality proves (\ref{eqn:reassembly-dd-proof-kant-converse}).
We can now define
$\phi^\star_{i,j}:=\phi^\star(\BIx_{i,j})$ for $j=1,\ldots,J_i$ and
conclude by (\ref{eqn:reassembly-dd-proof-kant-updated}) and (\ref{eqn:reassembly-dd-proof-kant-converse}) that 
$(\phi^\star_{i,j})_{j=1:J_i}$ is an optimal solution of (\ref{eqn:reassembly-dd-kant}),
and that the optimal value of (\ref{eqn:reassembly-dd-kant}) is equal to $W_1(\hat{\mu}_i,\mu_i)$.
This completes the proof of statement~\ref{props:reassembly-dd1}.

To prove statement~\ref{props:reassembly-dd2},
let us fix an arbitrary $i\in\{1,\ldots,N\}$ as well as an optimal solution 
$(\phi^\star_{i,j})_{j=1:J_i}\subset\R$ of the problem~(\ref{eqn:reassembly-dd-kant}),
and examine the first-order optimality condition of (\ref{eqn:reassembly-dd-kant})
with respect to $(\phi^\star_{i,j})_{j=1:J_i}$.
First, let us define the sets $(\widetilde{V}_{i,j})_{j=1:J_i}$ as follows: 
\begin{align}
    \begin{split}
    \widetilde{V}_{i,j}:=\big\{\BIz\in\R^{d_i}:\phi_{i,j}^\star-\|\BIx_{i,j}-\BIz\|>\phi_{i,j'}^\star-\|\BIx_{i,j'}-\BIz\|\;&\forall j'\in\{1,\ldots,J_i\}\setminus\{j\}\big\}\\
    & \qquad\qquad\qquad \forall 1\le j\le J_i.
    \end{split}
    \label{eqn:reassembly-dd-proof-Vtilde}
\end{align}
For $j=1,\ldots,J_i$, comparing (\ref{eqn:reassembly-dd-partition}) and (\ref{eqn:reassembly-dd-proof-Vtilde}) yields $\widetilde{V}_{i,j}\subseteq V_{i,j}$ and
\begin{align}
V_{i,j}\setminus\widetilde{V}_{i,j}\subseteq\bigcup_{j'\in\{1,\ldots,J_i\}\setminus\{j\}}\Big\{\BIz\in\R^{d_i}:\|\BIx_{i,j}-\BIz\|-\|\BIx_{i,j'}-\BIz\|=\phi_{i,j}^\star-\phi_{i,j'}^\star\Big\}.
\label{eqn:reassembly-dd-proof-boundary}
\end{align}
Recalling that $(\BIx_{i,j})_{j=1:J_i}$ are distinct by assumption, 
and applying Lemma~\ref{lem:strict-convexity} to each set in the union on the right-hand side of (\ref{eqn:reassembly-dd-proof-boundary}) with respect to
$\BIx_0\leftarrow \BIx_{i,j}$,
$\BIx_1\leftarrow\BIx_{i,j'}$,
$\beta\leftarrow \phi^\star_{i,j}-\phi^\star_{i,j'}$
shows that the right-hand side of (\ref{eqn:reassembly-dd-proof-boundary}) has Lebesgue measure~0,
and hence is $\mu_i$-negligible due to the assumption that $\mu_i$ is absolutely continuous with respect to the Lebesgue measure on $\R^{d_i}$.
Therefore, we conclude that 
\begin{align}
    \mu_i(\widetilde{V}_{i,j})=\mu_i(V_{i,j}) \qquad \forall 1\le j\le J_i.
    \label{eqn:reassembly-dd-proof-laguerre-meas-eq}
\end{align}
Moreover, using the property that $(\widetilde{V}_{i,j})_{j=1:J_i}$ are pair-wise disjoint,
we get $1\ge \mu_i\big(\bigcup_{j=1}^{J_i}\widetilde{V}_{i,j}\big)=\sum_{j=1}^{J_i}\mu_i(\widetilde{V}_{i,j})=\sum_{j=1}^{J_i}\mu_i(V_{i,j})\ge 1$,
which yields
\begin{align}
    \sum_{j=1}^{J_i}\INDI_{\widetilde{V}_{i,j}}(\BIz)
    =\sum_{j=1}^{J_i}\INDI_{V_{i,j}}(\BIz)
    = 1 \text{ for }\mu_i\text{-almost every } \BIz\in\R^{d_i}.
    \label{eqn:reassembly-dd-proof-laguerre-cover}
\end{align}
Next, let
$\Bphi^\star\in\R^{J_i}$,
$h:\R^{J_i}\times\R^{d_i}\to \R$,
$\Bzeta_{h}:\R^{d_i}\to\R^{J_i}$,
$u:\R^{J_i}\to\R$,
and $\Bzeta_{u}\in\R^{J_i}$ be defined as follows:
\begin{align*}
    \Bphi^\star&:=(\phi^\star_{i,1},\ldots,\phi^\star_{i,J_i})^\TRANSP,\\
    h(\Bphi,\BIz)&:= \max_{1\le j\le J_i}\big\{\phi_{i,j}- \|\BIx_{i,j}-\BIz\|\big\} && \forall \Bphi=(\phi_{i,1},\ldots,\phi_{i,J_i})^\TRANSP\in\R^{J_i},\; \forall \BIz\in\R^{d_i},\\
    \Bzeta_h(\BIz)&:= \big(\INDI_{\widetilde{V}_{i,1}}(\BIz),\ldots,\INDI_{\widetilde{V}_{i,J_i}}(\BIz)\big)^\TRANSP && \forall \BIz\in\R^{d_i},\\
    u(\Bphi)&:= \sum_{j=1}^{J_i}a_{i,j}\phi_{i,j} - \int_{\R^{d_i}}h(\Bphi,\BIz)\DIFFM{\mu_i}{\DIFF\BIz} && \forall \Bphi=(\phi_{i,1},\ldots,\phi_{i,J_i})^\TRANSP\in\R^{J_i},\\
    \Bzeta_u&:= \big(a_{i,1}-\mu_i(\widetilde{V}_{i,1}),\ldots,a_{i,J_i}-\mu_i(\widetilde{V}_{i,J_i})\big)^\TRANSP.
\end{align*}
For $j=1,\ldots,J_i$ and for any $\BIz\in \widetilde{V}_{i,j}$, it holds by the definition of $\widetilde{V}_{i,j}$ that there exists $\epsilon>0$ such that
\begin{align*}
    h(\Bphi^\star+\Blambda,\BIz)&=\phi^\star_{i,j}+\lambda_j-\|\BIx_{i,j}-\BIz\|=h(\Bphi^\star,\BIz) + \langle\Bzeta_h(\BIz),\Blambda\rangle \\ 
    & \hspace{130pt}\forall \Blambda=(\lambda_1,\ldots,\lambda_{J_i})^\TRANSP\in\big\{\BIw\in\R^{J_i}:\|\BIw\|_2<\epsilon\big\}.
\end{align*}
Hence, (\ref{eqn:reassembly-dd-proof-laguerre-cover}) guarantees that
\begin{align}
    \lim_{\Blambda\to\veczero}\frac{1}{\|\Blambda\|_2}\big|h(\Bphi^\star+\Blambda,\BIz)-h(\Bphi^\star,\BIz)-\langle\Bzeta_h(\BIz),\Blambda\rangle\big|=0 \text{ for }\mu_i\text{-almost every }\BIz\in\R^{d_i}.
    \label{eqn:reassembly-dd-proof-inner-gradient}
\end{align}
Moreover, let
\begin{align*}
    j(\Blambda,\BIz)\in \argmax_{1\le j\le J_i}\big\{\phi^\star_{i,j}+\lambda_j-\|\BIx_{i,j}-\BIz\|\big\} \qquad \forall \Blambda=(\lambda_1,\ldots,\lambda_{J_i})^\TRANSP\in\R^{J_i},\; \forall \BIz\in\R^{d_i}.
\end{align*}
It thus holds that
\begin{align*}
    h(\Bphi^\star+\Blambda,\BIz)-h(\Bphi^\star,\BIz)
    &=\phi^\star_{i,j(\Blambda,z)}+\lambda_{j(\Blambda,\BIz)}-\|\BIx_{i,j(\Blambda,\BIz)}-\BIz\| - \max_{1\le j\le J_i}\big\{\phi^\star_{i,j}-\|\BIx_{i,j}-\BIz\|\big\}\\
    &\le \phi^\star_{i,j(\Blambda,z)}+\lambda_{j(\Blambda,\BIz)}-\|\BIx_{i,j(\Blambda,\BIz)}-\BIz\| - \phi^\star_{i,j(\Blambda,\BIz)}+\|\BIx_{i,j(\Blambda,\BIz)}-\BIz\|\\
    &\le \|\Blambda\|_2 \qquad\qquad\qquad\qquad\qquad \forall \Blambda=(\lambda_1,\ldots,\lambda_{J_i})^\TRANSP\in\R^{J_i},\; \forall \BIz\in\R^{d_i},\\
    h(\Bphi^\star,\BIz)-h(\Bphi^\star+\Blambda,\BIz)
    &=\phi^\star_{i,j(\veczero,z)}-\|\BIx_{i,j(\veczero,\BIz)}-\BIz\| - \max_{1\le j\le J_i}\big\{\phi^\star_{i,j}+\lambda_{j}-\|\BIx_{i,j}-\BIz\|\big\}\\
    &\le \phi^\star_{i,j(\veczero,z)}-\|\BIx_{i,j(\veczero,\BIz)}-\BIz\| - \phi^\star_{i,j(\veczero,\BIz)} - \lambda_{j(\veczero,\BIz)}+\|\BIx_{i,j(\veczero,\BIz)}-\BIz\|\\
    &\le \|\Blambda\|_2 \qquad\qquad\qquad\qquad\qquad \forall \Blambda=(\lambda_1,\ldots,\lambda_{J_i})^\TRANSP\in\R^{J_i},\; \forall \BIz\in\R^{d_i}.
\end{align*}
Combining these two inequalities and using the Cauchy--Schwarz inequality yield
\begin{align}
    \begin{split}
    &\frac{1}{\|\Blambda\|_2}\big|h(\Bphi^\star+\Blambda,\BIz) - h(\Bphi^\star,\BIz) - \langle\Bzeta_h(\BIz),\Blambda\rangle\big|\\
    &\qquad\le \frac{1}{\|\Blambda\|_2}\big(\big|h(\Bphi^\star+\Blambda,\BIz) - h(\Bphi^\star,\BIz)\big|+\|\Bzeta_h(\BIz)\|_2\|\Blambda\|_2\big)\\
    &\qquad\le \frac{1}{\|\Blambda\|_2}\big(\|\Blambda\|_2+\|\Blambda\|_2\big) = 2 \qquad \forall \Blambda\in\R^{J_i},\; \forall \BIz\in\R^{d_i}.
    \end{split}
    \label{eqn:reassembly-dd-proof-dominated}
\end{align}
Consequently, we can use (\ref{eqn:reassembly-dd-proof-inner-gradient}),
(\ref{eqn:reassembly-dd-proof-dominated}), and Lebesgue's dominated convergence theorem to show that 
\begin{align*}
    &\lim_{\Blambda\to\veczero}\frac{1}{\|\Blambda\|_2}\big|u(\Bphi^\star+\Blambda)-u(\Bphi^\star)-\langle\Bzeta_{u},\Blambda\rangle\big|\\
    &\qquad= \lim_{\Blambda\to\veczero}\frac{1}{\|\Blambda\|_2} \bigg|\int_{\R^{d_i}}h(\Bphi^\star+\Blambda,\BIz)-h(\Bphi^\star,\BIz)-\langle\Bzeta_{h}(\BIz),\Blambda\rangle\DIFFM{\mu_i}{\DIFF\BIz}\bigg|\\
    &\qquad\le \lim_{\Blambda\to\veczero} \int_{\R^{d_i}}\frac{1}{\|\Blambda\|_2}\big|h(\Bphi^\star+\Blambda,\BIz)-h(\Bphi^\star,\BIz)-\langle\Bzeta_{h}(\BIz),\Blambda\rangle\big|\DIFFM{\mu_i}{\DIFF\BIz}\\
    &\qquad= \int_{\R^{d_i}}\lim_{\Blambda\to\veczero}\frac{1}{\|\Blambda\|_2} \big|h(\Bphi^\star+\Blambda,\BIz)-h(\Bphi^\star,\BIz)-\langle\Bzeta_{h}(\BIz),\Blambda\rangle\big|\DIFFM{\mu_i}{\DIFF\BIz}=0.
\end{align*}
Therefore, we have shown that $u:\R^{J_i}\to\R$ is differentiable at $\Bphi^\star$ with gradient $\Bzeta_{u}$.
Since $\Bphi^\star$ maximizes~$u$ by assumption,
we get $\Bzeta_u=\veczero$.
In view of (\ref{eqn:reassembly-dd-proof-laguerre-meas-eq}), 
we have shown that
$\mu_i(V_{i,j})=\mu_i(\widetilde{V}_{i,j})=a_{i,j}$ 
for $j=1,\ldots,J_i$.
This completes the proof of statement~\ref{props:reassembly-dd2}.

Lastly, to prove statement~\ref{props:reassembly-dd3},
let us first fix an arbitrary $i\in\{1,\ldots,N\}$.
Observe from statement~\ref{props:reassembly-dd2} and 
(\ref{eqn:reassembly-dd-proof-laguerre-cover})
that 
\begin{align*}
    \PROB[Z_i\in E]&=\sum_{j=1}^{J_i}\PROB[Z_i\in E | X_i=\BIx_{i,j}]\PROB[X_i=\BIx_{i,j}]\\
    &= \sum_{j=1}^{J_i} \frac{\mu_i(V_{i,j}\cap E)}{\mu_i(V_{i,j})} \hat{\mu}_i\big(\{\BIx_{i,j}\}\big)\\
    &=\sum_{j=1}^{J_i}\mu_i(V_{i,j}\cap E) = \mu_i(E) \qquad \forall E\in\CB(\R^{d_i}).
\end{align*}
Hence, the law of the random variable $Z_i:\Omega\to\R^{d_i}$ is~$\mu_i$.
Let $\gamma_i\in\CP(\R^{d_i}\times\R^{d_i})$ denote the law of the random variable 
$(X_i,Z_i):\Omega\to\R^{d_i}\times\R^{d_i}$.
It thus holds that $\gamma_i\in\Gamma(\hat{\mu}_i,\mu_i)$.
Next, let $\phi^\star:\R^{d_i}\to\R$ and $\psi^\star:\R^{d_i}\to\R$ be defined as follows:
\begin{align*}
    \phi^\star(\BIx)&:= \begin{cases}
        \phi^\star_{i,j} & \qquad\qquad \text{if }\BIx=\BIx_{i,j},\; \forall 1\le j\le J_i,\\
        \phi^\star_{i,1} - \|\BIx_{i,1}-\BIx\| & \qquad\qquad \forall \BIx\in\R^{d_i}\setminus \{\BIx_{i,j}:1\le j\le J_i\},
    \end{cases}\\
    \psi^\star(\BIz)&:=\max_{1\le j\le J_i}\big\{\phi^\star_{i,j}-\|\BIx_{i,j}-\BIz\|\big\} \hspace{-0.6pt}\qquad \forall \BIz\in\R^{d_i}.
\end{align*}
One can verify the following properties of $\phi^\star$ and $\psi^\star$:
\begin{align}
    \begin{split}
        \phi^\star(\BIx)-\psi^\star(\BIz)&=\phi^\star_{i,1}-\|\BIx_{i,1}-\BIx\| - \max_{1\le j\le J_i}\big\{\phi^\star_{i,j}-\|\BIx_{i,j}-\BIz\|\big\}\\
        &\le \phi^\star_{i,1}-\|\BIx_{i,1}-\BIx\| - \phi^\star_{i,1} + \|\BIx_{i,1}-\BIz\|\\
        &\le \|\BIx - \BIz\| \qquad \forall \BIx\in\R^{d_i}\setminus \{\BIx_{i,j}:1\le j\le J_i\},\; \forall \BIz\in\R^{d_i},
    \end{split}
    \label{eqn:reassembly-dd-proof-kant-property1}\\
    \phi^\star(\BIx_{i,j}) - \psi^\star(\BIz)&= \|\BIx_{i,j}-\BIz\| \hspace{75.0pt} \qquad \forall \BIz\in V_{i,j},\; \forall 1\le j\le J_i.
    \label{eqn:reassembly-dd-proof-kant-property2}
\end{align}
Combining (\ref{eqn:reassembly-dd-proof-kant-property1}),
(\ref{eqn:reassembly-dd-proof-kant-property2}),
and the properties of the random variable $(X_i,Z_i)$
yields that
$\phi^\star(\BIx)-\psi^\star(\BIz)\le \|\BIx-\BIz\|$ $\forall\BIx\in\R^{d_i}$, $\forall\BIz\in\R^{d_i}$,
as well as 
$\phi^\star(X_i)-\psi^\star(Z_i)=\|X_i-Z_i\|$ 
$\PROB$-almost surely.
Therefore, \citep[Theorem~5.10(iii)]{villani2008optimal}
guarantees that 
$\int_{\R^{d_i}\times\R^{d_i}}\|\BIx-\BIz\|\DIFFM{\gamma_i}{\DIFF\BIx,\DIFF\BIz}=W_1(\hat{\mu}_i,\mu_i)$.
Finally, 
let $\gamma\in\CP(\R^{d_1}\times\cdots\times\R^{d_N}\times \R^{d_1}\times\cdots\times\R^{d_N})$ denote the law of the random variable
$(X_1,\ldots,X_N,Z_1,\ldots,Z_N):\Omega\to \R^{d_1}\times\cdots\times\R^{d_N}\times \R^{d_1}\times\cdots\times\R^{d_N}$.
One verifies that $\gamma$
satisfies all the properties required by Definition~\ref{def:reassembly},
and therefore $\tilde{\mu}\in R(\hat{\mu};\mu_1,\ldots,\mu_N)$.
The proof is now complete.
\end{proof}

\subsection{Proofs of results in Section~\ref{ssec:momentsetdd}}\label{ssec:proof-momentsets}

Before presenting the proof in this subsection, let us first derive the following crucial properties of simplices.

\begin{lemma}\label{lem:simplices}
    Let $d\in\N$, let $n\in\N_0$ where $0\le n\le d$,
    and let $C\subset\R^d$ be an $n$-simplex, 
    i.e., $C$ is the convex hull of $n+1$ affinely independent points in $\R^d$. 
    Then, the following statements hold.
    \begin{enumerate}[label=(\roman*),beginpenalty=10000]
        \item\label{lems:simplices-funcs}%
        There exist affine functions 
        $\big(\bar{\lambda}_{C,\BIw}:\aff(C)\to\R\big)_{\BIw\in V(C)}$
        which satisfy 
        \begin{align}
            \sum_{\BIw\in V(C)}\bar{\lambda}_{C,\BIw}(\BIx)=1, \quad \sum_{\BIw\in V(C)}\bar{\lambda}_{C,\BIw}(\BIx)\BIw=\BIx \qquad \forall \BIx \in \aff(C).
            \label{eqn:simplices-funcs}
        \end{align}
        
        \item\label{lems:simplices-uniqueness}%
        For each $\BIx\in\aff(C)$, the representation of $\BIx$ as an affine combination of $V(C)$ in (\ref{eqn:simplices-funcs}) is unique, 
        that is,
        if there exist $(\tilde{\lambda}_{\BIw})_{\BIw\in V(C)}\subset\R$ that satisfy $\sum_{\BIw\in V(C)}\tilde{\lambda}_{\BIw}=1$ and 
        $\sum_{\BIw\in V(C)}\tilde{\lambda}_{\BIw}\BIw=\BIx$, then it necessarily holds that $\tilde{\lambda}_{\BIw}=\bar{\lambda}_{C,\BIw}(\BIx)$ $\forall\BIw\in V(C)$.
        
        \item\label{lems:simplices-face}%
        Every $F\in\FF(C)$ is an $n'$-simplex with $0\le n'\le n$, and 
        $F\in\FF(C)$ if and only if $V(F)\subseteq V(C)$.
        
        \item\label{lems:simplices-facefuncs}%
        Let $F\in\FF(C)$ and let $\big(\bar{\lambda}_{F,\BIw}:\aff(F)\to\R\big)_{\BIw\in V(F)}$ satisfy (\ref{eqn:simplices-funcs}) with $C$ replaced by $F$.
        Then, it holds that $\bar{\lambda}_{C,\BIw}(\BIx)=0$ $\forall \BIx\in F$, $\forall\BIw\in V(C)\setminus V(F)$,
        and
        $\bar{\lambda}_{C,\BIw}(\BIx)=\bar{\lambda}_{F,\BIw}(\BIx)$ $\forall \BIx\in F$, $\forall \BIw\in V(F)$.
        
        \item\label{lems:simplices-nonneg}%
        For all $\BIx\in\aff(C)$,
        $\BIx\in C$ if and only if $\bar{\lambda}_{C,\BIw}(\BIx)\ge \nobreak0$ $\forall \BIw\in V(C)$, 
        and $\BIx\in\relint(C)$ if and only if $\bar{\lambda}_{C,\BIw}(\BIx)>0$ $\forall \BIw\in V(C)$. 
        
    \end{enumerate}
\end{lemma}

\begin{proof}[Proof of Lemma~\ref{lem:simplices}]
    To begin, let us first enumerate $V(C)=\{\BIv_0,\BIv_1,\ldots,\BIv_n\}$,
    define
    \begin{align*}
        \bar{\Delta}_n:=\big\{(\lambda_0,\allowbreak\lambda_1,\ldots,\lambda_n)\allowbreak\in\R^{n+1}:{\textstyle\sum_{i=0}^{n}}\lambda_i=1\big\},
    \end{align*}
    and define $\bar{\BIz}:\bar{\Delta}_n\to\aff(C)$ as follows:
    \begin{align*}
        \bar{\BIz}(\lambda_0,\lambda_1,\ldots,\lambda_n):=\sum_{i=0}^{n}\lambda_i\BIv_i \qquad \forall (\lambda_0,\lambda_1,\ldots,\lambda_n)\in\bar{\Delta}_n.
    \end{align*}
    Since $\{\BIv_0,\BIv_1,\ldots,\BIv_n\}$ are affinely independent,
    $\bar{\BIz}$ is a bijective and affine function from 
    $\bar{\Delta}_n$ to $\aff(C)$,
    which admits an affine inverse that we denote by
    $\bar{\Blambda}(\,\cdot\,)=\big(\bar{\lambda}_{C,\BIv_0}(\,\cdot\,),\bar{\lambda}_{C,\BIv_1}(\,\cdot\,),\ldots,\bar{\lambda}_{C,\BIv_n}(\,\cdot\,)\big):\aff(C)\to\bar{\Delta}_n$.
    This yields the affine functions $\big(\bar{\lambda}_{C,\BIw}:\aff(C)\to\R\big)_{\BIw\in V(C)}$
    which satisfy (\ref{eqn:simplices-funcs}).
    Moreover, the bijectivity of $\bar{\Blambda}:\aff(C)\to\bar{\Delta}_n$ also implies that the representation of any $\BIx\in\aff(C)$ as an affine combination of $V(C)$ in (\ref{eqn:simplices-funcs}) is unique.
    We have thus completed the proofs of statement~\ref{lems:simplices-funcs} and statement~\ref{lems:simplices-uniqueness}.

    Next, statement~\ref{lems:simplices-face} follows directly from the definition of faces.
    To prove statement~\ref{lems:simplices-facefuncs},
    observe that by the properties of 
    $\big(\bar{\lambda}_{F,\BIw}:\aff(F)\to\R\big)_{\BIw\in V(F)}$
    and by statement~\ref{lems:simplices-face}, 
    each $\BIx\in F$ can be represented as an affine combination
    $\BIx=\sum_{\BIw\in V(C)}\tilde{\lambda}_{\BIw}\BIw$ where
    $\tilde{\lambda}_{\BIw}=\bar{\lambda}_{F,\BIw}(\BIx)$ $\forall \BIw\in V(F)$ and $\tilde{\lambda}_{\BIw}=\nobreak0$ $\forall \BIw\in V(C)\setminus V(F)$.
    Subsequently, statement~\ref{lems:simplices-facefuncs} follows from
    statement~\ref{lems:simplices-uniqueness}.
    
    Lastly, to prove statement~\ref{lems:simplices-nonneg},
    observe that $\bar{\BIz}\big(\bar{\Delta}_n\cap \R^{n+1}_+\big)=\conv\big(V(C)\big)=C$.
    Moreover, \citep[Theorem~6.6]{rockafellar1970convex} implies that
    $\bar{\BIz}\big(\bar{\Delta}_n\cap (0,1)^{n+1}\big)=\bar{\BIz}\big(\relint\big(\bar{\Delta}_n\cap \R^{n+1}_{+}\big)\big)=\relint\big(\bar{\BIz}\big(\bar{\Delta}_n\cap \R^{n+1}_{+}\big)\big)=\relint(C)$.
    Using the property that $\bar{\Blambda}(\,\cdot\,)$ is the inverse of $\bar{\BIz}(\,\cdot\,)$ then proves statement~\ref{lems:simplices-nonneg}.
    The proof is now complete.
\end{proof}

Let us also establish the following lemma applicable to polyhedral covers defined in Definition~\ref{def:polycover}, which contain polyhedra that are not necessarily simplices.
This lemma will be used in the proofs of Proposition~\ref{prop:momentset-simplex}, 
Proposition~\ref{prop:mmot-complexity-constant}, 
and Theorem~\ref{thm:momentset-euclidean} later.

% Additional Lemma that shows that relative interior of faces partition the union of the polyhedral cover
\begin{lemma}\label{lem:relint-partition} 
Let $d\in\N$, let $\CY\subseteq\R^d$, and let $\FC$ be a polyhedral cover of $\CY$; see Definition~\ref{def:polycover}. 
Then, the sets in $\big\{\relint(F):F\in\FF(\FC)\big\}$ are pair-wise disjoint and $\bigcup_{F\in\FF(\FC)}\relint(F)=\bigcup_{C\in\FC}C$. 
\end{lemma}

\begin{proof}[Proof of Lemma~\ref{lem:relint-partition}]
$\bigcup_{F\in\FF(\FC)}\relint(F)=\bigcup_{C\in\FC}C$ follows directly from \citep[Theorem~18.2]{rockafellar1970convex}. 
We will show that if $F,F'\in\FF(\FC)$ and $F\ne F'$ then $\relint(F)\cap\relint(F')=\emptyset$. 
Suppose that 
$F\in\FF(C)$ for $C\in\FC$, 
$F'\in\FF(C')$ for $C'\in\FC$, and 
$\relint(F)\cap\relint(F')\ne\emptyset$. 
Then, Definition~\ref{def:polycover} guarantees that
$C\hspace{1pt}\cap\hspace{1pt} C'\in \FF(C)\hspace{1pt}\cap\hspace{1pt} \FF(C')$. 
It hence holds that $\relint(F)\cap\relint(F')\subseteq F\cap F'\subseteq C\cap C'$. 
Thus, since $F$ is a convex set satisfying $\relint(F)\cap(C\cap C')\supseteq \relint(F)\cap\relint(F') \ne\emptyset$, 
where $C\cap C'\in \FF(C)$, we have by \citep[Theorem~18.1]{rockafellar1970convex} that $F\subseteq C\cap C'$. 
It hence follows from the definition of faces that $F\in \FF(C\cap C')$. 
Moreover, it follows from the same argument that $F'\in\FF(C\cap C')$, and thus 
$F=F'$ by \citep[Corollary~18.1.2]{rockafellar1970convex}. 
The proof is now complete.
\end{proof}

With Lemma~\ref{lem:simplices} and Lemma~\ref{lem:relint-partition} established, we are now ready to present the proof of Proposition~\ref{prop:momentset-simplex}.

% Proof of Proposition (simplicial cover)
\begin{proof}[Proof of Proposition~\ref{prop:momentset-simplex}]
    Let us first prove statement~\ref{props:momentset-simplex-bisection}.
    It holds by construction that, for every $r\in\N_0$,
    $\FC_r$ contains only $d$-simplices and that $\bigcup_{C\in\FC_r}C=C_0\supseteq\CY$.
    Observe that the property $C_1\cap C_2\in \FF(C_1)\cap \FF(C_2)$ $\forall C_1,C_2\in\FC_{r}$ holds for $r=\nobreak0$.
    In the following, we will use induction to show that this property holds for all $r\in\N_0$.
    To that end, let us assume that 
    $C_1\cap C_2\in \FF(C_1)\cap \FF(C_2)$ $\forall C_1,C_2\in\FC_{r-1}$ holds for some $r\in\N$.
    Let us take arbitrary $C_1,C_2\in\FC_{r}$.
    It follows from the construction of $\FC_{r}$ that there exist unique $\widetilde{C}_1,\widetilde{C}_2\in\FC_{r-1}$ such that 
    $C_1\subseteq\widetilde{C}_1$, $C_2\subseteq\widetilde{C}_2$.
    Subsequently, we define $\widetilde{F}:=\widetilde{C}_1\cap\widetilde{C}_2$, which belongs to $\FF(\widetilde{C}_1)\cap\FF(\widetilde{C}_2)$ by the induction hypothesis. 
    Note that $V(\widetilde{F})=V(\widetilde{C}_1)\cap V(\widetilde{C}_2)$ as a consequence of Lemma~\ref{lem:simplices}\ref{lems:simplices-face}.
    Moreover, we define $F:=\conv\big(V(C_1)\cap V(C_2)\big)\subseteq C_1\cap C_2$.
    Applying Lemma~\ref{lem:simplices}\ref{lems:simplices-face} then yields 
    $F\in\FF(C_1)\cap\FF(C_2)$.
    Thus, it remains to show that $C_1\cap C_2\subseteq F$.
    To that end, we will examine the relationships between $C_1$ and $\widetilde{C}_1$, and between $C_2$ and $\widetilde{C}_2$, when updating $\FC_{r-1}$ to $\FC_{r}$.
    Let $(\BIv,\BIv')\in\argmax_{(\BIw,\BIw')}\big\{\|\BIw-\BIw'\|:\BIw,\BIw'\in V(C),\; C\in \FC_{r-1}\big\}$ be the edge that is bisected when constructing $\FC_{r}$.
    The construction of $\FC_{r}$ implies that 
    \begin{align}
        \forall i\in\{1,2\},\; \begin{cases}
            \{\BIv,\BIv'\}\nsubseteq V(\widetilde{C}_i) \; \Leftrightarrow \; V(C_i)=V(\widetilde{C}_i), \\
            \{\BIv,\BIv'\}\subseteq V(\widetilde{C}_i) \; \Leftrightarrow \; V(C_i)=\big(V(\widetilde{C}_i)\setminus \{\BIw\}\big)\cup \big\{{\textstyle\frac{\BIv+\BIv'}{2}}\big\},\; \BIw\in\{\BIv,\BIv'\}.
        \end{cases}
        \label{eqn:momentset-simplex-proof-bisection-vertices}
    \end{align}
    We will prove the following statements:
    \begin{align}
        &\forall i\in\{1,2\},\; \forall \BIw\in\{\BIv,\BIv'\}: \nonumber\\
        \begin{split}
        &\qquad\qquad\qquad\qquad V(C_i)=\big(V(\widetilde{C}_i)\setminus \{\BIw\}\big)\cup \big\{{\textstyle\frac{\BIv+\BIv'}{2}}\big\},\; \{\BIv,\BIv'\}\nsubseteq V(\widetilde{F}) \\
        &\qquad\qquad\qquad\qquad\qquad\qquad\qquad\qquad \Rightarrow \; \widetilde{F}\cap C_i \subseteq \conv\big(V(\widetilde{F})\setminus \{\BIw\}\big),
        \end{split}
        \label{eqn:momentset-simplex-proof-bisection-property1}\\
        \begin{split}
        &\qquad\qquad\qquad\qquad V(C_i)=\big(V(\widetilde{C}_i)\setminus \{\BIw\}\big)\cup \big\{{\textstyle\frac{\BIv+\BIv'}{2}}\big\},\; \{\BIv,\BIv'\}\subseteq V(\widetilde{F}) \\
        &\qquad\qquad\qquad\qquad\qquad\qquad\qquad\qquad \Rightarrow \; \widetilde{F}\cap C_i \subseteq \conv\Big(\big(V(\widetilde{F})\setminus \{\BIw\}\big)\cup\big\{{\textstyle\frac{\BIv+\BIv'}{2}}\big\}\Big).
        \end{split}
        \label{eqn:momentset-simplex-proof-bisection-property2}
    \end{align}
    To show these, let $i=1$, $\BIw=\BIv$ without loss of generality,
    and assume that 
    $V(C_1)=\big(V(\widetilde{C}_1)\setminus \{\BIv\}\big)\cup \big\{{\textstyle\frac{\BIv+\BIv'}{2}}\big\}$.
    Then, (\ref{eqn:momentset-simplex-proof-bisection-vertices}) implies that 
    $\{\BIv,\BIv'\}\subseteq V(\widetilde{C}_1)$.
    Let $\BIz\in\widetilde{F}\cap C_1$ be arbitrary,
    and let $\big(\bar{\lambda}_{\widetilde{C}_1,\tilde{\BIw}}:\aff(\widetilde{C}_1)\to\R\big)_{\tilde{\BIw}\in V(\widetilde{C}_1)}$ be the functions given by 
    Lemma~\ref{lem:simplices}\ref{lems:simplices-funcs} with respect to $C\leftarrow\widetilde{C}_1$.
    It follows that 
    \begin{align}
        \begin{split}
            \BIz&=\sum_{\tilde{\BIw}\in V(\widetilde{C}_1)}\bar{\lambda}_{\widetilde{C}_1,\tilde{\BIw}}(\BIz)\tilde{\BIw} \\
            &= \Bigg(\sum_{\tilde{\BIw}\in V(\widetilde{C}_1)\setminus \{\BIv,\BIv'\}}\bar{\lambda}_{\widetilde{C}_1,\tilde{\BIw}}(\BIz)\tilde{\BIw}\Bigg) + 2\bar{\lambda}_{\widetilde{C}_1,\BIv}(\BIz)\big({\textstyle\frac{\BIv+\BIv'}{2}}\big) + \big(\bar{\lambda}_{\widetilde{C}_1,\BIv'}(\BIz)-\bar{\lambda}_{\widetilde{C}_1,\BIv}(\BIz)\big)\BIv'.
        \end{split}
        \label{eqn:momentset-simplex-proof-bisection-property-step1}
    \end{align}
    Subsequently, because 
    $V(C_1)=\big(V(\widetilde{C}_1\setminus\{\BIv,\BIv'\})\big)\cup \big\{\frac{\BIv+\BIv'}{2},\BIv'\big\}$ and $\BIz\in C_1$,
    applying Lemma~\ref{lem:simplices}\ref{lems:simplices-uniqueness} and 
    Lemma~\ref{lem:simplices}\ref{lems:simplices-nonneg} with respect to $C\leftarrow C_1$ to (\ref{eqn:momentset-simplex-proof-bisection-property-step1}) yields that 
    \begin{align}
        \bar{\lambda}_{\widetilde{C}_1,\tilde{\BIw}}(\BIz)\ge 0 \quad \forall \tilde{\BIw}\in V(\widetilde{C}_1)\setminus \{\BIv,\BIv'\}, \qquad 
        \bar{\lambda}_{\widetilde{C}_1,\BIv'}(\BIz)\ge \bar{\lambda}_{\widetilde{C}_1,\BIv}(\BIz)\ge\nobreak0.
        \label{eqn:momentset-simplex-proof-bisection-property-step2}
    \end{align}
    Moreover, since $\widetilde{F}\in\FF(\widetilde{C}_1)$,
    we apply Lemma~\ref{lem:simplices}\ref{lems:simplices-facefuncs} with respect to $C\leftarrow\widetilde{C}_1$, $F\leftarrow\widetilde{F}$ to get
    \begin{align}
        \bar{\lambda}_{\widetilde{C}_1,\tilde{\BIw}}(\BIz)=0 \quad \forall \tilde{\BIw}\in V(\widetilde{C}_1)\setminus V(\widetilde{F}).
        \label{eqn:momentset-simplex-proof-bisection-property-step3}
    \end{align}
    Now, let us apply (\ref{eqn:momentset-simplex-proof-bisection-property-step1}),
    (\ref{eqn:momentset-simplex-proof-bisection-property-step2}),
    and (\ref{eqn:momentset-simplex-proof-bisection-property-step3})
    to prove 
    (\ref{eqn:momentset-simplex-proof-bisection-property1})
    and (\ref{eqn:momentset-simplex-proof-bisection-property2}).
    In the case where $\BIv\notin V(\widetilde{F})$, 
    we get
    $\conv\big(V(\widetilde{F})\setminus\{\BIv\}\big)=\widetilde{F}\supseteq \widetilde{F}\cap C_1$.
    In the case where $\BIv'\notin V(\widetilde{F})$,
    (\ref{eqn:momentset-simplex-proof-bisection-property-step3}) shows that 
    $\bar{\lambda}_{\widetilde{C}_1,\BIv'}(\BIz)=\nobreak0$ and hence 
    (\ref{eqn:momentset-simplex-proof-bisection-property-step1}),
    (\ref{eqn:momentset-simplex-proof-bisection-property-step2}), and 
    Lemma~\ref{lem:simplices}\ref{lems:simplices-nonneg} imply that
    $\BIz\in\conv\big(V(\widetilde{F})\setminus \{\BIv,\BIv'\}\big)=\conv\big(V(\widetilde{F})\setminus \{\BIv\}\big)$.
    We have thus shown that (\ref{eqn:momentset-simplex-proof-bisection-property1}) holds.
    In the case where $\{\BIv,\BIv'\}\subseteq\widetilde{F}$,
    (\ref{eqn:momentset-simplex-proof-bisection-property-step1}),
    (\ref{eqn:momentset-simplex-proof-bisection-property-step2}),
    (\ref{eqn:momentset-simplex-proof-bisection-property-step3}),
    and Lemma~\ref{lem:simplices}\ref{lems:simplices-nonneg}
    imply that 
    $\BIz\in\conv\Big(V(\widetilde{F})\setminus\{\BIv,\BIv'\}\cup\big\{\frac{\BIv+\BIv'}{2},\BIv'\big\}\Big)=\conv\Big(\big(V(\widetilde{F})\setminus \{\BIv\}\big)\cup\big\{{\textstyle\frac{\BIv+\BIv'}{2}}\big\}\Big)$, which proves (\ref{eqn:momentset-simplex-proof-bisection-property2}).
    
    Now that (\ref{eqn:momentset-simplex-proof-bisection-property1})
    and (\ref{eqn:momentset-simplex-proof-bisection-property2}) have been established,
    let us prove $C_1\cap C_2\subseteq F$ by inspecting the four cases below in view of 
    (\ref{eqn:momentset-simplex-proof-bisection-vertices}).

    \underline{Case~1:}~$\{\BIv,\BIv'\}\nsubseteq V(\widetilde{C}_1)$ and $\{\BIv,\BIv'\}\nsubseteq V(\widetilde{C}_2)$.
    In this case,
    (\ref{eqn:momentset-simplex-proof-bisection-vertices}) implies that
    $V(\widetilde{F})\subseteq V(C_1)\cap V(C_2)$.
    Consequently, we get
    $C_1\cap C_2\subseteq \widetilde{C}_1\cap \widetilde{C}_2 = \widetilde{F}\subseteq \conv\big(V(C_1)\cap V(C_2)\big)=F$.

    \underline{Case~2:}~$\{\BIv,\BIv'\}\nsubseteq V(\widetilde{C}_1)$ and $\{\BIv,\BIv'\}\subseteq V(\widetilde{C}_2)$.
    In this case,
    we have 
    $\{\BIv,\BIv'\}\nsubseteq V(\widetilde{C}_1)\cap V(\widetilde{C}_2)=V(\widetilde{F})$,
    and (\ref{eqn:momentset-simplex-proof-bisection-vertices}) implies that 
    $V(C_1)=V(\widetilde{C}_1)$ and $V(C_2)=\big(V(\widetilde{C}_2)\setminus\{\BIw\}\big)\cup \big\{\frac{\BIv+\BIv'}{2}\big\}$ where $\BIw\in\{\BIv,\BIv'\}$.
    Subsequently, the premise of (\ref{eqn:momentset-simplex-proof-bisection-property1}) holds with respect to $i=2$, and it hence follows that 
    $\widetilde{F}\cap C_2\subseteq \conv\big(V(\widetilde{F})\setminus \{\BIw\}\big)\subseteq\conv\big(V(C_1)\cap V(C_2)\big)=F$.
    This yields
    $C_1\cap C_2\subseteq\widetilde{C}_1\cap \widetilde{C}_2\cap C_2=\widetilde{F}\cap C_2\subseteq F$.

    \underline{Case~3:}~$\{\BIv,\BIv'\}\subseteq V(\widetilde{C}_1)$ and $\{\BIv,\BIv'\}\nsubseteq V(\widetilde{C}_2)$.
    In this case, $C_1\cap C_2\subseteq F$ follows from the same argument as Case~2 by symmetry.

    \underline{Case~4:}~$\{\BIv,\BIv'\}\subseteq V(\widetilde{C}_1)$ and $\{\BIv,\BIv'\}\subseteq V(\widetilde{C}_2)$.
    In this case, we have
    $\{\BIv,\BIv'\}\subseteq V(\widetilde{C}_1)\cap V(\widetilde{C}_2)=V(\widetilde{F})$, 
    as well as 
    $V(C_1)=\big(V(\widetilde{C}_1)\setminus\{\BIw_1\}\big)\cup \big\{\frac{\BIv+\BIv'}{2}\big\}$,
    $V(C_2)=\big(V(\widetilde{C}_2)\setminus\{\BIw_2\}\big)\cup \big\{\frac{\BIv+\BIv'}{2}\big\}$,
    where $\BIw_1,\BIw_2\in\{\BIv,\BIv'\}$.
    Thus, the premise of (\ref{eqn:momentset-simplex-proof-bisection-property2}) holds with respect to both $i=1$ and $i=2$, and it follows that
    \begin{align}
        \begin{split}
            \widetilde{F}\cap C_1 &\subseteq \conv\Big(\big(V(\widetilde{F})\setminus \{\BIw_1\}\big)\cup\big\{{\textstyle\frac{\BIv+\BIv'}{2}}\big\}\Big), \\
            \widetilde{F}\cap C_2 &\subseteq \conv\Big(\big(V(\widetilde{F})\setminus \{\BIw_2\}\big)\cup\big\{{\textstyle\frac{\BIv+\BIv'}{2}}\big\}\Big).
        \end{split}
        \label{eqn:momentset-simplex-proof-bisection-case4-1}
    \end{align}
    If $\BIw_1=\BIw_2$, then 
    $\big(V(\widetilde{F})\setminus \{\BIw_1\}\big)\cup\big\{{\textstyle\frac{\BIv+\BIv'}{2}}\big\}=\big(V(\widetilde{F})\setminus \{\BIw_2\}\big)\cup\big\{{\textstyle\frac{\BIv+\BIv'}{2}}\big\}\subseteq V(C_1)\cap V(C_2)$,
    and (\ref{eqn:momentset-simplex-proof-bisection-case4-1}) yields
    $C_1\cap C_2 = \widetilde{C}_1\cap \widetilde{C}_2\cap C_1 \cap C_2=\big(\widetilde{F}\cap C_1\big)\cap \big(\widetilde{F}\cap C_2\big)\subseteq \conv\big(V(C_1)\cap V(C_2)\big)=F$.
    On the other hand, if $\BIw_1\ne\BIw_2$, then (\ref{eqn:momentset-simplex-proof-bisection-case4-1}) yields
    \begin{align}
        \begin{split}
        C_1\cap C_2&= \widetilde{C}_1\cap \widetilde{C}_2\cap C_1 \cap C_2=\big(\widetilde{F}\cap C_1\big)\cap \big(\widetilde{F}\cap C_2\big)\\
        &\subseteq \conv\Big(\big(V(\widetilde{F})\setminus \{\BIw_1\}\big)\cup\big\{{\textstyle\frac{\BIv+\BIv'}{2}}\big\}\Big) \cap \conv\Big(\big(V(\widetilde{F})\setminus \{\BIw_2\}\big)\cup\big\{{\textstyle\frac{\BIv+\BIv'}{2}}\big\}\Big).
        \end{split}
        \label{eqn:momentset-simplex-proof-bisection-case4-2}
    \end{align}
    Observe that 
    $\conv\Big(\big(V(\widetilde{F})\setminus \{\BIw_1\}\big)\cup\big\{{\textstyle\frac{\BIv+\BIv'}{2}}\big\}\Big)$ and $\conv\Big(\big(V(\widetilde{F})\setminus \{\BIw_2\}\big)\cup\big\{{\textstyle\frac{\BIv+\BIv'}{2}}\big\}\Big)$
    are two \mbox{$d$-simplices} formed by bisecting the simplex $\widetilde{F}$ at the midpoint of $\BIv,\BIv'\in V(\widetilde{F})$.
    Hence, it is straightforward to see that 
    $\conv\Big(\big(V(\widetilde{F})\setminus \{\BIw_1\}\big)\cup\big\{{\textstyle\frac{\BIv+\BIv'}{2}}\big\}\Big) \hspace{1pt}\cap\hspace{1pt} \conv\Big(\big(V(\widetilde{F})\setminus \{\BIw_2\}\big)\cup\big\{{\textstyle\frac{\BIv+\BIv'}{2}}\big\}\Big)=\conv\Big(\big(V(\widetilde{F})\setminus \{\BIv,\BIv'\}\big)\cup\big\{{\textstyle\frac{\BIv+\BIv'}{2}}\big\}\Big)$.
    Substituting this into (\ref{eqn:momentset-simplex-proof-bisection-case4-2}) yields 
    $C_1\cap C_2\subseteq \conv\Big(\big(V(\widetilde{F})\setminus \{\BIv,\BIv'\}\big)\cup\big\{{\textstyle\frac{\BIv+\BIv'}{2}}\big\}\Big)\subseteq \conv\big(V(C_1)\cap V(C_2)\big)=F$.

    We have thus shown that $C_1\cap C_2\subseteq F$ holds in all four cases, and hence $C_1\cap C_2 \in \FF(C_1)\cap \FF(C_2)$ holds for all $C_1,C_2\in\FC_{r}$.
    Therefore, we conclude by induction that $\bigcup_{C\in\FC_r}C\supseteq\CY$ and 
    $C_1\cap C_2\in\FF(C_1)\cap \FF(C_2)$ $\forall C_1,C_2\in\FC_{r}$ hold for all $r\in\N_0$.

    Lastly, we will show that 
    $\lim_{r\to\infty}\max_{C\in\FC_r}\max_{\BIv,\BIv'\in V(C)}\big\{\|\BIv-\BIv'\|_2\big\}=\nobreak0$ 
    and use the equivalence of norms on $\R^d$ to conclude that 
    $\lim_{r\to\infty}\max_{C\in\FC_r}\max_{\BIv,\BIv'\in V(C)}\big\{\|\BIv-\BIv'\|\big\}=\nobreak0$.
    To that end, for any $d$-simplex $C\subset\R^d$ and any distinct pair of its extreme points $\BIw,\BIw'\in V(C)$, we refer to $\conv\big(\{\BIw,\BIw'\}\big)$
    as the edge between $\BIw,\BIw'$.
    The construction of $\FC_r$ from $\FC_{r-1}$ corresponds to finding a longest edge $\conv\big(\{\BIv,\BIv'\}\big)$ in all $C\in\FC_{r-1}$ and subsequently bisecting each $d$-simplex $C\in\FC_{r-1}$ containing this edge at the midpoint of the edge.
    Let $C\in\FC_{r-1}$ be an arbitrary $d$-simplex that satisfy $\{\BIv,\BIv'\}\subseteq V(C)$ and let 
    $C_1:=\conv\Big(\big(V(C)\setminus \{\BIv\}\big)\cup \big\{\frac{\BIv+\BIv'}{2}\big\}\Big)$,
    $C_2:=\conv\Big(\big(V(C)\setminus \{\BIv'\}\big)\cup \big\{\frac{\BIv+\BIv'}{2}\big\}\Big)$
    be the two new $d$-simplices resulted from bisecting $C$ at the midpoint of $\BIv$ and $\BIv'$.
    One checks that
    \begin{align*}
        &\max_{\BIw,\BIw'\in V(C_1)}\big\{\|\BIw-\BIw'\|_2\big\} \vee \max_{\BIw,\BIw'\in V(C_2)}\big\{\|\BIw-\BIw'\|_2\big\} \\
        &\qquad= \max_{\BIw,\BIw'\in V(C),\,\{\BIw,\BIw'\}\ne\{\BIv,\BIv'\}}\big\{\|\BIw-\BIw'\|_2\big\} \vee \max_{\BIw\in V(C)}\Big\{\big\|\BIw-{\textstyle\frac{\BIv+\BIv'}{2}}\big\|_2\Big\},
    \end{align*}
    where $\max_{\BIw,\BIw'\in V(C),\,\{\BIw,\BIw'\}\ne\{\BIv,\BIv'\}}\big\{\|\BIw-\BIw'\|_2\big\}$ corresponds to the maximum $\|\cdot\|_2$-length of the edges of $C$ other than $\conv\big(\{\BIv,\BIv'\}\big)$,
    and 
    $\max_{\BIw\in V(C)}\Big\{\big\|\BIw-{\textstyle\frac{\BIv+\BIv'}{2}}\big\|_2\Big\}$
    corresponds to the maximum 
    $\|\cdot\|_2$-length of the newly added edges due to the bisection of $C$.
    It can be shown via an elementary argument in Euclidean geometry (e.g., by using the law of sines) that 
    $\max_{\BIw\in V(C)}\Big\{\big\|\BIw-{\textstyle\frac{\BIv+\BIv'}{2}}\big\|_2\Big\}\le \frac{\sqrt{3}}{2}\|\BIv-\BIv'\|_2$.
    One observes that the bisection of the simplex $C$ removes a longest edge with $\|\cdot\|_2$-length equal to $\|\BIv-\BIv'\|_2$ while adding finitely many new edges each with $\|\cdot\|_2$-length at most $\frac{\sqrt{3}}{2}\|\BIv-\BIv'\|_2$.
    Since the process of constructing $(\FC_r)_{r\in\N_0}$ begins with finitely many edges in $\FC_0$, we conclude that 
    $\lim_{r\to\infty}\max_{C\in\FC_r}\max_{\BIv,\BIv'\in V(C)}\big\{\|\BIv-\BIv'\|\big\}=\lim_{r\to\infty}\max_{C\in\FC_r}\max_{\BIv,\BIv'\in V(C)}\big\{\|\BIv-\BIv'\|_2\big\}=\nobreak0$.
    The proof of statement~\ref{props:momentset-simplex-bisection} is now complete.

    To prove statement~\ref{props:momentset-simplex-uniquerepresentation},
    we let $\big(\bar{\lambda}_{F,\BIw}:\aff(F)\to\R\big)_{\BIw\in V(F)}$ be given by Lemma~\ref{lem:simplices}\ref{lems:simplices-funcs} for each $F\in\FF(\FC)$, and define
    \begin{align}
        \lambda_{F,\BIw}(\BIx):=\begin{cases}
            \bar{\lambda}_{F,\BIw}(\BIx) & \forall \BIx\in F, \\
            0 & \forall \BIx \in \big(\bigcup_{C\in\FC}C\big)\setminus F, 
        \end{cases}\qquad \forall \BIw\in V(F),\; \forall F\in\FF(\FC).
        \label{eqn:momentset-simplex-proof-conversion}
    \end{align}
    One then checks that $\big(\lambda_{F,\BIw}:\bigcup_{C\in\FC}C\to\R\big)_{\BIw\in V(F),\,F\in\FF(\FC)}$ satisfy all the conditions in statement~\ref{props:momentset-simplex-uniquerepresentation}.

    To prove statement~\ref{props:momentset-simplex-continuity}, let us fix an arbitrary $\BIv\in V(\FC)$.
    It follows from statement~\ref{props:momentset-simplex-uniquerepresentation} and the definition of $g_{\BIv}$ in (\ref{eqn:momentset-simplex-testfuncs}) that $g_{\BIv}$ is non-negative. 
    To show the continuity of $g_{\BIv}$, let us first fix an arbitrary $F\in\FF(\FC)$ and let $\big(\bar{\lambda}_{F,\BIw}:\aff(F)\to\R\big)_{\BIw\in V(F)}$ be given by Lemma~\ref{lem:simplices}\ref{lems:simplices-funcs} with $C\leftarrow F$.
    Moreover, for each $F'\in\FF(F)$, let $\big(\bar{\lambda}_{F',\BIw}:\aff(F')\to\R\big)_{\BIw\in V(F')}$ be given by Lemma~\ref{lem:simplices}\ref{lems:simplices-funcs} with $C\leftarrow F'$.
    Subsequently, if $\BIv\notin V(F)$, then (\ref{eqn:momentset-simplex-testfuncs}) implies that $g_{\BIv}(\BIx)=\nobreak0$ $\forall\BIx\in \relint(F')$.
    If $\BIv\in V(F)\setminus V(F')$, then (\ref{eqn:momentset-simplex-testfuncs}), 
    (\ref{eqn:momentset-simplex-proof-conversion}), and
    Lemma~\ref{lem:simplices}\ref{lems:simplices-facefuncs} yield
    $g_{\BIv}(\BIx)=\nobreak0=\bar{\lambda}_{F,\BIv}(\BIx)$ $\forall\BIx\in \relint(F')$.
    If $\BIv\in V(F')$, then (\ref{eqn:momentset-simplex-testfuncs}), 
    (\ref{eqn:momentset-simplex-proof-conversion}), and
    Lemma~\ref{lem:simplices}\ref{lems:simplices-facefuncs}
    guarantee that
    $g_{\BIv}(\BIx)=\lambda_{F',\BIv}(\BIx)=\bar{\lambda}_{F',\BIv}(\BIx)=\bar{\lambda}_{F,\BIv}(\BIx)$ $\forall\BIx\in\relint(F')$.
    Since \citep[Theorem~18.2]{rockafellar1970convex} states that 
    $\big\{\relint(F'):F'\in\FF(F)\big\}$ are pair-wise disjoint and 
    $\bigcup_{F'\in\FF(F)}\relint(F')=F$, we can combine the three cases above to obtain
    \begin{align}
        \begin{split}
            \BIv\in V(F) \quad &\Rightarrow \quad g_{\BIv}(\BIx)=\bar{\lambda}_{F,\BIv}(\BIx) \qquad \forall \BIx\in F,\\
            \BIv\in V(\FC)\setminus V(F) \quad &\Rightarrow \quad g_{\BIv}(\BIx)=0 \hspace{30.2pt} \qquad \forall \BIx\in F,
        \end{split}
        \label{eqn:momentset-simplex-proof-piecewise-continuity}
    \end{align}
    where $\bar{\lambda}_{F,\BIv}$ is continuous on $F$ by Lemma~\ref{lem:simplices}\ref{lems:simplices-funcs}.
    In particular, (\ref{eqn:momentset-simplex-proof-piecewise-continuity}) shows that $g_{\BIv}$ is affine on each closed set $C\in\FC$,
    and hence $g_{\BIv}$ is also continuous on $\bigcup_{C\in\FC}C$.
    Furthermore, Lemma~\ref{lem:relint-partition},
    (\ref{eqn:momentset-simplex-proof-piecewise-continuity}), 
    and 
    the properties of the functions 
    $\big(\bar{\lambda}_{F,\BIw}:\aff(F)\to\R\big)_{\BIw\in V(F)}$ in Lemma~\ref{lem:simplices}\ref{lems:simplices-funcs} show that
    $\sum_{\BIv\in V(\FC)}g_{\BIv}(\BIx)=\sum_{\BIv\in V(F)}\bar{\lambda}_{F,\BIv}(\BIx)=\nobreak1$ $\forall\BIx\in F$,
    and hence 
    $\sum_{\BIv\in V(\FC)}g_{\BIv}(\BIx)=\nobreak1$ holds for all $\BIx\in\bigcup_{C\in\FC}C$.
    This completes the proof of statement~\ref{props:momentset-simplex-continuity}.

    Finally, to prove statement~\ref{props:momentset-simplex-upperbound},
    observe that $\FC$ is a bounded polyhedral cover since it satisfies the required properties in Definition~\ref{def:polycover}.
    Moreover, observe that statement~\ref{props:momentset-simplex-continuity}, 
    Lemma~\ref{lem:relint-partition},
    (\ref{eqn:momentset-simplex-proof-piecewise-continuity}), and the properties of the functions 
    $\big(\bar{\lambda}_{F,\BIw}:\aff(F)\to\R\big)_{\BIw\in V(F),\,F\in\FF(\FC)}$ in Lemma~\ref{lem:simplices}\ref{lems:simplices-funcs} guarantee that the continuous and non-negative functions in $\CG_{\mathsf{simp}}(\FC)$ satisfy the properties \ref{defs:cover-vif-normalize} and \ref{defs:cover-vif-disjoint} in Definition~\ref{def:cover-interp}.
    Consequently, applying Theorem~\ref{thm:momentset-euclidean}\ref{thms:momentset-euclidean-bounded} proves that 
    $W_1(\mu,\nu)\le 2\eta(\FC)$
    whenever $\mu,\nu\in\CP(\CY)$ satisfy
    $\mu\overset{\CG_{\mathsf{simp}}(\FC)}{\scalebox{3.5}[1]{$\sim$}}\nu$.
    The proof is now complete.
\end{proof}

\begin{proof}[Proof of Proposition~\ref{prop:momentset-errorcontrol}]
    For $i=1,\ldots,N$, 
    let us construct $\FC_i$ and $\widetilde{\CG}_i:=\CG_{\mathsf{simp}}(\FC_i)$ by Corollary~\ref{cor:momentset-scalability} with respect to 
    $p\leftarrow1$,
    $d\leftarrow d_i$,
    $\CY\leftarrow\CX_i$,
    $(\underline{M}_i,\overline{M}_i)_{i=1:d}\leftarrow(\underline{M}_{i,j},\overline{M}_{i,j})_{j=1:d_i}$,
    $\epsilon\leftarrow\frac{\epsilon-\epsilon_{\mathsf{LSIP}}}{NL_f}$,
    and 
    $C_{\|\cdot\|}\leftarrow C_{i,\|\cdot\|}$.
    Corollary~\ref{cor:momentset-scalability} guarantees that 
    $|\widetilde{\CG}_i|=\prod_{j=1}^{d_i}\bigg(1+\bigg\lceil\frac{2NL_f(\overline{M}_{i,j}-\underline{M}_{i,j})C_{i,\|\cdot\|}d_i}{\epsilon-\epsilon_{\mathsf{LSIP}}}\bigg\rceil\bigg)$ 
    and 
    $\specialoverline{W}_{1}(\mu_i,[\mu_i]_{\widetilde{\CG}_i})\le \sup_{\nu_i\overset{\widetilde{\CG}_i}{\sim}\nu'_i}\big\{W_1(\nu_i,\nu'_i)\big\} \le \frac{\epsilon-\epsilon_{\mathsf{LSIP}}}{NL_f}$.
    Moreover, observe that
    $\lspan_1(\widetilde{\CG}_i)=\lspan_1(\CG_i)$ for $i=1,\ldots,N$ and
    $(\mu_1,\ldots,\mu_N,f)\in\CA(L_f,\BCX,0,0,\ldots,0,0)$.
    Consequently, if
    $\hat{\mu}\in\Gamma\big([\mu_1]_{\CG_1},\ldots,[\mu_N]_{\CG_N}\big)$
    is an $\epsilon_{\mathsf{LSIP}}$-optimal solution of \eqref{eqn:mmotlb},
    then
    Theorem~\ref{thm:lowerbound}\ref{thms:lowerbound-epsilonoptimal} implies that every $\tilde{\mu}\in R(\hat{\mu};\mu_1,\ldots,\mu_N)$ is an $\epsilon_{\mathsf{apx}}$-optimal solution of \eqref{eqn:mmot}, 
    where $\epsilon_{\mathsf{apx}}:=\epsilon_{\mathsf{LSIP}}+\sum_{i=1}^NL_f\specialoverline{W}_{1}(\mu_i,[\mu_i]_{\CG_i})=\epsilon_{\mathsf{LSIP}}+\sum_{i=1}^NL_f\specialoverline{W}_{1}(\mu_i,[\mu_i]_{\widetilde{\CG}_i}) \le \epsilon$.
    The proof is now complete.
\end{proof}

The proof of Proposition~\ref{prop:mmot-complexity-constant} utilizes the following properties of $(g_{i,j})_{j=0:m_i,\,i=1:N}$ under Setting~\ref{sett:compact},
which directly follow from Proposition~\ref{prop:momentset-simplex} and its proof.
These properties will also be referred to later in the proof of Proposition~\ref{prop:boundedness-mmotlb-dual-lsip}.
\begin{enumerate}[label=\normalfont(P\arabic*),beginpenalty=10000]
    \item\label{pp:mmot-complexity-property-normalization}%
    For $i=1,\ldots,N$, $\sum_{j=0}^{m_i}g_{i,j}(\BIx)=1$ $\forall\BIx\in\CX_i$.
    
    \item\label{pp:mmot-complexity-property-orthonormality}%
    For $i=1,\ldots,N$,
    it holds that $g_{i,j}(\BIv_{i,j'})=\INDI_{\{j=j'\}}$ $\forall 0\le j\le m_i$, $\forall 0\le j'\le m_i$.
    In particular, it holds for $i=1,\ldots,N$ and for any $\BIy_i=(y_{i,1},\ldots,y_{i,m_i})^\TRANSP\in\R^{m_i}$ that
    $\langle\BIg_i(\BIv_{i,0}),\BIy_i\rangle=0$ and
    $\langle\BIg_i(\BIv_{i,j}),\BIy_i\rangle=y_{i,j}$ $\forall 1\le j\le m_i$.
    
    \item\label{pp:mmot-complexity-property-pwa}%
    For $i=1,\ldots,N$, for any $F\in\FF(\FC_i)$, and for any $\BIy_i\in\R^{m_i}$, 
    $F\ni \BIx\mapsto \langle\BIg_i(\BIx),\BIy_i\rangle\in\R$ is affine.
    
    \item\label{pp:mmot-complexity-property-disjoint}%
    For $i=1,\ldots,N$, $j=0,1,\ldots,m_i$, and for any $F\in\FF(\FC_i)$,
    it holds that 
    \begin{align*}
        \BIv_{i,j}\notin V(F)\quad&\Rightarrow\quad g_{i,j}(\BIx)=0 \qquad\forall \BIx\in F,\\
        \BIv_{i,j}\in V(F)\quad &\Rightarrow\quad g_{i,j}(\BIx)>0 \qquad\forall \BIx\in \relint(F).
    \end{align*}
\end{enumerate}
Our strategy to prove Proposition~\ref{prop:mmot-complexity-constant} is to begin with an arbitrary optimal solution $(y_0^\star,\BIy^\star)$ of \eqref{eqn:mmotlb-dual-lsip} and apply a sequence of operations to modify it, 
where each operation preserves the feasibility and optimality of 
$(y_0^\star,\BIy^\star)$.
After the operations, $(y_0^\star,\BIy^\star)$ is guaranteed to satisfy $\big\|(y_0^\star,\BIy^{\star\TRANSP})^\TRANSP\big\|_2\le 2L_f D(\BCX)(m+1)$.
The following lemma presents the operations that will be applied to $(y_0^\star,\BIy^\star)$ in the proof of Proposition~\ref{prop:mmot-complexity-constant}, and the properties that are preserved under these operations.

\begin{lemma}\label{lem:mmot-complexity-optimizers-bound}
Let the assumptions of Proposition~\ref{prop:mmot-complexity-constant} hold
and 
let $y_0\in\R$, $\BIy=(y_{1,1},\ldots,y_{1,m_1},\ldots,\allowbreak y_{N,1},\ldots,\allowbreak y_{N,m_N})^\TRANSP\in\R^m$ be feasible for \eqref{eqn:mmotlb-dual-lsip}.
Then, if we apply either of the two operations below to modify $(y_0,\BIy)$, the resulting $(y_0,\BIy)$ will remain feasible for \eqref{eqn:mmotlb-dual-lsip} and its objective value will not decrease.
\widowpenalty-1000
\begin{enumerate}[label=(\alph*),beginpenalty=10000]

\item\label{lems:mmot-complexity-optimizers-bound-y-above}%
If $\gamma_{\hat{i}}:=\max_{1\le j\le m_{\hat{i}}}\{y_{\hat{i},j}\}-2L_f D(\BCX) >0$ for some $\hat{i}\in\{1,\ldots,N\}$, then update $y_{0}\leftarrow y_{0}+\gamma_{\hat{i}}$, $y_{\hat{i},j}\leftarrow y_{\hat{i},j}-\gamma_{\hat{i}}$ for $j=1,\ldots,m_{\hat{i}}$, 
and leave the remaining terms unmodified.

\item\label{lems:mmot-complexity-optimizers-bound-y-below}%
If $y_{\hat{i},\hat{j}}<-2L_f D(\BCX)$ for some $\hat{i}\in\{1,\ldots,N\}$, $\hat{j}\in\{1,\ldots,m_{\hat{i}}\}$, then update $y_{\hat{i},\hat{j}}\leftarrow -2L_f D(\BCX)$,
and leave the remaining terms unmodified.
\end{enumerate}
\end{lemma}

\begin{proof}[Proof of Lemma~\ref{lem:mmot-complexity-optimizers-bound}]
In this proof, 
we let $\hat{y}_0$, $\hat{\BIy}_1=(\hat{y}_{1,1},\ldots,\hat{y}_{1,m_1})^\TRANSP,\ldots,\hat{\BIy}_N=(\hat{y}_{N,1},\ldots,\hat{y}_{N,m_N})^\TRANSP$, 
$\hat{\BIy}=(\hat{\BIy}_1^\TRANSP,\ldots,\hat{\BIy}_N^\TRANSP)^\TRANSP$
denote the values of 
$y_0$, $\BIy_1:=(y_{1,1},\ldots,y_{1,m_1})^\TRANSP,\ldots,\BIy_N:=(y_{N,1},\ldots,y_{N,m_N})^\TRANSP$, 
$\BIy=(\BIy_1^\TRANSP,\ldots,\BIy_N^\TRANSP)^\TRANSP$ after modification. 
Moreover, let us denote $J:=L_f D(\BCX)$ for notational simplicity.
Now, by the assumption that $\max_{\BIx\in\BCX}\big\{f(\BCX)\big\}=0$, the assumptions \ref{settc:compact-measures} and \ref{settc:compact-cost} in Setting~\ref{sett:compact},
and the feasibility of $(y_0^\star,\BIy^\star)$ for \eqref{eqn:mmotlb-dual-lsip},
we have 
\begin{align}
    \min_{\BIx\in\BCX}\big\{f(\BCX)\big\}\ge -J
    \label{eqn:mmot-complexity-optimizers-bound-proof-ineq1}
\end{align}
as well as
\begin{align}
    y_0+\sum_{i=1}^{N}\langle\BIg_i(\BIx_i),\BIy_i\rangle\le f(\BIx_1,\ldots,\BIx_N)\le 0 \qquad \forall(\BIx_1,\ldots,\BIx_N)\in\BCX.
    \label{eqn:mmot-complexity-optimizers-bound-proof-ineq2}
\end{align}

Let us first show that $(\hat{y}_0,\hat{\BIy})$ remains feasible for \eqref{eqn:mmotlb-dual-lsip} after operation~\ref{lems:mmot-complexity-optimizers-bound-y-above}.
Suppose without loss of generality that $\gamma_{1}:=\max_{1\le j\le m_{1}}\{y_{1,j}\}-2J >0$, and we apply operation~\ref{lems:mmot-complexity-optimizers-bound-y-above} with $\hat{i}\leftarrow 1$,
which results in
$\hat{y}_0=y_0+\gamma_1$,
$\hat{\BIy}_1=\BIy_1-\gamma_1\vecone_{m_1}$,
as well as 
$\hat{\BIy}_i=\BIy_i$ for $i=2,\ldots,N$.
For any $\BIx=(\BIx_1,\ldots,\BIx_N)\in\BCX$, it holds by the property~\ref{pp:mmot-complexity-property-normalization} that $1-\big(\sum_{j=1}^{m_{1}}g_{1,j}(\BIx_{1})\big)=g_{1,0}(\BIx_{1})$, and hence
\begin{align}
\begin{split}
\hat{y}_{0}+\sum_{i=1}^N\langle\BIg_i(\BIx_i),\hat{\BIy}_i\rangle&=y_{0}+\gamma_{1}+\Bigg(\sum_{j=1}^{m_{1}}(y_{1,j}-\gamma_{1})g_{1,j}(\BIx_{1})\Bigg)+
\Bigg(\sum_{i=2}^N\langle\BIg_i(\BIx_i),\BIy_i\rangle\Bigg)\\
&=\Bigg(y_0+\sum_{i=1}^N\langle\BIg_i(\BIx_i),\BIy_i\rangle\Bigg)+\gamma_{1}g_{1,0}(\BIx_{1}) \qquad \forall (\BIx_1,\ldots,\BIx_N)\in\BCX.
\end{split}
\label{eqn:mmot-complexity-optimizers-bound-proof-shifted}
\end{align}
Suppose for the sake of contradiction that there exists $(\bar{\BIx}_1,\ldots,\bar{\BIx}_N)\in\BCX$ with
$\hat{y}_0+ \sum_{i=1}^N \langle\BIg_i(\bar{\BIx}_i),\hat{\BIy}_i\rangle >f(\bar{\BIx}_1,\ldots,\bar{\BIx}_N)$.
It then follows from (\ref{eqn:mmot-complexity-optimizers-bound-proof-shifted})
and (\ref{eqn:mmot-complexity-optimizers-bound-proof-ineq2})
that $g_{1,0}(\bar{\BIx}_1)>0$.
Since Lemma~\ref{lem:relint-partition} implies that $\CX_{1}=\bigcup_{C\in\FC_{1}}C=\bigcup_{F\in\FF(\FC_1)}\relint(F)$, 
let us fix an $\overbar{F}\in\FF(\FC_{1})$ such that $\bar{\BIx}_{1}\in\relint(\overbar{F})$.
Then, $\overbar{F}$ is a $\bar{d}_1$-simplex for some $0\le \bar{d}_1\le d_1$.
Moreover, the property~\ref{pp:mmot-complexity-property-disjoint} guarantees that
$\BIv_{1,0}\in V(F)$.
Subsequently, let $\{\BIw_0,\BIw_1,\ldots,\BIw_{\bar{d}_{1}}\}$ be an enumeration of $V(\overbar{F})$ where $\BIw_0=\BIv_{1,0}$,
and express 
$\bar{\BIx}_{1}=\sum_{l=0}^{\bar{d}_{1}}\lambda_l\BIw_l$ where $\lambda_l>\nobreak 0$ $\forall 0\le l\le \bar{d}_{1}$ and $\sum_{l=0}^{\bar{d}_{1}}\lambda_l=1$;
see Lemma~\ref{lem:simplices}\ref{lems:simplices-uniqueness} and Lemma~\ref{lem:simplices}\ref{lems:simplices-nonneg}.
Notice that the property~\ref{pp:mmot-complexity-property-orthonormality} and (\ref{eqn:mmot-complexity-optimizers-bound-proof-ineq2}) imply that
\begin{align}
    \begin{split}
        y_{0}+y_{1,j}+\sum_{i=2}^N\langle\BIg_i(\bar{\BIx}_i),\hat{\BIy}_i\rangle &=y_{0}+\langle\BIg_1(\BIv_{1,j}),\BIy_1\rangle+\sum_{i=2}^N\langle\BIg_i(\bar{\BIx}_i),\BIy_i\rangle\le 0 \qquad \forall 1\le j\le m_1.
    \end{split}
\label{eqn:mmot-complexity-optimizers-bound-proof-nonpos-a}
\end{align}
It thus follows from 
the property~\ref{pp:mmot-complexity-property-orthonormality} and
(\ref{eqn:mmot-complexity-optimizers-bound-proof-ineq1}) that
\begin{align*}
\begin{split}
\hat{y}_{0}+\langle\BIg_1(\BIv_{1,0}),\hat{\BIy}_1\rangle+\sum_{i=2}^N\langle\BIg_i(\bar{\BIx}_i),\hat{\BIy}_i\rangle&=\Bigg(y_{0}+\max_{1\le j\le m_{1}}\{y_{1,j}\}+\sum_{i=2}^N\langle\BIg_i(\bar{\BIx}_i),\hat{\BIy}_i\rangle\Bigg)-2J\\
&\le -2J < f(\BIv_{1,0},\bar{\BIx}_2,\ldots,\bar{\BIx}_N),
\end{split}
\end{align*}
which shows that $\bar{\BIx}_{1}\ne\BIv_{1,0}=\BIw_0$ and that $\lambda_0\ne 1$. 
Therefore, we get $\lambda_0\in(0,1)$. 
Now, let us define $\bar{\BIx}^{\dagger}_{1}:=\frac{1}{1-\lambda_0}\sum_{l=1}^{\bar{d}_{1}}\lambda_l\BIw_l\in \conv\big(\{\BIw_1,\ldots,\BIw_{\bar{d}_{1}}\}\big)\in \FF(\overbar{F})$, 
and define 
$\psi_{a}(\nu):=\big\langle\BIg_1\big((1-\nu)\bar{\BIx}^\dagger_1+\nu\BIv_{1,0}\big),\hat{\BIy}_1\big\rangle$ for $\nu\in[0,1]$, which is an affine function on $[0,1]$ due to the property~\ref{pp:mmot-complexity-property-pwa}. 
Notice that $g_{1,0}(\bar{\BIx}^\dagger_{1})=0$ due to $\bar{\BIx}^{\dagger}_{1}\in \conv\big(\{\BIw_1,\ldots,\BIw_{\bar{d}_{1}}\}\big)\in\FF(\FC_1)$
and the property~\ref{pp:mmot-complexity-property-disjoint}.
Consequently, 
(\ref{eqn:mmot-complexity-optimizers-bound-proof-ineq2}),
(\ref{eqn:mmot-complexity-optimizers-bound-proof-shifted}), and the property \ref{pp:mmot-complexity-property-orthonormality} imply that
\begin{align*}
\psi_{a}(0)
&=\langle\BIg_1(\bar{\BIx}^\dagger_1),\hat{\BIy}_1\rangle\le f(\bar{\BIx}^\dagger_{1},\bar{\BIx}_2,\ldots,\bar{\BIx}_N)-\Bigg(\hat{y}_0+\sum_{i=2}^N \langle\BIg_i(\bar{\BIx}_i),\hat{\BIy}_i\rangle\Bigg),\\
\psi_{a}(\lambda_0)
&=\langle\BIg_1(\bar{\BIx}_1),\hat{\BIy}_1\rangle>f(\bar{\BIx}_{1},\bar{\BIx}_2,\ldots,\bar{\BIx}_N)-\Bigg(\hat{y}_0+\sum_{i=2}^N \langle\BIg_i(\bar{\BIx}_i),\hat{\BIy}_i\rangle\Bigg),\\
\psi_{a}(1)
&=\langle\BIg_1(\BIv_{1,0}),\hat{\BIy}_1\rangle=0.
\end{align*}
Note that we are using the crucial assumption $\bigcup_{C\in\FC_1}C=\CX_1$ here to guarantee that $\psi_{a}\big([0,1]\big)\subseteq \CX_1$.
In the following, we will show that $\psi_{a}$ is not affine on $[0,1]$, which leads to a contradiction.
On the one hand, we have by the assumptions \ref{settc:compact-measures} and \ref{settc:compact-cost} that
\begin{align*}
\frac{\psi_{a}(\lambda_0)-\psi_{a}(0)}{\lambda_0-0}&>\frac{1}{\lambda_0}\big(f(\bar{\BIx}_{1},\bar{\BIx}_2,\ldots,\bar{\BIx}_N)-f(\bar{\BIx}^\dagger_{1},\bar{\BIx}_2,\ldots,\bar{\BIx}_N)\big)\ge -\frac{L_f}{\lambda_0}\|\bar{\BIx}_1-\bar{\BIx}^\dagger_1\|.
\end{align*}
Moreover, it holds that 
\begin{align*}
\bar{\BIx}_{1}-\bar{\BIx}^\dagger_{1}&=\Bigg(\sum_{l=0}^{\bar{d}_{1}}\lambda_l\BIw_l\Bigg)-\frac{1}{1-\lambda_0}\Bigg(\sum_{l=1}^{\bar{d}_{1}}\lambda_l\BIw_l\Bigg)=\lambda_0\BIw_0-\frac{\lambda_0}{1-\lambda_0}\Bigg(\sum_{l=1}^{\bar{d}_{1}}\lambda_l\BIw_l\Bigg)=\lambda_0(\BIv_{1,0}-\bar{\BIx}^\dagger_{1}).
\end{align*}
Combining these yields
\begin{align}
\begin{split}
\frac{\psi_{a}(\lambda_0)-\psi_{a}(0)}{\lambda_0-0}&> -L_{f}\|\BIv_{1,0}-\bar{\BIx}^\dagger_{1}\|\ge -L_f D(\BCX)= -J.
\end{split}
\label{eqn:mmot-complexity-constants-proof-slope-a1}
\end{align}
On the other hand, we have by (\ref{eqn:mmot-complexity-optimizers-bound-proof-nonpos-a}) and (\ref{eqn:mmot-complexity-optimizers-bound-proof-ineq1}) that
\begin{align}
\frac{\psi_{a}(1)-\psi_{a}(\lambda_0)}{1-\lambda_0}&< \frac{1}{1-\lambda_0} \Bigg[y_{0}+\max_{1\le j\le m_{1}}\{y_{1,j}\}-2J+\Bigg(\sum_{i=2}^N \langle\BIg_i(\bar{\BIx}_i),\hat{\BIy}_i\rangle\Bigg)-f(\bar{\BIx}_1,\ldots,\bar{\BIx}_N)\Bigg]\nonumber\\
&\le \frac{1}{1-\lambda_0} \big(-2J-f(\bar{\BIx}_1,\ldots,\bar{\BIx}_N)\big)< -J.
\label{eqn:mmot-complexity-constants-proof-slope-a2}
\end{align}
Combining (\ref{eqn:mmot-complexity-constants-proof-slope-a1}) and (\ref{eqn:mmot-complexity-constants-proof-slope-a2}) 
leads to $\frac{\psi_{a}(\lambda_0)-\psi_{a}(0)}{\lambda_0-0}>\frac{\psi_{a}(1)-\psi_{a}(\lambda_0)}{1-\lambda_0}$, which
contradicts the affine property of $\psi_{a}(\,\cdot\,)$ on $[0,1]$. 
Therefore, we conclude that $(\hat{y}_0,\hat{\BIy})$ is feasible for \eqref{eqn:mmotlb-dual-lsip} after operation~\ref{lems:mmot-complexity-optimizers-bound-y-above}.
Furthermore, it follows from (\ref{eqn:mmot-complexity-optimizers-bound-proof-shifted}) that $\hat{y}_0+\langle\bar{\BIg},\hat{\BIy}\rangle=y_0+\langle\bar{\BIg},\BIy\rangle+\gamma_1\int_{\CX_1}g_{1,0}\DIFFX{\mu_1} \ge y_0+\langle\bar{\BIg},\BIy\rangle$,
and thus operation~\ref{lems:mmot-complexity-optimizers-bound-y-above} does not decrease the objective value of $(y_0,\BIy)$ for \eqref{eqn:mmotlb-dual-lsip}.

Next, let us show that $(\hat{y}_0,\hat{\BIy})$ remains feasible for \eqref{eqn:mmotlb-dual-lsip} after operation~\ref{lems:mmot-complexity-optimizers-bound-y-below}.
Suppose without loss of generality that $y_{1,1}<-2J$,
and we apply operation~\ref{lems:mmot-complexity-optimizers-bound-y-below} with $\hat{i}\leftarrow 1$, $\hat{j}\leftarrow 1$,
which results in 
$\hat{y}_0=y_0$,
$\hat{y}_{1,1}=-2J$,
$\hat{y}_{1,j}=y_{1,j}$ for $j=2,\ldots,m_1$,
as well as
$\hat{\BIy}_i=\BIy_i$ for $i=2,\ldots,N$.
Then, (\ref{eqn:mmot-complexity-optimizers-bound-proof-ineq2}) and the property \ref{pp:mmot-complexity-property-orthonormality} imply that
\begin{align}
\begin{split}
\hat{y}_{0}+\sum_{i=2}^N\langle\BIg_i(\BIx_i),\hat{\BIy}_i\rangle&=y_{0}+\langle\BIg_1(\BIv_{1,0}),\BIy_1\rangle+\sum_{i=2}^N\langle\BIg_i(\BIx_i),\BIy_i\rangle\le 0 \qquad \forall (\BIx_1,\ldots,\BIx_N)\in\BCX.
\end{split}
\label{eqn:mmot-complexity-optimizers-bound-proof-nonpos-b}
\end{align}
Suppose for the sake of contradiction that there exists $(\bar{\BIx}_1,\ldots,\bar{\BIx}_N)\in\BCX$ with $\hat{y}_{0}+\sum_{i=1}^N \langle\BIg_i(\bar{\BIx}_i),\hat{\BIy}_i\rangle >f(\bar{\BIx}_1,\ldots,\bar{\BIx}_N)$.
It then follows from (\ref{eqn:mmot-complexity-optimizers-bound-proof-ineq2}) that $g_{1,1}(\bar{\BIx}_1)>\nobreak0$.
Using the same argument in the proof of the properties of operation~\ref{lems:mmot-complexity-optimizers-bound-y-above}, let us fix an $\overbar{F}\in\FC_{1}$ such that $\bar{\BIx}_{1}\in\relint(\overbar{F})$,
enumerate $V(\overbar{F})=\{\BIw_0,\BIw_1,\ldots,\BIw_{\bar{d}_{1}}\}$ where $0\le \bar{d}_1\le d_1$, $\BIw_0=\BIv_{1,1}$,
and express 
$\bar{\BIx}_{1}=\sum_{l=0}^{\bar{d}_{1}}\lambda_l\BIw_l$ where $\lambda_l>\nobreak 0$ $\forall 0\le l\le \bar{d}_1$, 
$\sum_{l=0}^{\bar{d}_{1}}\lambda_l=1$.
It holds by the property \ref{pp:mmot-complexity-property-orthonormality}, (\ref{eqn:mmot-complexity-optimizers-bound-proof-nonpos-b}), and 
(\ref{eqn:mmot-complexity-optimizers-bound-proof-ineq1}) that 
\begin{align*}
\hat{y}_{0}+\langle\BIg_1(\BIv_{1,1}),\hat{\BIy}_{1}\rangle+\sum_{i=2}^N\langle\BIg_i(\bar{\BIx}_i),\hat{\BIy}_i\rangle\le -2J<f(\BIv_{1,1},\bar{\BIx}_2,\ldots,\bar{\BIx}_N),
\end{align*}
which shows that $\bar{\BIx}_{1}\ne\BIv_{1,1}=\BIw_0$ and thus $\lambda_0\ne 1$. 
Therefore, we get $\lambda_0\in(0,1)$. 
Now, let us define $\bar{\BIx}^{\dagger}_{1}:=\frac{1}{1-\lambda_0}\sum_{l=1}^{\bar{d}_{1}}\lambda_l\BIw_l\in \conv\big(\{\BIw_1,\ldots,\BIw_{\bar{d}_{1}}\}\big)\in \FF(\overbar{F})$, and define 
$\psi_{b}(\nu):=\big\langle\BIg_1\big((1-\nu)\bar{\BIx}^{\dagger}_{1}+\nu \BIv_{1,1}\big),\hat{\BIy}_1\big\rangle$ for $\nu\in[0,1]$, 
which is an affine function on $[0,1]$ due to the property~\ref{pp:mmot-complexity-property-pwa}. 
Notice that the property~\ref{pp:mmot-complexity-property-disjoint} implies that 
$g_{1,1}(\bar{\BIx}^\dagger_1)=\nobreak0$ and hence
$\langle\BIg_1(\bar{\BIx}^\dagger_1),\hat{\BIy}_1\rangle=\langle\BIg_1(\bar{\BIx}^\dagger_1),\BIy_1\rangle$.
Consequently, we get from (\ref{eqn:mmot-complexity-optimizers-bound-proof-ineq2}) and the property~\ref{pp:mmot-complexity-property-orthonormality} that
\begin{align*}
\psi_{b}(0)&=\langle\BIg_1(\bar{\BIx}^\dagger_{1}),\hat{\BIy}_1\rangle=\langle\BIg_1(\bar{\BIx}^\dagger_{1}),\BIy_1\rangle\le f(\bar{\BIx}^\dagger_{1},\bar{\BIx}_2,\ldots,\bar{\BIx}_N)-\Bigg(y_{0}+\sum_{i=2}^N\langle\BIg_i(\bar{\BIx}_i),\BIy_i\rangle\Bigg),\\
\psi_{b}(\lambda_0)&=\langle\BIg_1(\bar{\BIx}_{1}),\hat{\BIy}_1\rangle>f(\bar{\BIx}_{1},\bar{\BIx}_2,\ldots,\bar{\BIx}_N)-\Bigg(\hat{y}_{0}+\sum_{i=2}^N\langle\BIg_i(\bar{\BIx}_i),\hat{\BIy}_i\rangle\Bigg),\\
\psi_{b}(1)&=\langle\BIg_1(\BIv_{1,1}),\hat{\BIy}_1\rangle=\hat{y}_{1,1}=-2J.
\end{align*}
Here, we are again using the crucial assumption $\bigcup_{C\in\FC_1}C=\CX_1$ to guarantee that $\psi_{b}\big([0,1]\big)\subseteq\CX_1$.
In the following, we will show that $\psi_b$ is not affine on $[0,1]$ to get a contradiction.
On the one hand, since $\hat{y}_0=y_0$ and $\hat{\BIy}_i=\BIy_i$ for $i=2,\ldots,N$, we have by the assumptions \ref{settc:compact-measures} and \ref{settc:compact-cost} that
\begin{align*}
\begin{split}
\frac{\psi_{b}(\lambda_0)-\psi_{b}(0)}{\lambda_0-0}&>\frac{1}{\lambda_0}\big(f(\bar{\BIx}_{1},\bar{\BIx}_2,\ldots,\bar{\BIx}_N)-f(\bar{\BIx}^\dagger_{1},\bar{\BIx}_2,\ldots,\bar{\BIx}_N)\big)\ge -\frac{L_f}{\lambda_0}\|\bar{\BIx}_1-\bar{\BIx}^\dagger_1\|.
\end{split}
\end{align*}
Same as in the proof of the properties of operation~\ref{lems:mmot-complexity-optimizers-bound-y-above}, we have $\bar{\BIx}_{1}-\bar{\BIx}^\dagger_{1}=\lambda_0(\BIv_{1,1}-\bar{\BIx}^\dagger_{1})$, and we thus get
\begin{align}
\begin{split}
\frac{\psi_{b}(\lambda_0)-\psi_{b}(0)}{\lambda_0-0}&> -L_{f}\|\BIv_{1,1}-\bar{\BIx}^\dagger_{1}\|\ge -L_f D(\BCX)= -J.
\end{split}
\label{eqn:mmot-complexity-constants-proof-slope-b1}
\end{align}
On the other hand, we have by (\ref{eqn:mmot-complexity-optimizers-bound-proof-nonpos-b}) and (\ref{eqn:mmot-complexity-optimizers-bound-proof-ineq1}) that
\begin{align}
\begin{split}
\frac{\psi_{b}(1)-\psi_{b}(\lambda_0)}{1-\lambda_0}&< \frac{1}{1-\lambda_0}\Bigg[-2J-f(\bar{\BIx}_{1},\bar{\BIx}_2,\ldots,\bar{\BIx}_N)+\Bigg(\hat{y}_{0}+\sum_{i=2}^N\langle\BIg_i(\bar{\BIx}_i),\hat{\BIy}_i\rangle\Bigg)\Bigg] \\
&\le \frac{1}{1-\lambda_0} \big(-2J-f(\bar{\BIx}_{1},\bar{\BIx}_2,\ldots,\bar{\BIx}_N)\big)< -J.
\end{split}
\label{eqn:mmot-complexity-constants-proof-slope-b2}
\end{align}
Combining (\ref{eqn:mmot-complexity-constants-proof-slope-b1}) and (\ref{eqn:mmot-complexity-constants-proof-slope-b2})
yields $\frac{\psi_{b}(\lambda_0)-\psi_{b}(0)}{\lambda_0-0}>\frac{\psi_{b}(1)-\psi_{b}(\lambda_0)}{1-\lambda_0}$,
which contradicts the affine property of $\psi_b$ on $[0,1]$.
Therefore, we conclude that $(y_0,\BIy)$ remains feasible for \eqref{eqn:mmotlb-dual-lsip} after operation~\ref{lems:mmot-complexity-optimizers-bound-y-below}.
Lastly, since operation~\ref{lems:mmot-complexity-optimizers-bound-y-below} only modifies the value of $y_{1,1}<-2J$ to $\hat{y}_{1,1}=-2J$ and 
leaves the remaining terms unchanged,
it holds that 
$\hat{y}_0+\langle\bar{\BIg},\hat{\BIy}\rangle= y_0 +\langle\bar{\BIg},\BIy\rangle+(\hat{y}_{1,1}-y_{1,1})\int_{\CX_i}g_{1,1}\DIFFX{\mu_1}\ge y_0 +\langle\bar{\BIg},\BIy\rangle$.
The proof is now complete. 
\end{proof}

\begin{proof}[Proof of Proposition~\ref{prop:mmot-complexity-constant}] 
In this proof, let $J:=L_f D(\BCX)$ for notational simplicity. 
To begin, we observe that the condition~\ref{settc:algo-vertices} in Setting~\ref{sett:algo} is satisfied,
and it thus holds by Proposition~\ref{prop:boundedness-mmotlb-dual-lsip}\ref{props:boundedness-mmotlb-dual-lsip-boundedness} that 
\eqref{eqn:mmotlb-dual-lsip} admits an optimal solution.
Hence,
let us take an arbitrary optimal solution $(y^\star_0,\BIy^\star)$ of \eqref{eqn:mmotlb-dual-lsip} and denote $\BIy^\star=(y^\star_{1,1},\ldots,y^\star_{1,m_1},\ldots,y^\star_{N,1},\ldots,y^\star_{N,m_N})^\TRANSP$. 
Subsequently, we can apply a finite sequence of operations~\ref{lems:mmot-complexity-optimizers-bound-y-above} \ref{lems:mmot-complexity-optimizers-bound-y-below} in Lemma~\ref{lem:mmot-complexity-optimizers-bound} to modify $(y^\star_0,\BIy^\star)$ until $|y^\star_{i,j}|\le 2J$ holds for $j=1,\ldots,m_i$, $i=1,\ldots,N$.
The properties of these operations in Lemma~\ref{lem:mmot-complexity-optimizers-bound} guarantee that $(y^\star_0,\BIy^\star)$ remains an optimal solution of \eqref{eqn:mmotlb-dual-lsip} after these operations.
Now, it holds by the optimality of  $(y^\star_0,\BIy^\star)$ for \eqref{eqn:mmotlb-dual-lsip}, the property~\ref{pp:mmot-complexity-property-normalization}, and (\ref{eqn:mmot-complexity-optimizers-bound-proof-ineq1}) that
\begin{align*}
y^\star_{0}&=\inf_{(\BIx_1,\ldots,\BIx_N)\in\BCX}\left\{f(\BIx_1,\ldots,\BIx_N)-\left(\sum_{i=1}^N\sum_{j=1}^{m_i}y^{\star}_{i,j}g_{i,j}(\BIx_i)\right)\right\}\ge-(2N+1)J.
\end{align*}
On the other hand, it holds 
by the property~\ref{pp:mmot-complexity-property-orthonormality} that $y^\star_{0}=y^\star_{0}+\sum_{i=1}^N\sum_{j=1}^{m_i}y^{\star}_{i,j}g_{i,j}(\BIv_{i,0})\le f(\BIv_{1,0},\ldots,\BIv_{N,0})\le 0$.
Hence, we get
\begin{align*}
    \big\|(y^\star_0,\BIy^{\star\TRANSP})^\TRANSP\big\|_2^2\le (2N+1)^2J^2+\sum_{i=1}^N\sum_{j=1}^{m_i}(2J)^2=(4N^2+4N+1+4m)J^2< 4(m+1)^2J^2
\end{align*}
and thus $\big\|(y^\star_0,\BIy^{\star\TRANSP})^\TRANSP\big\|_2<2(m+1)J=2L_fD(\BCX)(m+1)$.
We have thus shown that \eqref{eqn:mmotlb-dual-lsip} admits an optimal solution 
$(y_0^\star,\BIy^\star)$ which satisfies 
$M_{\mathsf{opt}}:=\big\|(y^\star_0,\BIy^{\star\TRANSP})^\TRANSP\big\|_2<2L_fD(\BCX)(m+1)$.
This proves statement~\ref{props:mmot-complexity-constant-optimizer}.
Furthermore, 
notice that $\big\|\BIg_i(\BIx_i)\big\|_2\le 1$ $\forall \BIx_i\in\CX_i$, $\forall 1\le i\le N$ holds due to the property~\ref{pp:mmot-complexity-property-normalization},
and one verifies that all assumptions of Theorem~\ref{thm:mmot-complexity} are satisfied.
Consequently, it follows from Theorem~\ref{thm:mmot-complexity} that there exists an algorithm which,
for any $\epsilon_{\mathsf{LSIP}}>0$,
computes an $\epsilon_{\mathsf{LSIP}}$-optimal solution of \eqref{eqn:mmotlb-dual-lsip} with 
$O\big(m\log(mM_{\mathsf{opt}}/\epsilon_{\mathsf{LSIP}})\big)=O\big(m\log(L_fD(\BCX)m/\epsilon_{\mathsf{LSIP}})\big)$ calls to $\mathtt{Oracle}(\,\cdot\,)$ 
and $O\big(m^{\omega+1}\log(mM_{\mathsf{opt}}/\epsilon_{\mathsf{LSIP}})\big)=O\big(m^{\omega+1}\log(L_fD(\BCX)m/\epsilon_{\mathsf{LSIP}})\big)$ additional arithmetic operations.
The proof is now complete.
\end{proof}

\subsection{Proofs of results in Section~\ref{sec:numerics}}\label{ssec:proof-numerics}

% Proof of Proposition (boundedness of optimizers)
\begin{proof}[Proof of Proposition~\ref{prop:boundedness-mmotlb-dual-lsip}]
    Let us first prove statement~\ref{props:boundedness-mmotlb-dual-lsip-boundedness}.
    To begin, let us assume that \ref{settc:algo-fullsupport} holds and fix an arbitrary $i\in\{1,\ldots,N\}$.
    For any $C\in\FC_i$, observe from the property \ref{pp:mmot-complexity-property-pwa} that
    $\BIg_i(\,\cdot\,)$ is affine on $C$.
    Subsequently, \ref{settc:algo-fullsupport} implies that
    $\aff\big(\BIg_i(\CX_i\cap\hspace{1pt} C)\big)=\aff\big(\BIg_i(C)\big)$ $\forall C\in\FC_i$,
    and we hence get from the property \ref{pp:mmot-complexity-property-orthonormality} of $(g_{i,j})_{j=0:m_i,\,i=1:N}$ that
    \begin{align*}
        \aff\big(\{\BIg_i(\BIx):\BIx\in\CX_i\}\big) &= \aff\Bigg(\bigcup_{C\in\FC_i}\BIg_i(\CX_i\cap C)\Bigg)
         = \aff\Bigg(\bigcup_{C\in\FC_i}\aff\big(\BIg_i(\CX_i\cap C)\big)\Bigg)\\
        &= \aff\Bigg(\bigcup_{C\in\FC_i}\aff\big(\BIg_i(C)\big)\Bigg)
         = \aff\Big(\BIg_i\big({\textstyle\bigcup_{C\in\FC_i}}C\big)\Big)
         = \R^{m_i}.
    \end{align*}
    This shows that there exist $\BIx_{i,1},\ldots,\BIx_{i,m_i+1}\in\CX_i$ such that 
    $\BIg_i(\BIx_{i,1}),\ldots,\BIg_{i}(\BIx_{i,m_i+1})\in\R^{m_i}$ are affinely independent.
    We can conclude by Proposition~\ref{prop:duality-settings}\ref{props:duality-setting2} that the condition~\ref{thmc:duality1int} in Theorem~\ref{thm:duality} holds.

    In the case where \ref{settc:algo-vertices} holds, 
    it follows from 
    the properties 
    \ref{pp:mmot-complexity-property-normalization} and
    \ref{pp:mmot-complexity-property-orthonormality} of $(g_{i,j})_{j=0:m_i,\,i=1:N}$ that 
    $\int_{\CX_i}g_{i,j}\DIFFX{\mu_i}>\nobreak0$ $\forall 1\le j\le m_i$,  $\forall 1\le i\le N$,
    $\sum_{j=1}^{m_i}\int_{\CX_i}g_{i,j}\DIFFX{\mu_i} = 1 - \int_{\CX_i}g_{i,0}\DIFFX{\mu_i}<\nobreak1$, $\forall 1\le i\le N$, 
    as well as 
    \begin{align*}
        \conv\Big(\big\{\BIg_i(\BIx):\BIx\in\CX_i\big\}\Big)\supseteq \big\{(z_1,\ldots,z_{m_i})^\TRANSP\in\R^{m_i}_{+}:{\textstyle\sum_{j=1}^{m_i}}z_j\le 1\big\} \qquad \forall 1\le i\le N.
    \end{align*}
    We can thus conclude that
    \begin{align*}
        \bar{\BIg}&=(\bar{\BIg}_1^\TRANSP,\ldots,\bar{\BIg}_N^\TRANSP)^\TRANSP\in\bigtimes_{i=1}^{N}\inter\Big(\conv\Big(\big\{\BIg_i(\BIx):\BIx\in\CX_i\big\}\Big)\Big)=\inter\Big(\conv\Big(\big\{\BIg(\BIx):\BIx\in\BCX\big\}\Big)\Big),
    \end{align*}
    and that \ref{thmc:duality1int} holds.
    Now that we have shown that the condition \ref{thmc:duality1int} holds when either \ref{settc:algo-fullsupport} or \ref{settc:algo-vertices} holds,
    statement~\ref{props:boundedness-mmotlb-dual-lsip-boundedness} follows from Theorem~\ref{thm:duality}\ref{thms:duality2} 
    and the equivalence between (i) and (iii) in \citep[Corollary~9.3.1]{goberna1998linear}.

    Now, let us prove statement~\ref{props:boundedness-mmotlb-dual-lsip-explicit}.
    Observe from 
    Line~\ref{alglin:cpinit-minprob}, Line~\ref{alglin:cpinit-subtraction}, and Line~\ref{alglin:cpinit-zeroindices} of Algorithm~\ref{alg:cp-mmot-initialization}
    that the set $\CI$ is non-empty in each iteration of the while-loop in Line~\ref{alglin:cpinit-whileloop}, and that Algorithm~\ref{alg:cp-mmot-initialization} terminates after finitely many iterations.
    Thus, we let $T\in\N$ denote the total number of iterations of the while-loop before exiting, and
    we let $t\in\N$ denote the iteration counter for the while-loop throughout the proof of statement~\ref{props:boundedness-mmotlb-dual-lsip-explicit}. 
    Subsequently, 
    for $t=1,\ldots,T$,
    let 
    $\eta_{\min}^{(t)}$, 
    $\big(\BIx_{i}^{(t)}\big)_{i=1:N}$, 
    $\BIx^{(t)}$,
    $\CI^{(t)}$
    denote the values of 
    $\eta_{\min}$, 
    $(\BIx_{i})_{i=1:N}$, 
    $\BIx$,
    $\CI$
    in the $t$-th iteration of the while-loop, respectively, 
    let $\big(\eta_{i,j}^{(t)}\big)_{j=0:m_i,\,i=1:N}$ denote the values of $(\eta_{i,j})_{j=0:m_i,\,i=1:N}$ after the update in Line~\ref{alglin:cpinit-subtraction} in the $t$-th iteration of the while-loop, 
    and let $\big(r_i^{(t)}\big)_{i=1:N}$ and $\big(r_i^{(t+1)}\big)_{i=1:N}$ denote the values of $(r_i)_{i=1:N}$ before and after the update in Line~\ref{alglin:cpinit-advance} in the $t$-th iteration of the while-loop, respectively. 
    Moreover, let $\eta_{i,j}^{(0)}:=\int_{\CX_i}g_{i,j}\DIFFX{\mu_i}$ 
    $\forall 0\le\nobreak j\le\nobreak m_i$, $\forall 1\le\nobreak i\le\nobreak N$. 
    By Lines~\mbox{\ref{alglin:cpinit-subtraction}--\ref{alglin:cpinit-zeroindices}} and the property
    \ref{pp:mmot-complexity-property-orthonormality}
    of $(g_{i,j})_{j=0:m_i,\,i=1:N}$, 
    it holds that
    \begin{align*}
    \int_{\BCX}g_{i,j}\circ \pi_i\DIFFX{\hat{\mu}^{(0)}}&=\sum_{(\BIx,\eta)\in\CQ}\eta g_{i,j}\circ\pi_i(\BIx)\\
    &=\sum_{t=1}^{T}\eta_{\min}^{(t)} g_{i,j}\big(\BIv_{i,r_i^{(t)}}\big)=\sum_{t=1}^{T}\INDI_{\big\{r_i^{(t)}=j\big\}}\eta_{\min}^{(t)}
    \qquad \forall 0\le j\le m_i,\; \forall 1\le i\le N.
    \end{align*} 
    Thus, it suffices to show that $\sum_{t=1}^{T}\INDI_{\big\{r_i^{(t)}=j\big\}}\eta_{\min}^{(t)}=\int_{\CX_i}g_{i,j}\DIFFX{\mu_i}$ 
    $\forall 0\le j\le m_i$, $\forall 1\le i\le N$. 
    To that end, let us fix arbitrary $i\in\{1,\ldots,N\}$, $j\in\{0,1,\ldots,m_i\}$, 
    and let 
    \begin{align*}
        \underline{t}&:=\min\big\{t: 1\le t\le T,\;r_i^{(t)}=j\big\}, \qquad \overline{t}:=\max\big\{t: 1\le t\le T,\;r_i^{(t)}=j\big\}.
    \end{align*}
    Line~\ref{alglin:cpinit-advance} guarantees that 
    $r^{(t)}_i=j$ for $t=\underline{t},\ldots,\overline{t}$.
    For $t=1,\ldots,T$, 
    it follows from Line~\ref{alglin:cpinit-subtraction} that $\eta_{i,j}^{(t)}\ne \eta_{i,j}^{(t-1)}$ if and only if $r_i^{(t)}=j$. 
    Consequently, it holds that $\eta_{i,j}^{(\underline{t}-1)}=\eta_{i,j}^{(0)}$. 
    Moreover, Line~\ref{alglin:cpinit-advance} and Line~\ref{alglin:cpinit-zeroindices} imply that $\eta_{i,j}^{(\overline{t})}=0$.
    Therefore, we get
    \begin{align*}
        \sum_{t=1}^{T}\INDI_{\big\{r_i^{(t)}=j\big\}}\eta_{\min}^{(t)}&=\sum_{t=\underline{t}}^{\overline{t}}\eta_{\min}^{(t)}=\sum_{t=\underline{t}}^{\overline{t}}\eta_{i,j}^{(t-1)}-\eta_{i,j}^{(t)}=\eta_{i,j}^{(\underline{t}-1)}-\eta_{i,j}^{(\overline{t})}=\eta_{i,j}^{(0)}=\int_{\CX_i}g_{i,j}\DIFFX{\mu_i}.
    \end{align*}
    Denoting $\CQ=\big\{(\BIx,\eta_{\BIx}):\BIx\in\BCX^{\dagger(0)}\big\}$,
    we have shown that
    \begin{align}
        \int_{\BCX}g_{i,j}\circ \pi_i\DIFFX{\hat{\mu}^{(0)}}=\sum_{\BIx\in\BCX^{\dagger(0)}}g_{i,j}\circ \pi_i(\BIx)\eta_{\BIx}=\int_{\CX_i}g_{i,j}\DIFFX{\mu_i} \qquad\forall 0\le j\le m_i,\;\forall 1\le i\le N,
        \label{eqn:boundedness-mmotlb-dual-lsip-explicit-proof-integrals}
    \end{align}
    and it hence holds that
    $\hat{\mu}^{(0)}\in \Gamma\big([\mu_1]_{\CG_1},\ldots,[\mu_N]_{\CG_N}\big)$.
    Furthermore, observe that $(\eta_{\BIx})_{\BIx\in\BCX^{\dagger(0)}}$ is feasible for the following dual of the LP problem (\ref{eqn:boundedness-mmotlb-dual-lsip}):
    \begin{align}
        \begin{split}
        \minimize_{(\mu_{\BIx})}\quad & \sum_{\BIx\in\BCX^{\dagger(0)}}f(\BIx)\mu_{\BIx}\\
        \text{subject to}\quad & \sum_{\BIx\in\BCX^{\dagger(0)}}\mu_{\BIx}=1, \qquad \sum_{\BIx\in\BCX^{\dagger(0)}}\BIg(\BIx)\mu_{\BIx}=\bar{\BIg},\\
        & \mu_{\BIx}\ge 0 \qquad\forall \BIx\in\BCX^{\dagger(0)}.
        \label{eqn:boundedness-mmotlb-dual-lsip-explicit-proof-duallp}
        \end{split}
    \end{align}
    The weak duality between (\ref{eqn:boundedness-mmotlb-dual-lsip}) and (\ref{eqn:boundedness-mmotlb-dual-lsip-explicit-proof-duallp}) then shows that the optimal value of (\ref{eqn:boundedness-mmotlb-dual-lsip}) is bounded from above by $\sum_{\BIx\in\BCX^{\dagger(0)}}f(\BIx)\eta_{\BIx} <\nobreak \infty$, and that (\ref{eqn:boundedness-mmotlb-dual-lsip}) has an optimal solution.

    Lastly, let us assume $\mathtt{flag}=\nobreak0$ when Algorithm~\ref{alg:cp-mmot-initialization} terminates and prove that (\ref{eqn:boundedness-mmotlb-dual-lsip}) has bounded superlevel sets as well as $\big|\BCX^{\dagger(0)}\big|=m+1$.
    Notice that for $t=1,\ldots,T$, it holds by Line~\ref{alglin:cpinit-subtraction}, Line~\ref{alglin:cpinit-zeroindices}, and Line~\ref{alglin:cpinit-advance} that $\Big(\sum_{j=0}^{m_i}\eta_{i,j}^{(t)}\Big)-\Big(\sum_{j=0}^{m_i}\eta_{i,j}^{(t-1)}\Big)=\eta_{\min}^{(t)}>0$ for $i=1,\ldots,N$, i.e., the decrements in the values of $\sum_{j=0}^{m_i}\eta_{i,j}^{(t)}$ are equal to the same positive number for $i=1,\ldots,N$. 
    Combining this with the property that $\sum_{j=0}^{m_i}\eta_{i,j}^{(0)}=1$ $\forall 1\le i\le N$, 
    it holds that $r_i^{(T)}=m_i$ and $\eta_{i,m_i}^{(T)}=0$ $\forall 1\le i\le N$.  
    Moreover, since $\sum_{i=1}^N r_i^{(1)}=0$
    and since $\mathtt{flag}=\nobreak0$ when Algorithm~\ref{alg:cp-mmot-initialization} terminates, 
    it holds by Line~\ref{alglin:cpinit-flagcond} and Line~\ref{alglin:cpinit-advance} that
    $|\CI^{(t)}|=\nobreak1$ and 
    $\Big(\sum_{i=1}^N r_i^{(t+1)}\Big) - \Big(\sum_{i=1}^N r_i^{(t)}\Big)=1$ for $t=1,\ldots,T-1$,
    which yield
    \begin{align*}
        T-1=\sum_{t=1}^{T-1}\Bigg[\Bigg(\sum_{i=1}^N r_i^{(t+1)}\Bigg) - \Bigg(\sum_{i=1}^N r_i^{(t)}\Bigg)\Bigg]=\Bigg(\sum_{i=1}^N r_i^{(T)}\Bigg)-\Bigg(\sum_{i=1}^N r_i^{(1)}\Bigg)=\sum_{i=1}^Nm_i=m.
    \end{align*}
    Line~\ref{alglin:cpinit-zeroindices} and Line~\ref{alglin:cpinit-output} then guarantee that $\big|\BCX^{\dagger(0)}\big|=T=m+1$. 

    In the following, we let $\hat{i}^{(t)}\in\{1,\ldots,N\}$ denote the unique element of $\CI^{(t)}$ for $t=1,\ldots,m$.
    Note that $\BIx^{(1)}=(\BIv_{1,0},\ldots,\BIv_{N,0})$ since $r_1^{(1)}=\cdots=r_N^{(1)}=0$.
    The property
    \ref{pp:mmot-complexity-property-orthonormality}
    of $(g_{i,j})_{j=0:m_i,\,i=1:N}$ then
    implies that $\BIg\big(\BIx^{(1)}\big)=\big(\BIg_1(\BIv_{1,0})^\TRANSP,\ldots,\BIg_N(\BIv_{N,0})^\TRANSP\big)^\TRANSP=\nobreak\veczero_m$. 
    We will show that the $m$~vectors
    $\BIg\big(\BIx^{(2)}\big),\ldots,\BIg\big(\BIx^{(m+1)}\big)\in\R^{m}$ are linearly independent.
    To that end, let us fix arbitrary $(a_t)_{t=2:m+1}\subset\nobreak\R$ that satisfy $\sum_{t=2}^{m+1}a_t\BIg\big(\BIx^{(t)}\big)=\nobreak\veczero_m$
    and show via backward induction that $a_t=0$ for $t=2,\ldots,m+1$.
    Suppose for some $k\in\{2,\ldots,m+1\}$ that
    $a_t=\nobreak0$ holds for all $t=k+1,\ldots, m+1$;
    note that this holds vacuously in the case where $k=m+1$.
    In view of Line~\ref{alglin:cpinit-advance}, it holds that
    $r^{(k)}_{\hat{i}^{(k-1)}}=r^{(k-1)}_{\hat{i}^{(k-1)}} + 1$.
    Thus, let us denote $\tilde{i}:=\hat{i}^{(k-1)}$, $\tilde{j}:=r^{(k)}_{\hat{i}^{(k-1)}}$.
    It subsequently follows from Line~\ref{alglin:cpinit-subtraction}
    that
    $\BIx^{(k)}_{\tilde{i}}=\BIv_{\tilde{i},\tilde{j}}$
    and that 
    $\BIx^{(t)}_{\tilde{i}}\ne \BIv_{\tilde{i},\tilde{j}}$ for $t=1,\ldots,k-1$.
    Combining this with the property~\ref{pp:mmot-complexity-property-orthonormality}
    of $(g_{i,j})_{j=0:m_i,\,i=1:N}$ yields
    $g_{\tilde{i},\tilde{j}}\big(\BIx^{(k)}_{\tilde{i}}\big)=1$ 
    and 
    $g_{\tilde{i},\tilde{j}}\big(\BIx^{(t)}_{\tilde{i}}\big)=0$ 
    for $t=1,\ldots,k-1$.
    Now, using the assumption $\sum_{t=2}^{m+1}a_t\BIg\big(\BIx^{(t)}\big)=\nobreak\veczero_m$ and the induction hypothesis,
    we get
    \begin{align*}
        0=\sum_{t=2}^{m+1}a_tg_{\tilde{i},\tilde{j}}\big(\BIx^{(t)}_{\tilde{i}}\big)=a_{k} + \sum_{t=k+1}^{m+1}a_tg_{\tilde{i},\tilde{j}}\big(\BIx^{(t)}_{\tilde{i}}\big) = a_k,
    \end{align*}
    and hence $a_t=0$ holds for $t=k,\ldots,m+1$.
    Therefore, we conclude by induction that $a_t=0$ for $t=2,\ldots,m+1$,
    and thus
    $\BIg\big(\BIx^{(2)}\big),\ldots,\BIg\big(\BIx^{(m+1)}\big)$ are linearly independent.
    Since
    $\BCX^{\dagger(0)}=\big\{\BIx^{(1)},\ldots,\BIx^{(m+1)}\big\}$,
    it holds that
    $\BIg\big(\BIx^{(1)}\big),\ldots,\BIg\big(\BIx^{(m+1)}\big)$ are $m+1$ affinely independent points in $K:=\conv\Big(\big\{\BIg(\BIx):\BIx\in\nobreak\BCX^{\dagger(0)}\big\}\Big)$.
    Moreover, recall that (\ref{eqn:boundedness-mmotlb-dual-lsip-explicit-proof-integrals}) has shown 
    $\bar{\BIg}=\sum_{\BIx\in\BCX^{\dagger(0)}}\eta_{\BIx}\BIg(\BIx)=\sum_{t=1}^{m+1}\eta_{\BIx^{(t)}}\BIg\big(\BIx^{(t)}\big)$,
    and that $\eta_{\BIx^{(t)}}>0$ $\forall 1\le t\le m+1$. 
    This shows that $\bar{\BIg}\in\inter(K)$. 
    Let $U:=\cone\big(\big\{(1,\BIg(\BIx)^\TRANSP)^\TRANSP:\BIx\in\BCX^{(\dagger(0))}\big\}\big)$. 
    Since $(0,\veczero_m^\TRANSP)^\TRANSP,\big(1,\BIg\big(\BIx^{(1)}\big)^\TRANSP\big)^\TRANSP,\ldots,\big(1,\BIg\big(\BIx^{(m+1)}\big)^\TRANSP\big)^\TRANSP\in U$ are $m+2$ affinely independent points in $\R^{m+1}$, \citep[Corollary~6.8.1]{rockafellar1970convex} implies that $(1,\bar{\BIg}^\TRANSP)^\TRANSP\in\relint(U)=\inter(U)$, and the non-emptiness and boundedness of the set of optimal solutions of the LP problem (\ref{eqn:boundedness-mmotlb-dual-lsip}) is a consequence of \citep[Theorem~8.1(vi)]{goberna1998linear}, with $M\leftarrow U$, $c\leftarrow(1,\bar{\BIg}^\TRANSP)^\TRANSP$ in the notation of \citep{goberna1998linear}.
    By the equivalence between (i) and (iii) in \citep[Corollary~9.3.1]{goberna1998linear}, (\ref{eqn:boundedness-mmotlb-dual-lsip}) has non-empty and bounded superlevel sets.
    The proof is now complete.
\end{proof}

%% Proof of Proposition (extension of Lipschitz function)
\begin{proof}[Proof of Proposition~\ref{prop:extension-lip}]
For any $x,x'\in D$, it holds by the $L_f$-Lipschitz continuity of $f$ that $f(x)\le f(x')+L_fd_{\CY}(x,x')$. Subsequently, taking the infimum over $x'\in D$ yields $f(x)\le \tilde{f}(x)$. Moreover, $\tilde{f}(x)\le f(x)+L_fd_{\CY}(x,x)=f(x)$, which proves that $\tilde{f}(x)=f(x)$ for all $x\in D$. 
To prove the $L_f$-Lipschitz continuity of $\tilde{f}$, let $x_1,x_2\in\widetilde{D}$ be arbitrary. Then, it holds for any $x'\in D$ that $\tilde{f}(x_2)\le f(x')+L_fd_{\CY}(x_2,x')\le f(x')+L_fd_{\CY}(x_1,x_2)+L_fd_{\CY}(x_1,x')$ and hence
\begin{align*}
\tilde{f}(x_2)-L_fd_{\CY}(x_1,x_2)\le \inf_{x'\in D}\big\{f(x')+L_fd_{\CY}(x_1,x')\big\}=\tilde{f}(x_1).
\end{align*}
Subsequently, exchanging the roles of $x_1$ and $x_2$ yields $\big|\tilde{f}(x_1)-\tilde{f}(x_2)\big|\le L_fd_{\CY}(x_1,x_2)$, which shows that $\tilde{f}$ is $L_f$-Lipschitz continuous. The proof is now complete. 
\end{proof}

% Proof of Proposition (cutting-plane algorithm)
\begin{proof}[Proof of Proposition~\ref{prop:cpalgo-properties}]
Due to the compactness of $\CX_1,\ldots,\CX_N$ and the continuity of the test functions in $\CG_1,\ldots,\CG_N$, the set $\big\{\BIg(\BIx):\BIx\in\BCX\big\}$ is bounded. Moreover, the global minimization problem solved by $\mathtt{Oracle}(\,\cdot\,)$ in Line~\ref{alglin:cp-global} is bounded from below. 
Therefore, statement~\ref{props:cpalgo-termination} follows from \citep[Theorem~11.2]{goberna1998linear} with $g(\,\cdot\,,\cdot\,)\leftarrow\Big(\BCX\times\R^{1+m}\ni\big(\BIx,(y_0,\BIy)\big)\mapsto-\big(y_0+\langle\BIg(\BIx),\BIy\rangle-f(\BIx)\big)\in\R\Big)$. 

In order to prove statements~\ref{props:cpalgo-bounds}, \ref{props:cpalgo-dual}, and \ref{props:cpalgo-primal}, we will show that $(\hat{y}_0,\hat{\BIy})$ is a feasible solution of \eqref{eqn:mmotlb-dual-lsip} with objective value equal to $\alpha_{\mathsf{relax}}^{\SFL\SFB}$ and that $\hat{\mu}$ is a feasible solution of \eqref{eqn:mmotlb} with objective value equal to $\alpha_{\mathsf{relax}}^{\SFU\SFB}$. 
Subsequently, since Line~\ref{alglin:cp-termination1}, Line~\ref{alglin:cp-termination2}, and Line~\ref{alglin:cp-bounds} guarantee that $\alpha_{\mathsf{relax}}^{\SFU\SFB}-\alpha_{\mathsf{relax}}^{\SFL\SFB}=y^{(r)}_0-s^{(r)}\le\epsilon_{\mathsf{LSIP}}$, statements~\ref{props:cpalgo-bounds}, \ref{props:cpalgo-dual}, and \ref{props:cpalgo-primal} will follow from the weak duality in Theorem~\ref{thm:duality}\ref{thms:duality-weak}. 
On the one hand, by Line~\ref{alglin:cp-global} and Line~\ref{alglin:cp-dual}, it holds that 
\begin{align*}
\hat{y}_0+\langle\BIg(\BIx),\hat{\BIy}\rangle-f(\BIx)&=s^{(r)}+\langle\BIg(\BIx),\BIy^{(r)}\rangle-f(\BIx)\\
&= \min_{\BIx'\in\BCX}\big\{f(\BIx')-\langle\BIg(\BIx'),\BIy^{(r)}\rangle\big\} -\big(f(\BIx)-\langle\BIg(\BIx),\BIy^{(r)}\rangle\big)\le 0 \qquad \forall \BIx\in\BCX.
\end{align*}
Moreover, it follows from Line~\ref{alglin:cp-lp}, Line~\ref{alglin:cp-bounds}, and Line~\ref{alglin:cp-dual} that $\hat{y}_0+\langle\bar{\BIg},\hat{\BIy}\rangle=s^{(r)}+\langle\bar{\BIg},\BIy^{(r)}\rangle=\alpha^{(r)}-y^{(r)}_0+s^{(r)}=\alpha_{\mathsf{relax}}^{\SFL\SFB}$. 
This shows that $(\hat{y}_0,\hat{\BIy})$ is a feasible solution of \eqref{eqn:mmotlb-dual-lsip} with objective value~$\alpha_{\mathsf{relax}}^{\SFL\SFB}$. 
On the other hand, by Line~\ref{alglin:cp-lp}, $\big(\mu^{(r)}_{\BIx}\big)_{\BIx\in\BCX^{\dagger(r)}}$ is an optimal solution of the following LP problem, which corresponds to the dual of the LP problem in Line~\ref{alglin:cp-lp}:
\begin{align*}
\begin{split}
\minimize_{(\mu_{\BIx})_{\BIx\in\BCX^{\dagger(r)}}}\quad & \sum_{\BIx\in\BCX^{\dagger(r)}}f(\BIx)\mu_{\BIx} \\
\text{subject to}\quad & \sum_{\BIx\in\BCX^{\dagger(r)}}\mu_{\BIx}=1, \qquad \sum_{\BIx\in\BCX^{\dagger(r)}}\BIg(\BIx)\mu_{\BIx}=\bar{\BIg},\\
& \mu_{\BIx}\ge0 \qquad \forall\BIx\in\BCX^{\dagger(r)}.
\end{split}
%\label{eqn:cpalgo-primallp}
\end{align*}
Consequently, it follows from Line~\ref{alglin:cp-primal} that $\hat{\mu}$ is a positive Borel measure that satisfies $\hat{\mu}(\BCX)=\sum_{\BIx\in\BCX^{\dagger(r)}}\mu^{(r)}_{\BIx}{=1}$ and $\int_{\BCX}g_{i,j}\circ\pi_i\DIFFX{\hat{\mu}}=\sum_{\BIx\in\BCX^{\dagger(r)}}g_{i,j}\circ\pi_i(\BIx)\mu^{(r)}_{\BIx}=\int_{\CX_i}g_{i,j}\DIFFX{\mu_i}$ $\forall 1\le j\le m_i$, $\forall 1\le i\le N$.
This shows that $\hat{\mu}\in\Gamma\big([\mu_1]_{\CG_1},\ldots,[\mu_N]_{\CG_N}\big)$. 
Moreover, since $\BCX^{\dagger(r)}\supseteq\BCX^{\dagger(0)}$ by Line~\ref{alglin:cp-aggregate}, it follows from the assumption about $\BCX^{\dagger(0)}$ in Remark~\ref{rmk:cpalgo} that the LP problem in Line~\ref{alglin:cp-lp} is feasible and bounded from above. 
Subsequently, the strong duality of LP problems yields that 
$\int_{\BCX}f\DIFFX{\hat{\mu}}=\sum_{\BIx\in\BCX^{\dagger(r)}}f(\BIx)\mu^{(r)}_{\BIx}=\alpha^{(r)}$. 
Since $\alpha_{\mathsf{relax}}^{\SFU\SFB}=\alpha^{(r)}$ by Line~\ref{alglin:cp-bounds}, $\hat{\mu}$ is a feasible solution of \eqref{eqn:mmotlb} with objective value $\alpha_{\mathsf{relax}}^{\SFU\SFB}$.
Moreover, it holds that $\support(\hat{\mu})\subseteq\BCX^{\dagger(r)}$ is a finite set. 
The proof is now complete. 
\end{proof}

%% Proof of Theorem (MMOT algorithm)
\begin{proof}[Proof of Theorem~\ref{thm:mmotalgo}]
It follows from Proposition~\ref{prop:cpalgo-properties}\ref{props:cpalgo-bounds}--\ref{props:cpalgo-primal} that 
$\hat{y}_0+\langle\bar{\BIg},\hat{\BIy}\rangle=\alpha_{\mathsf{relax}}^{\mathsf{LB}}\le$ \eqref{eqn:mmotlb-dual-lsip} $\le\alpha_{\mathsf{relax}}^{\mathsf{UB}}=\int_{\BCX}f\DIFFX{\hat{\mu}}$, 
$\alpha_{\mathsf{relax}}^{\mathsf{UB}}-\alpha_{\mathsf{relax}}^{\mathsf{LB}}\le\epsilon_{\mathsf{LSIP}}$, 
and $\hat{\mu}$ is an $\epsilon_{\mathsf{LSIP}}$-optimal solution of \eqref{eqn:mmotlb-dual-lsip}.
Moreover, Line~\ref{alglin:mmot-reassembly} computes $\tilde{\mu}\in R(\hat{\mu};\mu_1,\ldots,\mu_N)\subseteq\Gamma(\mu_1,\ldots,\mu_N)$, and thus $\tilde{\mu}$ is feasible for \eqref{eqn:mmot}. 
Consequently, we obtain from Line~\ref{alglin:mmot-bounds} and Theorem~\ref{thm:lowerbound}\ref{thms:lowerbound-control} that
\begin{align}
\begin{split}
    \tilde{\epsilon}_{\mathsf{sub}}&=\int_{\BCX}f\DIFFX{\tilde{\mu}}-\alpha_{\mathsf{relax}}^{\mathsf{LB}}=\left(\int_{\BCX}f\DIFFX{\tilde{\mu}}-\int_{\BCX}f\DIFFX{\hat{\mu}}\right)+\left(\int_{\BCX}f\DIFFX{\hat{\mu}}-\alpha_{\mathsf{relax}}^{\mathsf{LB}}\right)\\
    &\le \epsilon_{\mathsf{LSIP}}+L_f\sum_{i=1}^N\specialoverline{W}_{1}(\mu_i,[\mu_i]_{\CG_i})\le \epsilon_{\mathsf{LSIP}}+L_f\sum_{i=1}^N\rho_i=\epsilon_{\mathsf{theo}}.
\end{split}
\label{eqn:mmotalgo-proof-error}
\end{align}
Since $\alpha^{\mathsf{LB}}=\alpha^{\mathsf{LB}}_{\mathsf{relax}}\le$ \eqref{eqn:mmotlb-dual-lsip} $=$ \eqref{eqn:mmotlb} $\le$ \eqref{eqn:mmot}, 
(\ref{eqn:mmotalgo-proof-error}) shows that statement~\ref{thms:mmotalgo-bounds} and statement~\ref{thms:mmotalgo-primal} hold.
Furthermore, one observes from Line~\ref{alglin:mmot-dual} that $(\tilde{h}_i)_{i=1:N}$ is feasible for \eqref{eqn:mmot-dual} 
with objective $\sum_{i=1}^{N}\int_{\CX_i}\tilde{h}_i\DIFFX{\mu_i}=\hat{y}_0 + \langle\bar{\BIg},\hat{\BIy}\rangle=\alpha_{\mathsf{relax}}^{\mathsf{LB}}=\alpha^{\mathsf{LB}}$, 
where
$\alpha^{\mathsf{LB}}=\int_{\BCX}f\DIFFX{\tilde{\mu}}-\tilde{\epsilon}_{\mathsf{sub}}\ge$ $\eqref{eqn:mmot}-\tilde{\epsilon}_{\mathsf{sub}}$ $\ge$ $\eqref{eqn:mmot-dual}-\nobreak\tilde{\epsilon}_{\mathsf{sub}}$.
This proves statement~\ref{thms:mmotalgo-dual}.
Lastly, 
to prove statement~\ref{thms:mmotalgo-control},
notice that
for $i=1,\ldots,N$,
the construction of $\FC_i$ satisfying $\eta(\FC_i)\le \frac{\epsilon-\epsilon_{\mathsf{LSIP}}}{2NL_f}$ from $\widetilde{\FC}_i$ can be done via the iterative bisection procedure in Proposition~\ref{prop:momentset-simplex}\ref{props:momentset-simplex-bisection}.
Subsequently, when using $(\CG_i)_{i=1:N}$, $\BIg(\,\cdot\,)$, $\bar{\BIg}$ constructed by \ref{settc:compact-testfuncs},
it is guaranteed by
(\ref{eqn:mmotalgo-proof-error}) 
and Proposition~\ref{prop:momentset-simplex}\ref{props:momentset-simplex-upperbound} 
that $\tilde{\epsilon}_{\mathsf{sub}}\le \epsilon_{\mathsf{LSIP}}+L_f\big(\sum_{i=1}^N\specialoverline{W}_{1}(\mu_i,[\mu_i]_{\CG_i})\big)\le \epsilon_{\mathsf{LSIP}}+L_f\big(\sum_{i=1}^N2\eta(\FC_i)\big)\le\nobreak\epsilon$.
The proof is now complete.
\end{proof}

%\pagebreak
\subsection{Proofs of results in Section~\ref{sec:experiments}}\label{ssec:proof-experiments}\nopagebreak
\begin{proof}[Proof of Proposition~\ref{prop:experiments-fluid}]
    To prove statement~\ref{props:experiments-fluid-setting},
    it suffices to show that $f$ is $(2L_{\Xi}+2)$-Lipschitz continuous.
    To that end,
    one observes that since $\Xi$ is $L_{\Xi}$-Lipschitz continuous and 
    $\FL_{[0,1]}\circ \Xi^{-1}=\FL_{[0,1]}$,
    it holds necessarily that $L_{\Xi}\ge 1$. 
    Consequently, we have
    \begin{align*}
        \big|f(\BIx)-f(\BIx')\big| &\le 2|x_N|\big|\Xi(x_1)-\Xi(x'_1)\big| + 2\big|\Xi(x'_1)\big||x_N-x'_N|\\
        &\qquad + \sum_{i=1}^{N-1} 2|x_i||x_{i+1}-x'_{i+1}| + 2|x'_{i+1}||x_{i}-x'_{i}|\\
        &\le (2L_{\Xi}+2)\sum_{i=1}^N|x_i-x'_i| \qquad \forall \BIx=(x_1,\ldots,x_N),\, \BIx'=(x'_1,\ldots,x'_N)\in[0,1]^N.
    \end{align*}
    Thus, statement~\ref{props:experiments-fluid-setting} holds.
    Next, statement~\ref{props:experiments-fluid-shift} and statement~\ref{props:experiments-fluid-primaldual-solutions}
    follow from the observations
    that the cost function of (\ref{eqn:experiments-fluid-OT}) satisfies
    \begin{align*}
        \big(x_N-\Xi(x_1)\big)^2 + \sum_{i=1}^{N-1}(x_i-x_{i+1})^2 &= f(x_1,\ldots,x_N) + x_1^2 + \Xi(x_1)^2 + 2\sum_{i=2}^{N} x_i^2 \\
        &\qquad \qquad\qquad\qquad \qquad \forall (x_1,\ldots,x_N)\in[0,1]^N,
    \end{align*}
    and that
    \begin{align*}
        \int_{[0,1]^N} {\textstyle x_1^2 + \Xi(x_1)^2 + 2\sum_{i=2}^{N} x_i^2} \DIFFM{\mu}{\DIFF x_1,\ldots,\DIFF x_N}&=\int_{[0,1]}\Xi(x)^2 + (2N-1)x^2 \DIFFM{\FL_{[0,1]}}{\DIFF x} \\
        &= 2N\int_{0}^{1}x^2\DIFFX{x}=\frac{2N}{3} \qquad \forall \mu\in\Gamma(\mu_1,\ldots,\mu_N).
    \end{align*}

    Lastly, statement~\ref{props:experiments-fluid-reassembly} can be proved by a direct application of Proposition~\ref{prop:reassembly-1d}\ref{props:reassembly-1d-semidisc}.
    To see this, let 
    $\big(F_{\mu_i}(\,\cdot\,)\big)_{i=1:N}$, 
    $\big(F_{\mu_i}^{-1}(\,\cdot\,)\big)_{i=1:N}$ be defined as in Proposition~\ref{prop:reassembly-1d},
    let 
    $\sigma_1:\{1,\ldots,J\}\to\{1,\ldots,J\}$ be the identity mapping,
    and define $c_{0,j}:=\sum_{l=1}^{j-1}a_l$ $\forall 1\le j\le J$.
    One checks that $F_{\mu_i}^{-1}(u)=u$ $\forall u\in(0,1)$, $\forall 1\le i\le N$,
    and that $c_{0,j}=c_{1,j}$ $\forall 1\le j\le J$.
    Moreover, let $(\Omega,\CF,\PROB)$ be a probability space with 
    $\Omega:=[0,1]$, $\CF:=\CB([0,1])$, $\PROB:=\FL_{[0,1]}$,
    define 
    $U_0:\Omega\to[0,1]$ by $U_0(\omega):=\omega$ $\forall \omega\in[0,1]$,
    define $\big(U_{i,j}:\Omega\to[0,1]\big)_{j=1:J,\,i=1:N}$ by Remark~\ref{rmk:reassembly-1d-uniform-examples}\ref{rmks:reassembly-1d-uniform-examples-deterministic},
    and define 
    $(Z_i:\Omega\to\R)_{i=1:N}$ by (\ref{eqn:reassembly-1d-invtrans}).
    One verifies that
    $Z_i(\omega)=\Xi_i(\omega)$ $\forall \omega\in(0,1)$ for $i=1,\ldots,N$,
    which shows that 
    $\mu_1\circ(\Xi_1,\ldots,\Xi_N)^{-1}$ is the law of the random variable $(Z_1,\ldots,Z_N):\Omega\to[0,1]^N$.
    Consequently, it holds by Proposition~\ref{prop:reassembly-1d}\ref{props:reassembly-1d-semidisc} that $\mu_1\circ(\Xi_1,\ldots,\Xi_N)^{-1}\in R(\hat{\mu};\mu_1,\ldots,\mu_N)$.
    The proof is now complete.
\end{proof}

\begin{proof}[Proof of Proposition~\ref{prop:experiments-barycenter-mmot}]
    Throughout this proof, let us define $(f_i:\CX_i\times\R^d\to\R)_{i=1:N}$ as follows:
    \begin{align*}
        f_i(\BIx_i,\BIz):=\|\BIz\|_2^2-2\langle\BIx_i,\BIz\rangle \qquad\forall \BIx_i\in\CX_i,\;\forall\BIz\in\R^d,\; \forall 1\le i\le N.
    \end{align*}
    Statement~\ref{props:experiments-barycenter-mmot-formulation}
    follows directly from \citep[Proposition~3]{carlier2010matching}.
    To prove statement~\ref{props:experiments-barycenter-setting},
    observe that 
    \begin{align*}
        %\begin{split}
        \bar{\BIz}(\BIx_1,\ldots,\BIx_N)&= \argmin_{\BIz\in\R^d}\Bigg\{\sum_{i=1}^N f_i(\BIx_i,\BIz)\Bigg\} \qquad \forall (\BIx_1,\ldots,\BIx_N)\in\BCX,
        %\end{split}
        %\label{eqn:experiments-barycenter-mmot-proof-mean}
    \end{align*}
    and that
    \begin{align}
        \begin{split}
        f(\BIx_1,\ldots,\BIx_N)&=\frac{1}{N}\sum_{i=1}^N f_i\big(\BIx_i,\bar{\BIz}(\BIx_1,\ldots,\BIx_N)\big)\\
        &=\min_{\BIz\in\CZ}\Bigg\{\frac{1}{N}\sum_{i=1}^N f_i(\BIx_i,\BIz)\Bigg\} \qquad \forall (\BIx_1,\ldots,\BIx_N)\in\BCX.
        \end{split}
        \label{eqn:experiments-barycenter-mmot-proof-objective}
    \end{align}
    Subsequently, one checks that
    \begin{align}
        \big|f_i(\BIx_i,\BIz)-f_i(\hat{\BIx}_i,\BIz)\big| \le 2\sup_{\BIz\in\CZ}\big\{\|\BIz\|_2\big\} \|\BIx_i-\hat{\BIx}_i\|_2 \qquad\forall \BIx_i,\hat{\BIx}_i\in\CX_i,\;\forall \BIz\in\CZ, \; \forall 1\le i\le N,
        \label{eqn:experiments-barycenter-couplings-proof-lipschitzeach}
    \end{align}
    and hence (\ref{eqn:experiments-barycenter-mmot-proof-objective}) implies that
    \begin{align*}
        f(\BIx)-f(\hat{\BIx}) &\le \frac{1}{N}\sum_{i=1}^{N} f_i\big(\BIx_i,\bar{\BIz}(\hat{\BIx})\big) - f_i\big(\hat{\BIx}_i,\bar{\BIz}(\hat{\BIx})\big) \\
        &\le \frac{2}{N}\sup_{\BIz\in\CZ}\big\{\|\BIz\|_2\big\} \sum_{i=1}^N \|\BIx_i-\hat{\BIx}_i\|_2  \qquad \forall \BIx=(\BIx_1,\ldots,\BIx_N),\hat{\BIx}=(\hat{\BIx}_1,\ldots,\hat{\BIx}_N)\in\BCX.
    \end{align*}
    Exchanging the roles of $\BIx$ and $\hat{\BIx}$ in the inequalities above proves statement~\ref{props:experiments-barycenter-setting}.
    Moreover, (\ref{eqn:experiments-barycenter-mmot-proof-objective}) also shows that
    \begin{align}
        \min_{\BIz\in\CZ}\Bigg\{\frac{1}{N}\sum_{i=1}^N \|\BIx_i-\BIz\|_2^2\Bigg\} = f(\BIx_1,\ldots,\BIx_N) + \frac{1}{N}\sum_{i=1}^{N}\|\BIx_i\|_2^2 \quad \forall (\BIx_1,\ldots,\BIx_N)\in\BCX.
        \label{eqn:experiments-barycenter-mmot-proof-quadratic-terms}
    \end{align}
    Combining (\ref{eqn:experiments-barycenter-mmot-proof-quadratic-terms}) with statement~\ref{props:experiments-barycenter-mmot-formulation} then proves 
    statement~\ref{props:experiments-barycenter-bounds}.
    Furthermore, combining (\ref{eqn:experiments-barycenter-mmot-proof-quadratic-terms}) 
    and Theorem~\ref{thm:mmotalgo}\ref{thms:mmotalgo-dual} proves
    statement~\ref{props:experiments-barycenter-dual-solution}.

    Next, note that since $\tilde{\mu}\circ (\pi_i,\bar{\BIz})^{-1}\in\Gamma(\mu_i,\tilde{\nu})$ for $i=1,\ldots,N$, we have
    \begin{align*}
        \frac{1}{N}\sum_{i=1}^NW_2\big(\mu_i,\tilde{\nu}\big)^2 
        & \le \int_{\BCX} {\textstyle \frac{1}{N}\sum_{i=1}^N} \big\|\BIx_i-\bar{\BIz}(\BIx_1,\ldots,\BIx_N)\big\|_2^2 \DIFFM{\tilde{\mu}}{\DIFF\BIx_1,\ldots,\DIFF\BIx_N} \\
        & =\int_{\BCX} \min_{\BIz\in\CZ}\Big\{{\textstyle\frac{1}{N}\sum_{i=1}^N}\|\BIx_i-\BIz\|_2^2\Big\} \DIFFM{\tilde{\mu}}{\DIFF\BIx_1,\ldots,\DIFF\BIx_N}.
    \end{align*}
    Combining this with Theorem~\ref{thm:mmotalgo}\ref{thms:mmotalgo-primal}, (\ref{eqn:experiments-barycenter-mmot-proof-quadratic-terms}),
    and statement~\ref{props:experiments-barycenter-mmot-formulation} proves statement~\ref{props:experiments-barycenter-primal-solution-W1}.

    Lastly, let us prove statement~\ref{props:experiments-barycenter-primal-solution-W2}.
    Observe that the existence of $\hat{\gamma}$ is guaranteed by repeated applications of the gluing lemma (i.e., Lemma~\ref{lem:gluing}), similar to the proof of Lemma~\ref{lem:reassembly}.
    Subsequently, since $\breve{\mu}\circ (\pi_i,\bar{\BIz})^{-1}\in\Gamma(\mu_i,\breve{\nu})$ for $i=1,\ldots,N$,
    it holds by (\ref{eqn:experiments-barycenter-mmot-proof-quadratic-terms})
    and (\ref{eqn:experiments-barycenter-mmot-proof-objective})
    that
    \begin{align}
        \begin{split}
            \frac{1}{N}\sum_{i=1}^{N}W_2(\mu_i,\breve{\nu})^2 &\le \int_{\BCX}{\textstyle\frac{1}{N}\sum_{i=1}^{N}}\big\|\BIx_i-\bar{\BIz}(\BIx_1,\ldots,\BIx_N)\big\|_2^2 \DIFFM{\breve{\mu}}{\DIFF\BIx_1,\ldots,\DIFF\BIx_N}\\
            &=\int_{\BCX}\min_{\BIz\in\CZ}\Big\{{\textstyle\frac{1}{N}\sum_{i=1}^{N}}\|\BIx_i-\BIz\|_2^2\Big\}\DIFFM{\breve{\mu}}{\DIFF\BIx_1,\ldots,\DIFF\BIx_N}\\
            &=\int_{\BCX}f\DIFFX{\breve{\mu}}+C_{\mathsf{quad}}=\beta^{\mathsf{UB}}.
        \end{split}
        \label{eqn:experiments-barycenter-proof-primal-solution-W2-step1}
    \end{align}
    This proves that $\beta^{\mathsf{UB}}$ is an upper bound for the optimal value of (\ref{eqn:experiments-barycenter-definition}).
    Moreover,
    since statement~\ref{props:experiments-barycenter-bounds} has already established that
    $\tilde{\alpha}^{\mathsf{LB}}+C_{\mathsf{quad}}\le \inf_{\nu\in\CP_2(\R^d)}\big\{\frac{1}{N}\sum_{i=1}^N W_2(\mu_i,\nu)^2\big\}$,
    we get 
    \begin{align*}
        \Bigg(\frac{1}{N}\sum_{i=1}^{N}W_2(\mu_i,\breve{\nu})^2\Bigg) - \inf_{\nu\in\CP_2(\R^d)}\Bigg\{\frac{1}{N}\sum_{i=1}^N W_2(\mu_i,\nu)^2\Bigg\} &\le \beta^{\mathsf{UB}} - \tilde{\alpha}^{\mathsf{LB}}-C_{\mathsf{quad}} = \tilde{\xi}_{\mathsf{sub}},
    \end{align*}
    and hence $\breve{\nu}$ is a $\tilde{\xi}_{\mathsf{sub}}$-optimal solution of (\ref{eqn:experiments-barycenter-definition}).
    It remains to show that $\tilde{\xi}_{\mathsf{sub}}\le \epsilon_{\mathsf{theo}}$.
    Continuing from (\ref{eqn:experiments-barycenter-proof-primal-solution-W2-step1}), it holds by the properties of  $(\hat{\gamma}_i)_{i=1:N}$ that
    \begin{align}
        \begin{split}
            \beta^{\mathsf{UB}}&=\int_{\BCX}\min_{\BIz\in\CZ}\Big\{{\textstyle\frac{1}{N}\sum_{i=1}^{N}}\|\BIx_i-\BIz\|_2^2\Big\}\DIFFM{\breve{\mu}}{\DIFF\BIx_1,\ldots,\DIFF\BIx_N}\\
            &\le \int_{\BCX\times\CZ}{\textstyle\frac{1}{N}\sum_{i=1}^{N}}\|\BIx_i-\BIz\|_2^2\DIFFM{\hat{\gamma}}{\DIFF\BIx_1,\ldots,\BIx_N,\BIz}\\
            &= \frac{1}{N}\sum_{i=1}^{N}\int_{\CX_i\times\CZ}\|\BIx_i-\BIz\|_2^2\DIFFM{\hat{\gamma}_i}{\DIFF\BIx_i,\DIFF\BIz}\\
            &= C_{\mathsf{quad}} + \frac{1}{N}\sum_{i=1}^{N}\inf_{\gamma_i\in\Gamma(\mu_i,\hat{\nu})}\bigg\{\int_{\CX_i\times\CZ}f_i(\BIx_i,\BIz)\DIFFM{\gamma_i}{\DIFF\BIx_i,\DIFF\BIz}\bigg\}.
        \end{split}
        \label{eqn:experiments-barycenter-proof-primal-solution-W2-step2}
    \end{align}
    Next,
    let us denote $\bar{\CX}_i:=\CX_i$ for $i=1,\ldots,N$,
    denote $\overbar{\BCX}:=\bigtimes_{i=1}^{N}\bar{\CX}_i$, 
    denote the marginal of $\hat{\mu}$ on $\CX_i$ by $\hat{\mu}_i$ for $i=1,\ldots,N$, and consider $\gamma\in\CP(\BCX\times \overbar{\BCX})$ that satisfies the conditions in Definition~\ref{def:reassembly}.
    It subsequently follows from 
    the definition of $\hat{\nu}$,
    (\ref{eqn:experiments-barycenter-couplings-proof-lipschitzeach}),
    (\ref{eqn:experiments-barycenter-mmot-proof-objective}), 
    Proposition~\ref{prop:cpalgo-properties}\ref{props:cpalgo-primal}, 
    and 
    Proposition~\ref{prop:cpalgo-properties}\ref{props:cpalgo-bounds} that
    \begin{align}
        \begin{split}
            \hspace{14pt}&\hspace{-20pt} \frac{1}{N}\sum_{i=1}^{N}\inf_{\gamma_i\in\Gamma(\mu_i,\hat{\nu})}\bigg\{\int_{\CX_i\times\CZ}f_i(\BIx_i,\BIz)\DIFFM{\gamma_i}{\DIFF\BIx_i,\DIFF\BIz}\bigg\} \\
            &\le \int_{\BCX\times\overbar{\BCX}} \textstyle\frac{1}{N}\sum_{i=1}^N f_i\big(\BIx_i,\bar{\BIz}(\hat{\BIx}_1,\ldots,\hat{\BIx}_N)\big) \DIFFM{\gamma}{\DIFF\hat{\BIx}_1,\ldots,\DIFF\hat{\BIx}_N,\DIFF\BIx_1,\ldots,\DIFF\BIx_N}\\
            &\le \int_{\BCX} {\textstyle\frac{1}{N}\sum_{i=1}^N} f_i\big(\hat{\BIx}_i,\bar{\BIz}(\hat{\BIx}_1,\ldots,\hat{\BIx}_N)\big) \DIFFM{\hat{\mu}}{\DIFF\hat{\BIx}_1,\ldots,\DIFF\hat{\BIx}_N} + \frac{2}{N}\sup_{\BIz\in\CZ}\big\{\|\BIz\|_2\big\}\sum_{i=1}^N W_1(\hat{\mu}_i,\mu_i)\\
            &\le \int_{\BCX}f\DIFFX{\hat{\mu}}+ \frac{2}{N}\sup_{\BIz\in\CZ}\big\{\|\BIz\|_2\big\}\sum_{i=1}^N \specialoverline{W}_{1}(\mu_i,[\mu_i]_{\CG_i})\\
            &\le \tilde{\alpha}^{\mathsf{LB}}+\epsilon_{\mathsf{LSIP}} + \frac{2}{N}\sup_{\BIz\in\CZ}\big\{\|\BIz\|_2\big\}\sum_{i=1}^N \rho_i = \tilde{\alpha}^{\mathsf{LB}} + \epsilon_{\mathsf{theo}}.
        \end{split}
        \label{eqn:experiments-barycenter-proof-primal-solution-W2-step3}
    \end{align}
    Now, combining (\ref{eqn:experiments-barycenter-proof-primal-solution-W2-step2}) and (\ref{eqn:experiments-barycenter-proof-primal-solution-W2-step3}) 
    proves that $\tilde{\xi}_{\mathsf{sub}}=\beta^{\mathsf{UB}}-\tilde{\alpha}^{\mathsf{LB}}-C_{\mathsf{quad}}\le \epsilon_{\mathsf{theo}}$.
    The proof is now complete. 
\end{proof}

\subsection{Proofs of results in Section~\ref{sec:momentset-extended}}\label{ssec:proof-momentset-extended}

\begin{proof}[Proof of Proposition~\ref{prop:momentset-simplex-VIFS}]
    The claims of this proposition have already been established in the proof of Proposition~\ref{prop:momentset-simplex}.
\end{proof}

% Proof of Proposition (hyper-rectangular cover and associated IFB)
\begin{proof}[Proof of Proposition~\ref{prop:cover-cube}]
    This proof is structured as follows.
    We will first construct a VIFS 
    $\CG_{\mathsf{VIFS}}:=\big\{g_{\BIv}:\BIv\in V(\FC)\big\}\subset\lspan_1(\CG_0)$ for~$\FC$.
    Since $V(\FC_0)=V(\FC)$, this set will also be a VIFS for $\FC_0$, which will prove statement~\ref{props:cover-cube-bounded}.
    Moreover, since $\CG_0\subset\CG_p$ for all $p\in[1,\infty)$,
    it follows that $\CG_{\mathsf{VIFS}}\subset\lspan_1(\CG_p)$ for all $p\in[1,\infty)$.
    Hence, we will subsequently prove that
    the functions $\CG_{p\text{-}\mathsf{RFS}}$ defined in statement~\ref{props:cover-cube-unbounded} satisfy the property 
    \ref{defs:cover-rf-bound} as well as $\CG_{p\text{-}\mathsf{RFS}}\subset\lspan_1(\CG_p)$ in order to complete the proof of statement~\ref{props:cover-cube-unbounded}.

    Let us begin by introducing the notations used in this proof. 
    We denote the $d$-dimensional vector with all entries equal to $\infty$ by $\vecinfty$. 
    For any $\BIx=(x_1,\ldots,x_d)^\TRANSP,\BIx'=(x'_1,\ldots,x_d')^\TRANSP\in\big(\R\cup\{-\infty,\infty\}\big)^d$, 
    we denote $(\BIx,\BIx'):=\big\{(z_1,\ldots,z_d)^\TRANSP\in\R^d:x_i<z_i<x'_i\;\forall 1\le i\le d\big\}$, 
    $(\BIx,\BIx']:=\big\{(z_1,\ldots,z_d)^\TRANSP\in\nobreak\R^d:x_i<z_i\le x'_i$ $\forall 1\le i\le d\big\}$, 
    and $[\BIx,\BIx']:=\big\{(z_1,\ldots,z_d)^\TRANSP\in\R^d:x_i\le z_i\le x'_i$ $\forall 1\le i\le d\big\}$.
    Throughout the proof, we will adopt the following equivalent expression of $\CG_0$:
    \begin{align}
    \begin{split}
    \CG_0&=\bigg\{\R^d\ni(x_1,\ldots,x_d)^\TRANSP\mapsto {\textstyle\bigvee_{i=1}^{d}} \beta_i^{-1}(x_i-\kappa_i)^+ : \kappa_i\in\{\kappa_{i,0},\ldots,\kappa_{i,n_i},\infty\}\;\forall 1\le i\le d \bigg\}.
    \end{split}
    \label{eqn:cover-cube-proof-basis}
    \end{align}

    The proof is divided into eight steps as follows.
    \begin{itemize}
        \item%
        \textit{Step~1: defining an affine transformation $\BIT:\R^d\to\R^d$ such that $\BIT(V(\FC))$ becomes a grid $\BIL:=\bigtimes_{i=1}^{d}\{0,1,\ldots,n_i\}$ in which each grid point has integer-valued coordinates.}
        
        \item%
        \textit{Step~2: defining an index set $\BIJ:=\bigtimes_{i=1}^{d}\{0,1,\ldots,n_i-1,\infty\}$, a ``predecessor'' function $\BIp(\,\cdot\,)$ on~$\BIJ$, an expanded index set $\overline{\BIJ}:=\BIJ\cup\BIp(\BIJ)$, and a bijection $\BIl:\BIJ\to\BIL$.}
        
        \item%
        \textit{Step~3: defining a class of functions $\big\{\big(Q(\,\cdot\,;\BIj):\R^d\to[0,1]\big):\BIj\in\overline{\BIJ}\big\}$, and showing the property that $Q(\BIz;\BIj)=\mu_{\BIz}\big((-\vecinfty,\BIj]\big)$ $\forall\BIz\in\R^d$, $\forall\BIj\in\overline{\BIJ}$,
        where $\mu_{\BIz}\in\CP(\R^d)$ is the law of a random variable that is uniformly distributed on the line segment $\conv\big(\{\BIz-\vecone,\BIz\}\big)$.}
        
        \item%
        \textit{Step~4: defining a class of functions $\big\{\big(M(\,\cdot\,;\BIj',\BIj):\R^d\to[0,1]\big):\BIj',\BIj\in\BIJ,\;\BIj'\le\BIj\big\}$, and showing the property that
        $M(\BIz;\BIj',\BIj)=\mu_{\BIz}\big((\BIp(\BIj'),\BIj]\cap(\BIz-\vecone,\BIz)\big)$ $\forall\BIz\in\R^d$, $\forall\BIj',\BIj\in\BIJ$, $\BIj'\le\BIj$.}
        
        \item%
        \textit{Step~5: defining two index sets $\BIK_{\TL}:=\bigtimes_{i=1}^{d}\{-\infty,0,\ldots,n_i\}$, $\BIK_{\TR}:=\bigtimes_{i=1}^{d}\{0,\ldots,n_i,\infty\}$ to represent the transformed faces $\big\{\BIT(F):F\in\FF(\FC)\big\}$, as well as two mappings 
        $\BIq_{\TL}:\BIK_{\TL}\to\BIJ$, $\BIq_{\TR}:\BIK_{\TR}\to\BIJ$, and characterizing certain cases in which $M(\,\cdot\,;\BIj',\BIj)$ evaluates to~0 or~1.}
        
        \item%
        \textit{Step~6: showing by linear independence that each function in $\big\{M(\,\cdot\,;\BIj',\BIj):\BIj',\BIj\in\BIJ,\;\BIj'\le\BIj\big\}$ can be represented as a linear combination of $\big\{M(\,\cdot\,;\BIj,\BIj):\BIj\in\BIJ\big\}$.}
        
        \item%
        \textit{Step~7: defining $g_{\BIT^{-1}(\BIl(\BIj))}(\BIx):=M\big(\BIT(\BIx);\BIj,\BIj\big)$ $\forall\BIx\in\R^d$, $\forall \BIj\in\BIJ$ and showing that $\big\{g_{\BIT^{-1}(\BIl(\BIj))}:\BIj\in\BIJ\big\}\subset\lspan_1(\CG_0)$ 
        contains non-negative and continuous functions which
        satisfy the properties \ref{defs:cover-vif-normalize} and \ref{defs:cover-vif-disjoint}.}
        
        \item%
        \textit{Step~8: showing that the functions $\CG_{p\text{-}\mathsf{RFS}}$ satisfy $\CG_{p\text{-}\mathsf{RFS}}\subset\lspan_1(\CG_p)$ and the property~\ref{defs:cover-rf-bound}.}
    \end{itemize}

    \textit{Step~1.}
    One can observe from the definition of $\FC$ that
    \begin{align}
    \begin{split}
    \FF(\FC)&=\Big\{I_1\times\cdots\times I_d: I_i\in\big\{(-\infty,\kappa_{i,0}],\{\kappa_{i,0}\},[\kappa_{i,0},\kappa_{i,1}],\ldots,\\
    &\hspace{140pt}[\kappa_{i,n_i-1},\kappa_{i,n_i}],\{\kappa_{i,n_i}\},[\kappa_{i,n_i},\infty)\big\}\;\forall 1\le i\le d\Big\},
    \end{split}
    \label{eqn:cover-cube-proof-face-original}
    \end{align}
    and that $V(\FC)=\bigtimes_{i=1}^d\{\kappa_{i,0},\ldots,\kappa_{i,n_i}\}$.
    Let us define $\BIT:\R^d\to\R^d$ as follows:
    \begin{align*}
        \BIT(\BIx)&:=\big(\beta_1^{-1}(x_1-\underline{\kappa}_1),\ldots,\beta_d^{-1}(x_d-\underline{\kappa}_d)\big)^\TRANSP \qquad \forall\BIx=(x_1,\ldots,x_d)^\TRANSP\in\R^d,
    \end{align*}
    and define $\BIL:=\bigtimes_{i=1}^{d}\{0,1,\ldots,n_i\}$.
    Notice that $\BIT$ is affine and bijective, and it holds that
    \begin{align}
        \BIT(V(\FC))&=\bigtimes_{i=1}^{d}\{0,\ldots,n_i\}=:\BIL, 
        \label{eqn:cover-cube-proof-vertex-rep}\\
        \begin{split}
        F\in\FF(\FC) \; & \Rightarrow\; \exists \BIk_{\TL}\in\bigtimes_{i=1}^{d}\{-\infty,0,\ldots,n_i\},\; \exists \BIk_{\TR}\in\bigtimes_{i=1}^{d}\{0,\ldots,n_i,\infty\},\\
        &\phantom{\Rightarrow}\hspace{9.0pt} \BIk_{\TL}\le \BIk_{\TR},\; \BIT(F)=[\BIk_{\TL},\BIk_{\TR}],\; \BIT(V(F))=[\BIk_{\TL},\BIk_{\TR}] \cap \BIL.
        \end{split}
        \label{eqn:cover-cube-proof-face-rep}
    \end{align}

    \textit{Step~2.}
    Let us define $\BIJ:=\bigtimes_{i=1}^{d}\{0,1,\ldots,n_i-1,\infty\}$,
    and define $\BIp(\,\cdot\,)$ as follows:
    \begin{align*}
        p_i(j)&:=\begin{cases}
            -\infty & \text{if }j=0,\\
            j-1 & \text{if }j\in\{1,\ldots,n_i-1\},\\
            n_i-1 & \text{if }j=\infty,
        \end{cases}
        \qquad\forall 1\le i\le d,\\
        \BIp(\BIj)&:=\big(p_1(j_1),\ldots,p_d(j_d)\big)^\TRANSP \hspace{1pt}\qquad \forall \BIj=(j_1,\ldots,j_d)^\TRANSP\in\BIJ.
    \end{align*}
    Let the expanded index set $\overline{\BIJ}$ be defined as 
    $\overline{\BIJ}:=\BIJ\cup\BIp(\BIJ)=\bigtimes_{i=1}^{d}\{-\infty,0,\ldots,n_i-1,\infty\}$.
    Moreover, let $\BIl(\,\cdot\,)$ be defined as follows:
    \begin{align*}
        l_i(j)&:=\begin{cases}
            j & \text{if }j\in\{0,\ldots,n_i-1\}, \\
            n_i & \text{if }j=\infty, 
        \end{cases}
        \hspace{14.5pt}\qquad \forall 1\le i\le d,\\
        \BIl(\BIj)&:=\big(l_1(j_1),\ldots,l_d(j_d)\big)^\TRANSP \qquad \forall \BIj=(j_1,\ldots,j_d)^\TRANSP\in\BIJ.
    \end{align*}
    Note that $\BIl:\BIJ\to\BIL$ is a bijection.

    \textit{Step~3.}
    To begin, let us define
    \begin{align*}
        b(\BIz;\BIy)&:=\bigvee_{i=1}^{d}(z_i-y_i)^+ \qquad \forall \BIz=(z_1,\ldots,z_d)^\TRANSP\in\R^{d},\; \forall \BIy=(y_1,\ldots,y_d)^\TRANSP\in\big(\R\cup\{\infty\}\big)^d.
    \end{align*}
    Subsequently, let the class of functions 
    $\big\{\big(Q(\,\cdot\,;\BIj):\R^d\to[0,1]\big):\BIj\in\overline{\BIJ}\big\}$
    be defined as follows:
    \begin{align*}
        Q(\BIz;\BIj)&:=\begin{cases}
            1 - b(\BIz;\BIj) + b(\BIz;\BIj+\vecone) & \forall \BIj\in\BIJ,\\
            0 & \forall \BIj\in\overline{\BIJ}\setminus\BIJ,
        \end{cases}\qquad \forall\BIz\in\R^d.
    \end{align*}
    One verifies by (\ref{eqn:cover-cube-proof-basis}) that 
    \begin{align}
        \big\{Q(\BIT(\,\cdot\,);\BIj):\BIj\in\overline{\BIJ}\big\}\subset \lspan_1(\CG_0).
        \label{eqn:cover-cube-proof-Q-span}
    \end{align}
    Moreover, one checks that $Q(\,\cdot\,;\cdot\,)$ can also be expressed as follows:
    \begin{align}
        \begin{split}
        Q(\BIz;\BIj)&= 1 \wedge \Bigg(1-\bigvee_{i=1}^{d}(z_i-j_i)\Bigg)^+ \\
        &= \bigwedge_{i=1}^{d}\big[\big(j_i-(z_i-1)\big)^{+}\wedge 1\big] \quad\;\;\; \forall \BIz=(z_1,\ldots,z_d)^\TRANSP\in\R^{d},\; \forall \BIj=(j_1,\ldots,j_d)^\TRANSP\in\overline{\BIJ}.
        \end{split}
        \label{eqn:cover-cube-proof-Q-continuity}
    \end{align}
    Subsequently, if we fix $\BIz\in\R^d$ and treat $\BIj$ as the argument,
    then one can observe that 
    $Q(\BIz;\BIj)$ resembles a distribution function.
    Concretely,
    let us define $\big\{\big(\widetilde{Q}(\,\cdot\,;\BIz):\big(\R\cup\{-\infty,\infty\}\big)^d\to[0,1]\big):\BIz\in\R^d\big\}$ as follows:
    \begin{align*}
        \widetilde{Q}(\BIy;\BIz)&:=\bigwedge_{i=1}^{d}\big[\big(y_i-(z_i-1)\big)^{+}\wedge 1\big]\\
        &\qquad \forall \BIy=(y_1,\ldots,y_d)^\TRANSP\in\big(\R\cup\{-\infty,\infty\}\big)^d,\; \forall \BIz=(z_1,\ldots,z_d)^\TRANSP\in\R^{d}.
    \end{align*}
    Then, notice that for $i=1,\ldots,d$, 
    $\big(\R\cup\{-\infty,\infty\}\big)\ni y_i\mapsto \big(y_i-(z_i-1)\big)^+\wedge 1 \in [0,1]$
    corresponds to the distribution of a random variable which is uniformly distributed on the interval $[z_i-1,z_i]\subset\R$.
    Consequently, the definition of the comonotonicity copula and Sklar's theorem (see, e.g., \citep[Equation~(5.7) \& Equation~(5.3) \& Theorem~5.3]{mcneil2005quantitative})
    guarantee that
    $\widetilde{Q}(\,\cdot\,;\BIz)$ is the distribution function of a random variable which is uniformly distributed on the line segment
    $\conv\big(\{\BIz-\vecone,\BIz\}\big)\subset\R^d$, for every $\BIz\in\R^d$.
    Let $\mu_{\BIz}\in\CP(\R^d)$ denote the law of this random variable, i.e.,
    $\mu_{\BIz}$ satisfies 
    \begin{align*}
        \mu_{\BIz}\big((-\vecinfty,\BIy]\big) = \widetilde{Q}(\BIy;\BIz) \qquad \forall \BIy\in \big(\R\cup\{-\infty,\infty\}\big)^d,\; \forall \BIz\in\R^d.
    \end{align*}
    In particular, it holds that
    \begin{align}
        Q(\BIz;\BIj)&=\widetilde{Q}(\BIj;\BIz)=\mu_{\BIz}\big((-\vecinfty,\BIj]\big) \qquad \forall \BIz\in\R^{d},\; \forall \BIj\in\overline{\BIJ},
        \label{eqn:cover-cube-proof-Q-meas}
    \end{align}
    and 
    \begin{align}
        \mu_{\BIz}\big((\BIz-\vecone,\BIz)\big)=1 \qquad \forall \BIz\in\R^{d}.
        \label{eqn:cover-cube-proof-Q-concentration}
    \end{align}
    This completes Step~3.

    \textit{Step~4.}
    Let $\Bsigma:\overline{\BIJ}\times\overline{\BIJ}\times\{0,1\}^d\to\overline{\BIJ}$ and the class of functions 
    $\big\{\big(M(\,\cdot\,;\BIj',\BIj):\R^d\to[0,1]\big):\BIj',\BIj\in\BIJ,\;\BIj'\le\BIj\big\}$
    be defined as follows:
    \begin{align}
        \sigma_i(j',j,\iota)&:=\begin{cases}
            j & \text{if }\iota=0,\\
            j' & \text{if }\iota=1,
        \end{cases}\qquad \forall j',j\in\{-\infty,0,\ldots,n_i-1,\infty\},\nonumber\\
        \Bsigma(\BIj',\BIj,\Biota)&:=\big(\sigma_1(j_1',j_1,\iota_1),\ldots,\sigma_d(j_d',j_d,\iota_d)\big)^\TRANSP \nonumber\\
        &\qquad \forall \BIj'=(j_1',\ldots,j_d')^\TRANSP,\BIj=(j_1,\ldots,j_d)^\TRANSP\in\overline{\BIJ},\; \forall \Biota=(\iota_1,\ldots,\iota_d)^\TRANSP\in\{0,1\}^d. \nonumber\allowdisplaybreaks\\
        M(\BIz;\BIj',\BIj)&:=\sum_{\Biota\in\{0,1\}^d}(-1)^{\|\Biota\|_1} Q\big(\BIz;\Bsigma(\BIp(\BIj'),\BIj,\Biota)\big) \qquad \forall \BIz\in\R^d,\; \forall \BIj',\BIj\in\BIJ,\; \BIj'\le \BIj.
        \label{eqn:cover-cube-proof-M-def}
    \end{align}
    It subsequently follows from (\ref{eqn:cover-cube-proof-Q-meas}),
    (\ref{eqn:cover-cube-proof-Q-concentration}), and the following combinatorial identity:
    \begin{align*}
        \INDI_{(\BIp(\BIj'),\BIj]}(\BIz) = \sum_{\Biota\in\{0,1\}^d}(-1)^{\|\Biota\|_1}\INDI_{\big(-\vecinfty,\Bsigma(\BIp(\BIj'),\BIj,\Biota)\big]}(\BIz) \qquad \forall \BIz\in\R^{d},\; \forall \BIj',\BIj\in\BIJ,\;\BIj'\le\BIj
    \end{align*}
    that the functions $\big\{M(\,\cdot\,;\BIj',\BIj):\BIj',\BIj\in\BIJ,\;\BIj'\le\BIj\big\}$ possess the following property:
    \begin{align}
        M(\BIz;\BIj',\BIj)&=\mu_{\BIz}\big((\BIp(\BIj'),\BIj]\big) = \mu_{\BIz}\big((\BIp(\BIj'),\BIj] \cap (\BIz-\vecone,\BIz)\big) \quad \forall \BIz\in\R^{d},\; \forall \BIj',\BIj\in\BIJ,\;\BIj'\le\BIj.
        \label{eqn:cover-cube-proof-M-meas}
    \end{align}

    \textit{Step~5.}
    Let us define $\BIK_{\TL}:=\bigtimes_{i=1}^{d}\{-\infty,0,\ldots,n_i\}$, $\BIK_{\TR}:=\bigtimes_{i=1}^{d}\{0,\ldots,n_i,\infty\}$.
    Recall from (\ref{eqn:cover-cube-proof-face-rep}) that
    if $F\in\FF(\FC)$, then $\BIT(F)=[\BIk_{\TL},\BIk_{\TR}]$ where $\BIk_{\TL}\in\BIK_{\TL}$, $\BIk_{\TR}\in\BIK_{\TR}$.
    Subsequently, let us define $\BIq_{\TL}:\BIK_{\TL}\to\BIJ$,
    $\BIq_{\TR}:\BIK_{\TR}\to\BIJ$ as follows:
    \begin{align*}
        q_{i,\TL}(k)&:=\begin{cases}
            0 & \text{if }k=-\infty,\\
            k & \text{if }k\in\{0,\ldots,n_i-1\}, \\
            \infty & \text{if }k=n_i,
        \end{cases}\qquad \forall 1\le i\le d,\\
        q_{i,\TR}(k)&:=\begin{cases}
            k & \text{if }k\in\{0,\ldots,n_i-1\}, \\
            \infty & \text{if }k\in\{n_i,\infty\},
        \end{cases}\hspace{1pt}\qquad \forall 1\le i\le d,\\
        \BIq_{\TL}(\BIk_{\TL})&:= \big(q_{1,\TL}(k_1),\ldots,q_{d,\TL}(k_d)\big)^\TRANSP \hspace{0.8pt}\qquad\qquad \forall \BIk_{\TL}=(k_1,\ldots,k_d)^\TRANSP\in\BIK_{\TL}, \\
        \BIq_{\TR}(\BIk_{\TR})&:= \big(q_{1,\TR}(k_1),\ldots,q_{d,\TR}(k_d)\big)^\TRANSP \hspace{-0.9pt}\qquad\qquad \forall \BIk_{\TR}=(k_1,\ldots,k_d)^\TRANSP\in\BIK_{\TR}.
    \end{align*}
    With these notions, we will establish the following properties of 
    $\big\{M(\,\cdot\,;\BIj',\BIj):\BIj',\BIj\in\BIJ,\;\BIj'\le\BIj\big\}$:
    \begin{align}
        \begin{split}
            &\forall \BIj\in\BIJ,\; \forall \BIk_{\TL}\in\BIK_{\TL},\; \BIk_{\TR}\in\BIK_{\TR},\; \forall \BIz\in\R^{d}:\\
            &\qquad\BIz\in[\BIk_{\TL},\BIk_{\TR}],\; \BIl(\BIj)\notin[\BIk_{\TL},\BIk_{\TR}] \; \Rightarrow\; \begin{cases}
                M\big(\BIl(\BIj); \BIq_{\TL}(\BIk_{\TL}),\BIq_{\TR}(\BIk_{\TR})\big) = 0, \\
                M\big(\BIz; \BIq_{\TL}(\BIk_{\TL}),\BIq_{\TR}(\BIk_{\TR})\big) = 1, \\
                M(\BIz;\BIj,\BIj)=0.
            \end{cases}
        \end{split}
        \label{eqn:cover-cube-proof-M-01}
    \end{align}
    In view of the property~(\ref{eqn:cover-cube-proof-M-meas})
    and the definitions of $\BIq_{\TL}(\,\cdot\,)$, $\BIq_{\TR}(\,\cdot\,)$, $\BIp(\,\cdot\,)$, $\BIl(\,\cdot\,)$,
    it suffices to prove the following coordinate-wise statement:
    \begin{align*}
        &\forall i\in\{1,\ldots,d\},\; \forall j\in\{0,\ldots,n_i-1,\infty\},\; 
        \forall k_{\TL}\in\{-\infty,0,\ldots,n_i\},\; 
        \forall k_{\TR}\in\{0,\ldots,n_i,\infty\},\;
        \forall z\in\R\!:\\
        &\qquad z\in [k_{\TL},k_{\TR}],\; l_i(j)\notin [k_{\TL},k_{\TR}] \;\Rightarrow\; \begin{cases}
            \big(p_i(q_{i,\TL}(k_{\TL})),q_{i,\TR}(k_{\TR})\big] \cap \big(l_i(j)-1,l_i(j)\big) = \emptyset,\\
            \big(p_i(q_{i,\TL}(k_{\TL})),q_{i,\TR}(k_{\TR})\big] \supset (z-1,z),\\
            \big(p_i(j),j\big] \cap (z-1, z) = \emptyset,
        \end{cases}
    \end{align*}
    which can be directly verified from 
    the definitions of 
    $\big(q_{i,\TL}(\,\cdot\,)\big)_{i=1:d}$, 
    $\big(q_{i,\TR}(\,\cdot\,)\big)_{i=1:d}$,
    $\big(p_{i}(\,\cdot\,)\big)_{i=1:d}$, and
    $\big(l_{i}(\,\cdot\,)\big)_{i=1:d}$.

    \textit{Step~6.}
    Similar to Step~5, 
    we can show the following property of 
    $\big\{M(\,\cdot\,;\BIj',\BIj):\BIj',\BIj\in\BIJ,\;\BIj'\le\BIj\big\}$:
    \begin{align}
        M\big(\BIl(\BIj');\BIj,\BIj\big)&= \INDI_{\{\BIj=\BIj'\}} \qquad \forall \BIj,\BIj'\in\BIJ
        \label{eqn:cover-cube-proof-M-orthonormality}
    \end{align}
    by (\ref{eqn:cover-cube-proof-M-meas}), the definitions of $\BIp(\,\cdot\,)$, $\BIl(\,\cdot\,)$, and by directly verifying the following coordinate-wise statements:
    \begin{align*}
        &\forall i\in\{1,\ldots,d\},\; \forall j, j' \in\{0,\ldots,n_i-1,\infty\}: \\
        &\qquad  j\ne j' \;\Rightarrow\; \big(p_i(j),j\big]\cap \big(l_i(j')-1,l_i(j')\big)=\emptyset, \\
        &\qquad  j=j' \; \Rightarrow\; \big(p_i(j),j\big]\supset \big(l_i(j')-1,l_i(j')\big).
    \end{align*}
    Thus, we conclude by the property (\ref{eqn:cover-cube-proof-M-orthonormality}) that the functions 
    $\big\{M(\,\cdot\,;\BIj,\BIj):\BIj\in\BIJ\big\}$ are linearly independent.
    Moreover, observe from (\ref{eqn:cover-cube-proof-M-def}) that each function in 
    $\big\{M(\,\cdot\,;\BIj',\BIj):\BIj',\BIj\in\BIJ,\;\BIj'\le\BIj\big\}$
    is a linear combination of the functions 
    $\big\{Q(\,\cdot\,;\BIj):\BIj\in\BIJ\big\}$;
    recall that 
    $Q(\,\cdot\,;\BIj)=0$ for all $\BIj\in\overline{\BIJ}\setminus\BIJ$.
    Therefore, $\big\{M(\,\cdot\,;\BIj,\BIj):\BIj\in\BIJ\big\}$ form a basis of the $|\BIJ|$-dimensional linear space spanned by $\big\{M(\,\cdot\,;\BIj',\BIj):\BIj',\BIj\in\BIJ,\;\BIj'\le\BIj\big\}$,
    and it holds by (\ref{eqn:cover-cube-proof-M-orthonormality}) that
    \begin{align}
        M(\BIz;\hat{\BIj}',\hat{\BIj})&=\sum_{\BIj\in\BIJ}M\big(\BIl(\BIj);\hat{\BIj}',\hat{\BIj}\big)M(\BIz;\BIj,\BIj) \qquad \forall \BIz\in\R^d,\; \forall \hat{\BIj}',\hat{\BIj}\in\BIJ,\; \hat{\BIj}'\le\hat{\BIj}.
        \label{eqn:cover-cube-proof-M-representation}
    \end{align}

    \textit{Step~7.}
    Now, let us define $g_{\BIT^{-1}(\BIl(\BIj))}(\BIx):=M\big(\BIT(\BIx);\BIj,\BIj\big)$ $\forall\BIx\in\R^d$, $\forall \BIj\in\BIJ$.
    First, observe from (\ref{eqn:cover-cube-proof-vertex-rep}) and the bijectivity of $\BIl$ that $\BIT^{-1}(\BIl(\BIJ))=\BIT^{-1}(\BIL)=V(\FC)$.
    Thus, for each $\BIv\in V(\FC)$,
    there exists a unique $\BIj\in\BIJ$ such that $\BIT^{-1}(\BIl(\BIj))=\BIv$.
    This means that 
    we can re-index the class of functions $\big\{g_{\BIT^{-1}(\BIl(\BIj))}:\BIj\in\BIJ\big\}$ by $\big\{g_{\BIv}:\BIv\in V(\FC)\big\}$.
    Second, for every $\BIj\in\BIJ$, 
    it follows from (\ref{eqn:cover-cube-proof-M-meas}), (\ref{eqn:cover-cube-proof-Q-continuity}), (\ref{eqn:cover-cube-proof-M-def}), and (\ref{eqn:cover-cube-proof-Q-span}) that 
    $g_{\BIT^{-1}(\BIl(\BIj))}$ is non-negative, continuous, and
    $g_{\BIT^{-1}(\BIl(\BIj))}\in\lspan_1(\CG_0)$.
    It remains to show that 
    $\big\{g_{\BIv}:\BIv\in V(\FC)\big\}$ satisfies the properties \ref{defs:cover-vif-normalize} and \ref{defs:cover-vif-disjoint}.
    To that end, let us fix an arbitrary face $F\in\FF(\FC)$ as well as an arbitrary point $\BIx\in F$.
    It follows from (\ref{eqn:cover-cube-proof-face-rep}) that there exist
    $\BIk_{\TL}\in\BIK_{\TL}$ and $\BIk_{\TR}\in\BIK_{\TR}$ which satisfy
    $\BIk_{\TL}\le\BIk_{\TR}$,
    $\BIT(F)=[\BIk_{\TL},\BIk_{\TR}]$, and
    $\BIT(V(F))=[\BIk_{\TL},\BIk_{\TR}]\cap\BIL=\big\{\BIl(\BIj):\BIl(\BIj)\in [\BIk_{\TL},\BIk_{\TR}],\; \BIj\in\BIJ\big\}$.
    Since $\BIT(\BIx)\in[\BIk_{\TL},\BIk_{\TR}]$, it subsequently holds by the properties (\ref{eqn:cover-cube-proof-M-01}), 
    (\ref{eqn:cover-cube-proof-M-representation}), and (\ref{eqn:cover-cube-proof-M-orthonormality}) that
    \begin{align*}
        1&=M\big(\BIT(\BIx);\BIq_{\TL}(\BIk_{\TL}),\BIq_{\TR}(\BIk_{\TR})\big)=\sum_{\BIj\in\BIJ}M\big(\BIl(\BIj);\BIq_{\TL}(\BIk_{\TL}),\BIq_{\TR}(\BIk_{\TR})\big)M\big(\BIT(\BIx);\BIj,\BIj\big)\\
        &=\sum_{\BIj\in\BIJ}\INDI_{\big\{\BIl(\BIj)\in[\BIk_{\TL},\BIk_{\TR}]\big\}}g_{\BIT^{-1}(\BIl(\BIj))}(\BIx)=\sum_{\BIv\in V(F)}g_{\BIv}(\BIx).
    \end{align*}
    This shows that $\big\{g_{\BIv}:\BIv\in V(\FC)\big\}$ satisfies the property \ref{defs:cover-vif-normalize}.
    Moreover, let us take an arbitrary $\hat{\BIv}\in V(\FC)\setminus V(F)$ and let $\hat{\BIj}:=\BIl^{-1}(\BIT(\hat{\BIv}))$.
    It holds that $\BIl(\hat{\BIj})=\BIT(\hat{\BIv})\notin [\BIk_{\TL},\BIk_{\TR}]$,
    and thus the property~(\ref{eqn:cover-cube-proof-M-01}) implies that
    \begin{align*}
        g_{\hat{\BIv}}(\BIx)&=g_{\BIT^{-1}(\BIl(\hat{\BIj}))}(\BIx)=M(\BIT(\BIx);\hat{\BIj},\hat{\BIj})=0.
    \end{align*}
    This shows that $\big\{g_{\BIv}:\BIv\in V(\FC)\big\}$ satisfies the property \ref{defs:cover-vif-disjoint}.
    In summary, we have shown that $\CG_{\mathsf{VIFS}}:=\big\{g_{\BIv}:\BIv\in V(\FC)\big\}\subset\lspan_1(\CG_0)\subset\lspan_1(\CG_p)$ is a VIFS for both $\FC$ and $\FC_0$,
    and the proof of statement~\ref{props:cover-cube-bounded} is now complete.

    \textit{Step~8.}
    Now, let us fix an arbitrary $p\in[1,\infty)$.
    For $i=1,\ldots,d$, let $\BIe_i$ denote the $i$-th standard basis vector of $\R^d$.
    One checks from the definition of $\FC$ that $D(\FC)=\{-\BIe_1,\BIe_1,\ldots,-\BIe_d,\BIe_d\}$. 
    Let us label the functions in $\CG_{p\text{-}\mathsf{RFS}}$ as follows:
    \begin{align*}
    \overline{g}_{\BIe_i}(\BIx)&:=\big(C_{\|\cdot\|}(x_i-\kappa_{i,n_i})^+\big)^p \qquad \forall \BIx=(x_1,\ldots,x_d)^\TRANSP\in\R^{d},\; \forall 1\le i\le d,\\
    \overline{g}_{-\BIe_i}(\BIx)&:=\big(C_{\|\cdot\|}(\kappa_{i,0}-x_i)^+\big)^p \hspace{4.0pt} \qquad \forall \BIx=(x_1,\ldots,x_d)^\TRANSP\in\R^{d},\; \forall 1\le i\le d.
    \end{align*}
    In the case where $p=1$, it follows from the definition of $\CG_1$ that $\overline{g}_{\BIe_i}\in\lspan_1(\CG_1)$ for $i=1,\ldots,d$. 
    Moreover, due to the following identity
    \begin{align*}
    C_{\|\cdot\|}(\kappa_{i,0}-x_i)^+=C_{\|\cdot\|}\kappa_{i,0}-C_{\|\cdot\|} x_i+C_{\|\cdot\|}(x_i-\kappa_{i,0})^+ \qquad \forall 1\le i\le d,
    \end{align*}
    we have $\overline{g}_{-\BIe_i}\in\lspan_1(\CG_1)$ for $i=1,\ldots,d$.
    In the case where $p>1$, it follows from the definition of $\CG_p$ that $\overline{g}_{\BIe_i},\overline{g}_{-\BIe_i}\in\lspan_1(\CG_p)$ for $i=1,\ldots,d$.
    This shows that $\CG_{p\text{-}\mathsf{RFS}}\subset\lspan_1(\CG_p)$.
    To prove the property \ref{defs:cover-rf-bound}, let us fix an arbitrary $F\in\FF(\FC)$ that is unbounded, and fix an arbitrary $\BIx=(x_1,\ldots,x_d)^\TRANSP\in F$. 
    By (\ref{eqn:cover-cube-proof-face-original}), $F$ can be expressed as $F=\bigtimes_{i=1}^{d}I_i$, where $I_i\in\big\{(-\infty,\kappa_{i,0}],\{\kappa_{i,0}\},[\kappa_{i,0},\kappa_{i,1}],\ldots,[\kappa_{i,n_i-1},\kappa_{i,n_i}],\{\kappa_{i,n_i}\},[\kappa_{i,n_i},\infty)\big\}$ for $i=1,\ldots,d$. Observe that for $i=1,\ldots,d$, $\BIe_i\in D(F)$ if and only if $I_i=[\kappa_{i,n_i},\infty)$, and that $-\BIe_i\in D(F)$ if and only if $I_i=(-\infty,\kappa_{i,0}]$. 
    Let us define $(\widetilde{I}_i)_{i=1:d}$ as follows:
    \begin{align*}
    \widetilde{I}_i:=\begin{cases}
    \{\kappa_{i,n_i}\} & \text{if }\BIe_i\in D(F),\\
    \{\kappa_{i,0}\} & \text{if }{-\BIe_i}\in D(F),\\
    I_i & \text{if }\BIe_i\notin D(F)\text{ and }{-\BIe_i}\notin D(F),
    \end{cases}\qquad \forall 1\le i\le d.
    \end{align*}
    We thus get $\conv(V(F))=\bigtimes_{i=1}^{d}\widetilde{I}_i$. 
    Now, let us define
    \begin{align*}
        \hat{x}_i&:=\begin{cases}
            \kappa_{i,n_i} & \text{if }\BIe_i\in D(F), \\
            \kappa_{i,0} & \text{if }{-\BIe_i}\in D(F),\\
            x_i & \text{if }\BIe_i\notin D(F)\text{ and }{-\BIe_i}\notin D(F),
        \end{cases}\qquad \forall 1\le i\le d,
    \end{align*}
    and define $\hat{\BIx}:=(\hat{x}_1,\ldots,\hat{x}_d)^\TRANSP$. 
    Hence, it holds that $\hat{\BIx}\in\conv(V(F))$.
    From the above definitions, it follows that
    \begin{align*}
        |x_i-\hat{x}_i|^p&=\begin{cases}
            \big((x_i-\kappa_{i,n_i})^+\big)^p=C_{\|\cdot\|}^{-p}\overline{g}_{\BIe_i}(\BIx) & \text{if }\BIe_i\in D(F), \\
            \big((\kappa_{i,0}-x_i)^+\big)^p=C_{\|\cdot\|}^{-p}\overline{g}_{-\BIe_i}(\BIx) & \text{if }{-\BIe_i}\in D(F),\\
            0 & \text{if }\BIe_i\notin D(F)\text{ and }{-\BIe_i}\notin D(F),
        \end{cases} \qquad \forall 1\le i\le d,
    \end{align*}
    and thus $\|\BIx-\hat{\BIx}\|_p^p=\sum_{i=1}^{d}|x_i-\hat{x}_i|^p=C_{\|\cdot\|}^{-p}\sum_{\BIu\in D(F)}\overline{g}_{\BIu}(\BIx)$.
    Consequently, we get
    \begin{align*}
    \left(\min_{\BIy\in\conv(V(F))}\big\{\|\BIx-\BIy\|\big\}\right)^p&\le\|\BIx-\hat{\BIx}\|^p\le C_{\|\cdot\|}^p\|\BIx-\hat{\BIx}\|_p^p=\sum_{\BIu\in D(F)}\overline{g}_{\BIu}(\BIx).
    \end{align*}
    Therefore, we have shown that $\CG_{p\text{-}\mathsf{RFS}}$ satisfies the property \ref{defs:cover-rf-bound}.
    The proof is now complete.
\end{proof}

% Proof of Theorem (moment set $W_p$ control)
\begin{proof}[Proof of Theorem~\ref{thm:momentset-euclidean}]
    Since all functions in 
    a VIFS or a $p$-RFS are defined on $\bigcup_{C\in\FC}C$ and every $\mu\in\CP_p(\CY;\CG)$ can be extended to a probability measure in $\CP_p\big(\bigcup_{C\in\FC}C;\CG\big)$, we can assume without loss of generality that $\CY=\bigcup_{C\in\FC}C$.
    Subsequently, let us recall from Lemma~\ref{lem:relint-partition} that $\big\{\relint(F):F\in\FF(\FC)\big\}$ is a disjoint partition of $\CY$. 
    Throughout the proof, let $\big\{g_{\BIv}:\BIv\in V(\FC)\big\}\subset\lspan_1(\CG)$ be a VIFS for $\FC$,
    and let $\big\{\overline{g}_{\BIu}:\BIu\in D(\FC)\big\}$ be a possibly empty $p$-RFS for $\FC$.
    The proof is divided into five steps as follows.
    \begin{itemize}
        \item%
        \emph{Step~1: decomposing an arbitrary $\mu\in\CP_p(\CY;\CG)$ into a mixture of probability measures
        $\big\{\mu_F\in\CP_p(\CY):F\in\FF_{\mu}(\FC)\big\}$
        where $\FF_{\mu}(\FC):=\big\{F\in\FF(\FC):\mu(\relint(F))>0\big\}$,
        and each $\mu_F$ is concentrated on $\relint(F)$ for $F\in\FF_{\mu}(\FC)$.}

        \item%
        \emph{Step~2: approximating each $\mu_F$ by a discrete $\hat{\mu}_F\in\CP_p(\CY)$ whose support lies in~$V(F)$.}

        \item%
        \emph{Step~3: deriving an upper bound for each $W_p(\mu_F,\hat{\mu}_F)$.}

        \item%
        \emph{Step~4: taking the mixture of $\big\{\hat{\mu}_{F}:F\in\FF_{\mu}(\FC)\big\}$ to make a discrete approximation $\hat{\mu}\in\CP_p(\CY)$ of $\mu$ and deriving an upper bound for $W_p(\mu,\hat{\mu})$.}
        
        \item%
        \emph{Step~5: showing that if $\nu\in\CP_p(\CY;\CG)$ satisfies $\nu\overset{\CG}{\sim}\mu$, then the discrete approximation $\hat{\nu}\in\CP_p(\CY)$ of $\nu$ constructed via Steps~1--4 is equal to $\hat{\mu}$, yielding the upper bound $W_p(\mu,\nu)\le W_p(\mu,\hat{\mu})+W_{p}(\nu,\hat{\nu})$.}

    \end{itemize}

    \emph{Step~1.}
    Let us fix an arbitrary $\mu\in\CP_p(\CY;\CG)$ and define
    \begin{align*}
        \FF_{\mu}(\FC):=\big\{F\in\FF(\FC):\mu(\relint(F))>0\big\}.
    \end{align*} 
    Then, it follows from Lemma~\ref{lem:relint-partition} that $\mu\left(\bigcup_{F\in\FF_{\mu}(\FC)}\relint(F)\right)=1$. 
    Subsequently, let us define $\big\{\mu_{F}\in\CP_p(\CY):F\in\FF_{\mu}(\FC)\big\}$ as follows:
    \begin{align*}
        \mu_F(\DIFF\BIx):=\frac{\INDI_{\relint(F)}}{\mu(\relint(F))}\mu(\DIFF\BIx)\qquad \forall F\in\FF_{\mu}(\FC).
    \end{align*}
    It thus holds that
    \begin{align*}
        \sum_{F\in\FF_{\mu}(\FC)}\mu(\relint(F))\mu_F(E)&=\mu\Big(\big({\textstyle\bigcup_{F\in\FF_{\mu}(\FC)}}\relint(F)\big)\cap E\Big)=\mu(E) \qquad \forall E\in\CB(\CY),
    \end{align*}
    and hence we get
    \begin{align}
        \sum_{F\in\FF_{\mu}(\FC)}\mu(\relint(F))\mu_F(\DIFF\BIx)=\mu(\DIFF\BIx).
        \label{eqn:momentset-euclidean-proof-mixtureweight-sum}
    \end{align}

    \emph{Step~2.}
    Let $(\hat{\mu}_F)_{F\in\FF_{\mu}(\FC)}$ be defined as follows:
    \begin{align*}
        \hat{\mu}_F(\DIFF\BIx):=\sum_{\BIv\in V(F)}\bigg(\int_{\CY}g_{\BIv}\DIFFX{\mu_F}\bigg)\delta_{\BIv}(\DIFF\BIx) \qquad \forall F\in\FF_{\mu}(\FC).
    \end{align*}
    Then, using the property~\ref{defs:cover-vif-normalize} of $\big\{g_{\BIv}:\BIv\in V(\FC)\big\}$, we get 
    \begin{align*}
        \hat{\mu}_F(\CY)&=\sum_{\BIv\in V(F)}\int_{\CY}g_{\BIv}\DIFFX{\mu_F}=\int_{\CY}\Big(\textstyle\sum_{\BIv\in V(F)}g_{\BIv}\Big)\DIFFX{\mu_F}\\
        &=\frac{1}{\mu(\relint(F))}\int_{\CY}\Big(\textstyle\sum_{\BIv\in V(F)}g_{\BIv}\Big)\INDI_{\relint(F)}\DIFFX{\mu}=1 \qquad \forall F\in\FF_{\mu}(\FC).
    \end{align*}
    Since $\int_{\CY}g_{\BIv}\DIFFX{\mu_F}\ge 0$ for all $\BIv\in V(\FC)$, it holds that $\hat{\mu}_F\in\CP_p(\CY)$ for all $F\in\FF_{\mu}(\FC)$. 
    Moreover, for every $\BIv\in V(\FC)$, it follows from the orthonormality property of $\big\{g_{\BIv}:\BIv\in V(\FC)\big\}$
    in (\ref{eqn:cover-vif-orthonormality})
    that 
    \begin{align}
        \int_{\CY}g_{\BIv}\DIFFX{\hat{\mu}_F}=\sum_{\BIv'\in V(F)}\bigg(\int_{\CY}g_{\BIv'}\DIFFX{\mu_F}\bigg)g_{\BIv}(\BIv')=\int_{\CY}g_{\BIv}\DIFFX{\mu_F} \qquad \forall \BIv\in V(F),\;\forall F\in\FF_{\mu}(\FC).
        \label{eqn:momentset-euclidean-proof-discrete-eq1}
    \end{align}
    Moreover, it follows from the property \ref{defs:cover-vif-disjoint} that 
    \begin{align}
        \int_{\CY}g_{\BIv}\DIFFX{\hat{\mu}_F}=0=\int_{\CY}g_{\BIv}\DIFFX{\mu_F} \qquad \forall \BIv\in V(\FC)\setminus V(F),\;\forall F\in\FF_{\mu}(\FC).
        \label{eqn:momentset-euclidean-proof-discrete-eq2}
    \end{align}
    Combining (\ref{eqn:momentset-euclidean-proof-discrete-eq1}) and (\ref{eqn:momentset-euclidean-proof-discrete-eq2}) yields
    \begin{align}
    \int_{\CY}g_{\BIv}\DIFFX{\hat{\mu}_F}=\int_{\CY}g_{\BIv}\DIFFX{\mu_F}\qquad \forall \BIv\in V(\FC),\;\forall F\in\FF_{\mu}(\FC).
    \label{eqn:momentset-euclidean-proof-discrete-characterization}
    \end{align}

    \emph{Step~3.}
    For each $F\in\FF_{\mu}(\FC)$ that is bounded
    and for any $\gamma_{F}\in\Gamma(\mu_F,\hat{\mu}_F)$,
    it holds that $\support(\gamma_F)\subseteq F\times F = \conv(V(F))\times \conv(V(F))$. 
    It thus follows from a convex maximization argument that
    \begin{align}
    \begin{split}
        W_p(\mu_F,\hat{\mu}_F)&=\inf_{\gamma_F\in\Gamma(\mu_F,\hat{\mu}_F)}\bigg\{\int_{\CY\times\CY}\|\BIx-\BIz\|^p\DIFFM{\gamma_F}{\DIFF\BIx,\DIFF\BIz}\bigg\}^{\frac{1}{p}}
        \\
        &\le \max_{\BIv,\BIv'\in V(F)}\big\{\|\BIv-\BIv'\|^p\big\}^{\frac{1}{p}}\le\eta(\FC) \qquad \forall F\in\FF_{\mu}(\FC),\; F\text{ is bounded}.
    \end{split}
    \label{eqn:momentset-euclidean-proof-coupling-boundedface}
    \end{align}
    For each unbounded $F\in\FF_{\mu}(\FC)$, let $T_F:F\to F$ be a Borel measurable mapping such that $T_F(\BIx)\in\argmin_{\BIz\in\conv(V(F))}\big\{\|\BIx-\BIz\|\big\}$ for all $\BIx\in F$, which exists by \citep[Proposition~7.33]{bertsekas1978stochastic}.
    Subsequently, each $\gamma_F\in\Gamma(\mu_F\circ T_F^{-1},\hat{\mu}_F)$ satisfies $\support(\gamma_F)\subseteq \conv(V(F))\times \conv(V(F))$.
    It follows from the same convex maximization argument as above that
    \begin{align}
        W_p(\mu_F\circ T_{F}^{-1},\hat{\mu}_F)\le \eta(\FC) \qquad \forall F\in\FF_{\mu}(\FC),\; F\text{ is unbounded}.
        \label{eqn:momentset-euclidean-proof-coupling-unboundedface-1}
    \end{align}
    Moreover, the property \ref{defs:cover-rf-bound} of $\big\{\overline{g}_{\BIu}:\BIu\in D(\FC)\big\}$ implies that
    \begin{align}
        \begin{split}
        W_p(\mu_F,\mu_F\circ T_{F}^{-1}) &\le \bigg(\int_{\CY}\big\|\BIx-T_F(\BIx)\big\|^p\DIFFM{\mu_F}{\DIFF\BIx}\bigg)^{\frac{1}{p}}\\
        &=\bigg(\int_{\CY}\left(\min_{\BIz\in\conv(V(F))}\big\{\|\BIx-\BIz\|\big\}\right)^{p}\DIFFM{\mu_F}{\DIFF\BIx}\bigg)^{\frac{1}{p}}\\
        &\le \bigg(\int_{\CY}{\textstyle\sum_{\BIu\in D(F)}}\overline{g}_{\BIu}(\BIx)\DIFFM{\mu_F}{\DIFF\BIx}\bigg)^{\frac{1}{p}} \\
        &\le \Bigg(\sum_{\BIu\in D(\FC)}\int_{\CY}\overline{g}_{\BIu}\DIFFX{\mu_F}\Bigg)^{\frac{1}{p}} \qquad \forall F\in\FF_{\mu}(\FC),\; F\text{ is unbounded}.
        \end{split}
        \label{eqn:momentset-euclidean-proof-coupling-unboundedface-2}
    \end{align}
    Combining (\ref{eqn:momentset-euclidean-proof-coupling-unboundedface-1}) and (\ref{eqn:momentset-euclidean-proof-coupling-unboundedface-2}) leads to
    \begin{align}
    W_p(\mu_F,\hat{\mu}_F)\le\eta(\FC)+\Bigg(\sum_{\BIu\in D(\FC)}\int_{\CY}\overline{g}_{\BIu}\DIFFX{\mu_F}\Bigg)^{\frac{1}{p}} \qquad \forall F\in\FF_{\mu}(\FC),\; F\text{ is unbounded}.
    \label{eqn:momentset-euclidean-proof-coupling-unboundedface}
    \end{align}

    \emph{Step~4.}
    Now, let us define $\hat{\mu}:=\sum_{F\in\FF_{\mu}(\FC)}\mu(\relint(F))\hat{\mu}_F$.
    Since $\mu\Big(\bigcup_{F\in\FF_{\mu}(\FC)}\relint(F)\Big)=1$, we have $\hat{\mu}\in\CP_p(\CY)$. 
    Moreover, it holds that $\support(\hat{\mu})=\bigcup_{F\in\FF_{\mu}(\FC)}\support(\hat{\mu}_F)\subseteq V(\FC)$.
    Consequently, by the orthonormality property of $\big\{g_{\BIv}:\BIv\in V(\FC)\big\}$ in (\ref{eqn:cover-vif-orthonormality}) 
    as well as 
    (\ref{eqn:momentset-euclidean-proof-discrete-characterization})
    and 
    (\ref{eqn:momentset-euclidean-proof-mixtureweight-sum}), we get
    \begin{align}
    \begin{split}
    \hat{\mu}(\{\BIv\})&=\int_{\CY}g_{\BIv}\DIFFX{\hat{\mu}}=\sum_{F\in\FF_{\mu}(\FC)}\mu(\relint(F))\int_{\CY}g_{\BIv}\DIFFX{\hat{\mu}_F}\\
    &=\sum_{F\in\FF_{\mu}(\FC)}\mu(\relint(F))\int_{\CY}g_{\BIv}\DIFFX{\mu_F}=\int_{\CY}g_{\BIv}\DIFFX{\mu} \qquad \forall \BIv\in V(\FC).
    \end{split}
    \label{eqn:momentset-euclidean-proof-invariance}
    \end{align}
    Next, for each $F\in\FF_{\mu}(\FC)$, let $\gamma_{F}\in\Gamma(\mu_F,\hat{\mu}_F)$ be an optimal coupling of $\mu_F$ and $\hat{\mu}_F$
    with respect to the cost function $\CY\times\CY\ni (\BIx,\BIz)\mapsto \|\BIx-\BIz\|^p\in\R$,
    i.e.,
    $\int_{\CY\times\CY}\|\BIx-\BIz\|^p\DIFFM{\gamma_F}{\DIFF\BIx,\DIFF\BIz}=W_p(\mu_F,\hat{\mu}_F)^p$ $\forall F\in\FF_{\mu}(\FC)$.
    Subsequently, let us define 
    $\gamma:=\sum_{F\in\FF_{\mu}(\FC)}\mu(\relint(F))\gamma_F$.
    One checks that $\gamma\in\Gamma(\mu,\hat{\mu})$.
    In the case where $\FC$ is bounded, every $F\in\FF_{\mu}(\FC)$ is bounded, and it hence follows from (\ref{eqn:momentset-euclidean-proof-coupling-boundedface}) that
    \begin{align*}
        W_{p}(\mu,\hat{\mu})&\le\bigg(\int_{\CY\times \CY}\|\BIx-\BIz\|^p\DIFFM{\gamma}{\DIFF\BIx,\DIFF\BIz}\bigg)^\frac{1}{p}\\
        &=\Bigg(\sum_{F\in\FF_{\mu}(\FC)}\mu(\relint(F))\int_{\CY\times \CY}\|\BIx-\BIz\|^p\DIFFM{\gamma_F}{\DIFF\BIx,\DIFF\BIz}\Bigg)^{\frac{1}{p}}\\
        &=\Bigg(\sum_{F\in\FF_{\mu}(\FC)}\mu(\relint(F))W_p(\mu_F,\hat{\mu}_F)^{p}\Bigg)^{\frac{1}{p}} \le\eta(\FC).
    \end{align*}
    In the case where $\FC$ is not bounded, it follows from (\ref{eqn:momentset-euclidean-proof-coupling-boundedface}), (\ref{eqn:momentset-euclidean-proof-coupling-unboundedface}), and the Minkowski inequality that
    \begin{align*}
        W_p(\mu,\hat{\mu})&\le\bigg(\int_{\CY\times \CY}\|\BIx-\BIz\|^p\DIFFM{\gamma}{\DIFF\BIx,\DIFF\BIz}\bigg)^\frac{1}{p}\\
        &=\Bigg(\sum_{F\in\FF_{\mu}(\FC)}\mu(\relint(F))\int_{\CY\times \CY}\|\BIx-\BIz\|^p\DIFFM{\gamma_F}{\DIFF\BIx,\DIFF\BIz}\Bigg)^{\frac{1}{p}}\allowdisplaybreaks\\
        &= \Bigg(\sum_{F\in\FF_{\mu}(\FC)}\mu(\relint(F))W_p(\mu_F,\hat{\mu}_F)^p\Bigg)^{\frac{1}{p}}\allowdisplaybreaks\\
        &\le \Bigg(\sum_{F\in\FF_{\mu}(\FC)}\mu(\relint(F))\Bigg(\eta(\FC)+\Bigg(\sum_{\BIu\in D(\FC)}\int_{\CY}\overline{g}_{\BIu}\DIFFX{\mu_F}\Bigg)^{\frac{1}{p}}\Bigg)^p\Bigg)^{\frac{1}{p}}\allowdisplaybreaks\\
        &\le \Bigg(\sum_{F\in\FF_{\mu}(\FC)}\mu(\relint(F))\eta(\FC)^p\Bigg)^{\frac{1}{p}} + \Bigg(\sum_{F\in\FF_{\mu}(\FC)}\mu(\relint(F))\sum_{\BIu\in D(\FC)}\int_{\CY}\overline{g}_{\BIu}\DIFFX{\mu_F}\Bigg)^{\frac{1}{p}}\\
        &=\eta(\FC)+\Bigg(\sum_{\BIu\in D(\FC)}\int_{\CY}\overline{g}_{\BIu}\DIFFX{\mu}\Bigg)^{\frac{1}{p}}.
    \end{align*}
    In summary, we have derived the upper bound 
    $W_{p}(\mu,\hat{\mu})\le\eta(\FC)$ for the case where $\FC$ is bounded,
    and the upper bound
    $W_{p}(\mu,\hat{\mu})\le\eta(\FC)+\big(\sum_{\BIu\in D(\FC)}\int_{\CY}\overline{g}_{\BIu}\DIFFX{\mu}\big)^{\frac{1}{p}}$
    for the case where $\FC$ is not bounded.

    \emph{Step~5.}
    Finally, let us take an arbitrary $\nu\in\CP_p(\CY;\CG)$ which satisfies $\mu\overset{\CG}{\sim}\nu$, and repeat Steps~1--4 above to construct $\hat{\nu}\in\CP_p(\CY)$. 
    By the same arguments above, we have $W_p(\nu,\hat{\nu})\le\eta(\FC)$ in the case where $\FC$ is bounded.
    Using the assumption that $\big\{\overline{g}_{\BIu}:\BIu\in D(\FC)\big\}\subset\lspan_1(\CG)$,
    we have $W_p(\nu,\hat{\nu})\le\eta(\FC)+\big(\sum_{\BIu\in D(\FC)}\int_{\CY}\overline{g}_{\BIu}\DIFFX{\nu}\big)^{\frac{1}{p}}=\eta(\FC)+\big(\sum_{\BIu\in D(\FC)}\int_{\CY}\overline{g}_{\BIu}\DIFFX{\mu}\big)^{\frac{1}{p}}$ in the case where $\FC$ is not bounded.
    Moreover, by (\ref{eqn:momentset-euclidean-proof-invariance}) and the assumption that $\big\{g_{\BIv}:\BIv\in V(\FC)\big\}\subset\lspan_1(\CG)$, we have $\hat{\nu}(\{\BIv\})=\int_{\CY}g_{\BIv}\DIFFX{\nu}=\int_{\CY}g_{\BIv}\DIFFX{\mu}=\hat{\mu}(\{\BIv\})$ for all $\BIv\in V(\FC)$.
    But since $\support(\hat{\mu})\subseteq V(\FC)$ and $\support(\hat{\nu})\subseteq V(\FC)$, this implies that $\hat{\mu}=\hat{\nu}$. 
    Thus, it holds that $W_p(\mu,\nu)\le W_p(\mu,\hat{\mu})+W_p(\hat{\mu},\hat{\nu})+W_p(\nu,\hat{\nu})\le 2\eta(\FC)$ in the case where $\FC$ is bounded, and $W_p(\mu,\nu)\le W_p(\mu,\hat{\mu})+W_p(\hat{\mu},\hat{\nu})+W_p(\nu,\hat{\nu})\le 2\eta(\FC)+2\left(\sum_{\BIu\in D(\FC)}\int_{\CY}\overline{g}_{\BIu}\DIFFX{\mu}\right)^{\frac{1}{p}}$ in the case where $\FC$ is not bounded.
    The proof is now complete.
\end{proof}

%% Proof of Corollary (number of test functions)
\begin{proof}[Proof of Corollary~\ref{cor:momentset-scalability}]
We will first prove the following inequality:
\begin{align}
    \max_{C\in\FC}\max_{\BIv,\BIv'\in V(C)}\big\{\|\BIv-\BIv'\|_{\infty}\big\}\le \frac{\epsilon}{2C_{\|\cdot\|} d^{1/p}},
    \label{eqn:momentset-scalability-proof-ineq}
\end{align}
which will then lead to 
$\eta(\FC)\le C_{\|\cdot\|}\max_{C\in\FC}\max_{\BIv,\BIv'\in V(C)}\big\{\|\BIv-\BIv'\|_{p}\big\}\le \frac{\epsilon}{2}$.
To prove (\ref{eqn:momentset-scalability-proof-ineq}),
let us suppose for the sake of contradiction that 
there exist $C\in\FC$ and $\BIv=(v_1,\ldots,v_d)^\TRANSP,\BIv'=(v'_1,\ldots,v'_d)^\TRANSP\in V(C)$ 
with $\|\BIv-\BIv'\|_{\infty}>\frac{\epsilon}{2C_{\|\cdot\|} d^{1/p}}$.
Then, in view of the definition of $\widehat{V}$
and the property that
$\frac{\overline{M}_i-\underline{M}_i}{n_i}\le \frac{\epsilon}{2C_{\|\cdot\|} d^{1/p}}$ for $i=1,\ldots,d$,
there exists $\hat{i}\in\{1,\ldots,d\}$ such that 
$|v_{\hat{i}}-v'_{\hat{i}}|\ge \frac{2(\overline{M}_{\hat{i}}-\underline{M}_{\hat{i}})}{n_{\hat{i}}}$.
Let us assume without loss of generality that $v_{\hat{i}}-v'_{\hat{i}}\ge\frac{2(\overline{M}_{\hat{i}}-\underline{M}_{\hat{i}})}{n_{\hat{i}}}$,
and let
$\BIe_{\hat{i}}\in\R^d$ denote the $\hat{i}$-th standard basis vector of $\R^d$.
It subsequently follows from the definition of $\widehat{V}$ that both
$\BIw:=\BIv - \frac{\overline{M}_{\hat{i}}-\underline{M}_{\hat{i}}}{n_{\hat{i}}}\BIe_{\hat{i}}$
and $\BIw':=\BIv' + \frac{\overline{M}_{\hat{i}}-\underline{M}_{\hat{i}}}{n_{\hat{i}}}\BIe_{\hat{i}}$ belong to $\widehat{V}$.
Next, let $\BIb=(b_1,\ldots,b_d)^\TRANSP\in\R^d$ and $r>0$ denote the center and radius of the circum-hypersphere $S_{C}$ of $C$, 
i.e., $S_C=\big\{\BIx\in\R^d: \|\BIx-\BIb\|_2=r\big\}$.
It thus holds that 
$\|\BIv-\BIb\|_2=\|\BIv'-\BIb\|_2=r$.
Moreover, since $\BIw,\BIw'\in V(\FC)$, 
the Delaunay condition guarantees that 
$\|\BIw-\BIb\|_2\ge r$, $\|\BIw'-\BIb\|_2\ge r$.
Combining these properties
yields 
\begin{align*}
    r^2 &\le \Big\|\BIv-\BIb-{\textstyle\frac{\overline{M}_{\hat{i}}-\underline{M}_{\hat{i}}}{n_{\hat{i}}}}\BIe_{\hat{i}}\Big\|_2^2=r^2 + \frac{2(\overline{M}_{\hat{i}}-\underline{M}_{\hat{i}})}{n_{\hat{i}}}(b_{\hat{i}}-v_{\hat{i}}) + \bigg(\frac{\overline{M}_{\hat{i}}-\underline{M}_{\hat{i}}}{n_{\hat{i}}}\bigg)^2,\\
    r^2 &\le \Big\|\BIv'-\BIb+{\textstyle\frac{\overline{M}_{\hat{i}}-\underline{M}_{\hat{i}}}{n_{\hat{i}}}}\BIe_{\hat{i}}\Big\|_2^2=r^2 + \frac{2(\overline{M}_{\hat{i}}-\underline{M}_{\hat{i}})}{n_{\hat{i}}}(v'_{\hat{i}}-b_{\hat{i}}) + \bigg(\frac{\overline{M}_{\hat{i}}-\underline{M}_{\hat{i}}}{n_{\hat{i}}}\bigg)^2.
\end{align*}
Summing up the two inequalities above,
we get $v_{\hat{i}}-v'_{\hat{i}}\le \frac{\overline{M}_{\hat{i}}-\underline{M}_{\hat{i}}}{n_{\hat{i}}}$
which contradicts 
$v_{\hat{i}}-v'_{\hat{i}}\ge \frac{2(\overline{M}_{\hat{i}}-\underline{M}_{\hat{i}})}{n_{\hat{i}}}$.
We have thus proved that 
$\eta(\FC)\le \frac{\epsilon}{2}$.

Consequently, it holds by Proposition~\ref{prop:momentset-simplex-VIFS}
that $\CG_{\mathsf{simp}}(\FC)$ is a VIFS for~$\FC$,
and thus Theorem~\ref{thm:momentset-euclidean}\ref{thms:momentset-euclidean-bounded} guarantees
that $W_p(\mu,\nu)\le 2\eta(\FC)\le \epsilon$ for any $\mu,\nu\in\CP_p\big(\CY;\CG_{\mathsf{simp}}(\FC)\big)$ satisfying 
$\mu\overset{\CG_{\mathsf{simp}}(\FC)}{\scalebox{3.5}[1]{$\sim$}}\nu$. 
Moreover, it holds that
\begin{align*}
    \big|\CG_{\mathsf{simp}}(\FC)\big|=|\widehat{V}|=\prod_{i=1}^d(1+n_i)=\prod_{i=1}^d\Bigg(1+\Bigg\lceil\frac{2(\overline{M}_i-\underline{M}_i)C_{\|\cdot\|} d^{1/p}}{\epsilon}\Bigg\rceil\Bigg).
\end{align*}
The proof is now complete.
\end{proof}

%% Proof of Corollary (Breeden-Litzenberger)
\begin{proof}[Proof of Corollary~\ref{cor:breedenlitzenberger}]
For $i=1,\ldots,d$, let $\CI_i:=\big\{(-\infty,\kappa_{i,0}]$, $[\kappa_{i,0},\kappa_{i,1}],\ldots,[\kappa_{i,n_i-1},\kappa_{i,n_i}],\allowbreak[\kappa_{i,n_i},\infty)\big\}$, and let $\FC:=\big\{I_1\times\cdots\times I_d : I_i\in\CI_i\;\forall 1\le i\le d\big\}$. 
One observes that $\FC$ is a polyhedral cover of $\CY$ consisting of $d$-dimensional cubes with side lengths all equal to~$\beta$,
and hence $\eta(\FC)\le C_{\|\cdot\|}\beta\|\vecone\|_p= C_{\|\cdot\|} d^{1/p}\beta$.
By letting $\beta_1=\cdots=\beta_d=\beta$ in Proposition~\ref{prop:cover-cube} and then multiplying each function in $\CG_0$ by the positive constant $\beta$, it 
holds that $\lspan_1(\CG_p)$ for $\CG_p$ defined in Corollary~\ref{cor:breedenlitzenberger} is identical to $\lspan_1(\CG_p)$ for $\CG_p$ defined in Proposition~\ref{prop:cover-cube}.
Thus, (\ref{eqn:breedenlitzenberger-bound}) follows from Proposition~\ref{prop:cover-cube}\ref{props:cover-cube-unbounded}
and Theorem~\ref{thm:momentset-euclidean}\ref{thms:momentset-euclidean-unbounded}.

To prove the last statement in Corollary~\ref{cor:breedenlitzenberger}, let us fix an arbitrary $\epsilon>0$ and let $\beta=\frac{\epsilon}{4C_{\|\cdot\|} d^{1/p}}$. For $i=1,\ldots,d$, since $\mu_i\in\CP_p(\R)$, it holds by Lebesgue's dominated convergence theorem that there exists $\underline{\kappa}_i\in\R$ such that
\begin{align*}
\int_{\R}\big((\underline{\kappa}_i-x_i)^+\big)^p\DIFFM{\mu_i}{\DIFF x_i}\le \frac{1}{2d}\left(\frac{\epsilon}{4C_{\|\cdot\|}}\right)^p.
\end{align*}
Subsequently, by applying Lebesgue's dominated convergence theorem again, there exists $n_i\in\N$ such that
\begin{align*}
\int_{\R}\big((x_i-\underline{\kappa}_i-n_i\beta)^+\big)^p\DIFFM{\mu_i}{\DIFF x_i}\le \frac{1}{2d}\left(\frac{\epsilon}{4C_{\|\cdot\|}}\right)^p.
\end{align*}
With these choices of $\beta$, $(\underline{\kappa}_i)_{i=1:d}$, and $(n_i)_{i=1:d}$, we have by (\ref{eqn:breedenlitzenberger-bound}) that
\begin{align*}
\specialoverline{W}_{p}(\mu,[\mu]_{\CG_p})&\le 2C_{\|\cdot\|} d^{1/p}\beta+2C_{\|\cdot\|}\Bigg(\sum_{i=1}^d\int_{\R}\big((\kappa_{i,0}-x_i)^+\big)^p+\big((x_i-\kappa_{i,n_i})^+\big)^p\DIFFM{\mu_i}{\DIFF x_i}\Bigg)^{\frac{1}{p}}\le \epsilon.
\end{align*}
The proof is now complete. 
\end{proof}

\bibliographystyle{abbrvnat}
\bibliography{references} 

\end{document}